\documentclass[11pt]{article}
\usepackage[margin=1in]{geometry}

\usepackage{authblk} 
\usepackage{graphicx}
\usepackage{amssymb}
 
\usepackage{amsmath}
\usepackage{enumerate}
\usepackage{amsfonts}
\usepackage[mathcal]{euscript}
\usepackage{amsthm}
\usepackage[all,2cell,import]{xy} \UseAllTwocells
\usepackage{xcolor}
\usepackage{tkz-euclide}
\usepackage{tikz}
	\usetikzlibrary{3d} 

\usepackage{url}
\makeatletter
\def\url@leostyle{%
	\@ifundefined{selectfont}{\def\UrlFont{\sf}}{\def\UrlFont{\small\ttfamily}}}
\makeatother
\usepackage[refpage]{nomencl}

\makenomenclature
\usepackage{makeidx}
\makeindex
\numberwithin{equation}{section}

\usepackage[colorlinks=true, pdfstartview=FitV, linkcolor=black,
			   citecolor=black, urlcolor=black]{hyperref}
\usepackage{bookmark}
\usepackage[titletoc]{appendix} 
\hypersetup{hypertexnames=false} 

          \newcommand{\nc}{\newcommand}
          \nc{\DMO}{\DeclareMathOperator}	
          \nc{\commentout}[1]{}

          \nc{\newnotation}{\nomenclature}
          \nc{\wrap}{\cW}
          \nc{\Tw}{\mathsf{Tw}}
          \nc{\loc}{\mathsf{Loc}}
          \nc{\Top}{\mathcal{T}\!\operatorname{op}}
          \nc{\Diff}{\mathcal{D}\!\operatorname{iff}}
          \DMO{\emb}{Emb}
          \nc{\Ind}{\mathrm{Ind}}
          \nc{\Loc}{\mathsf{Loc}}
          \nc{\Cob}{\mathsf{Cob}}
          \nc{\mul}{\mathsf{Mul}}
          \nc{\fat}{\mathsf{fat}}
          \nc{\cob}{\mathsf{Cob}}
          \nc{\coh}{\mathsf{Coh}}
          \nc{\Liouaut}{\Aut}
          \nc{\Liouauto}{{\Aut^o}}
          \nc{\Liouautb}{\Aut^{b}}
          \nc{\Liouautgr}{\Aut^{gr}}
          \nc{\Liouautgrb}{\Aut^{gr,b}}
          \nc{\idem}{\mathsf{Idem}}
          \nc{\sets}{\mathsf{Sets}}
          \nc{\near}{\mathsf{near}}
          \nc{\sing}{\mathsf{Sing}}
          \nc{\Sing}{\mathsf{Sing}}
          \nc{\perf}{\mathsf{Perf}}
          \nc{\linear}{\mathsf{linear}}
          \nc{\block}{\mathsf{block}}
          \nc{\ssets}{\mathsf{sSets}}
          \nc{\cmpct}{\mathsf{cmpct}}
          \nc{\compact}{\mathsf{cmpct}}
          \nc{\pwrap}{\mathsf{PWrap}}
          \nc{\coder}{\mathsf{Coder}}
          \nc{\bimod}{\mathsf{Bimod}}
          \nc{\grmod}{\mathsf{GrMod}}
          \nc{\Morita}{\mathsf{Morita}}
          \nc{\morita}{\mathsf{Morita}}
          \nc{\spaces}{\mathsf{Spaces}}
          \nc{\posets}{\mathsf{Poset}}
          \nc{\pwrms}{\mathsf{PWrFuk}_{M,S}}
          \nc{\pwrmf}{\mathsf{PWrFuk}_{M,F}}
          \nc{\pwrapmf}{\mathsf{PWrFuk}_{M,F}}
          \nc{\fuk}{\mathsf{Fukaya}}
          \nc{\infwr}{\mathsf{InfWr}}
          \nc{\fukaya}{\mathsf{Fukaya}}
          \nc{\autml}{\mathsf{Aut}_{M,\Lambda}}
          \nc{\fukml}{\mathsf{Fukaya}_{M,\Lambda}}
          \nc{\fukmle}{\mathsf{Fukaya}_{M,\Lambda,\epsilon}}
          \nc{\fukmod}{\wrfukcompact(M)\modules}
          \nc{\lag}{\mathsf{Lag}}
          \nc{\lagm}{\lag_M}
          \nc{\lago}{\lag^o}
          \nc{\lagml}{\lag_{M,\Lambda}} 
          \nc{\lagmle}{\lag_{M,\Lambda,\epsilon}}
          \nc{\Fun}{\mathsf{Fun}}
          \nc{\fun}{\mathsf{Fun}}
          \nc{\vect}{\mathsf{Vect}}
          \nc{\chain}{\mathsf{Chain}}
          \nc{\chainn}{Chain}
          \nc{\wrfuk}{\mathsf{WrFukaya}}
          \nc{\wrfukcompact}{\mathsf{WrFukaya}_{\mathsf{cmpct}}}
          \nc{\pwrfuk}{\mathsf{PWrFukaya}}
          \nc{\inffuk}{\mathsf{InfFuk}}
          \nc{\pwrfukml}{\mathsf{PWrFukaya}_{M,\Lambda}}
          \nc{\inffukml}{\mathsf{InfFuk}_{M,\Lambda}}
          \nc{\nattrans}{\mathsf{NatTrans}}
          \nc{\corres}{\mathsf{Corres}}
          \nc{\fukep}{\fukaya_\Lambda(M,\epsilon)}
          \nc{\fukepop}{\fukaya_\Lambda(M,\epsilon)^{\op}}
          \nc{\lagep}{\lag_\Lambda(M,\epsilon)}
          \DMO{\cyl}{cyl} 
          \nc{\dbcoh}{D^b\mathsf{Coh}}
          \nc{\corr}{\mathsf{Corr}}

          \nc{\cat}{\mathsf{Cat}}
          \nc{\Cat}{\mathsf{Cat}}
          \nc{\ainfty}{\mathsf{A}_\infty}
          \nc{\inftycat}{\mathcal{C}\!\operatorname{at}_\infty}
          \nc{\inftyCat}{\mathcal{C}\!\operatorname{at}_\infty}
          \nc{\inftyGpd}{\mathcal{G}\!\operatorname{pd}_\infty}
          \nc{\Ainftycat}{\mathcal{C}\!\operatorname{at}_{A_\infty}}
          \nc{\dgcat}{\mathcal{C}\!\operatorname{at}_{dg}}
          \nc{\ainftycat}{\mathcal{C}\!\operatorname{at}_{A_\infty}}
          \nc{\stablecat}{\mathcal{C}\!\operatorname{at}_\infty^{\Ex}}

          \DMO{\op}{op}
          \DMO{\sd}{sd} 
          \DMO{\im}{im}
          \DMO{\ev}{ev}
          \DMO{\st}{st}
          \DMO{\stable}{Ex}
          \DMO{\map}{Map}
          \nc{\sdcoll}{\sd^{\coll}}
          \nc{\Excoll}{\Ex^{\coll}}
          \DMO{\inj}{inj}
          \DMO{\fib}{fib}
          \DMO{\conf}{Conf}
          \DMO{\chains}{Chains}
          \DMO{\cochains}{Cochains}
          \DMO{\cone}{Cone}
          \DMO{\Map}{Map}
          \DMO{\ran}{Ran}
          \DMO{\rot}{Rot}
          \DMO{\leg}{Leg}
          \DMO{\imm}{Imm}
          \DMO{\adj}{adj}
          \DMO{\symp}{Symp}
          \DMO{\tree}{Tree}
          \DMO{\cube}{Cube}
          \DMO{\weak}{weak}
          \DMO{\strong}{strong}
          \DMO{\Hoch}{Hoch}
          \DMO{\front}{front}
          \DMO{\flow}{Flow}
          \DMO{\floer}{Floer}
          \DMO{\Maps}{Maps}
          \DMO{\exact}{exact}
          \DMO{\excess}{Excess}
          \DMO{\Decomp}{Decomp}
          \DMO{\decomp}{Decomp}
          \DMO{\collar}{collar}
          \DMO{\yoneda}{Yoneda}
          \DMO{\hamspace}{Ham}
          \DMO{\sympspace}{Symp}
          \DMO{\holomaps}{Holomaps}
          \DMO{\comp}{Comp}
          \DMO{\crit}{Crit}
          \DMO{\test}{{test}}
          \DMO{\triv}{triv}
          \DMO{\sign}{sign}
          \DMO{\topp}{top}
          \DMO{\movie}{movie}
          \DMO{\indx}{Index}
          \DMO{\Break}{Break} 
          \DMO{\zero}{zero} 
          \DMO{\ob}{Ob}
          \DMO{\gr}{Gr} 
          \DMO{\Gr}{Gr} 
          \DMO{\cl}{Cl} 
          \DMO{\grlag}{GrLag}
          \DMO{\Pin}{Pin}
          \DMO{\Graph}{Graph}
          \DMO{\pin}{Pin}
          \DMO{\gap}{Gap}
          \DMO{\Ex}{Ex}
          \DMO{\id}{id}
          \DMO{\End}{End}
          \DMO{\sym}{Sym}
          \DMO{\aut}{Aut}
          \DMO{\Aut}{Aut}
          \DMO{\haut}{hAut}
          \DMO{\hAut}{hAut}
          \DMO{\DK}{DK} 
          \DMO{\poly}{poly} 
          \DMO{\diff}{Diff}
          \DMO{\coll}{coll}
          \DMO{\dist}{dist} 
          \DMO{\coker}{coker} 
          \nc{\kernel}{\ker} 
          \DMO{\sspan}{span}
          \DMO*{\colim}{colim}
          \DMO*{\hocolim}{hocolim}	
          \DMO*{\holim}{holim}
          \DMO{\sk}{sk}

          \DMO{\ho}{ho}
          \DMO{\fin}{fin}
          \DMO{\tor}{Tor}
          \DMO{\ext}{Ext}
          \DMO{\ret}{Ret}
          \DMO{\ham}{Ham}
          \DMO{\con}{con}
          \DMO{\leaf}{leaf}
          \DMO{\supp}{supp}
          \DMO{\edge}{edge}
          \DMO{\edges}{edges}
          \DMO{\Image}{image}
          \DMO{\roots}{roots}
          \DMO{\height}{height}
          \DMO{\finmod}{FinMod}
          \DMO{\leaves}{leaves}
          \DMO{\planar}{planar}
          \DMO{\vertices}{vertices}

          \nc{\lagg}{\lag^{\cG}}
          \nc{\iso}{\mathsf{Iso}}
          \nc{\Set}{\mathsf{Set}}
          \nc{\Ass}{\mathsf{ \bf Ass}}
          \nc{\Mod}{\mathsf{Mod}}
          \nc{\modules}{\mathsf{Mod}}
          \nc{\sset}{\mathsf{sSet}}
          \nc{\liou}{\mathsf{Liou}}
          \nc{\poset}{\mathsf{Poset}}
          \nc{\trno}{T^*\RR^n_{\geq 0}}
          \nc{\spectra}{\mathsf{Spectra}}
          \nc{\tensorfin}{\tensor^{\fin}}
          \nc{\lagptg}{\lag_{pt,pt}^{\cG}}
          \nc{\Fin}{\mathcal{F}\mathsf{in}}
          \nc{\lagnl}{\lag_{N,\Lambda}}
          \nc{\lagmlg}{\lag_{M,\Lambda}^{\cG}}
          \nc{\lagsplit}{\lag^{\mathsf{split}}}
          \nc{\lagktimes}{(\lag^{\dd k})^\times}
          \nc{\lagplanar}{\lag^{\times,\planar}}
          \nc{\Cont}{\text{\rm Cont}}
          \nc{\Ham}{\text{\rm Ham}}
          \nc{\Dev}{\text{\rm Dev}}
          \nc{\Lin}{\text{\rm Lin}}
          \nc{\Int}{\text{\rm Int}}
          \nc{\Hom}{\text{\rm Hom}}
          \nc{\Chord}{\text{\rm Chord}}
          \nc{\nbhd}{\mathcal{N}\text{\rm{bhd}}}

          \nc{\smsh}{\wedge}
          \nc{\un}{\underline}
          \nc{\xto}{\xrightarrow}
          \nc{\xra}{\xto}
          \nc{\tensor}{\otimes}
          \nc{\del}{\partial}
          \nc{\dd}{\diamond}
          \nc{\tri}{\triangle}
          \nc{\bb}{\Box}
          \nc{\into}{\hookrightarrow}
          \nc{\onto}{\twoheadrightarrow}
          \nc{\contains}{\supset}
          \nc{\transverse}{\pitchfork}
          \nc{\uncirc}{\underline{\circ}}

          \nc{\eqn}{\begin{equation}}
          \nc{\eqnn}{\begin{equation}\nonumber}
          \nc{\eqnd}{\end{equation}}
          \nc{\enum}{\begin{enumerate}}
          \nc{\enumd}{\end{enumerate}}
          \nc{\beastar}{\begin{eqnarray*}}
          \nc{\eeastar}{\end{eqnarray*}}

          \def\cA{\mathcal A}\def\cB{\mathcal B}\def\cC{\mathcal C}\def\cD{\mathcal D}
          \def\cE{\mathcal E}\def\cG{\mathcal G}
          \def\cL{\mathcal L}
          \def\cP{\mathcal P}
          
          \def\cW{\mathcal W}

          \def\CC{\mathbb C}\def\DD{\mathbb D}
          \def\GG{\mathbb G}

          \def\RR{\mathbb R}\def\SS{\mathbb S}
          \def\VV{\mathbb V}\def\WW{\mathbb W}
          \def\ZZ{\mathbb Z}

          \def\bL{\mathbf L}

          \def\fC{\mathfrak C}

          \nc{\Euc}{\mathsf{Euc}}
          \nc{\mfld}{\mathsf{Mfld}}
          \nc{\DTop}{\mathsf{DTop}}
          \nc{\simp}{\mathsf{Simp}}
          \nc{\Ainftycatt}{A_\infty Cat}
          \nc{\dgcatt}{dg Cat}
          \nc{\StableCat}{StableCat}
          \nc{\subdivision}{\mathsf{subdiv}}
          \nc{\Kan}{\mathcal{K}\mathsf{an}}

          \nc{\deR}{\operatorname{deR}}

          \theoremstyle{definition}
          \newtheorem{theorem}{Theorem}[section]
          
          \newtheorem{prop}[theorem]{Proposition}
          \newtheorem{lemma}[theorem]{Lemma}
          \newtheorem{warning}[theorem]{Warning}
          \newtheorem{cor}[theorem]{Corollary}
          \newtheorem{corollary}[theorem]{Corollary}
          \newtheorem{construction}[theorem]{Construction}

          \newtheorem{defn}[theorem]{Definition}
          \newtheorem{notation}[theorem]{Notation}
          \newtheorem{example}[theorem]{Example}
          \newtheorem{choice}[theorem]{Choice}
          
          \newtheorem{recollection}[theorem]{Recollection}
          \newtheorem{remark}[theorem]{Remark}
          \newtheorem{figurelabel}[theorem]{Figure} 

          \newtheorem{maintheorem}{Theorem}

\nc{\shrink}{\mathsf{shrink}}
\nc{\subdiv}{\mathsf{sd}}
\nc{\bsd}{\mathsf{bsd}}
\nc{\eqvs}{\mathsf{eq}}
\nc{\fs}{\mathfrak{s}}
\nc{\kihara}{\mathfrak{Kihara}}
\DMO{\collared}{coll}

\nc{\stab}{\mathsf{Stab}}

\nc{\wein}{\mathsf{Wein}}
\nc{\weinlocal}{\mathcal{W}\!\operatorname{ein}^{\dd}}
\nc{\LLiou}{\mathcal{L}\!\operatorname{iou}}
\nc{\lioulocal}{\mathcal{L}\!\operatorname{iou}^{\dd}}
\nc{\LLioustab}{\mathcal{L}\!\operatorname{iou}^{\dd}}
\nc{\flexsub}{\mathfrak{s}}
\nc{\liouvit}{\mathsf{Liou}_{\mathsf{Vit}}}
\nc{\weinvit}{\mathsf{Wein}_{\mathsf{Vit}}}
\nc{\weintop}{\mathsf{Wein}}
\nc{\weinstr}{\mathsf{Wein}_{\mathsf{str}}}
\nc{\lioustrstab}{\mathsf{Liou}^{\dd}_{\mathsf{str}}}
\DMO{\lioudmo}{Liou}
\nc{\lioutop}{\lioudmo^{\Top}}
\nc{\lioudiff}{\lioudmo^{\Diff}}
\nc{\lioucoll}{\lioudmo^{\collared}}
\nc{\embliou}{\emb_{\liou}}
\nc{\emblioucoll}{\emb_{\liou}^{\collared}}
\nc{\emblioudef}{\emb_{\liou}^{\mathcal{D}\!\operatorname{ef}}}
\nc{\weinstab}{\mathsf{Wein}^{\dd}}
\DMO{\sub}{sub}

\nc{\semi}{\mathsf{semi}}
\nc{\semisets}{\ssets_{\semi}}
\nc{\cLstr}{\cL_{\mathsf{str}}}

\DMO{\kan}{Kan}
\DMO{\Pic}{Pic}
\DMO{\finite}{finite}
\DMO{\nonunital}{nu}
\nc{\catsimp}{\mathcal{C}\!\operatorname{at}^{\Delta}}
\nc{\catsimpnon}{\mathcal{C}\!\operatorname{at}^{\Delta}_{\nonunital}}
\nc{\fcnu}{\fC_{\nonunital}}

\nc{\Toppara}{\operatorname{Top}}
\nc{\Topparavec}{\operatorname{Vec}}
\nc{\Toptriv}{\operatorname{Vec}_{\operatorname{Charts}}}

\nc{\liouenr}{\mathsf{Liou}_{\mathsf{enr}}}
\nc{\lioustr}{\mathsf{Liou}_{\mathsf{str}}}
\nc{\lioumoore}{\mathsf{Liou}_{\mathsf{forms}}}
\nc{\lioucollar}{\mathsf{Liou}_{\mathsf{collar}}}
\nc{\eqs}{\mathfrak{eq}}

\nc{\lioudelta}{\mathsf{Liou}_\Delta}
\nc{\lioudeltadef}{\mathsf{Liou}^{\mathcal{D}\!\operatorname{ef}}_\Delta}
\nc{\lioudeltadefstab}{(\lioudeltadef)^{\dd}}
\nc{\liousub}{\mathbb{L}\mathrm{iou}}

\nc{\embtop}{\mathcal{E}\kern-0.22em \operatorname{mb}}

\nc{\strict}{\mathrm{str}}
\nc{\defliou}{\mathcal{D}\!\operatorname{ef}}
\nc{\deflioumoore}{\mathcal{D}\!\operatorname{ef}_{\operatorname{Moore}}}
\nc{\liougen}{\liou^\mathsf{gen}}

\makeatletter
\newcommand{\LiouUnder@}[2]{%
  \vtop{\m@th\ialign{##\cr
    \hfil$#1\operator@font Liou$\hfil\cr
    \noalign{\nointerlineskip\kern1.5\ex@}#2\cr
    \noalign{\nointerlineskip\kern-\ex@}\cr}}%
}
\newcommand{\LiouUnder}{
  \mathop{\mathpalette\LiouUnder@{\rightarrowfill@\textstyle}}\nmlimits@
}
\makeatother

\begin{document}
\title{The $\infty$-category of stabilized Liouville sectors}
\author{Oleg Lazarev, Zachary Sylvan, and Hiro Lee Tanaka}
\maketitle

\begin{abstract}

We prove the surprising fact that the infinity-category of stabilized Liouville sectors is a localization of an ordinary category of stabilized Liouville sectors and strict sectorial embeddings. 

From the perspective of homotopy theory, this result continues a trend of realizing geometrically meaningful mapping spaces through the categorically formal process of localizing. From the symplectic viewpoint,
these results allow us to reduce highly non-trivial coherence results to much simpler verifications. For example, we prove that the wrapped Fukaya category is coherently functorial on stabilized Liouville sectors: Not only does a wrapped category receive a coherent action from stabilized automorphism spaces of a Liouville sector, spaces of sectorial embeddings map to spaces of functors between wrapped categories in a way respecting composition actions. 
As a consequence, we observe that wrapped Floer theory for sectors works in families.
As we will explain, our methods immediately establish such coherence results for most known sectorial invariants, including Lagrangian cobordisms.

As another application, we show that this infinity-category admits a symmetric monoidal structure, given by direct product of underlying sectors. The existence of this structure relies on a computation---familiar from the foundations of factorization homology---that localizations detect certain isotopies of smooth manifolds. Moreover, we characterize the symmetric monoidal structure using a universal property, again producing a simple-as-possible criterion for verifying whether invariants are both continuously and multiplicatively coherent in a compatible way. 

Much of this paper is devoted to the foundational work of rigorously constructing the infinity-category of stabilized Liouville sectors, where any sector $M$ is identified with its stabilization $M \times T^*[0,1]$. We begin with a verification that our construction is indeed an infinity-category (which relies on new constructions in Liouville geometry), then show that our infinity-category computes the ``correct'' mapping spaces (employing new simplicial techniques). A significant ingredient is a demonstration that two notions of sectorial maps give rise to homotopy equivalent spaces of maps.
\end{abstract}

\setcounter{tocdepth}{2}
\tableofcontents

\clearpage
\section{Introduction}
The theory of Liouville manifolds, and Liouville sectors more generally, has lacked a convenient framework for proving fundamental categorical questions regarding symplectic invariants. For example, is the wrapped Fukaya category functor sensitive to the topology and continuous  actions of Liouville embedding spaces?

Prior to the present work, the ``obvious'' way to try and answer the above Floer-theoretic question would have been to write down (or prove the existence of) a coherent system of auxiliary Floer data, and prove that the associated moduli spaces of pseudoholomorphic curves compactify in such a way as to output the desired $A_\infty$-categorical formulas. Such an approach is, while well-established, time-consuming. Not only that, different questions of coherence raise different flavors of difficulties when constructing relevant moduli spaces, requiring ad hoc techniques in each setting---as an example, we invite the reader to ponder how to prove continuous coherences are compatible with multiplicative/Kunneth-type formulas in wrapped Floer theory. A more fundamental scientific issue is the difficulty for an independent user or referee to verify such methods efficiently. Even if most agree ``it can be done,'' the classical techniques also result in a troubling number who agree ``I don't want to write or read the details.'' The works~\cite{oh-tanaka-actions, oh-tanaka-smooth-approximation} exhibited one method of proving coherence results while avoiding the gnarliest of these details. See also Remarks~\ref{remark. another plus} and~\ref{remark. comparing is easier now}.

As it turns out, and as we show in this work and its follower(s), one can avoid such details
and prove certain coherence results with minimal Floer theory---i.e., while constructing only minimally-intricate moduli spaces of holomorphic disks. 

More fundamentally, we construct here an $\infty$-category 
$\lioulocal$ of stabilized  Liouville sectors (Notation~\ref{notation. lioulocal}). This $\infty$-category has excellent formal properties, and the results here open the door for applying powerful $\infty$-categorical machinery to Liouville symplectic geometry.

\begin{remark}
To our readers unfamiliar with $\infty$-categories: One would not lose much intuition by pretending that an $\infty$-category is simply a category wherein the collection of morphisms may form a space (and not just a set). For example, the collection of spaces can form an $\infty$-category in two distinct ways. First, we may  think of the collection of continuous functions as simply a set. Second, one may endow the collection of continuous functions with the compact-open topology and enjoy the fact that composition of continuous functions is continuous.
\end{remark}

\begin{warning}
Throughout this work, the word ``stable'' refers to the geometric process of increasing dimension. It does not refer to stability in the sense of higher algebra or pretriangulated categories. 
For example, 
$\lioulocal$ is not a stable $\infty$-category in the sense of Lurie~\cite{higher-algebra}.
\end{warning}

First, some preliminaries. Let $M$ and $N$ be Liouville sectors (Definition~\ref{defn. sector}). 

\begin{defn}
A smooth, codimension zero, proper embedding $f: M \to N$ is a {\em sectorial embedding} if $f^*\lambda^N = \lambda^M + dh$ for some compactly supported smooth function $h$ on $M$. A sectorial embedding is {\em strict} if $f^*\lambda^N = \lambda^M$.\footnote{This is not the notion of strictness used  in~\cite{abouzaid-seidel-recombination}.} 

We say (a not necessarily strict) $f$ is a {\em sectorial equivalence}, or an equivalence of Liouville sectors, if there exist a (not necessarily strict) sectorial embedding $g: N\to M$ and isotopies $g\circ f \sim \id_N$, $f \circ g \sim \id_M$, through (not necessarily strict)  sectorial embeddings. 
\end{defn}
Sectorial equivalences are a natural class of maps known to leave wrapped Fukaya categories unchanged (see Remark~2.45 of~\cite{last-flexible}, Lemma~3.33 of~\cite{gps}, and Proposition~\ref{prop. wrapped category invariant under sectorial equivalences} below).

Finally, for the sake of exposition, we present an undefined idea: If $M$ and $N$ are Liouville sectors, let
	\eqn\label{eqn. emb informal}
	\embtop^{\dd}(M,N) \approx \{f: M \times T^*[0,1]^k \to N \times T^*[0,1]^{(\dim M + 2k - \dim N)/2},
	\qquad
	k>>0\}
	\eqnd
denote the space of stable (not necessarily strict) sectorial embeddings. 

\begin{warning}\label{warning. stable embeddings undefined}
This space is not well-defined as written above---hence the approximation symbol---because even if $f$ is a sectorial embedding, its stabilization $f \times \id_{T^*[0,1]}$ is typically not a sectorial embedding, as the compact-support condition on $h$ is not preserved. At the same time, it is natural to contemplate a ``stabilized'' mapping space: For Weinstein sectors, the wrapped category is known to be unchanged under stabilization~\cite{gps-descent}. We give a precise formulation of this mapping space in Definition~\ref{defn. stabilized embedding space}.
\end{warning}

\subsection{A coherence result}
In the introduction, we will denote by capital letters (A, B, C, et cetera) what we view to be the central results . 

To begin, we informally state a result that we make formal in Theorem~\ref{theorem. localization}. It guarantees that certain invariants known only to be sensitive to the {\em set} of   strict embeddings can in fact be made sensitive to the {\em space} of stable, not-necessarily-strict embeddings.

\begin{theorem}\label{theorem. informal functoriality}
Let $\WW$ be an invariant of Liouville sectors taking values in an $\infty$-category admitting filtered colimits and assume the following:
\enum[(a)]
	\item\label{item. W is strict functorial} ($\WW$ is functorial with respect to strict inclusions.) To any {\em strict} sectorial embedding $f: M \to N$, one can functorially assign a pushforward map $f_*: \WW(M) \to \WW(N)$.

	\item\label{item. W stabilizes assumption} ($\WW$ stabilizes.) 	$\WW$ is equipped with natural maps $\WW(M) \to \WW(M \times T^*[0,1])$. This allows us to define the ``stabilized'' invariant associated to $M$, which we call $\WW^{\dd}(M)$.\footnote{
	Concretely, $\WW^{\dd}(M)$ is the colimit of the diagram $\WW(M) \to \WW(M \times T^*[0,1]) \to \WW(M \times T^*[0,1]^2) \to \ldots$. This filtered colimit exists by assumption; see also Remark~\ref{remark. filtered colimits}
	}.
	
	\item\label{item. W inverts equivalences assumption} ($\WW$ stably inverts equivalences.) If $f: M \to N$ is a strict sectorial embedding that happens to be an equivalence of Liouville sectors, then the induced map $f_*: \WW^{\dd}(M) \to \WW^{\dd}(N)$ is an equivalence.
\enumd
Then, for any pair of Liouville sectors $M$ and $N$, $\WW$ induces a continuous map from the space $\embtop^{\dd}(M,N)$ to the space of morphisms from $\WW^{\dd}(M)$ to $\WW^{\dd}(N)$ and this map extends the assignments~\eqref{item. W is strict functorial}. Finally, upon varying $M$ and $N$, these maps coherently respect composition.
\end{theorem}

\begin{remark}
In fact, one may replace~\eqref{item. W inverts equivalences assumption} with the assumption that $\WW$ stably inverts any of three other natural notions of equivalences---trivial inclusions, bordered deformation equivalences, or movie inclusions---and the conclusion of the theorem still holds (see Section~\ref{section. equivalent localizations}). We note that, in practice, some invariants may satisfy a stronger assumption than~\eqref{item. W inverts equivalences assumption}: If $f$ is a strict equivalence, the map $f_*: \WW(M) \to \WW(N)$ may be an equivalence even before passing to the stabilized invariants. This is the case when $\WW$ is the wrapped Fukaya category. 
\end{remark}

\begin{remark} \label{remark. filtered colimits}
In many examples of $\infty$-categories, a filtered colimit may be modeled by an increasing union along well-behaved inclusions (of possibly infinitely many objects). So one would not lose much intuition by imagining that $\WW^{\dd}(M)$ is a union $\bigcup_{n \geq 0} \WW(M \times T^*[0,1]^n)$ along the stabilization maps~\eqref{item. W stabilizes assumption}. Finally, most invariants take values in familiar $\infty$-categories---of groups, modules, categories, sets, spaces, chain complexes et cetera---where such filtered colimits exist, so the assumption that $\WW$ takes values in a setting with filtered colimits is not onerous. 
\end{remark}

The wrapped Fukaya category\footnote{This functoriality is proven in~\cite{gps}, and other foundational works include~\cite{abouzaid-loops, abouzaid-seidel-wrapped, sylvan-wrapped}. We also point out the font difference between $\WW$ (which refers to a general invariant) and $\cW$ (which refers to the wrapped Fukaya category).} $\cW(M)$ and the Lagrangian cobordism $\infty$-category\footnote{In the notation of~\cite{nadler-tanaka}, we set $\lag(M) = \lag_\Lambda(M)$ where $\Lambda$ is the skeleton of $M$.}  $\lag(M)$
are both known to satisfy the hypotheses of Theorem~\ref{theorem. informal functoriality}. In fact, for all Liouville sectors, $\lag(M) \simeq \lag^{\dd}(M)$.
So we have the following as an example application:

\begin{theorem}\label{theorem. mapping spaces}
Let $M$ and $N$ be Liouville sectors. Then 
\enum
	\item There are continuous maps
	\eqnn
	\embtop^{\dd}(M,N) \to \fun(\lag(M),\lag(N)),
	\qquad
	\embtop^{\dd}(M,N) \to \fun(\cW^{\dd}(M),\cW^{\dd}(N))
	\eqnd
to the space of functors. 
	\item Further, these maps can be made to respect composition operations homotopy-coherently.
\enumd
\end{theorem}

We emphasize the uniformity of the result. Before this work, one would have had to use Floer-theoretic techniques to prove the wrapped Fukaya category result and Lagrangian cobordism techniques to prove---you get the idea. At the expense of passing to {\em stabilized} invariants, our results allow one to cordon off all non-uniform treatments to the verification of the hypotheses in Theorem~\ref{theorem. informal functoriality}. Note also that one could consider a chain-complex valued, or space-valued, invariant of sectors (rather than a categorical invariant); Theorem~\ref{theorem. informal functoriality} still guarantees continuous maps to the relevant mapping spaces. This uniformity is useful for comparing invariants as well---see Corollary~\ref{cor. natural transformations of invariants}. 

The theorem induces a map of $A_\infty$-spaces from $\embtop^{\dd}(M,M)$ to $\fun(\cW^{\dd}(M),\cW^{\dd}(M))$. (The ``homotopy coherence'' of the theorem guarantees that these are indeed maps respecting the $A_\infty$ structures.) It is then automatic that the higher homotopy groups of the former map to the negative Hochschild cohomology groups of $\cW(M)$. Indeed, observe the isomorphisms
    \begin{align}\nonumber
        \pi_i\fun(\cW(M),\cW(M))
            &\cong
          HH^{1-i}(\cW(M),\cW(M))
          &
          (i > 1),
         \\
         \pi_1\fun(\cW(M),\cW(M))
           & \cong
          (HH^{0}(\cW(M),\cW(M)))^\times
          \nonumber
    \end{align}
where the $\times$ superscript indicates the units of the ring.
Then we have:

\begin{corollary}
    There exists a natural map
    \begin{equation}\nonumber
        \pi_i 	\embtop^{\dd}(M,M)
        \to
          HH^{1-i}(\cW(M),\cW(M))
          \qquad
          (i > 0).
    \end{equation}
And for any other sector $N$, the following diagram 
commutes:
    \begin{equation}
        \nonumber
        \xymatrix{
        \pi_i \embtop^{\dd}(M,N) \times \pi_i \embtop^{\dd}(M,M)
        \ar[r] \ar[d]
        & \pi_i \embtop^{\dd}(M,N) \ar[d] \\
        \pi_i \fun(\cW^{\dd}(M),\cW^{\dd}(N)) \times
        \pi_i \fun(\cW^{\dd}(M),\cW^{\dd}(M))
        \ar[r]
        & 
        \pi_i \fun(\cW^{\dd}(M),\cW^{\dd}(N))
        }
    \end{equation}
\end{corollary}

\begin{remark}\label{remark. no analysis needed for families}
This is a stable analogue of results from~\cite{oh-tanaka-actions}---indeed, to our knowledge, ibid. was the first work which demonstrated that localization techniques in $\infty$-category theory (and not just localization along positive wrappings as in~\cite{gps}) yield geometrically meaningful results in wrapped Floer theory. One major departure from~\cite{oh-tanaka-actions} is the absence of analytical arguments regarding $J$-holomorphic disks in the present work -- indeed, ibid. requires a set-up for Floer theory in families, while here the $\infty$-categorical techniques allow us to avoid them completely, and simply piggy-back on the results of~\cite{gps}.
\end{remark}

\begin{example}[Floer theory works in families]
In particular, suppose that $E \to X$ is a bundle of Liouville sectors over some base space $X$ with fiber $F$ and let $\Aut(F)$ denote the space of (not necessarily strict) Liouville automorphisms. Then the composition
    \eqnn
    X \to B\Aut(F)
    \to B\embtop^{\dd}(F,F)
    \to B\aut(\WW(F))
    \eqnd
where the last map is guaranteed to exist by Theorem~\ref{theorem. informal functoriality}, is an invariant of the bundle. In fact, the last part of the composition gives a way to study the topology of Liouville automorphism spaces.  For example, when $\WW$ is the wrapped Fukaya category, and noting the natural map from $B\Aut(\WW(F))$ to the shift of Hochschild cochains of $\WW(F)$, we obtain a Liouville version of the Seidel homomorphism~\cite{seidel-representation} (for all higher homotopy groups, not just $\pi_1$, of $\Aut(F)$). Put another way, wrapped Floer theory behaves well in families. As in Remark~\ref{remark. no analysis needed for families}, we obtain these results with no need to address compactness of moduli of disks in families, nor to carefully study degenerations of disks in families.

\end{example}

\begin{remark}[Another advantage]
\label{remark. another plus}
Because current definitions of ($\ZZ$-linear) Fukaya categories only depend on counting 0-dimensional components of moduli spaces, as do the definitions of most $A_\infty$ functors, the combinatorics involved in proving coherence results are largely manageable even if tedious. However, for spectrally enriched Floer invariants, which necessarily extract data from all higher-dimensional moduli spaces, far more delicate book-keeping would be necessary to construct families of functors out of families of embeddings. Theorem~\ref{theorem. informal functoriality} allows us to deduce the coherent functoriality of such spectral invariants simply by verifying ``discrete'' or ``embedding-by-embedding'' (as opposed to families of embeddings) properties, which are far easier to establish.
\end{remark}

\begin{remark}
Finally, many known methods for constructing actions of {\em unstabilized} sectorial embedding spaces on the above invariants actually factor through the stabilized embedding spaces anyway.\footnote{While this expectation is intuitive, the hard work is in translating known constructions---see~\cite{oh-tanaka-actions} for example---to the methods used in this paper.} So for invariants satisfying the hypotheses of Theorem~\ref{theorem. informal functoriality}, no information is lost upon passage to stabilization. 
\end{remark}

\subsection{The infinity-categories of Liouville sectors}
Let us now dip our toes into the details---by highlighting our main constructions, and one fundamental property. These details serve the most important philosophical take-away from our work: $\infty$-categorical localizations give rise to meaningful objects in symplectic geometry. 
We refer the reader to Table~\ref{table. the categories of sectors} for a summary of the various $\infty$-categories we construct.

The first construction of our work (Definition~\ref{defn. lioudelta}) is an $\infty$-category 
	\eqnn
	\lioudelta
	\eqnd
whose objects are Liouville sectors, and whose commutative diagrams consist of certain data organized by the combinatorics of barycentric subdivisions. 

\begin{example}
A morphism from $M_0$ to $M_1$ is not the data of a sectorial embedding, but equivalent to the data of two {\em strict} embeddings
	\eqnn
	M_0 \times T^*[0,\epsilon] \to M_{01} \times T^*[0,1] \leftarrow M_1 \times T^*[1-\epsilon,1]
	\eqnd
where the left-pointing arrow is demanded to be an isomorphism to its image, where both embeddings are collared in an appropriate way, and where we specify a Liouville structure on $M_{01} \times T^*[0,1]$ which may not be a direct product of the  Liouville structures of the two factors. The Liouville  structure on $M_{01} \times T^*[0,1]$ will be such that it determines a 1-parameter family of Liouville structures on $M_1$ itself, through deformations by derivatives of compactly supported functions.

Note that the above diagram is in the ``shape'' of the barycentric subdivision of a 1-simplex. 
Likewise, a $k$-simplex in $\lioudelta$ will consist of a diagram of sectors in the shape of a barycentric subdivision of a $k$-simplex. See~\eqref{eqn. 2-simplex isotopy} for the case $k=2$.
\end{example}

\begin{maintheorem}[Theorem~\ref{theorem. lioudelta is an infinity-cat}]
\label{theorem. liou delta is oo-cat}
$\lioudelta$ is an $\infty$-category.
\end{maintheorem}

The construction of $\lioudelta$ is the technical heart of our paper. Indeed, even the verification that $\lioudelta$ is an $\infty$-category is not formal, as one requires results from Liouville geometry (the existence of isotopy extensions and the ability to ``squeeze'' a sector away from its boundary and corner strata). Moreover, we first define $\lioudelta$ as a semisimplicial set; the geometry just mentioned allows us to use a result of Steimle~\cite{steimle} to conclude $\lioudelta$ is an $\infty$-category, and a result of the third author~\cite{tanaka-non-strict} to deduce the existence of certain functors out of $\lioudelta$. 

We then prove:

\begin{maintheorem}[Theorem~\ref{theorem. EmbLiou is homLiouDelta}]
\label{theorem. hom liou delta is hom liou}
Fix two Liouville sectors $M$ and $N$, and let $\embtop_{\liou}(M,N)$ denote the space of sectorial embeddings from $M$ to $N$, topologized to permit control near infinity over compact families. (See Definition~\ref{defn. topology on embliou}.) Then there exists a homotopy equivalence
between the (unstabilized) mapping spaces
	\eqnn
	\embtop_{\liou}(M,N)
	\simeq
	\hom_{\lioudelta}(M,N).
	\eqnd
\end{maintheorem}
The underlying philosophy of the theorem is that families of (not necessarily strict) embeddings may be modeled by certain diagrams in the shapes of barycentric subdivisions (consisting only of strict embeddings). The take-away is that one should think of $\lioudelta$ 
as a different $\infty$-categorical model of the usual space-enriched category of sectors, just as some chain complexes admit equivalent but differently-behaved models.

\begin{remark}
\label{remark. hom versus cat equivalence of lioudelta}
It is possible to construct an equivalence of $\infty$-categories between $\lioudelta$ and the space-enriched category of Liouville sectors, but the combinatorial details are quite involved so we leave this for a separate work. Regardless, we give a ``moral'' argument in Section~\ref{section. lioudelta equivalence sketch} as to why one should expect such an equivalence.
\end{remark}

To obtain a stabilization of $\lioudelta$, we construct in Definition~\ref{defn. lioudeltadef} an auxiliary $\infty$-category 
	\eqnn
	\lioudeltadef
	\eqnd
equivalent to $\lioudelta$ (Theorem~\ref{theorem. lioudelta is lioudeltadef}), but whose diagrams allow for families of Liouville structures that are not necessarily related by compactly supported deformations. The constraints imposed by strictness in the barycentric subdivision diagrams, and the freedoms granted by this larger class of deformations, pay off in the following observation: The stabilization assignment $M \mapsto T^*[0,1] \times M$ is a functor from $\lioudeltadef$ to itself.\footnote{
\label{footnote.direct product}
In contrast, this assignment does not know what to do for non-strict diagrams of Liouville sectors for the reasons raised in Warning~\ref{warning. stable embeddings undefined}. Indeed, much of our simplicial finagling began so we could avoid dealing with this issue. We also note the need to pass to a larger class of deformations (not just compactly supported deformations) for the same reason.} So we can now take the colimit (in our setting, an increasing union)
	\eqn\label{eqn. lioudeltadefstab}
	\lioudeltadefstab := \colim \left(\lioudeltadef \xrightarrow{T^*[0,1] \times - }
	\lioudeltadef \xrightarrow{T^*[0,1] \times - } \ldots \right).
	\eqnd
Of course, because $\lioudelta \simeq \lioudeltadef$, we may also construct an $\infty$-category 
	\eqn\label{eqn. lioudeltastab}
	\lioudelta^{\dd}:= \colim \left(\lioudelta \xrightarrow{T^*[0,1] \times - }
	\lioudelta \xrightarrow{T^*[0,1] \times - } \ldots \right)
	\eqnd
where we have abused notation to let $T^*[0,1] \times -$ denote the induced endofunctor of $\lioudelta$ as well.  $\lioudelta^{\dd}$ and $\lioudeltadefstab$ are equivalent as $\infty$-categories; we call them both the $\infty$-category of stabilized Liouville sectors. (See also Notation~\ref{notation. lioulocal}.)

\begin{remark}\label{remark. stable Emb}
Concretely, an object of $\lioudelta^{\dd}$ is a pair $(M,k)$ where $M$ is a Liouville sector and $k$ is a non-negative integer, up to the relation $(M,k) \sim (T^*[0,1]^N \times M, k+N)$. In particular, there are some objects such as $(T^*\RR^0, k)$ which one might think of as the $k$th ``destabilization'' of $T^*\RR^0$. 

Now let us discuss the morphisms. If there is a morphism from $(M,k)$ to $(M',k')$, it must be that $\dim(M) + 2k' = \dim(M') + 2k$. Thus, $\lioudelta^{\dd}$ is a disjoint union of many subcategories, indexed by the difference between $2k$ and the underlying dimension of $M$. 
By Theorem~\ref{theorem. hom liou delta is hom liou}, we see that a morphism from $M$ to $N$ is, up to contractible choice, the data of a (not necessarily strict) sectorial embedding $ T^*[0,1]^a  \times M \to T^*[0,1]^b \times N$ for $\dim M + a = \dim N + b$. 
\end{remark}

In light of Theorem~\ref{theorem. mapping spaces} and our ability to stabilize~\eqref{eqn. lioudeltastab}, the mapping spaces $\hom_{\lioudelta^{\dd}}(M,N)$ are {\em the} model for $\embtop^{\dd}(M,N)$ we use in this work. This makes precise the informal description in~\eqref{eqn. emb informal}:

\begin{defn}[$\embtop^{\dd}$]\label{defn. stabilized embedding space}
For two Liouville sectors $M$ and $N$, we define
    \eqnn
     \embtop^{\dd}(M,N):= \hom_{\lioudelta^{\dd}}(M,N).
    \eqnd
\end{defn}

 \begin{remark}
Moreover, for any two choices of even integers $a$ and $a'$, the full subcategory of those $(M,k)$ for which $\dim M - 2k = a$ is equivalent to the full subcategory of those $(M',k')$ for which $\dim M' - 2k' = a'$. And, by definition of $\WW^{\dd}$ (in Theorem~\ref{theorem. informal functoriality}), we see that $\WW^{\dd}(M,k) \simeq \WW^{\dd}(M,0)$ regardless of $k$. 
Thus, in studying the invariant $\WW^{\dd}$, one may safely consider only the subcategory of $\lioudelta^{\dd}$ for which $\dim M = 2k$.

However, the subcategories with $\dim M \neq 2k$ will become relevant when we consider symmetric monoidal functors out of $\lioudelta^{\dd}$. See Remark~\ref{remark. monoidal structure sees other k components}
\end{remark}

\begin{table}[tbp]
\centering
\resizebox{\textwidth}{!}{
\begin{tabular}{p{0.13\textwidth}p{0.35\textwidth}p{0.35\textwidth}}
\hline \hline
$\lioustr$&
   {A category of Liouville sectors and  only strict embeddings. Notation~\ref{notation. lioustr}.} 
   & Sees no homotopical information of embedding spaces, but easy to construct functors out of.\\ \hline  \hline
$\lioustrstab$ &
  {A category of stabilized Liouville   sectors and strict embeddings. \eqref{eqn. lioustr defn}} 
  &   Easy to define as a colimit because $\lioustr$ admits an endofunctor   called stabilization.\\ \hline  \hline
 $\lioudelta$ &
  An $\infty$-category of Liouville sectors. Morphism spaces are equivalent to   spaces of sectorial embeddings. Definition~\ref{defn. lioudelta}. &
  Diagrams involve only strict embeddings, but can encode deformations of   Liouville structures by compactly supported exact deformations. \\  
$\lioudeltadef$ &
  An $\infty$-category of Liouville sectors. Morphism spaces are equivalent to   spaces of sectorial embeddings. Definition~\ref{defn. lioudeltadef}.&
  Diagrams involve only strict embeddings, but can encode deformations of   Liouville structures by arbitrary exact deformations. Morphism spaces do not admit obvious composition law.\\ \hline  \hline
$\lioudeltadefstab$ &
  An $\infty$-category of stabilized Liouville sectors. \eqref{eqn. lioudeltadefstab}&
  Easy to define as a colimit because $\lioudeltadef$ admits an endofunctor   called stabilization. \\   
$\lioudelta^{\dd}$ &
  An $\infty$-category of stabilized Liouville sectors. \eqref{eqn. lioudeltastab}&
  Defined indirectly, as $\lioudelta$ does not admit obvious stabilization endofunctor. \\   
$ \lioustrstab[(\eqs^\dd)^{-1}]$ &
  {Localization of $\lioustrstab$ along (strict) sectorial equivalences} 
  & Useful universal property.\\  
$\lioulocal$ &
  {A convenient notation denoting any of the three equivalent $\infty$-categories of stabilized sectors. Notation~\ref{notation. lioulocal}.} 
  & Notation is clean.\\ \hline\hline
$\lioulocal_{/(B' \to B)}$ &
  {A version of $\lioulocal$ where stabilized sectors are equipped with tangential data, and morphisms are maps equipped with compatibilities of tangential data. Construction~\ref{construction. Liou with structures}.} 
  & Necessary to articulate functoriality of various versions of Fukaya and microlocal categories.\\ \hline\hline
\end{tabular}
}
  \caption{Summary of the various categories and $\infty$-categories of Liouville sectors. The horizontal double-lines delineate equivalences classes of $\infty$-categories. For example, $\lioudelta$ and $\lioudeltadef$ are equivalent $\infty$-categories.}
  \label{table. the categories of sectors}
\end{table}

\subsection{Stabilized Liouville sectors via localization}
We now come to the one main property: $\lioudeltadefstab$ is a localization. 

\begin{remark}
To situate the reader, recall that a functor $\cA \to \cB$ is called a localization if it exhibits $\cB$ as a universal (in this case, initial) $\infty$-category that inverts a particular collection of morphisms in $\cA$. In particular, if $F: \cA \to \cC$ is a functor sending all morphisms in this collection to equivalences in $\cC$, then $F$ canonically factors through $\cB$ up to homotopy. A central tool in modern homotopy theory is to begin with $\cA$ a discrete category (i.e., a category in the classical sense~\cite{mac-lane}), and to localize $\cA$ to an $\infty$-category with rich topological information. For example, if one localizes the ordinary category of CW complexes along homotopy equivalences, one recovers the $\infty$-category of CW complexes (and in particular, the homotopy type of all mapping {\em spaces} between CW complexes).\footnote{To see this, apply the results of Dwyer-Kan~\cite{dwyer-kan-function-complexes} to a convenient model category of spaces, and note that Dwyer-Kan localization models $\infty$-categorical localizations---e.g., via Proposition~1.2.1 of~\cite{hinich-dwyer-kan}.} In the present work, we recover the stabilized sectorial embedding spaces using this process (Theorem~\ref{theorem. localization} below).
\end{remark}

So, what is $\lioudeltadefstab$ a localization of? 

\begin{notation}\label{notation. lioustr}\label{notation. lioustrstab}
Let $\lioustr$ be the category (in the classical sense) whose objects are Liouville sectors, and whose morphisms are strict sectorial embeddings. We stabilize to define a category (in the classical sense) 
	\eqn\label{eqn. lioustr defn}
	\lioustrstab := \colim \left( \lioustr \xrightarrow{T^*[0,1] \times -}  \lioustr \xrightarrow{T^*[0,1] \times -} \ldots \right)
	\eqnd 
of stabilized Liouville sectors, where morphism sets are sets of strict stable sectorial embeddings. (The $T^*[0,1]\times -$ functor is well-defined on the nose because $\lioustr$ only contains strict sectorial embeddings.) 
\end{notation}
Based on our models, one can straightforwardly construct a functor (Construction~\ref{construction. j from lioustrstab to lioudeltadefstab})
	\eqn\label{eqn. the functor}
	\lioustrstab \to \lioudeltadefstab.
	\eqnd
And though sectorial equivalences are of course not invertible in $\lioustrstab$, their images become invertible up to homotopy in $\lioudeltadefstab$ (Proposition~\ref{prop. homotopic maps are homotopic}).

Let $\eqs^{\dd} \subset \lioustrstab$ denote the collection of (strict) sectorial equivalences.

\begin{maintheorem}\label{theorem. localization}
The functor~\eqref{eqn. the functor} is a localization of $\lioustrstab$ along (strict) sectorial equivalences. More precisely, the induced map
	\eqnn
	 \lioustrstab[(\eqs^\dd)^{-1}] \to \lioudeltadefstab
	\eqnd
is an equivalence of $\infty$-categories.
\end{maintheorem}

\begin{remark}
Let us explain where stabilization (the $\dd$ superscript) enters the picture. The point is that $\lioudelta$ already incorporates movies (Definition~\ref{defn. movie object}) in defining its morphisms, and receives a natural map from the localization of $\lioustr$ -- however, to exhibit an inverse to this map, we resort to a construction (Construction~\ref{construction. alpha}) where the movies in our morphisms must be interpretable as objects. In particular, we must stabilize our sectors to treat objects and movies on the same footing.
\end{remark}

Combining with Theorem~\ref{theorem. hom liou delta is hom liou}, we see that a category (in the classical sense) localizes to give an $\infty$-category with rich geometric information---e.g., the homotopy type of spaces of stable sectorial embeddings.

\begin{notation}[$\lioulocal$]
\label{notation. lioulocal}
The three $\infty$-categories 
	$ \lioustrstab[(\eqs^\dd)^{-1}]$, $\lioudeltadefstab$, and $\lioudelta^{\dd}$ are hence all equivalent. Each model has their advantages, but notations are all clunky, so we will use the notation
	\eqnn
	\lioulocal
	\eqnd
to denote $ \lioustrstab[(\eqs^\dd)^{-1}]$ (and hence, up to equivalence, $\lioudelta^{\dd}$ and $\lioudeltadefstab$). 
In fact, there are three other localizations all equivalent to $ \lioustrstab[(\eqs^\dd)^{-1}]$ (Theorem~\ref{thm. three localizations of lioustr}) and we will notate them all by $\lioulocal$.

We call $\lioulocal$ the $\infty$-category of stabilized Liouville sectors.
\end{notation}

\subsection{Localization exhibits coherent functoriality}

Assuming Theorem~\ref{theorem. hom liou delta is hom liou} and Theorem~\ref{theorem. localization}, we may now prove Theorem~\ref{theorem. informal functoriality}:

\begin{proof}[Proof of Theorem~\ref{theorem. informal functoriality}.]
Let us first make precise the hypotheses of Theorem~\ref{theorem. informal functoriality}. The first hypothesis states that $\WW$ is in fact a functor $\lioustr \to \cD$, where $\cD$ is some $\infty$-category 
 admitting filtered colimits. 
The third assumption is the data of a diagram
	\eqn\label{eqn. stabilization triangle}
	\xymatrix{
		\lioustr \ar[r]^{\WW} & \cD \\
		\lioustr \ar[u]^{T^*[0,1] \times -} \ar[ur]_{\WW}
	}
	\eqnd
commuting up to specified natural transformation.
By passing to the colimit~\eqref{eqn. lioustr defn}, one obtains a functor $\WW^{\dd}: \lioustrstab \to \cD$. By the second assumption, and by the universal property of localization, Theorem~\ref{theorem. localization} guarantees a factorization up to homotopy
	\eqnn
	\xymatrix{
	\lioustrstab \ar[rr]^-{\WW^{\dd}} \ar[d] && \cD. \\
	\lioulocal \ar@{-->}[urr]^{\exists}
	}
	\eqnd
Because Theorem~\ref{theorem. hom liou delta is hom liou} tells us that the mapping spaces of $\lioulocal$ are precisely the stabilized sectorial embedding spaces (see also Remark~\ref{remark. stable Emb}), this dashed functor produces the desired maps
	\eqnn
	\embtop^{\dd}(M,N) \simeq \hom_{\lioudelta^{\dd}}(M,N) \to \hom_{\cD}(\WW^{\dd}(M),\WW^{\dd}(N)).
	\eqnd
The ``coherence'' with respect to compositions is a vague term whose specific meaning is that the dashed arrow is a functor.
\end{proof}

\begin{proof}[Proof of Theorem~\ref{theorem. mapping spaces}.]
For $\lag$, let $\cD$ be the $\infty$-category of stable $\infty$-categories; then the hypotheses of Theorem~\ref{theorem. informal functoriality} are known. For the wrapped Fukaya category, we note that the work ~\cite{gps} establishes a functor from $\lioustr$ to the ordinary category of $A_\infty$-categories; this in turn maps to the $\infty$-category of $A_\infty$-categories $\Ainftycat$---this is the localization of the ordinary category of $A_\infty$-categories along quasi-equivalences. Thus one may take $\cD = \Ainftycat$ and apply Theorem~\ref{theorem. informal functoriality}. 
\end{proof}

For the record, and now that we have the notation, let us state the precise formulation of Theorem~\ref{theorem. mapping spaces}:

\begin{theorem}[Precise restatement of Theorem~\ref{theorem. mapping spaces}]
Let $\Ainftycat$ denote the $\infty$-category of $A_\infty$-categories.  Then the wrapped Fukaya category construction, along with the functoriality established in~\cite{gps}, induces a functor of $\infty$-categories
	\eqnn
	\lioulocal \to \Ainftycat
	\eqnd
which sends a Weinstein sector $M$ to $\cW(M)$ and a Liouville sector more generally to $\cW^{\dd}(M)$.
Likewise, $\lag$ defines a functor from $\lioulocal$ to the $\infty$-category of stable $\infty$-categories.
\end{theorem}

\begin{remark}
Of course, there are natural maps from the unstabilized embedding spaces $\embtop(M,N)$ to the stabilized embedding spaces $\embtop^{\dd}(M,N)$. So one would expect Theorem~\ref{theorem. informal functoriality} to imply that, for invariants satisfying the hypotheses of the theorem, any ``naturally'' continuous extension of $\WW$ to an $\infty$-category of unstabilized Liouville sectors will actually factor through the stabilization, and in particular only be sensitive to stabilized embedding spaces. We refer to this as an expectation in the present work, rather than a theorem, because we do not extend Theorem~\ref{theorem. hom liou delta is hom liou} to the level of $\infty$-categories. (The theorem only equates mapping spaces, without checking compositional compatibilities.) See also Remark~\ref{remark. hom versus cat equivalence of lioudelta}. 
\end{remark}

\begin{corollary}[of Theorem~\ref{theorem. localization}]\label{cor. natural transformations of invariants}
Let $\WW$ and $\VV$ be functors $\lioustr \to \cD$ equipped with stabilizing natural transformations $\WW(-) \to \WW(- \times T^*[0,1])$ and $\VV(-) \to \VV(- \times T^*[0,1])$. Further assume that one is supplied with a natural transformation $\eta: \WW \to \VV$ compatible with stabilization. Concretely, $\eta$ is the data of a map $\Delta^2 \times \Delta^1 \to \inftycat$ where $\Delta^2 \times \{0\}$ is the diagram~\eqref{eqn. stabilization triangle}, and likewise $\Delta^2 \times \{1\}$ is the corresponding diagram for $\VV$.

Then there is a natural transformation
	$
	\eta^{\dd} : \WW^{\dd} \to \VV^{\dd}
	$
between functors from $\lioulocal$ to $\cD$. In particular, the diagram of mapping spaces
	\eqnn
	\xymatrix{
	\embtop^{\dd}(M,N) \ar[r] \ar[d]
	& \hom_{\cD}(\WW^{\dd}(M),\WW^{\dd}(N)) \ar[d]^{(\eta_N)_*} \\
	\hom_{\cD}(\VV^{\dd}(M),\VV^{\dd}(N)) \ar[r]^{(\eta_M)^*}
	& \hom_{\cD}(\WW^{\dd}(M),\VV^{\dd}(N))
	}
\eqnd
commutes up to homotopy. In particular, if $\cD$ is tensored over spaces this implies a compatibility of module actions:
    \eqnn
    \xymatrix{
    \embtop^{\dd}(M,N)
    \tensor \WW^{\dd}(M) \ar[r] \ar[d]^{\eta_M} & \WW^{\dd}(N) \ar[d]^{\eta_N}	 \\
    \embtop^{\dd}(M,N)
    \tensor \VV^{\dd}(M) \ar[r]  & \VV^{\dd}(N) 
    }
    \eqnd
If $\eta: \WW \to \VV$ is further a natural equivalence of functors from $\lioustr$ to $\cD$, then $\eta^{\dd}$ is a natural equivalence of functors from $\lioulocal$ to $\cD$.
\end{corollary}

\begin{remark}\label{remark. comparing is easier now}
Again, the power is in only having to check properties of $\eta$ that avoid mention of families of embeddings. As example applications: (I) If one has a $G$-action on $M \in \lioulocal$ by some topological group $G$, one sees that the map $\WW^{\dd}(M) \to \VV^{\dd}(M)$ is $G$-equivariant (by only checking discrete properties of $\eta$). (II) If one has several models for spectrally enriched wrapped Fukaya categories, one need only check that they are equivalent as functors out of $\lioustr$ to conclude on may compute mapping-space information about one model in terms of mapping-space information about the other.
\end{remark}

\subsection{Multiplicative coherence}
So far, we have only discussed the ``continuous'' coherence of various invariants (i.e., whether invariants also map mapping spaces coherently to each other). We now turn our attention to multiplicative coherence (i.e., how invariants behave with respect to direct products of sectors---examples include Kunneth formulas).

Let us first make the simple observation that $\lioustr$ is symmetric monoidal under direct product.  We lift this structure through the following theorem.

\begin{maintheorem}\label{theorem. lioustab is symmetric monoidal}
The functor $\lioustr \to \lioulocal$ induces a symmetric monoidal structure on $\lioulocal$. This symmetric monoidal structure
\enum
	\item Acts on objects by direct product of Liouville sectors, and
	\item Is universal for symmetric monoidal functors out of $\lioustr$ that invert sectorial equivalences, and that send $T^*[0,1]$ to a monoidally invertible object.
\enumd
\end{maintheorem}

\begin{remark}
 
Let $\cD$ be a symmetric monoidal $\infty$-category with monoidal product $\tensor_\cD$. Recall that an object $X$ is called {\em (monoidally) invertible} if there exists an object $Y$ for which $X \tensor_{\cD} Y$ is equivalent to the monoidal unit $1_{\cD}$. 

As an example, if $\cD$ is the $\infty$-category of chain complexes over $R$ with $\tensor_{\cD}$ (a model for) the derived tensor product, then a chain complex equivalent to $R[n]$---$R$ concentrated in a single degree---is an invertible object with inverse $R[-n]$. In particular, $R[n]$ for $n \neq 0$ is an invertible object that is not equivalent to the monoidal unit $R$.
\end{remark}

\begin{remark}
 
The proof of Theorem~\ref{theorem. lioustab is symmetric monoidal} is not formal. The proof hinges on a non-trivial computation relating the space of smooth self-embeddings of $[0,1]^3$ to certain mapping spaces in a localization of the category of compact manifolds with corners. (See Lemma~\ref{lemma. main monoidal lemma} and Remark~\ref{remark. the least formal part of monoidal structure}.) Such results---relating the categorically formal process of localization to meaningful mapping spaces---are familiar from the foundations of factorization homology for manifolds~\cite{aft-1,aft-2}. Indeed, we make use of a proof strategy from Lurie's work~\cite{higher-algebra} (where factorization homology is referred to as topological chiral homology).

After this computation, however, the proof of Theorem~\ref{theorem. lioustab is symmetric monoidal} comes down to the simple geometric fact that the $(123)$ permutation of the coordinates of $[0,1]^3$ is isotopic to the identity morphism.
\end{remark}

\begin{remark}
The first part of Theorem~\ref{theorem. lioustab is symmetric monoidal} would, to put it mildly, be a pain in the backside to prove ``by hand.'' To see why, the reader may appreciate that the direct product of two sectorial embeddings is not a sectorial embedding (as Liouville-flow-equivariance near infinity is violated upon taking direct product); this issue was raised already in Warning~\ref{warning. stable embeddings undefined}. The standard techniques to create a symmetric monoidal structure in such a setting would (i) construct collections of coCartesian fibrations whose fibers involve spaces of deformations of $f_1 \times f_2$ to honest sectorial embeddings, and (ii) lose the attention of many potential readers. We avoid (i) (and we hope (ii) has not happened yet).

In contrast, Theorem~\ref{theorem. localization}
allows us to prove Theorem~\ref{theorem. lioustab is symmetric monoidal} without contemplating spaces of ways to deform products of sectorial embeddings.

Finally, it is one thing to produce a symmetric monoidal structure, and it is another to characterize it by a universal property. This is the other appealing feature of Theorem~\ref{theorem. lioustab is symmetric monoidal}.
\end{remark}

\begin{example}
Theorem~\ref{theorem. lioustab is symmetric monoidal} immediately implies that the space of stable sectorial self-embeddings of a point---i.e., of $T^*[0,1]^\infty$---is an $E_\infty$-algebra. (One could also prove this using a standard Eckmann-Hilton argument, but such an argument would get hairy quickly because one must unpack and utilize a particular model of stabilized mapping spaces.)

We remark in passing that this $E_\infty$-space is not group-like, and hence does not define a spectrum without passing to its units. (To see the space is not group-like, one simply embeds an exact Lagrangian copy of $T^*_0\RR^N$ as an interesting object in the Fukaya category of $T^*[0,1]^N$, for example, as a zero object; the endofunctor resulting from the induced sectorial embedding is not an equivalence.) This is in contrast to the classical analogue: The space of smooth self-embeddings of $[0,1]^\infty$ is homotopy equivalent to the infinite orthogonal group $O = O(\infty)$ (Proposition~\ref{prop. O_n}) which is group-like.
\end{example}

\begin{remark}
\label{remark. monoidal structure sees other k components}
 
We gave a description of the objects of $\lioulocal$ (equivalently, of $\lioudelta^{\dd}$) in Remark~\ref{remark. stable Emb}. 
Let us remark why the objects with $2k \neq \dim M$ become crucial for understanding Theorem~\ref{theorem. lioustab is symmetric monoidal}.

First, the symmetric monoidal structure on objects is given by $(M,k) \tensor (M',k') \simeq (M \tensor M', k + k')$ where $M \tensor M'$ is the manifold $M \times M'$ equipped with the product Liouville structure (Definition~\ref{defn. products of sectors}).

Note that via the inclusion $\lioustr \to \lioulocal$, the sector $T^*[0,1]$ is identified with the object $(T^*[0,1],0)$. This object is not the monoidal unit. (The object $(T^*[0,1],1) \sim (\RR^0,0)$ is the unit.) But it is invertible, with inverse $(\RR^0,1)$, as
	\eqnn
	(\RR^0,1) \tensor (T^*[0,1],0) 
	\simeq
	(\RR^0 \tensor T^*[0,1],1 + 0)
	\simeq
	(T^*[0,1],1)
	\simeq
	(\RR^0,0).
	\eqnd
One sees that stabilization thus renders the operation $-\times T^*[0,1]$ invertible, in a symmetric monoidal fashion.
\end{remark}

Here is a corollary of Theorem~\ref{theorem. lioustab is symmetric monoidal}. It is a symmetric monoidal version of Theorem~\ref{theorem. informal functoriality}.

\begin{theorem}\label{theorem. symmetric monoidal functors}
Let $\cD$ be a symmetric monoidal $\infty$-category and let $\WW: \lioustr \to \cD$ be a symmetric monoidal functor sending $T^*[0,1]$ to a monoidally invertible object of $\cD$, and inverting sectorial equivalences. Then $\WW $ canonically extends to a symmetric monoidal functor $\lioulocal \to \cD$.
\end{theorem}

\begin{remark}
Suppose we have a symmetric monoidal functor $\WW$ from $\lioustr$ to some symmetric monoidal $\infty$-category $\cD$, and that $\WW$ sends $T^*[0,1]$ to an invertible object, while sending sectorial equivalences to equivalences in $\cD$. By Theorem~\ref{theorem. symmetric monoidal functors}, we are guaranteed a symmetric monoidal extension of $\WW$, which we will denote by $\widetilde{\WW}: \lioulocal \to \cD$. Let us also denote $\WW(T^*[0,1])$ by $\cL$ (to invoke a line bundle---a prototypical example of an invertible non-unit).

The object $T^*[0,1]$ of $\lioustr$ is then identified with the object $(T^*[0,1],0)$ of $\lioulocal$, so $\widetilde{\WW}(T^*[0,1],0) \simeq \cL$. More generally, we have that $\widetilde{\WW}(T^*[0,1],k) \simeq \cL^{\tensor_\cD 1-k}$, where the negative tensor powers make sense because $\cL$ is invertible. Even more generally, for any Liouville sector $M$, we have that $\widetilde{\WW}(M,k) \simeq \WW(M) \tensor_\cD \cL^{\tensor_\cD {-k}}$. Note that these equivalences are specified by the symmetric monoidal structure of the functor $\widetilde{\WW}$.

Thus, in contrast with the outcome of Theorem~\ref{theorem. informal functoriality}, the data of $\widetilde{\WW}$ is not formally determined by  its restriction to the full subcategory of $\lioulocal$ with $\dim M - 2k = 0$. (See Remark~\ref{remark. stable Emb}.)
\end{remark}

\begin{remark}
In contrast with Theorem~\ref{theorem. informal functoriality}, there is no need to assume that $\cD$ admits filtered colimits, and there is no assumption on $\WW$ admitting a stabilization natural transformation $\WW(-) \to \WW(- \times T^*[0,1])$.

Indeed, there is a symmetry-breaking in the literature when the natural transformation $\WW(M) \to \WW(M \times T^*[0,1])$, is constructed; in the wrapped Fukaya category context, one typically chooses a cotangent fiber $T^*_p[0,1]$ of $[0,1]$ and (on objects) the natural transformation is defined as $L \mapsto L \times T^*_p[0,1]$. Meanwhile, the stabilization of Theorem~\ref{theorem. lioustab is symmetric monoidal} does not rely on such symmetry-breaking data; instead, the symmetric monoidal structure, together with a check of a {\em property} (that $T^*[0,1]$ is sent to a monoidally invertible object) guarantees the existence of a symmetric monoidal functor from the stabilized localization.

Finally, we note the utility of Theorem~\ref{theorem. lioustab is symmetric monoidal} and versatility of $\lioulocal$. 
Even when the invariant $\WW(T^*[0,1])$ is not the monoidal unit (hence when $\WW(M \times T^*[0,1])$ need not necessarily equal $\WW(M)$) one can produce a meaningful multiplicative invariant. A simple example is the functor to chain complexes sending $M$ to its compactly supported and supported-away-from-the-boundary deRham forms, which sends $T^*[0,1]$ to a complex equivalent to $\RR$ concentrated in cohomological degree 2. In particular, while invariants of the type in Theorem~\ref{theorem. informal functoriality} are insensitive to the underlying dimension of a sector $M$ in $\lioustr$,  invariants arising in Theorem~\ref{theorem. symmetric monoidal functors} may still be sensitive to dimension.
\end{remark}

Unlike Theorem~\ref{theorem. mapping spaces}, the verifications that $\lag$ and $\cW$ may be promoted to be symmetric monoidal are yet to be written in the literature. For $\lag$, the problem is still wide open, even for the Weinstein case. For $\cW$, the Kunneth formula is proven for wrapped categories of Weinsteins, while a proof of the full symmetric monoidal structure  (even for Weinsteins) is not yet recorded, though the details are not expected to be onerous. 
Regardless, the reader may appreciate that lifting the {\em discrete} functors out of $\lioustr$ to be symmetric monoidal is far easier than trying to write, by hand, a compatibility between continuous coherences and multiplicative coherences. Theorem~\ref{theorem. symmetric monoidal functors} eliminates the necessity for writing down such a compatibility.

\begin{cor}
Let $\WW: \lioustr\to \cD$ be a symmetric monoidal functor satisfying the hypotheses of Theorem~\ref{theorem. symmetric monoidal functors}. Then whenever $M$ is an $E_n$-algebra in $\lioulocal$, $\WW(M)$ is an $E_n$-algebra in $\cD$.
\end{cor}

\begin{remark}
In particular, if $\WW$ lands in $\infty$-categories, any $E_n$-algebra $M$ in stabilized Liouville sectors induces an $E_{n+1}$-algebra structure on endomorphisms of the unit of $\WW(M)$. At present, we know no way of producing algebra objects in stabilized Liouville sectors, but in~\cite{last-flexible}, we show that upon localizing with respect to another natural class of sectorial maps (the {\em subcritical} maps), one can find a plethora of $E_\infty$-algebras, and these conjecturally represent symmetric monoidal localizations of the stable homotopy category. 
\end{remark}

\subsection{Results for Liouville sectors with tangential structures}
Finally, by now we know that we should equip our Liouville sectors
with
tangential structures to, for example, construct invariants over different base rings. (Without a trivialization of $\det(TM)^{\tensor 2}$, for example, $\cW(M)$ is not $\ZZ$-graded.) We prove in this work that  all the results stated so far have analogues in the setting where our sectors are equipped with tangential structures. Let us make this precise.

Fix $E_\infty$-spaces $B'$ and $B$, and fix $E_\infty$ maps $B' \to B$ and $BU \to B$. Then for any Liouville sector $M$, by virtue of its natural map to $BU$ classifying the compatible-with-symplectic-form complex tangent bundle, we may ask for a $B'$-reduction of its $B$-tangential structure. 

\begin{example}
For example, if $BU \to BU(1) = B$ classifies $\det^2$ and $B'$ is a point, a $B'$-reduction amounts to a trivialization of the stable $\det^2$-bundle.
\end{example}

We let $(\lioustrstab)_{/(B' \to B)}$ denote the $\infty$-category whose objects are Liouville sectors equipped with $B'$-reductions of $B$-structures, and where morphisms are pairs $(f, h)$ where $f$ is a strict sectorial inclusion, and $h$ is a homotopy of $B'$-reductions from the reduction on the domain to the reduction pulled back along $f$.
We let $\eqs_{/B' \to B} \subset (\lioustrstab)_{/(B' \to B)}$ denote those $(f,h)$ where $f$ is a sectorial equivalence.

Finally, we let $\lioulocal_{/(B' \to B)}$ denote the $\infty$-category of stabilized Liouville sectors with reductions of $B$-structures to $B'$. 

\begin{maintheorem}\label{maintheorem. B structures}
The map from the localization
	\eqnn
	(\lioustrstab)_{/(B' \to B)}[\eqs_{/B' \to B}^{-1}] \to \lioulocal_{/(B' \to B)}	
	\eqnd 
is an equivalence of $\infty$-categories. 
\end{maintheorem}

We have an induced symmetric monoidal structure on $\lioulocal_{/(B' \to B)}$ characterizable by two different universal properties---one induced from the localization, and the other by virtue of being a pullback of symmetric monoidal $\infty$-categories. See Section~\ref{section. tangential}.

\subsection{The geometric insights}
This paper emerged as the authors were trying to create the correct formalism for articulating a {\em symmetric monoidal} localization process for Liouville sectors. 
(Indeed, the foundations from the present work allow for the symmetric monoidal localizations---for categories of sectors, and conjecturally of wrapped Fukaya categories themselves---in~\cite{last-flexible}.) 

While we have so far conveyed the $\infty$-categorical results, let us now speak to the new geometric insights we incorporate in the present work. The categorical results were not at all obligated to exist based on the geometry known to us when we first began this project.

The first innocent observation is that while stabilization $M \mapsto M \times T^*[0,1]$ is easy to define for objects (sectors), it does not preserve the standard class of morphisms (sectorial embeddings). 
At first it seemed natural to thus create an $\infty$-category equipped with data of smooth isotopies through sectorial embeddings, to deform {\em morphisms} to render them sectorial after stabilizing.

However, also floating in the air was the knowledge that carrying the data of a deformation of {\em Liouville structures} (on objects) was natural for both Fukaya-categorical and symplectic reasons.

By Moser's Trick, it seems precisely articulable that carrying auxiliary data of deformations of Liouville structures is somehow equivalent to carrying auxiliary data of deforming the morphisms themselves. We do not explore this equivalence, but suffice it to say that---faced with the fork in the road of deforming morphisms or deforming Liouville structures---we settled on carrying deformations of Liouville structures as the $\infty$-categorical data. This is how movies, and barycentric subdivision diagrams, came about. Each object $M \times T^*\Delta^k$ at a barycenter encodes a $\Delta^k$-family of Liouville structures on $M$.

Then, to preserve closure under direct products, it was natural to consider only maps that were {\em strict} sectorial embeddings, then identify a class of deformations closed under products. This class had to be made surprisingly large to accommodate both useful examples and closure under products. (See Definition~\ref{defn. various deformation embeddings}---in particular, our deformations of Liouville structures are not through only compactly supported deformations.) The key theorem that justifies this choice of class is that the space of strict-morphisms-equipped-with-this-class-of-deformations (where the deformations need not be compactly supported)
is homotopy equivalent to the usual  space of all sectorial embeddings 
(which respect the Liouville form up to $d$ of a {\em compactly} supported function).
See Theorem~\ref{theorem. embedding spaces are deformation embedding spaces} and the equivalences of~\eqref{eqn. diagram of embedding spaces}. 
In other words, up to homotopy equivalence, we can recover the ``standard'' morphism spaces for Liouville sectors. Carrying around deformations does not change the homotopy type of morphism spaces.

Then, we had to prove some generalizations of standard facts from the differential topology of manifolds with corners: isotopy extension, the ability to shrink a Liouville sector to have image away from its own boundary, and the ability to create a single (higher-dimensional) object for which a single map encodes a family of maps for lower-dimensional objects. 
We call these objects movies (Section~\ref{section.movies}), and they generalize a construction we learned from~\cite{eliashberg-revisited}.

\subsection{Later works and other future directions}
\label{section. later works}

We will show in later works that $\lioulocal$ has finite colimits. 

\begin{remark}
 Note that $\lioudelta$ (i.e., before stabilizing) does not admit all colimits, for the same reason that the category of sets with injections does not. For example, if $X$ and $Y$ are both non-empty and zero-dimensional sectors, there is no categorical coproduct $X \coprod Y$. 
\end{remark}

 As we will show in later works, direct product is compatible with finite colimits in each variable. So $\lioulocal$ can be Ind-completed to be a presentably symmetric monoidal $\infty$-category. Further, building on results of~\cite{gps-descent} one can show that the wrapped Fukaya category functor preserves all colimits.  By the adjoint functor theorem, the wrapped Fukaya category functor admits a right adjoint. 
In particular, to any stable $\infty$-category, we can canonically associate an Ind-stabilized-Liouville-sector. 

Let us note that neither the unit nor the counit of this adjunction is expected to be an equivalence. Regardless, the adjunction shows that for certain Ind-Liouville sectors, the space of morphisms into them may be computed purely algebraically as a space of functors.
The ability to even state the ideas in the previous paragraphs precisely is, we note, another outcome of the present work.

Another direction of interest is the articulation of the Weinstein version of the tools we use here. Let $\weinlocal \subset \lioulocal$ denote the full subcategory of those (stabilized) sectors admitting Weinstein structures.\footnote{We do not define in this paper any notion of maps between Weinstein sectors that are ``compatible'' with Weinstein structures; for our purposes, a map of Liouville sectors-that-happen-to-admit-Weinstein-structures is simply a (not necessarily strict) sectorial embedding.}
This is of course a natural $\infty$-category to study. 
However, we do not characterize $\weinlocal$ as a localization of anything, and in particular, we do not give a proof that invariants defined only for Weinstein sectors (and not for Liouville sectors generally) induce compatible maps from stabilized mapping spaces. A notable invariant which, at present, is only defined for Weinstein sectors is the microlocal category $\mathfrak{Sh}(M)$ of Nadler-Shende~\cite{nadler-shende}.

Regardless, it is straightforward to construct a Weinstein version of $\lioudelta$---an $\infty$-category whose simplices consist of barycentric subdivision diagrams of {\em Weinstein} sectors with strict inclusions. In a previous draft of this work, we expressed our belief that morphism spaces arising from $\wein_\Delta$ are the only candidate we know for defining a rich ``space of Weinstein morphisms.'' However, work in progress seems to show that this candidate space is homotopy equivalent to a fully faithful space of Liouville embeddings (i.e., a subspace consisting of certain connected components of the Liouville embedding spaces.) Moreover, we expect the methods of this work carry over to show that $\wein_\Delta^{\dd}$ is a localization of $\wein^{\dd}$ along Weinstein equivalences. (A Weinstein equivalence is a sectorial embedding $f: M \to N$ admitting an inverse-up-to-Weinstein-isotopy $g: N\to M$. More precisely, the isotopies $g \circ f \sim \id$ and $f \circ g \sim \id$ have movies---Example~\ref{example. isotopies are 2 simplices}--- that admit Weinstein structures.)
Thus, for every pair of objects $X,Y$, it seems that the natural map $\hom_{\wein^{\dd}_\Delta}(X,Y) \into \hom_{\lioulocal}(X,Y)$ is an inclusion of connected components, and that invariants like those of Nadler-Shende receive actions from Liouville embedding spaces whose elements are Liouville-homotopic to maps respecting Weinstein structures.

It seems to remain an important open problem to construct an $\infty$-category with meaningful spaces of ``Weinstein morphisms'' (of which no homotopically useful definition exists at present). A ``correct'' definition is delicate, analogous to the way that one must carefully define what we mean by families of handle decompositions (Morse functions) to create invariants sensitive to simple homotopy equivalences and K-theoretic data.

\subsection{This paper is \texorpdfstring{$\infty$}{infinity}-categorical in most of its tools}
While we hope the results of this work will appeal to both homotopy theorists and symplectic geometers, we remark that we rely heavily on $\infty$-categorical machinery. 

\subsection{Acknowledgments}
 We would like to thank Yasha Eliashberg for very helpful conversations. 
	The first author was supported by the NSF postdoctoral fellowship, Award \#1705128. 
	The first and second authors were partially supported by the Simons Foundation through grant \#385573, the Simons Collaboration on Homological Mirror Symmetry. 
	The third author was supported by a Texas State University
	 Research Enhancement Program grant, the Texas State University Presidential Seminar and Valero Awards, and by an NSF CAREER Grant under Award Number 2044557; he would also like to thank Rune Haugseng and Jacob Lurie for helpful communications.
	
\clearpage
\section{Some Liouville geometry}

Here we collect some facts about Liouville geometry and establish notation, conventions, and definitions along the way. The main results original to this work are:
\begin{itemize}
\item A movie construction, which provides a way to construct a single (higher-dimensional) Liouville sector encoding a family of Liouville structures on a manifold $M$ (Proposition~\ref{prop. movies are sectors}). This is a higher-dimensional, and minor, generalization of a construction of Eliashberg~\cite{eliashberg-revisited}.
\item A ``shrinking isotopy'' existence theorem assuring that we may always find a self-equivalence of a Liouville sector whose image is bounded away from the sectorial boundary (Section~\ref{section. shrinking}).
\item An isotopy extension theorem for Liouville sectors (Proposition~\ref{prop. isotopy extension}), and
\item That the space of ``embeddings equipped with deformation data rendering the embeddings strict'' is homotopy equivalent to the space of sectorial embeddings (Section~\ref{section. different deformation spaces}). 
\end{itemize}

\subsection{Basic definitions}
	
	\subsubsection{Liouville sectors}
	
	The following agrees with the notion from~\cite{gps-descent} of a \emph{straightened} Liouville sector with corners. Because sectors will always be straightened in our work, we remove the adjective.
	
	\newenvironment{sector-properties}{
	  \renewcommand*{\theenumi}{(S\arabic{enumi})}
	  \renewcommand*{\labelenumi}{(S\arabic{enumi})}
	  \enumerate
	}{
	  \endenumerate
}

	\begin{defn}[Liouville sector]
	\label{defn. sector}
Fix an exact symplectic manifold $( M, \omega=d\lambda )$ possibly with corners, together with the data, for each $ x \in \partial^i M $ in a codimension $i$ corner $i \geq 1$, of a 
		neighborhood $\nbhd(x)$ inside $M$ and a codimension-preserving symplectic submersion
		\eqn\label{eqn. corner projections}
		\pi_x \colon \nbhd(x) \to T^*[0, 1)^i.
		\eqnd
		We say this collection of data is a {\em Liouville sector} if it satisfies the following:
		\begin{sector-properties}
			\item\label{item. finite type} $\lambda$ has finite type. This means that $M$ admits a proper, smooth function $M \to \RR_{\geq 0}$ which, outside some compact subset of $M$, is linear with respect to the Liouville flow of $M$.
			\item\label{item. splitting lambda} Each $\pi_x$ is flat, and $\lambda$ is split with respect to $\{\pi_x\}_{x \in \del M}$. This means: Using the flat splitting of $\nbhd(x)$ as a product of 
			a fiber of $\pi_x$ and the image of $\pi_x$, $\lambda|_{\nbhd(x)} =  \lambda^{\text{fiber}} + \pi_x^*\mathbf{p}d\mathbf{q}$. (See Remark~\ref{remark. splitting symplectic submersions}.)
			\item\label{item. projections compatible} If $y \in \partial^j M \cap \nbhd(x)$ with $j \leq i$, then on the overlap we have $\pi_y = \pi_{yx} \circ \pi_x$, where $\pi_{yx} \colon T^*[0, 1)^i \to T^*[0,1)^j$ is a projection to $j$ components (not necessarily respecting the order of coordinates).
		\end{sector-properties}
	We will usually denote a Liouville sector by $(M, \lambda)$ or just $M$, leaving the family of projections $\{\pi_x\}$ implicit. 
	\end{defn}

	\begin{remark}[The splitting]
	\label{remark. splitting symplectic submersions}
	\label{remark. liouville flow is tangent to boundary}
	Let us explain the splitting in~\ref{item. splitting lambda}. When $\pi: E \to B$ is a symplectic submersion---so that $E$ is symplectic and the symplectic form on $E$ restricts to a symplectic form on each fiber of $\pi$---one has a natural horizontal distribution (given by the symplectic orthogonal to the vertical tangent spaces). We say that the submersion is flat if this distribution is integrable, and in particular, defines a flat connection. We assume this of our projection maps, so $\pi$ canonically splits $E$ as a smooth manifold,  locally on $E$. In particular, note that we may always shrink $\nbhd(x)$ so that $\nbhd(x)$ itself admits a splitting; indeed, the split expression for $\lambda$ in~\ref{item. splitting lambda} assumes we have chosen such a neighborhood for every $x$.

	To see a sample consequence, let $v_\lambda$ denote the Liouville vector field (obtained as the symplectic dual to $\lambda$). Condition~\ref{item. splitting lambda} guarantees not only a splitting of the symplectic form, but also a splitting of the Liouville form; as a result, $v_\lambda$ splits. In particular, because each $\pi_x$ is codimension-preserving, we observe that the Liouville flow---i.e., the flow of $v_\lambda$---is tangent to each boundary and corner stratum of $M$. Moreover, because the Liouville form on the $T^*[0,1)^i$ factor is ${\mathbf p}d{\mathbf q}$, any trajectory of the Liouville flow which begins away from $\partial M$ must remain away from $\partial M$.

	A further consequence, by shrinking $\nbhd(x)$ so the image of $\pi_x$ is a box, is that the composite projection $\nbhd(x) \xrightarrow{\pi_x} T^*[0,1]^i \xrightarrow{\mathbf q} [0,1)^i$ induces a collaring of the stratum of $x$ in $\nbhd(x)$ by choosing a fiber of ${\mathbf q} \circ \pi_x$ and writing $\nbhd(x) \cong \text{fiber} \times \Image({\mathbf q} \circ \pi_x)$.  This exhibits a collar of $\nbhd(x)$ as a smooth manifold with corners.

	\end{remark} 
	
	\begin{remark}[Completeness]
	\label{remark. sectors are completions}
	Let $\rho: M \to \RR_{\geq 0}$ be a smooth function as in Condition~\ref{item. finite type}, and let $v_\lambda$ denote the Liouville vector field, defined by $d\lambda(v_\lambda,-) = \lambda(-)$. 
	The linearity in Condition~\ref{item. finite type} means that $v_\lambda(\rho) = \rho$ (outside some compact subset of $M$). 
	 This has the following consequences.
	
	First, all large-enough positive real numbers are regular values of $\rho$. Because $\rho$ is proper and takes non-negative values, we thus see that $K = \rho^{-1}([0,R])$ for large enough $R$ is some compact, smooth, codimension-zero submanifold-with-corners of $M$. Because the Liouville flow is tangent to the boundary and corner strata of $M$ (Remark~\ref{remark. liouville flow is tangent to boundary}),  every point in the complement of $K$ ``flows to infinity'' under the Liouville flow. So, $M$ is the completion of a compact Liouville domain-with-corners (where the contact boundary may also have corners). As a result, one has a well-defined ``contact boundary at infinity'' (possibly with corners) which we will denote by 
		\eqnn
		\del_\infty M.
		\eqnd
	Moreover, because $M$ is the completion of a compact domain under positive Liouville flow, the Liouville flow is complete (in both the positive and negative flow directions).
	 
	This has the following implication for our submersions. For any $x \in \del^i M$, there is some point $x'$ in the same codimension $i$ stratum as $x$ whose image under the projection $\pi_{x'}$ is in the zero section of $T^*[0,1)^i$. For example, one can take $x'$ to be the $t \to \infty$ limit of $x$ under the negative Liouville flow.
 
	\end{remark}
	 
	Before we further explore Definition~\ref{defn. sector}, we recall a fact from smooth topology. For any compact smooth manifold $X$ with corners and any element $x \in \del^i X$ of the codimension $i$ locus, consider the stratum $(\del^i X)_x$ of $x$. (Such a stratum is always an $i$-dimensional manifold without boundary or corners, always a connected component of $\del^i X$, and often non-compact.) The fact we wish to recall is that $(\del^i X)_x$ admits a compactification 
		\eqnn
		\overline{(\del^i X)_x}
		\eqnd
	to a smooth $i$-dimensional manifold possibly with corners, equipped with an immersion to $X$ extending the embedding of $(\del^i X)_x$.	
	 
	Despite our overline notation, this compactification is not a closure in $X$, and the immersion is often not an embedding. One construction of this compactification is as follows. A point $\tilde y$ in the compactification is an equivalence class of a triplet $(y, u, \phi)$ where $y$ is a point in the closure of $(\del^i X)_x$ inside $M$, $u$ is an element of $(\RR_{\geq 0})^{\dim M - i}$, and  $\phi$ is a germ about $u$ of a codimension-preserving, smooth embedding of some open subset of $(\RR_{\geq 0})^{\dim M -i}$ into the closure of $(\del^i X)_x$ inside $X$, satisfying $\phi(u) = y$. One declares an equivalence $(y,u,\phi) \sim (y',u',\phi')$ when $y = y'$ and there exists a germ of a codimension-preserving diffeomorphism $j$ for which $j(u) = u'$.

	\begin{example}\label{example. rounded square}
	Let $X$ be the manifold with corners obtained by smoothing three of the corners of the square $[0,1]^2$. The codimension 2 stratum is its own compactification; it consists of a point. Meanwhile, the codimension 0 stratum (the interior of $X$) compactifies to $X$ itself. 
	 
	The codimension 1 case is of most interest. The lone codimension-one stratum $\del^1 X$ is diffeomorphic to an open interval, and its compactification $\overline{\del^1 X}$ is diffeomorphic to a closed interval (as opposed to the closure of $\del^1 X$ in $X$, which is homeomorphic to a circle), immersing to the square with both endpoints mapping to the unsmoothed corner.
	\end{example}

	\begin{remark}\label{remark. immersion restricts to embedding} 
	Note that the immersion $\overline{\del^1 X} \to X$ restricts to an embedding on the original $\del^1 X$. We will see identical behavior below when a sector $\overline{F}$ immerses to $M$, and the immersion restricts to an embedding on the interior $F \subset \overline{F}$.
	\end{remark}

	\begin{remark}[Liouville sectors associated to strata]\label{remark. sector fibers at each stratum}
	We will now see that every stratum of codimension $i$ has the structure of a trivial fiber bundle over $\RR^i$, with fiber $F$ naturally identified as the interior of a Liouville sector $\overline{F}$.
	This fiber depends not only on the codimension, but also on the stratum itself (i.e., on the connected component within the codimension $i$ locus).
	
	Let $K$ be the compact subset from Remark~\ref{remark. sectors are completions}, and fix an element $x \in K \cap \del^i M \subset \del^i K$; we denote the stratum of $\del^i K$ containing $x$ by $(\del^i K)_x$. By compactness of $K$,~\ref{item. projections compatible} guarantees a submersion 
		\eqnn
		\pi: \nbhd((\del^i K)_x) \to T^*[0,1)^i
		\eqnd
	defined on a neighborhood of the whole of $(\del^i K)_x$ (and not only locally). 
	
	As per the discussion preceding this remark, let us compactify $(\del^i K)_x$ to a smooth manifold $\overline{(\del^i K)_x}$ of dimension $\dim M - i$, possibly with boundary and corners, equipped with a natural immersion to $M$. 
	$\overline{(\del^i K)_x}$ inherits a natural 1-form by pulling back $\lambda$, and is compact because $K$ is. Again by compactness and~\ref{item. projections compatible}, the maps $\{\pi_y\}_{y \in \text{closure of $(\del^i K)_x$ in $M$}}$ glue together to give a well-defined submersion 
		\eqnn
		\overline{\pi}: \overline{(\del^i K)_x} \to T^*_{\{0\}}[0,1)^i \cong \RR^i
		\eqnd
	to the cotangent fiber of the origin. Note that (i) there is no well-defined analogue of $\overline{\pi}$ on the closure of $(\del^i K)_x$ inside $M$, precisely because of the ambiguity in permuting the coordinates of $T^*[0,1)^i$ in~\ref{item. projections compatible}, and (ii) the the fibers of $\overline{\pi}$ are compact manifolds with corners.
	 
	By Remark~\ref{remark. sectors are completions}, the image of $\overline{\pi}$ contains the origin. Let $\overline{F_K}$ be the fiber of $\overline{\pi}$ 
above the origin, 
	and consider its completion $\overline{F}$ with respect to the Liouville flow (with respect to the pullback of $\lambda$). Then $\overline{F}$ admits a natural immersion to $M$ by extending, via the Liouville flow, the immersion of $\overline{F_K}$. We observe $\overline{F}$ is a Liouville sector, where~\ref{item. finite type} is satisfied for example by pulling back a $\rho$ exhibiting~\ref{item. finite type} for $M$.  
	 
	Letting $F$ denote the interior of $\overline{F}$, the immersion of $\overline{F}$ restricts to an embedding of $F$ to $M$  (see Remark~\ref{remark. immersion restricts to embedding}). So we may identify $F$ with the Liouville completion of the fiber of $\pi$ above the origin of $T^*_{\{0\}}[0,1)^i$.
	Next, note that the splitting condition~\ref{item. splitting lambda} guarantees that the fibers of $\pi$ are all isomorphic, hence the fibers of $\overline{\pi}$ are isomorphic; this shows that each codimension $i$ stratum $(\del^i M)_x$ of $M$ fibers over $T^*_{\{0\}}[0,1)^i \cong \RR^i$ with fibers isomorphic to $F$, though we caution that the only fiber-inclusion into $M$ commuting with the Liouville flow is the inclusion of the fiber above the origin. 
	 
	Finally, we note that by realizing $(\del^i M)_x$ as a Liouville completion of $(\del^i K)_x$, and using the compactness of $K$, we may arrange for each $\pi_y, y \in (\del^i M)_x$ to have image with uniform $q$ coordinate, giving a collaring of the entire stratum:
		\eqn\label{eqn. nbhd of boundary collared}
		\nbhd( (\del^i M)_x ) \cong F \times T^*[0,q_0)^i
		\eqnd
	where $F$ is the interior of the Liouville sector associated to the stratum of $x$ and $q_0>0$ is some positive real number. 
 	\end{remark}
	
	\begin{example} 
	Consider the case where $M$ has boundary but no corners. By Remark~\ref{remark. sector fibers at each stratum}, each boundary component of $M$ has associated to it a canonical Liouville sector $F$ (where $F$ may depend on the boundary component, and $F = \overline{F}$ because $M$ has no corners).  
	 
	Letting $q_0$ be as in~\eqref{eqn. nbhd of boundary collared}, choose a diffeomorphism from $[0,q_0)$ to $[0,1)$. Then every boundary component of $M$ has a neighborhood admitting an identification---as smooth manifolds with boundary equipped with a 1-form---with $F \times T^*[0,1)$. (Where $F$ again depends on the boundary component.) As a smooth manifold with boundary, we thus have a natural collaring
		\eqnn
		(F \times T^*_{\{0\}}[0,1)) \times [0,1).
		\eqnd
identifying the boundary component of $M$ with the smooth manifold $F \times T^*_{\{0\}}[0,1) \cong F \times \RR$.
	\end{example}
	
	\begin{defn}[Sectorial boundary]
		Let $M$ be a Liouville sector. Then the {\em sectorial boundary} of $M$ is the boundary of $M$ when considered as a smooth manifold with corners---in other words, the union of all boundary and corner strata of $M$.
	\end{defn}

	\begin{defn}[Products]
	\label{defn. products of sectors}
	Given two Liouville sectors $(M, \lambda^M), (N, \lambda^N)$, we define the product sector $M \tensor N$ to be $(M \times N, \lambda^M + \lambda^N)$ with the induced projection maps $\{\pi_{(x,y)} = \pi_x \times \pi_y\}$.

	For any integer $k \geq 1$, the sector $M \tensor T^*D^k$ is called a {\em stabilization} of the sector $M$. 
	\end{defn}
	
	\begin{remark}
	In this work, we will frequently consider product manifolds whose Liouville forms do not split (see Construction~\ref{construction. movie}). We reserve the notation $M \tensor N$ for the split setting.
	\end{remark}
	
	\begin{example}
	Let $Q$ be a smooth manifold with corners. Then for every $x$ in a codimension $i$ stratum, $\nbhd(x)$ admits a (codimension-preserving)  collaring projection $\nbhd(x) \to [0,1)^{i}$. The space of such projections is contractible, and a choice of such projection for each $x$ induces projections~\eqref{eqn. corner projections} on $T^*Q$. With this structure, it is standard to see that $T^*Q$ is a Liouville sector in the sense of Definition~\ref{defn. sector}.
	
	For more examples, we refer the reader to Section 2 of~\cite{last-flexible}, and especially to the construction of sectors from stopped domains.
	\end{example}
	
	\subsubsection{Families of exact deformations of Liouville structures}
	Though most of the literature contemplates 1-parameter families of Liouville structures, we will contemplate families of Liouville structures parametrized by smooth manifolds $S$. In fact, we will typically take $S$ to be a simplex $\Delta^k$ (which is a smooth manifold with corners).
	
	\begin{defn}[Deformations of Liouville structures]
	\label{defn. deformations of liouville structures}
	Fix a Liouville sector $(N,\lambda)$. Let $S$ be a smooth manifold, possibly with corners, and fix a smooth function $h: N \times S \to \RR$. Then we say that the family of 1-forms 
		\eqn\label{eqn. d_N notation}
		\{\lambda_s = \lambda + d_N h_s\}_{s \in S}
		\eqnd
	is an $S$-parametrized {\em exact family} of Liouville structures, or an exact {\em deformation} of Liouville structures, if for every $s \in S$, $(N,\lambda_s)$ is a Liouville sector, and

	\newenvironment{deformation-tameness}{
	  \renewcommand*{\theenumi}{(T)}
	  \renewcommand*{\labelenumi}{(T)}
	  \enumerate
	}{
	  \endenumerate
}
	\begin{deformation-tameness}
	\item\label{item. tameness} 
	There exists a function $R: N \times S \to \RR_{\geq 0}$ such that, locally on $S$, $R$ is proper and one can find a compact subset $K \subset N$ for which $R_s$ is $\lambda_s$-linear outside $K$.  
	\end{deformation-tameness}
	Here, by $R$ being proper locally on $S$, we mean that for every $s \in S$, there is a compact neighborhood $A_s$ of $s$ so that $R|_{N \times A_s}$ is proper.
	The local existence of $K$ means: Let For every $s \in S$, there exists an open $U \subset S$ with $s \in U$ and a compact subset $K_s \subset M$ for which $\lambda_{s'}(X_{R_{s'}})(x) = R_{s'}(x)$ for all $x\not \in K$ and for all $s' \in U$. 
	 
	Here, $X_{R_{s'}}$ is the Hamiltonian vector field on $N$ associated to $R_{s'}$.
	\end{defn}

	\begin{remark}
	For every $s \in S$, $h_s : N \cong N \times \{s\} \to \RR$ is the restriction of $h$. $d_N h_s = d h_s$ is the de~Rham derivative of $h_s$ with respect to the de~Rham operator on $N$. In Definition~\ref{defn. deformations of liouville structures}, the subscript $N$ is used to emphasize that no derivatives are being taken in the $S$ directions.
	\end{remark}
	
	\begin{remark}
	 
	The reader may ask, ``a deformation of {\em what?}'' If one chooses a basepoint $s_0 \in S$, then one may consider the $S$-parametrized family as deforming $\lambda_{s_0}$. In the absence of a natural basepoint, ``family'' may be a better term than deformation. In most of the examples of this paper, $S$ will be a simplex with ordered vertices, hence there will be a natural basepoint given by the terminal vertex of the simplex.
	\end{remark}
	
	\begin{remark}
	Condition~\ref{item. tameness} is a ``tameness-near-infinity'' assumption on the family. Let us motivate it. It turns out there are various topologies one could put on the space of possible deformations. We choose to topologize as follows: One filters the set of all possible deformations by the compact subsets of $K$---for any $K$, we consider only deformations that are all linear outside $K$, meaning roughly that we always have a symplectization coordinate outside of $K$. For every $K$, we give the $K$-supported deformation space a weak $C^\infty$ topology, then we topologize the space of all deformations by taking the direct limit topology (by filtering by the variable $K$).\footnote{See also Definition~\ref{defn. topology on embliou}.} Unwinding definitions, this means that if we have an $S$-parametrized family of deformations, then locally on $S$, we can find compact $K$ (which is not uniformly defined over all $S$) controlling the linearity of the family.
	
	This condition did not appear out of thin air. First, it is natural to ask for a single $K$ to control behavior-at-infinity when $S$ is compact. Then, the local existence of $K$ is an inevitable consequence of sheafifying this desideratum (and hence when incorporating non-compact $S$).
		\end{remark}
	
	\begin{remark}
	\label{remark. families for S compact}
	In most of our examples, $S$ will be compact. When $S$ is compact, we may assume $R: N \times S \to \RR_{\geq 0}$ is itself proper. We can further choose a single $K$ so that regardless of $s$, $R_s$ is $\lambda_s$-linear outside the fixed $K$. Then, if necessary, we may enlarge $K$ so that $K$ contains the union of all the skeleta/cores (i.e., the locus of points that do not escape to infinity) of $(M,\lambda_s)$. This enlargement ensures that, outside of this $K$, any trajectory of the Liouville flow of $\lambda_s$ will escape toward infinity (as evidenced by the growth of $R$). 
	\end{remark}
	
	\begin{example}
	Consider an exotic Weinstein structure on $\RR^{2n}$ that contains a standard Weinstein $\RR^{2n}$ in the interior. One can then isotope the ``exotic'' regions toward infinity while leaving the standard copy in place. One can arrange for this to result in a smooth family of Liouville structures, defined on $S=[0,1]$, from the exotic structure to the standard structure. However, this family does not satisfy Condition~\ref{item. tameness}. In particular, Condition~\ref{item. tameness} precludes such ill-behaved ``isotopies'' between exotic and standard structures. (See also Remark~11.24 of~\cite{cieliebak-eliashberg}.)
	\end{example}

	\begin{defn}
	\label{defn. deformation types of liouville structures}
		$S$-parametrized exact deformations (Definition~\ref{defn. deformations of liouville structures}) will further be called (in order of increasing restrictiveness):
		\begin{itemize}
			\item {\em Bordered} if for each $s$, $\lambda_s$ respects the splitting~\ref{item. splitting lambda}.	    
			\item {\em Interior} if $\lambda_s$ is constant (i.e., $s$-independent) near $\del N$ and bordered. 
			\item {\em Compactly supported} if (i) locally in $S$, there exists a compact set $K \subset N$ for which $\supp(\lambda_s - \lambda_0) \subset K$, and (ii) the deformation is interior. 
		\end{itemize}
	\end{defn}

	\begin{remark}
	\label{remark. bordered homotopy}
		For a deformation $\{\lambda_s\}$ to be bordered means exactly that $\{\lambda_s\}$, in a neighborhood of $\del N$, is a direct product of deformations---a deformation of Liouville structure of the fiber $F$, and a constant (non-)deformation of the structure on $T^*[0,1)^i$---with respect to the splitting in~\ref{item. splitting lambda}.
	\end{remark}
	
	\begin{remark}
		If $\{\lambda_s\}$ is a compactly supported deformation, the functions $h_s$ can be chosen to vanish near $\del N$ and $\{\lambda_s\}$ is constant (i.e., $s$-independent) near this boundary.
		
		Put another way, we abusively use compactly supported to mean ``interior and compactly supported.''
	\end{remark}

	\begin{remark}
	Suppose $S$ is compact and the family $\{\lambda_s\}_{s \in S}$ is compactly supported. Then the compact subset $K \subset N$ may be chosen uniformly (i.e., independent of $s \in S$).
	\end{remark}
	
	\begin{remark}
	When $S$ is the smooth $n$-simplex equipped with a collaring, we will also define what it means for $\{\lambda_s\}_{s \in S}$ to be collared (Definition~\ref{defn. collared h}).
	\end{remark}
			
	\subsubsection{Various morphisms}
	We discuss maps between Liouville sectors. Here is the most basic class:
	
\begin{defn}\label{defn. sectorial embedding}
Fix two Liouville sectors $(M,\lambda^M)$ and $(N,\lambda^N)$. A {\em sectorial embedding}, or a map of Liouville sectors, is a function $f: M \to N$ satisfying the following: $f$ is proper, $f$ is a codimension-zero smooth embedding, and there exists some smooth, compactly supported function $h: M \to \RR$ for which $f^*\lambda^N = \lambda^M + dh$. 

If $f$ is further a diffeomorphism, we call $f$ an {\em isomorphism} of Liouville sectors.
\end{defn}

	\begin{remark}
	 
	By the assumption that $f$ is proper and codimension zero, one can equally ask for the existence of a compactly supported $g$ on $N$ for which $f^*(\lambda^N + dg) = \lambda^M$ (rather than asking for the compactly supported function to be defined on $M$).
	\end{remark}

\begin{defn}\label{defn. strict embedding}
We call a sectorial embedding $f$ {\em strict} if $f^*\lambda^N = \lambda^M$. We further call $f$ a {\em strict isomorphism} if $f$ is a diffeomorphism.
\end{defn}

\begin{remark}
In~\cite{last-flexible}, what we here call a strict sectorial embedding is called a strict proper inclusion.
\end{remark}

	\begin{remark}
		Note we allow the sectorial boundary of $M$ to intersect the sectorial boundary of $N$. (Compare with Convention~3.1 of~\cite{gps}.) 
		If $M$ is contained in the interior of $N$, then  the complement  $N \backslash i(M)$ is also a Liouville sector,  with a strict sectorial embedding into $N$; if $N$ is further a  Weinstein sector, then so is $N\backslash i(M)$.
	\end{remark}	
	
	\begin{remark}\label{remark. strict is monoidal}
		If $i$ and $i'$ are strict sectorial embeddings, so is the product $i \times i'$. In contrast, if $i$ and $i'$ are sectorial embeddings, then the product may not be (due to the compact support condition).
	\end{remark}

\begin{remark}
\label{remark. sectorial embeddings are eventually conical}
Let $f: M \to N$ be a sectorial embedding. Outside a compact region, $f$ strictly respects $\lambda$. Because Liouville sectors may, outside a compact set, be modeled as symplectizations of a contact manifold-with-corners (Remark~\ref{remark. sectors are completions}), we can contemplate the induced map $\RR_t \times \del_\infty M \to \RR_t \times \del_\infty N$, where both domain and codomain have Liouville form given by $\lambda \cong e^t \lambda|_{\del_\infty}$. A standard computation shows that any such map must be of the form $(t,x) \mapsto (t + \tau(x), y(x))$ where $y$ is a contact embedding of $\del_\infty M$ into $\del_\infty N$ for which the contact form is scaled by $y^*(\lambda^N) = e^{-\tau(x)} \lambda^M$ on $\del_{\infty} M$. Note that both $\tau: \del_\infty M \to \RR$ and $y$ are independent of $t$.
\end{remark}

\begin{defn}\label{defn. sectorial equivalence}
Let $f: M \to N$ be a sectorial embedding. We call $f$ an {\em equivalence}, or {\em sectorial equivalence}, or an {\em equivalence of Liouville sectors} if there exists a sectorial embedding $g: N \to M$ so that $g \circ f$ and $f \circ g$ are smoothly isotopic, through sectorial embeddings, to the identity functions of $M$ and $N$, respectively.
\end{defn}

\begin{example}\label{example. from isotopies to equivalences}
Fix two smooth manifolds $R$ and $S$, possibly with corners. Then $T^*R$ and $T^*S$, equipped with the canonical Liouville form $pdq$, are Liouville sectors. 

Let $f: S \to R$ and $g: R \to S$ be smooth maps so that $g \circ f$ is smoothly isotopic to $\id_S$ and $f \circ g$ is smoothly isotopic to $\id_R$. (An example: Take $S = D^n$ and $R = [0,1]^n$. We can choose $f$ and $g$ to be smooth embeddings given by a sufficient scaling $x \mapsto t x$ for $t$ a positive real number.) 

Then $T^*f$ and $T^*g$ are both sectorial equivalences.

If $M$ and $N$ are Liouville sectors, let
	$
	M \tensor T^*R
	$
	and
	$
	N \tensor T^*S
	$
denote the direct product of sectors (Definition~\ref{defn. products of sectors}).
If $\phi: M \to N$ is a strict isomorphism of Liouville sectors---meaning $\phi$ is a diffeomorphism and $\phi^*\lambda^N = \lambda^M$---then the map
	\eqnn
	\phi \times T^*f : M \tensor T^*R \to N \tensor T^*S
	\eqnd
is an equivalence of Liouville sectors.
\end{example}

	The above morphisms are specified completely by a single mapping. We now introduce a class of morphisms specified by a mapping, together with a family of Liouville structures in the codomain. 
	
	\begin{defn}
	\label{defn. various deformation embeddings}
	Let $(M,\lambda^M)$ and $(N,\lambda^N)$ be Liouville sectors.
	Consider a pair $(f, \{\lambda^N_s\}_{s \in [-1,1]})$ where $f$ is a smooth, proper embedding (with a priori no compatibility with $\lambda^N_s$) and $\{\lambda^N_s\}$ is a 1-parameter exact family of Liouville structures of $N$ with $\lambda^N_1 = \lambda^N$. 
	
	We say that $(f,\{\lambda^N_s\})$ is a \textit{bordered, interior, or compactly supported deformation embedding} if 
	\enum
	\item $\{\lambda^N_s\}$ is a bordered, interior, or compactly supported family (respectively),
	\item The family $\lambda^N_s$ is constant (i.e., $s$-independent) near $s=-1$ and near $s=1$, and 
	\item $f^*\lambda^N_{-1} = \lambda^M$.    
	\enumd
	\end{defn}
	
	\begin{remark}
	\label{remark. why composition of families is annoying}
	Note that composition of deformation morphisms is cumbersome to define on the nose, if not impossible. For example, given two deformation embeddings $f_{01} : M_0 \to M_1, f_{12}: M_1 \to M_2$, one would have to make a non-canonical choice of extending $\{(f_{12})_*\lambda^{M_1}_s\}_{s}$ and implement a Moore-path-space-like concatenation of $s$-intervals to define a composition and define a topologically enriched category of Liouville sectors with morphisms given by deformation morphisms. 
	
	This at least partially motivates the reason we use $\infty$-categorical models in this work. The higher simplices of an $\infty$-category {\em are} the data of all potentially non-canonical choices---so we need not make any choices---and (by definition of $\infty$-category) we need not define any compositions that depend on such choices.
	\end{remark}
	
\subsection{The movie construction}\label{section.movies}
The following generalizes a construction of~\cite[Section~3.3]{eliashberg-revisited}.

\begin{construction}[The movie construction]
\label{construction. movie}
Let $(M,\lambda)$ be a Liouville sector and
let $S$ be a smooth manifold, possibly with corners. Fix an $S$-parametrized exact family of Liouville structures encoded by a smooth function $h: M \times S \to \RR$ (Definition~\ref{defn. deformations of liouville structures}). Then the {\em movie construction} associated to $h$ is the pair
	\eqn\label{eqn. movie form}
	(M \times T^*S, \lambda + {\bf p}d{\bf q} + dh)
	\eqnd
where  ${\bf p}d{\bf q}$ is the standard Liouville form on $T^*S$, and we out of sloth identify ${\bf p}d{\bf q}$ and $\lambda$ with their pullbacks to $M \times T^*S$.
\end{construction}

Observe that the movie construction is an exact symplectic manifold. We will see in Proposition~\ref{prop. movies are sectors} that the movie construction is always a Liouville sector provided that the deformation encoded by $h$ is bordered and collared  (Definition~\ref{defn. collaring of S}) along the corners of $S$.

\begin{remark}
Setting $\lambda_s = \lambda + d_M h_s$ as in~\eqref{eqn. d_N notation}, we observe
	\eqnn
	\lambda + {\bf p}d{\bf q} + dh
	=
	\lambda_s + {\bf p}d{\bf q} + d_{S} h.
	\eqnd
where the last term denotes we only take the deRham derivative of $h$ in the $S$ directions. Both sides of the above equation are written informally; the most rigorous expression for this 1-form is
	\eqnn
	\pi_M^* \lambda + \pi_{T^*S}^* {\bf p}d{\bf q} + \pi_{M \times S}^*dh
	\eqnd
where $\pi_M$, $\pi_{T^*S}$, and $\pi_{M \times S}$ are the projections from $M \times T^*S$ to $M, T^*S$, and $M \times S$ respectively.
\end{remark}

\begin{remark}
 
The presentation of the form as $\lambda + {\bf p}d{\bf q} + dh$ breaks symmetry by apotheosizing a particular Liouville structure $\lambda$ on $M$.
Later, we will often take $S = \Delta^n$ to be a simplex with ordered vertices, and justify the symmetry-breaking by further demanding that $h_s$ vanishes (as a function on $M$) in a neighborhood of the terminal vertex of $\Delta^n$.
\end{remark}

The main result we will use about the movie construction is that the construction itself is a Liouville sector (Proposition~\ref{prop. movies are sectors}), provided that $h$ is collared appropriately. We say what we mean by an appropriate collar now.

\begin{defn}[Collaring]
\label{defn. collaring of S}
Let $S$ be a smooth manifold with corners. A {\em collaring} of $S$ is the data, for every codimension $i \geq 1$ boundary stratum $\del^i S \subset S$, of a smooth, codimension-respecting embedding
	\eqnn
	\phi_i: \del^i S \times [-1,\epsilon_i)^i \to S,
	\qquad
	(x, \vec u) \mapsto \phi_i(x,\vec u)
	\eqnd
for which $\phi_i(x,0) = x$, and for which there exists a $\delta=\delta(i,j)$ for which
	\eqnn
	\phi_i(\phi_j(x,\vec u'), \vec u'')
	=
	\phi_j(x, \vec u)
	\eqnd
whenever $\phi_j(x,\vec u') \in \del^i S$ and $|\vec u'|, |\vec u''|< 1 + \delta$. Here, $\vec u \in [-1,\epsilon_j)^j$ is a vector obtained by replacing some of the null coordinates of $\vec u'$ by the coordinates of $\vec u''$, up to some re-ordering of coordinates.
\end{defn}

\begin{remark}
In short, a collaring is a trivialization of tubular neighborhoods of each stratum using cubical coordinates. The compatibility condition states that these trivializations respect the passage to higher-codimension strata. We will later choose collarings for simplices $S = \Delta^k$, and so that the collarings are compatible with respect to natural simplicial face maps. (See Definition~\ref{defn. collaring convention}.)
\end{remark}

\begin{remark}
It would be most natural to not demain that collarings always extend beyond 0. However, here (and everywhere else in this work) we demand that collaring coordinates extend beyond the interval $[-1,0]$ to guarantee a {\em standard} embedding of $[-1,0]$ into these collaring neighborhoods. The reasoning behind this is ultimately a coherence/categorical issue: These collars allow us to make sure that our eventual $\infty$-categories $\lioudelta$ and $\lioudeltadef$ are comprised of honestly commuting diagrams of collared objects, rather than of data that allows for shrinking isotopies and reparametrizations.
\end{remark}

\begin{defn}[Collared family]
\label{defn. collared h}
\label{defn. collared}
Fix a collaring on $S$. We say a smooth function $h: M \times S \to \RR$ is {\em collared} (with respect to the collaring on $S$) if, for every codimension $i \geq 1$ stratum of $S$, there is some neighborhood of $\del^i S$ so that $h$ factors through the projection	
	\eqnn
	M \times \nbhd(\del^i S) \to M \times \del^i S.
	\eqnd
In other words, $h$ is collared if it is independent of the ``normal directions'' to $\del^i S$ specified by the collaring $\{\phi_i\}$.

An $S$-parametrized exact family of Liouville structures encoded by a smooth function $h: M \times S \to \RR$ (Definition~\ref{defn. deformations of liouville structures}) will also be called collared if $h$ is collared.

More generally, any structure on $S$---such as a function on $S$, a family over $S$, et cetera---is called collared if it is constant along the normal directions determined by the collaring.
\end{defn}

\begin{example}
Let $S = \Delta^I$ be an $I$-simplex with a collaring. If $h: M \times S \to \RR$ is collared, then $h$ is independent of the $\Delta^I$-coordinate near every vertex of $\Delta^I$. More precisely, given a vertex $v$, there exists a neighborhood $\nbhd(v) \subset \Delta^I$ so that $h|_{M \times \nbhd(v)}$ factors as
	\eqnn
	M \times \nbhd(v) 
	\to
	M
	\xrightarrow{h_v}
	\RR
	\eqnd
for some smooth function $h_v$.
\end{example}

\begin{prop}
\label{prop. movies are sectors}
Let $(M,\lambda)$ be a Liouville sector and $S$ a compact manifold-with-corners equipped with a collaring. Fix a bordered, collared exact deformation encoded by a smooth function $h: M \times S \to \RR$ (Definitions~\ref{defn. deformations of liouville structures} and~\ref{defn. collared h}). Then the movie construction of $M$ (Construction~\ref{construction. movie}) is a Liouville sector.
\end{prop}

\begin{defn}
\label{defn. movie object}
When the deformation is bordered (Definition~\ref{defn. deformation types of liouville structures}) and collared (Definition~\ref{defn. collared}), we will call the Liouville sector~\eqref{eqn. movie form} an {\em $S$-movie}, or simply a {\em movie} when $S$ is clear or irrelevant. When the deformation is further compactly supported, we will call the movie compactly supported as well.

(We will use this terminology most often when $S$ is a simplex---i.e., we will often consider $\Delta^I$-movies.)
\end{defn}

\begin{proof}[Proof of Proposition~\ref{prop. movies are sectors}.]
Because $(M,\lambda)$ is a sector, it comes with the data~\eqref{eqn. corner projections} of projections $\pi_x$ near its corner strata. Further, the collaring of $S$ (Definition~\ref{defn. collaring of S}) induces such projections near the corner strata of $T^*S$. The products of these define the projections for $M \times T^*S$. We must now ensure that these projections, together with the 1-form~\eqref{eqn. movie form}, satisfy the properties~\ref{item. finite type}, \ref{item. splitting lambda}, and~\ref{item. projections compatible} from Definition~\ref{defn. sector}. We proceed in reverse order.

\ref{item. projections compatible} is satisfied, as the projection maps are products of maps satisfying \ref{item. projections compatible}.

\ref{item. splitting lambda} is satisfied because the function $h$ is collared (so that $d_Sh \equiv 0$ near $\del T^*S$, see Definition~\ref{defn. collared h}), because ${\bf p}d{\bf q}$ satisfies~\ref{item. splitting lambda}, and because $\{\lambda_s\}$ is a bordered deformation (Definition~\ref{defn. deformation types of liouville structures}).

To verify~\ref{item. finite type}, it suffices to exhibit a compact subset $A \subset M \times T^*S$ outside of which every element escapes to infinity. For this, because $S$ is compact, we may assume there is a single compact $K \subset M$ outside which every $\lambda_s$-trajectory escapes to infinity. ($K$ is independent of $s$. See Remark~\ref{remark. families for S compact}.) 

We define $A$ to be the compact set of those $(x,s,\sigma)$ for which $x \in K$ and $|\sigma| \leq |d_Sh(x,s)|$. 

If we have $(x,s,\sigma) \in (M \times T^*S) \setminus A$ for which $x \not \in K$, then by assumption on $K$, we know that the Liouville flow of $(x,s,\sigma)$ (which is given by the flow of $\lambda_s$ in the $M$ component) will force the $x$ component to flow toward infinity, hence force $(x,s,\sigma)$ to flow toward infinity. On the other hand, if $x \in K$, we know by definition of $A$ that $|\sigma| > |d_S(h(x,s)|$. Noting that the Liouville vector field on $M \times T^*S$ is given, in the $T^*S$ component, in local coordinates as
	\eqnn
	\sum \left(p_i + {\frac{\del h}{\del q_i}(x,s)} \right) \del_{p_i}
	\eqnd
the definition of $A$ shows that the standard Liouville flow given by ${\bf p}{\del_{\bf p}}$ dominates. In particular, we see that the $\sigma$ coordinate of $(x,s,\sigma)$ will tend to infinity as we flow $(x,s,\sigma)$ along the Liouville flow of the movie. This proves~\ref{item. finite type}.
\end{proof}

\begin{remark}
 
The movie of any collared, bordered deformation is Liouville. However, unless the deformation takes place through Weinsten deformations, the resulting movie is not Weinstein. Nor, as far as we know, does the movie admit any deformation of Liouville structures to be a Weinstein sector. 
\end{remark}

\subsection{The movie construction for maps}
We have see that a deformation of Liouville structures may be encoded into a single sector by the movie construction. Likewise, suppose we now have an $S$-parametrized collection $\{y_s: M \to N\}_{s \in S}$ of (not necessarily strict) sectorial embeddings. We will see here that the data of $\{y_s\}$ may be converted into the data of a single map of movies, and in fact, a single {\em strict} map of movies.

We leave the proof of the following as an exercise for the reader. We remind the reader that $M \tensor T^*S$ is our notation for a sector with the product Liouville structure, while $M \times T^*S$ is direct product of manifolds with a Liouville structure that may not necessarily be the product Liouville structure. (See Definition~\ref{defn. products of sectors}.)

\begin{prop}
\label{prop. movies of maps}
Let $S$ be a compact manifold-with-corners equipped with a collaring.
Let $M \times T^*S$ and $N \times T^*S$ be movies (associated to bordered and collared families of Liouville structures $\{\lambda^M_s\}_{s \in S}$, $\{\lambda^N_s\}_{s \in S}$; see Definition~\ref{defn. movie object}). Assume we are given a {\em strict} sectorial embedding $\widetilde{ y}: M \times T^*S \to N \times T^*S$ for which there exists a function $y: M \times S \to N, (x,s) \mapsto y_s(x),$ fitting into the following commutative diagram:
	\eqn\label{eqn. movie square}
	\xymatrix{
	M \times T^*S \ar[d] \ar[rr]^{\widetilde{y}} 
		&& N \times T^*S \ar[d] \\
	M \times S \ar[rr]^{(x,s) \mapsto (y_s(x),s)} 
		&& N \times S.
	}
	\eqnd
Then for each $s \in S$, $y_s$ is a strict sectorial embedding for the Liouville structures $\lambda^M_s$ and $\lambda^N_s$.

Conversely, suppose we have a smooth, collared, continuous family $\{y_s: M \to N\}_{s \in S}$ of not-necessarily-strict sectorial embeddings (where for each $s \in S$, $y_s$ is a sectorial embedding from $(M,\lambda^M_s)$ to $(N,\lambda^N_s)$). Then there exists (i) a compactly supported function $g: N \times S \to \RR$ and (ii) a smooth map $\widetilde{y}: M \times T^*S \to N \times T^*S$ so that (a) $\widetilde y$ makes~\eqref{eqn. movie square} commute, and (b) $\widetilde y$ is a {\em strict} sectorial embedding when the domain is given the movie structure associated to $\{\lambda^M_s\}$, and when the codomain is given the movie sectorial structure associated to $\{\lambda^N_s + d(g_s)\}_{s \in S}$.
\end{prop}

\begin{defn}
\label{defn. movie of maps}
Fix two movies $M \times T^*S$ and $N \times T^*S$. A map $\widetilde{y}: M \times T^*S \to N \times T^*S$ is called a {\em movie of embeddings} if 
	\enum
	\item $\widetilde{y}$ is a strict sectorial embedding, and
	\item There exists a smooth map $y: M \times S \to N$ rendering the diagram~\ref{eqn. movie square} commutative. (By Proposition~\ref{prop. movies of maps}, then each $y_s$ is necessarily a sectorial embedding.)
	\enumd
We will say that $\widetilde{y}$ is the movie of the family of sectorial embeddings $\{y_s : M \to N\}_{s \in S}$.
\end{defn}

\begin{remark}
\label{remark. movies give smooth maps to embedding spaces}
Suppose we are given a movie of embeddings $\widetilde{y}: M \times T^*S \to N \times T^*S$ where $S$ is collared. Suppose we are further given that the movies of $M$ and $N$ are both movies of compactly supported deformations (Definition~\ref{defn. movie object}). Then the family $\{y_s: M \to N\}$---guaranteed to exist by Proposition~\ref{prop. movies of maps} and obviously uniquely determined by $\widetilde{y}$---determines a smooth and continuous map from $S$ to the space $\embtop_{\liou}(M,N)$ of sectorial embeddings from $M$ to $N$ (Definitions~\ref{defn. topology on embliou} and~\ref{defn. smooth map}). The space $\embtop_{\liou}(M,N)$ is defined unambiguously because  for any $s,s'\in S$ the Liouville structures $\lambda^M_{s}$, $\lambda^M_{s'}$ are related by $d$ of a compactly supported function. (In particular, the identity function is a Liouville isomorphism from $(M,\lambda^M_s)$ to $(M,\lambda^M_{s'})$.) Likewise for $N$.

If further $\widetilde{y}$ is assumed collared, then $\{y_s\}$ defines a smooth and collared and continuous map from $S$.
\end{remark}

\subsection{Shrinking near the boundary, rounding corners}
\label{section. shrinking}

\begin{defn}\label{defn:shrinking_isotopy}
  Let $M$ be any Liouville sector. A \emph{shrinking isotopy} of $M$ is a smooth family $\{\sigma_s: M \to M\}_{s \in [0,1]}$ of strict sectorial embeddings such that $\sigma_0 = \id_M$ and $\sigma_s(\partial M) \cap \partial M = \emptyset$ for $s>0$.
\end{defn}
\begin{lemma}\label{lem:id_is_deformable}
  Every Liouville sector admits a shrinking isotopy.
\end{lemma}
\begin{proof}
  Fix a Liouville sector $M$. After changing the projections $\pi_x$ near infinity, we can arrange that $\pi_x$ is Liouville equivariant up to permutation. In particular, there is some $\varepsilon > 0$ so that the $\pi_x$ assemble to a collection of surjections
  \[
    \boldsymbol{\pi}^i \colon \nbhd(\partial^i X) \to T^*[0,\varepsilon]^i/S_i
  \]
  with $S_i$ the symmetric group acting by permutation. This decomposes $M$ into pieces indexed by maximal nearby codimension. On each piece, the desired isotopy is induced by applying an embedding $f_t \colon[0,\varepsilon] \to [0,\varepsilon]$ to each component which is identity near $\varepsilon$ and sends $0$ to $\frac{t}{2\varepsilon}$.
\end{proof}

\begin{corollary}\label{cor:rounding_corners}
  Every Liouville sector $X$ is equivalent to a Liouville sector $X^{rc}$ with smooth boundary, in the sense that there are maps $X \to X^{rc}$ and $X^{rc} \to X$ whose compositions are isotopic to identity.
\end{corollary}
Here, \emph{rc} stands for \emph{rounded corners}.
\begin{proof}
  Take a cutoff function $\kappa \colon [0,\varepsilon] \to [0,1]$ which vanishes at $0$ and equals $1$ near $\varepsilon$, and has strictly positive derivative on $\kappa^{-1}[0,1)$. Define
  \[
    \boldsymbol\kappa^i \colon [0,\varepsilon]^i \to [0,1]
  \]
  to be the product of $\kappa$ on each component. $\boldsymbol\kappa^i$ is a submersion over $(0,1)$, and we can split it by declaring $\nabla\boldsymbol\kappa^i$ to be horizontal, where $\nabla$ is the gradient with respect to the standard Euclidean metric. This splitting allows us to define the cotangent map $T^*\boldsymbol\kappa^i$, which is a symplectic submersion over $T^*(0,1)$.

  Since $\boldsymbol\kappa^i$ is permutation-invariant, we can consider the composition
  \[
    \pi^{rc}|_{\nbhd^\circ(\partial^i X)} = T^*\boldsymbol\kappa^i \circ \boldsymbol\pi^i,
  \]
  where $\nbhd^\circ$ refers to a punctured neighborhood and $\boldsymbol\pi^i$ is as in the proof of Lemma \ref{lem:id_is_deformable}. The local $\pi^{rc}$ assemble to a global symplectic submersion $\nbhd^\circ(\partial X) \to T^*(0,1)$ with split Liouville form.

  We can now define the rounded-corner sector
  \[
    X^{rc} := X^\circ \setminus (\pi^{rc})^{-1}T^*\left(0, \tfrac12\right).
  \]
\end{proof}

In fact, we will often be more interested in the projection $\pi^{rc}$ than the rounded sector $X^{rc}$, so let us make the following definition.

\begin{defn}
  \label{defn:corner_round_proj}
  For a Liouville sector $X$, a \emph{corner-rounding projection} is a map $\pi^{rc} \colon \nbhd^\circ(\partial X) \to T^*(0,1)$ such that
  \begin{enumerate}[(i)]
    \item The composition
      \[
        \nbhd^\circ(\partial X) \xrightarrow{\pi^{rc}} T^*(0,1) \to (0,1)
      \]
      extends to a smooth map
      \[
        \nbhd(\partial X) \to [0,1)
      \]
      sending $\partial X$ to $0$.
    \item $\pi^{rc}$ is a symplectic fibration over $T^*(0,1)$ which splits the Liouville form as
      \[
        \lambda_X = \lambda_F + pdq.
      \]
  \end{enumerate}
  When working in a neighborhood of $\partial X$, we will sometimes abuse notation and write $(\pi^{rc})^{-1}T^*[0,a)$ to mean $\partial X \cup (\pi^{rc})^{-1}T^*(0,a)$.
\end{defn}

\begin{corollary}
  \label{cor:corner_round_proj}
  Every Liouville sector admits a corner-rounding projection.
  \qed
\end{corollary}

\subsection{Isotopy extension}

Let $M$ and $N$ be smooth manifolds, and suppose one is given a smooth isotopy $\{f_t: M \to N\}_{t \in [0,1]}$. The classical isotopy extension theorem guarantees conditions under which $f_t$ may be extended to an ambient isotopy $\{g_t: N \to N\}$ so that $g_t \circ f_0 = f_t$. This is also true in families---i.e., when there is an $(s,t)$-dependence for $s$ a smooth parameter.

Here we prove the Liouville analogue when the $f_{s,t}$ are (codimension zero) sectorial embeddings.  

\begin{remark}
Even classically, one faces issues when $M$ and $N$ are non-compact, or have boundary. A typical example is to consider $M = [-1,1]$ and $N= [-2,2]$, and let $f_t$ be an isotopy that expands $M$ to become all of $N$---e.g., $f_t(x) = x + tx$. Clearly such an isotopy cannot extend to one of $N$. However, after shrinking $N$ in a $t$-dependent manner (and hence changing the original isotopy of $M$), one can indeed find an isotopy extension. This is the reason for the shrinking in Proposition~\ref{prop. isotopy extension} below.
\end{remark}

\begin{prop}[Isotopy extension for sectorial embeddings]\label{prop. isotopy extension}
Fix 
\begin{itemize}
\item a smooth manifold $S$ (possibly with corners),
\item $f: M \to N$ a sectorial embedding,
\item a smooth family of sectorial embeddings $\{\widetilde{j_{s,0}} : N \to N\}_{s \in S}$, and
\item a smooth family of sectorial embeddings $\{j_{s,t} : M \to N\}_{(s,t) \in S \times [0,1]}$, for which $j_{s,0} = \widetilde{j_{s,0}} \circ f$. 
\end{itemize}
Then the following are true:

\enum[(i)]
\item \label{item. need to shrink isotopies}
For any smoothly-$t$-dependent family of shrinking isotopies $\shrink_{t}: N \to N$ with $\shrink_{0} = \id_N$, there exists a smooth family of sectorial embeddings
	\eqnn
	\{ \widetilde{j_{s,t}} : N \to N\}_{(s,t) \in S \times [0,1]}
	\eqnd
so that $\widetilde{ j_{s,t}} \circ f = \shrink_{t} \circ j_{s,t}$.

\item \label{item. when we dont need to shrink isotopies} Suppose there is an $(s,t)$-dependent $\nbhd(\del N)_{s,t}$ such that
	\eqn\label{eqn. isotopies away from del N}
	\forall s,t
	\qquad
	\nbhd(\del N)_{s,t} \cap j_{s,t}(M) = \emptyset.
	\eqnd
Then one can choose $\shrink_{t} = \id_N$ for all $t$. In particular, there exists a smooth family of sectorial embeddings
	\eqnn
	\{ \widetilde{j_{s,t}} : N \to N\}_{(s,t) \in S \times [0,1]}
	\eqnd
so that $\widetilde{ j_{s,t}} \circ f = j_{s,t}$.

In words, assuming~\eqref{eqn. isotopies away from del N}, any isotopy of a Liouville subsector $M$ extends to an isotopy of the parent sector $N$, and smoothly so in $S$-families.
\enumd
\end{prop}

\begin{proof}
Consider the $(s,t)$-dependent vector fields $X_{s,t} := {\frac {\del j_{s,t}}{\del t}}$ defined on the images $j_{s,t}(M)$. Because each $j_{s,t}$ is a sectorial embedding, we know there exist compactly supported smooth functions $h_{s,t} : M \to \RR$ for which $j_{s,t}^*{\lambda^N} = \lambda^M + dh_{s,t}$. 
Along the images $j_{s,t}(M) \subset N$, the Lie derivative of $\lambda^N$ along $X_{s,t}$ is given by
	\eqnn
	\cL_{X_{s,t}} \lambda^N = \lim_{r \to 0} 
	{\frac
	{j_{s,t+r}^*\lambda^N - \lambda^N}
	{r}
	}
	= \lim_{r \to 0} 
	{\frac
	{dh_{s,t+r}}
	{r}
	}
	=
	{\frac {\del (d h_{s,t}) }{\del t}}
	=
	d({\frac {\del h_{s,t}}{\del t}}).
	\eqnd
While Cartan's magic formula gives
		\eqnn
	\cL_{X_{s,t}} \lambda^N = 
	d (\iota_{X_{s,t}}(\lambda^N))
	+
	\iota_{X_{s,t}}(d \lambda^N)).
	\eqnd
 
Thus,
	\eqnn
	\omega(-,X_{s,t})
	=
	d\lambda^N(-,X_{s,t})
	=
	d\left(
	\lambda^N(X_{s,t})
	-
	{\frac {\del h_{s,t}} {\del t} }
	\right)
	\eqnd

In other words, for every $(s,t)$, we see that $X_{s,t}$ is a Hamiltonian vector field with associated Hamiltonian
	\eqnn
	H_{s,t} = \lambda^N(X_{s,t}) - {\frac {\del h_{s,t}}{\del t}}.
	\eqnd
We can extend each $H_{s,t}$ to a smooth function $\widetilde{H_{s,t}}: N \to \RR$ which satisfies $\widetilde{H_{s,t}} = d\widetilde{H_{s,t}}(Z^N)$ outside of an $(s,t)$-dependent compact subset, and we can choose such extensions smoothly with respect to $(s,t)$. (Here, $Z^N$ is the Liouville vector field on $N$; the equation guarantees that the flow of $X_{\widetilde H_{s,t}}$ generates sectorial embeddings.) We thus obtain Hamiltonian vector fields $X_{\widetilde{H_{s,t}}}$ on $N$. If these vector fields were complete on $N$, their $t$-dependent flows (pre-composed with $\widetilde{j_{s,0}}$) would extend the isotopies $j_{s,t}$, as desired.

If we assume~\eqref{eqn. isotopies away from del N}, we can obviously choose the extensions $\widetilde{H_{s,t}}$ to be supported away from $\del N$. (Note that the support of $\widetilde{H_{s,t}}$ does not even need to be uniform in the $(s,t)$ variables.) Thus the associated Hamiltonian flow is complete, and \eqref{item. when we dont need to shrink isotopies} is proven.

To prove the more general setting in~\eqref{item. need to shrink isotopies}, choose a shrinking isotopy $\shrink_t: N \to N$ (Section~\ref{section. shrinking}). So  $\shrink_0 = \id
_N$, and for every $t>0$, $\shrink_t$ is a proper embedding with $\shrink_t(N) \cap \del N = \emptyset$. Let $X_{s,t}$ be the vector field associated to the flow of $\shrink_{t} \circ j_{s,t}$ (rather than $j_{s,t}$); as shown above, this is a Hamiltonian vector field (because $\shrink_t \circ j_{s,t}$ is a sectorial embedding for every $t$), so let $G_{s,t}$ be the associated Hamiltonian. Then by design, we can choose an extension $\widetilde{G_{s,t}}$ whose the $t$-dependent Hamiltonian flow is complete and extends $\shrink_t \circ j_{s,t}$, as desired. 
\end{proof}

In fact, the exact same proof method yields the following: 
\begin{prop}\label{prop. isotopy extension 2}
Fix 
\begin{itemize}
\item a smooth manifold $S$ (possibly with corners),
\item $f: M \to \widetilde{M}$ a sectorial embedding,
\item a smooth family of sectorial embeddings $\{\widetilde{j_{s,0}} : \widetilde{M} \to N\}_{s \in S}$, and
\item a smooth family of sectorial embeddings $\{j_{s,t} : M \to N\}_{(s,t) \in S \times [0,1]}$, for which $j_{s,0} = \widetilde{j_{s,0}} \circ f$. 
\end{itemize}
Then the following are true:

\enum[(i)]
\item 
There exists a smoothly-$t$-dependent family of shrinking isotopies $\shrink_{t}: N \to N$ with $\shrink_{0} = \id_N$, and a smooth family of sectorial embeddings
	\eqnn
	\{ \widetilde{j_{s,t}} : \widetilde{M} \to N\}_{(s,t) \in S \times [0,1]}
	\eqnd
so that $\widetilde{ j_{s,t}} \circ f = \shrink_{t} \circ j_{s,t}$.

\item   Suppose there is an $(s,t)$-dependent $\nbhd(\del N)_{s,t}$ such that
	\eqnn
	\forall s,t
	\qquad
	\nbhd(\del N)_{s,t} \cap j_{s,t}(M) = \emptyset.
	\eqnd
Then one can choose $\shrink_{t} = \id_N$ for all $t$. In particular, there exists a smooth family of sectorial embeddings
	\eqnn
	\{ \widetilde{j_{s,t}} : \widetilde{M} \to N\}_{(s,t) \in S \times [0,1]}
	\eqnd
so that $\widetilde{ j_{s,t}} \circ f = j_{s,t}$.

In words, assuming~\eqref{eqn. isotopies away from del N}, any isotopy of a Liouville subsector $M$ extends to an isotopy of the parent sector $\widetilde{ M}$, and smoothly so in $S$-families.
\enumd
\end{prop}

\begin{remark}\label{remark. can isotopy extend while collaring}
At the heart of isotopy extension is the ability to choose Hamiltonians $\widetilde{H_{s,t}}$. This flexibility is quite useful. For example, suppose that the original isotopies are collared in some way in the $(s,t)$ variable---say, so that the isotopies are $s$-independent along some open subset of $S$, or $t$-independent in some neighborhood of 0 or 1 in $[0,1]$. Then we may obviously choose $\widetilde{H_{s,t}}$ with the same collaring properties, so that the chosen isotopy extensions $\widetilde{j_{s,t}}$ inherit the collaring properties of $j_{s,t}$. (For $t$-independence, one would have to modify the $t$-dependence of the shrinking isotopies $\shrink_t$ accordingly; this is straightforward.)
\end{remark}

\subsection{Spaces of sectorial embeddings}

\begin{defn}[The topology on $\embtop_{\liou}(M,N)$.]
\label{defn. topology on embliou}
Fix two Liouville sectors $M$ and $N$, along with a compact subset $K \subset M$. We let
	\eqnn
	\embtop_{\liou}^K(M,N)
	\eqnd
denote the set of sectorial embeddings $f: M \to N$ such that, writing $f^*\lambda^N = \lambda^M + dh$, $\supp h \subset K$. The topology is inherited from the collection of all smooth maps equipped with the weak $C^\infty$ Whitney topology\footnote{We can motivate the use of the weak Whitney topology, as opposed to the strong Whitney topology, by noting that the natural map of diffeomorphism spaces $\diff(Q,Q') \to \diff(T^*Q,T^*Q')$ is not continuous if the codomain is given the strong topology. See related discussions in~\cite{oh-tanaka-smooth-approximation}.}. We topologize the collection of all sectorial embeddings using a direct limit:
	\eqnn
	\embtop_{\liou}(M,N) := \bigcup_{K} \embtop_{\liou}^K(M,N)
	\eqnd
where the union runs over all compact subsets $K \subset M$. 
\end{defn}

\begin{remark}
This topology has the following consequence. Let $S$ be a manifold. Then for every continuous map $j: S \to \embtop_{\liou}(M,N)$, we are guaranteed the following analogue of Property~\ref{item. tameness}: For every $s_0 \in S$, there exists a neighborhood $U \subset S$ containing $s_0$, and a compact subset $K \subset M$, so that whenever $s \in U$, we know that $j_s^*\lambda^N - \lambda^M$ is $d$ of some function with support in $K$. Note that $K$ is uniform for all $s \in U$. 

In particular, when $S$ is compact, there is a single compact $K$---i.e., one may take $U = S$.
\end{remark}

\begin{notation}[$\emb_{\liou}(M,N)$]
\label{notation. emb liou}
We let
	\eqnn
	\emb_{\liou}(M,N) := \sing(\embtop_{\liou}(M,N))
	\eqnd
denote the singular complex. Explicitly, $\emb_{\liou}(M,N)$ is the simplicial set for which a $k$-simplex is a continuous map $\Delta^k \to \embtop_{\liou}(M,N)$.

We emphasize the font difference between $\embtop$ and $\emb$. 
\end{notation}

\begin{remark}
Writing spaces of maps between non-compact manifolds as a colimit over subspaces of maps behaving well on compact objects is a standard procedure in both smooth geometry and in symplectic geometry. See for example Remark~10.5 of~\cite{mcduff-salamon-intro}.
\end{remark}

Let us now say what we mean by a smooth map into a mapping space.

\begin{defn}[Smooth map into a mapping space]
\label{defn. smooth map}
Let $S,M,N$ be smooth manifolds (possibly with corners) and let $C^\infty(M,N)$ be the set of smooth maps from $M$ to $N$. Then a function $S \to C^\infty(M,N)$ is set to be {\em smooth} if the induced function $S \times M \to N$ is smooth in the usual sense.
\end{defn}

\begin{remark}
Definition~\ref{defn. smooth map} is a natural definition for smoothness, and is employed in many other settings, perhaps most notably in the setting of diffeological spaces.
\end{remark}

\begin{remark}[Smooth and continuous maps into $\embtop_{\liou}(M,N)$]
\label{remark. smooth and continuous}
So fix a smooth manifold $S$ (possibly with corners) and a function $S \to \embtop_{\liou}(M,N)$. Suppose this function is smooth in the sense of Definition~\ref{defn. smooth map}, so that $S \times M \to N$ is a smooth map. We note that smoothness does not impose the local tameness-at-infinity of Condition~\ref{item. tameness}, i.e., the continuity with respect to the colimit-topology of Definition~\ref{defn. topology on embliou}. 

For this reason, we will often talk about {\em smooth and continuous} maps into $\embtop_{\liou}(M,N)$ to consider smooth maps satisfying Condition~\ref{item. tameness}. This is of course dissonant with the usual fact that smoothness implies continuity; perhaps a better terminology would have been {\em smooth and locally tame near infinity}.
But this is more linguistic than mathematical. One could have built in the $K$-filtered behavior into our definition of smoothness to avoid having to say ``smooth and continuous.''
\end{remark}

\begin{notation}
Given a function $j: S \to \embtop_{\liou}(M,N)$, we will denote by $j_s$ the induced map
	\eqnn
	M \cong M \times \{s\} \into M \times S \to N.
	\eqnd
We will also write $j(s,x)$ to denote $j(s)(x)$. So for example, $j_s(x) = j(s,x)$.
\end{notation}

\subsection{Spaces of (bordered) deformation embeddings}
\label{section. spaces of deformation embeddings}
Here, we will define various spaces of deformation embeddings between two given Liouville sectors. In Section~\ref{section. different deformation spaces}, we will prove these spaces are all homotopy equivalent. To begin, we need a space of deformation of Liouville structures which is sensitive to Condition~\ref{item. tameness}.

\begin{defn}
  \label{defn:liou_struct_space}
  For $N$ a Liouville sector and $K \subset N$ a compact subset, define $\bL_K(N)$ to be the space of pairs $(\lambda_N+dh, r) \in \Omega^1(N) \times C^\infty(N)$ for which $(N, \omega=d(\lambda+dh))$ is a Liouville sector with respect to the projections $\pi_x$ of $N$, and such that $r$ is positive and of linear growth with respect to $\lambda$ outside $K$. (The function $r$ exhibits that $\lambda$ is finite type as in~\ref{item. finite type}.)

  The \emph{space of bordered modifications} of $N$ is
  \[
    \bL(N) := \left( \bigcup_K \bL_K(N) \right) / \sim,
  \]
  where the equivalence relation forgets the second component $r$.

As usual, by a {\em smooth} map $S \to \bL(N)$, we mean a function for which the adjoint is smooth function (so that the $S$-parametrized family of $(\lambda,r)$ depends smoothly on $S$). 
\end{defn}

We now define two homotopy equivalent (in fact, contractible)  spaces of deformations.

\begin{defn}
  \label{defn. deformation space}
  Let $N$ be a Liouville sector. The \emph{space of bordered Liouville deformations} of $N$ is the subspace
  \[
  \defliou(N) \subset C^\infty\bigl( \Delta^1, \bL(N)\bigr)_{\lambda_t}
  \]
  for which $\lambda_t$ agrees with $\lambda^N$ in a neighborhood of the terminal vertex of $\Delta^1$ and is $t$-independent in a neighborhood of the initial vertex of $\Delta^1$.
  
   The \emph{Moore path space of bordered Liouville deformations} of $N$ is a subspace
  \[
    \deflioumoore(N) \subset
    \bigl(\RR_{\ge0}\bigr)_T \times C^\infty\bigl((\RR)_t, \bL(N)\bigr)_{\lambda_t}
  \]
  
  consisting of the following elements. Given 
  	\eqnn
	\left(T, (t \mapsto (\lambda_t,r_t))\right)
	\in
	\bigl(\RR_{\ge0}\bigr)_T \times C^\infty\bigl((\RR)_t, \bL(N)\bigr)_{\lambda_t}
	\eqnd 
	we demand that $\lambda_t$ agrees with $\lambda^N$ for $t\le0$ and is $t$-independent for $t\ge T$.
  
We note that $\defliou(N)$ and $\deflioumoore(N)$ both have obvious notions of smooth maps from a smooth manifold $B$.
\end{defn}

\begin{remark}
The two spaces in Definition~\ref{defn. deformation space} weakly homotopy equivalent, and in fact contractible; but the utility of the Moore deformation space is the evident concatenation operation. Likewise, the two spaces in Definition~\ref{defn. compact deformation space} below are homotopy equivalent.
\end{remark}

\begin{defn}
  \label{defn. compact deformation space}
  The \emph{space of compactly supported Liouville deformations} of $N$ is the subspace
  \[
  \defliou^{\cmpct}(N) \subset\defliou(N)   
  \]
  for which $\lambda_t$ is a compactly supported deformation. We likewise have the Moore path space analogue.
  
   The \emph{Moore path space of compactly supported Liouville deformations} of $N$ is the subspace
  \[
    \deflioumoore^{\cmpct} \subset \deflioumoore(N).
  \]
These again have obvious notions of smooth maps from a smooth manifold $B$.
\end{defn}

This allows us to define models for spaces of deformation embeddings.

\begin{defn}
\label{defn. deformation morphism space}
  The \emph{ space of deformation embeddings} of Liouville sectors from $M$ to $N$ is the subspace
  \[
    \embtop_{\liou}^{\defliou}(M,N)
    \subset
    \defliou(N)_{\lambda^N_t} \times C^\infty(M, N)_{f}
  \]
  such that, near the initial vertex of $\Delta^1$, we have that $f^*\lambda^N_t = \lambda^M$.  Note that $C^\infty$ is given the weak Whitney topology.
  
   The \emph{Moore path space of deformation embeddings} of $N$ is the subspace
  \[
  \embtop_{\liou,\operatorname{Moore}}^{\defliou}(M,N)
    \subset
    \deflioumoore(N)_{\lambda^N_t} \times C^\infty(M, N)_{f}
  \]
  for which $f^*\lambda^N_0 = \lambda^M$.
  
  Again, both spaces have natural notions of smooth maps to them.
\end{defn}

We also have the following versions:

\begin{defn}
\label{defn. compactly supported deformation morphism space}
  The \emph{ space of compactly supported deformation embeddings} of Liouville sectors from $M$ to $N$ is the subspace
  \[
    \embtop_{\liou}^{\defliou,\cmpct}(M,N)\subset
    \embtop_{\liou}^{\defliou}(M,N)
  \]
  such that $\lambda^N_t$ is a compactly supported family of deformations. Likewise, the \emph{Moore path space of compactly supported deformation embeddings} of $N$ is the subspace
  \[
  \embtop_{\liou,\operatorname{Moore}}^{\defliou,\cmpct}(M,N)
  \subset
  \embtop_{\liou,\operatorname{Moore}}^{\defliou}(M,N)
  \]
  such that $\lambda^N_t$ is a compactly supported family of deformations.
  
  Again, both spaces have natural notions of smooth maps to them.
\end{defn}

We have bombarded the reader with various notions of morphisms between Liouville sectors. Here are the three most important notions in this work:
\enum
\item Sectorial embeddings (Definition~\ref{defn. sectorial embedding}), each of which is the datum of a single map $f$. 
\item Compactly supported deformation embeddings (Definition~\ref{defn. various deformation embeddings}), which are maps $f$ equipped with a compactly supported (and hence bordered) deformation of Liouville structures in the codomain.
\item (Bordered) deformation embeddings (Definition~\ref{defn. various deformation embeddings}), which are maps $f$ equipped with a bordered deformation of Liouville structures in the codomain.
\enumd

Having defined the spaces of these morphisms, we may observe maps between them as follows:
\eqn\label{eqn. diagram of embedding spaces}
\xymatrix{
	&& \embtop_{\liou,\operatorname{Moore}}^{\defliou,\cmpct}(M,N)  \ar[rr]
	&& \embtop_{\liou,\operatorname{Moore}}^{\defliou}(M,N)\\
\embtop_{\liou}(M,N) 
	&& \ar[ll]_{\sim}  \embtop_{\liou}^{\defliou,\cmpct}(M,N) \ar[rr] \ar[u]^{\sim} 
	&& \embtop_{\liou}^{\defliou}(M,N)\ar[u]^{\sim} 
}
\eqnd
The vertical equivalences 
are induced by the usual equivalences between path spaces and their Moore path space models. The left-pointing equivalence is induced by noting that the space of compactly supported functions is convex, hence contractible.

We will later prove (Theorem~\ref{theorem. embedding spaces are deformation embedding spaces}) that the right-pointing arrows are also weak homotopy equivalences.

\clearpage
\section{\texorpdfstring{$\infty$}{infinity}-categorical tools}
\label{section. infty cat tools}

Here we set up notation and introduce one new result: Proposition~\ref{prop. Exsim(C) is C}. It allows us to replace an $\infty$-category $\cC$ with an equivalent $\infty$-category $\Ex_{\simeq}(\cC)$, consisting of certain barycentric subdivision diagrams in $\cC$. The expert reader will recognize that barycentric subdivision diagrams can be used to model localizations; this philosophy will allow us to construct a model of the localization $ \lioustrstab[(\eqs^\dd)^{-1}]$ in Sections~\ref{section. lioudelta} and~\ref{section. proof of localization}.

We expect many readers will not be familiar with some $\infty$-categorical arguments, so we include references in the footnotes.

\begin{notation}[Simplices]
Let $I$ be a finite, non-empty, linearly ordered set. We let $\Delta^I$ denote the simplicial set given by the nerve of $I$. Concretely, a $k$-simplex of $\Delta^I$ is a length $k$ ascending chain $i_0 \leq \ldots \leq i_k$ in $I$. We call $\Delta^I$ the standard (combinatorial) $I$-simplex.

There is likewise a topological space
	\eqn\label{eqn. usual simplex}
	\Delta^I := \{(x_i)_{i \in I} \in \RR_{\geq 0}^I \, | \sum_{i \in I} x_i = 1 \}.
	\eqnd
While it is sometimes common to write the space by the notation $|\Delta^I|$ (to avoid double-booking notation) we will not follow this convention, with the exception of Sections~\ref{section. weak kan} and~\ref{section. lioudelta homs}, where we will repeatedly use the topological spaces, not just the combinatorial simplices. When the vertical bars $| \bullet |$ are missing, it will be clear from context whether we are referring to the combinatorial simplex, or the simplex as a topological space.

Note $\Delta^I$ is also a smooth manifold with corners in the standard way. 

Finally, we will often depart from the convention of~\eqref{eqn. usual simplex} and allow our simplices to have coordinates as small as -1 (see Section~\ref{section. -1 convention for simplex}). This is to more easily define collarings of simplices later, but does not in any meaningful way affect the underlying ideas of our constructions. 
\end{notation}

\subsection{Simplicial and semisimplicial sets}

\begin{notation}[$\Delta_{\inj}$]
Recall that $\Delta$ is the category whose objects are finite, non-empty, linearly ordered sets $I$. A morphism is a weakly order-preserving function. 

We let $\Delta_{\inj}$ denote the subcategory of $\Delta$ with the same objects, but whose morphisms are injections (equivalently, the morphisms must be strictly order-preserving). 
\end{notation}

\begin{defn}
A {\em simplicial set} is a functor $\Delta^{\op} \to \sets$ to the category of sets. A map of simplicial sets is a natural transformation. A {\em semisimplicial set} is a functor $(\Delta_{\inj})^{\op} \to \sets$. A map of semisimplicial sets is a natural transformation.
\end{defn}

Fix a simplicial set $X$. Given an injection $I' \to I$, we call the induced map $X_I \to X_{I'}$ a face map. If $I' \to I$ is a surjection, the induced map is called a degeneracy map. Informally, a semisimplicial set is a simplicial set without the degeneracy maps.

We let $\Lambda^n_k$ delete the $k$th $n$-horn---informally, this is a simplicial set obtained by deleting the interior, and the face opposite the $k$th vertex, of the $n$-simplex. A simplicial set $X$ is called an {\em $\infty$-category} if it satisfies the weak Kan condition: For any $0<k<n$, any map $\Lambda^n_k \to X$ extends to $\Delta^n$. Such an extension is often called a filler.

\begin{remark}
The distinction between the notion of simplicial and semisimplicial sets will become prominent especially in Section~\ref{section. lioudelta is oo cat}. This is because, when the simplicial set is an $\infty$-category, degeneracy maps encode identity morphisms and the higher coherences of identity morphisms. The point of constructing a well-behaved {\em semisimplicial} set is to construct an ``$\infty$-category'' that admits units, but for which we specify none.

It is often the case that one can easily define face maps, but cannot easily define degeneracy maps. This is common in many geometric settings--e.g., cobordism categories, and in settings where strict collaring conditions must be respected (as in our setting---see Definition~\ref{defn. lioudelta}).  
\end{remark}

We now review formal facts from the theory of $\infty$-categories that tell us when semisimplicial sets may be promoted to be simplicial sets, and in fact, $\infty$-categories.

\begin{defn}[\cite{steimle}, Definition~1.1]
\label{defn. idempotent equivalence}
Given a semisimplicial set $X$, an edge $e$ is {\em idempotent} if there is a 2-simplex of $X$ whose three boundary edges are all given by $e$. The edge $e$ is an equivalence if (a) for any horn $\Lambda^n_n \to X$ whose last edge is $e$, there is a filler $\Delta^n \to X$, and (b) the same holds for any horn $\Lambda^n_0 \to X$ whose first edge is $e$.
\end{defn}

\begin{theorem}
\label{theorem. steimle}
(Theorem~1.2 of~\cite{steimle}.) Let $X$ be a semisimplicial set satisfying the semisimplicial weak Kan condition, and for which every object admits an idempotent self-equivalence. Then $X$ admits a simplicial structure with the same underlying face maps, whose degenerate edges are the idempotent self-equivalences, and for which $X$ is an $\infty$-category.

(Theorem~2.1 of ibid.) Moreover, fix a semisimplicial subset $A \subset X$ with a specified simplicial structure. Then one may instead choose a simplicial structure on $X$ extending the degeneracy maps of $A$. 
\end{theorem}

\begin{theorem}[\cite{tanaka-non-strict}]
\label{theorem. hiro non strict}
Let $X$ and $Y$ be $\infty$-categories, and fix $f_{\semi}: X \to Y$ a {\em semisimplicial} map (meaning $f$ respects face maps, but may not respect degeneracy maps). If $f_{\semi}$ sends all degenerate edges of $X$ to equivalences in $Y$, then there exists a {\em simplicial} set map $f: X \to Y$ such that, for any simplex $e$ of $X$, $f_{\semi}(e)$ is naturally homotopic to $f(e)$.
\end{theorem}

In some sense, the degeneracies of a simplicial set are the ``least homotopical'' ingredient in the theory of $\infty$-categories, as they demand particular choices and {\em strict} equations to be satisfied by these choices. The above two theorems allow us to bypass this issue, at the level of both individual $\infty$-categories and functors between them.

\subsection{Subdivisions and \texorpdfstring{$\Ex$}{Ex}}
We recall some standard constructions in simplicial homotopy theory~\cite{kan-css, goerss-jardine}

\begin{notation}
Let $I$ be a finite, non-empty, linearly ordered set. By $\cP'(I)$, we mean the collection of all non-empty subsets of $I$.
We note that $\cP'(I)$ is itself a (non-linear) poset, ordered by the subset relation.
\end{notation}

\begin{notation}[$\sd$]
Let $\Delta^I$ be the standard $I$-simplex. 
We let $\sd(\Delta^I)$ denote the simplicial set given by the nerve of $\cP'(I)$. We call it the {\em barycentric subdivision} of $\Delta^I$. Concretely, a 0-simplex of $\sd(\Delta^I)$ is an element of $\cP'(I)$. A $k$-simplex is a chain $A_0 \subset A_1 \subset \ldots \subset A_k$ of elements in $\cP'(I)$, where $\subset$ need not be a proper inclusion.

We have defined $\sd$ on simplices. Following custom, we denote the left Kan extension to all simplicial sets by the same notation. Concretely, given a simplicial set $\cC$, we have
	\eqnn
	\sd(\cC) \cong \colim_{(\Delta^k \to \cC) \in \Delta_{/\cC}} \sd(\Delta^k).
	\eqnd
Informally, $\sd(\cC)$ is the simplicial set obtained from $\cC$ by replacing every $k$-simplex of $\cC$ by $\sd(\Delta^k)$.
\end{notation}

\begin{defn}
\label{defn. max localizing}
We will say that an edge $(A \subset B)$ of $\sd(\Delta^I)$ is {\em $\max$-localizing} if $\max A = \max B$. 
\end{defn}

\begin{remark}\label{remark. max is a localization}
Consider the function
	\eqnn
	\max: \cP'(I) \to I, 
	\qquad
	A \mapsto \max A.
	\eqnd
(Here, $\max A \in A \subset I$ is the maximal element of $A$ with respect to the linear order induced from $I$.)
Then $\max$ is a map of posets. It has a fully faithful right adjoint
	\eqn\label{eqn. max right adjoint}
	I \to \cP'(I),
	\qquad
	i \mapsto [\min I, i] = \{i' \in I \, | \, \min I \leq i' \leq i\}.
	\eqnd
Thus, $\max: \sd(\Delta^I) \to \Delta^I$ is a localization of $\infty$-categories\footnote{Definition~5.2.7.2 of~\cite{htt}.}. (We have abused notation and used $\max$ to also denote the induced map of nerves.) 

In particular, for any $\infty$-category $\cC$, the map of functor $\infty$-categories
	\eqnn
	\fun(\Delta^I, \cC) \to \fun(\sd(\Delta^I),\cC)
	\eqnd
is fully faithful, with essential image consisting of those functors that send $\max$-localizing edges to equivalences in $\cC$.\footnote{Proposition~5.2.7.12 of~\cite{htt}.}
\end{remark}

\begin{warning}
The right adjoints~\eqref{eqn. max right adjoint} are not natural in the $I$ variable.
\end{warning}

\begin{defn}[$\Ex$]
Let $\cC$ be a simplicial set. We define $\Ex(\cC)$ to be the simplicial set
	\eqnn
	\Ex(\cC) : \Delta^{\op} \to \sets,
	\qquad
	I \mapsto \hom_{\sset}(\sd(\Delta^I), \cC)
	\eqnd
so a $k$-simplex of $\cC$ is the data of a simplicial set map from $\sd(\Delta^k)$ to $\cC$.\footnote{The $\Ex$ stands for ``extension'' according to Kan's original work~\cite{kan-css}.}
\end{defn}

\begin{remark}
Given any (not necessarily order-respecting) function $f: I \to J$, we have an induced map of posets $\cP'(I) \to \cP'(J)$ by sending a subset $A \subset I$ to $f(A) \subset J$. Note $f^*: \Ex(\cC)_J \to \Ex(\cC)_I$ is precomposition with this induced map.
\end{remark}

\begin{notation}\label{notation. id to Ex}
There is a natural transformation $\max^*: \id_{\sset} \to \Ex$. Explicitly, for any simplicial set $\cC$, the map 
	\eqnn
	\max\!^*: \cC \to \Ex(\cC)
	\eqnd 
is induced by precomposition with $\max : \cP'(I) \to I, A \mapsto \max A$. \end{notation}

\subsection{\texorpdfstring{$\Ex_{\simeq}$}{Ex sim} of \texorpdfstring{$\infty$}{infinity}-categories}

\begin{notation}[$\Ex_{\simeq}$]\label{notation. Exsim}
Let $\cC$ be an $\infty$-category. We let $\Ex_{\simeq}(\cC)  \subset \Ex(\cC)$ 
denote the simplicial set whose $I$-simplices consist only of those maps $\sd(\Delta^I) \to \cC$ that send $\max$-localizing edges to equivalences.
\end{notation}

\begin{notation}[$m$]
\label{notation. map to Exsimeq}
The natural map $\cC \to \Ex(\cC)$ from Notation~\ref{notation. id to Ex} factors through $\Ex_{\simeq}(\cC)$ because the composition $\sd(\Delta^I) \xrightarrow{\max} \Delta^I \to \cC$ sends $\max$-localizing edges (Definition~\ref{defn. max localizing}) to (degenerate) equivalences in $\cC$.
We will denote this map of simplicial sets by
	\eqnn
	m: \cC \to \Ex_{\simeq}(\cC).
	\eqnd
We note that the assignment $\cC \mapsto \Ex_{\simeq}(\cC)$ is a functor (from the category of $\infty$-categories to the category of simplicial sets) and $m$ defines a natural transformation from the identity functor.

In Proposition~\ref{prop. Ex is homotopical}, we will prove that the $\Ex_{\simeq}$ functor may be enriched over simplicial sets; moreover, we will see shortly that $\Ex_{\simeq}(\cC)$ is an $\infty$-category if $\cC$ is (Proposition~\ref{prop. Exsim(C) is C}). Thus $\Ex_{\simeq}$ defines a functor from the $\infty$-category of $\infty$-categories to itself. 
\end{notation}

\begin{notation}
For this section only, we will use the notation $\fun^\ast$ from time to time. Concretely, $\fun(X,\cC)$ is the simplicial set of maps from $X$ to $\cC$, and the notation $\fun^\ast(X,\cC)$ emphasizes the simplicial structure: $\fun^n(X,\cC)$ is the set of maps $X \times \Delta^n \to \cC$. The reason for the $\ast$ symbol is that we will soon have expressions such as $\fun^\ast(\Delta^\bullet,\cC)$ depending on two indices $\ast$ and $\bullet$, and we want to distinguish their roles.
\end{notation}

\begin{prop}\label{prop. Exsim(C) is C}
Let  $\cC$ be an $\infty$-category.
\enum
	\item $\Ex_{\simeq}(\cC)$ is an $\infty$-category.
	\item The map (Notation~\ref{notation. map to Exsimeq})
	$
	m: \cC \to \Ex_{\simeq}(\cC)
	$
is an equivalence of $\infty$-categories.
\enumd
\end{prop}

\begin{proof}
Consider the simplicial space
	\eqnn
	\fun^\ast(\Delta^\bullet, \cC)^{\sim}
	\eqnd
where $\ast$ is the ``space'' index and $\bullet$ is the ``simplicial'' index.
If $Y$ is an $\infty$-category, $\fun^\ast(X,Y)$ is also an $\infty$-category (regardless of $X$).\footnote{See Proposition~1.2.7.3 of~\cite{htt}.} The notation $\fun^\ast(X,Y)^{\sim} \subset \fun^\ast(X,Y)$ indicates the largest $\infty$-groupoid inside $\fun^\ast(X,Y)$. In short, $\fun^\ast(\Delta^I, \cC)^{\sim}$ is the space (Kan complex) of functors from $\Delta^I$ to $\cC$. Then $\fun^\ast(\Delta^\bullet, \cC)^{\sim}$ is a complete Segal space because $\cC$ is an $\infty$-category.\footnote{See Proposition~4.10 of~\cite{joyal-tierney}, where $\Gamma(X)_{m,n}$ stands for our $\fun^n(\Delta^m,X)^\sim$.}

Let us now define the simplicial space
	\eqnn
	\fun^\ast_{\simeq}(\sd(\Delta^\bullet),\cC)^{\sim}.
	\eqnd
For any $I$, $\fun^\ast_{\simeq}(\sd(\Delta^I),\cC)^{\sim}
	\subset \fun^\ast(\sd(\Delta^I),\cC)^{\sim}$ is declared to be the full Kan complex spanned by those simplicial set maps $\sd(\Delta^I) \to \cC$ sending $\max$-localizing edges to equivalences in $\cC$. (In other words, $\fun^\ast_{\simeq}(\sd(\Delta^I),\cC)^{\sim}$ consists exactly of those connected components containing functors that send $\max$-localizing edges to equivalences.) 

We now claim $\fun^\ast_{\simeq}(\sd(\Delta^\bullet),\cC)^{\sim}$ is a complete Segal space. We will see that $\fun^\ast_{\simeq}(\sd(\Delta^\bullet),\cC)^{\sim}$ is Reedy fibrant in Lemma~\ref{lemma. Exeq is fibrant} below. Further, by Remark~\ref{remark. max is a localization}, we know that $\Delta^I$ is a localization of $\sd(\Delta^I)$ along the $\max$-localizing edges---so the natural map
	$
	\fun^\ast(\Delta^I, \cC)^{\sim}
	\to
	\fun^\ast_{\simeq}(\sd(\Delta^I),\cC)^{\sim}
	$
is a homotopy equivalence of Kan complexes for every $I$. Because $\fun^\ast(\Delta^\bullet, \cC)^{\sim}$ satisfies the completeness and Segal conditions, the claim is proven.

We witnessed in the previous paragraph that 
	$
	\fun^\ast(\Delta^\bullet, \cC)^{\sim}
	\to
	\fun^\ast_{\simeq}(\sd(\Delta^ \bullet),\cC)^{\sim}
	$
is an equivalence of complete Segal spaces. 
On the other hand, if $X_{\bullet,\ast}$ is a complete Segal space, it is known that (i) $X_{\bullet,0}$ is an 
$\infty$-category\footnote{Corollary~3.6 of~\cite{joyal-tierney}.
}, and (ii) any equivalence of complete Segal spaces $W_{\bullet,\ast} \to X_{\bullet,\ast}$ induces an equivalence of 
$\infty$-categories\footnote{This follows from the fact that $X_{\bullet,\ast} \mapsto X_{\bullet,0}$ is right Quillen. (See~\cite{joyal-tierney}, where the notation $i_1^*$ denotes this right adjoint.)  Because $i_1^*$ is right Quillen, it preserves weak equivalences between fibrant objects. On the other hand, complete Segal spaces and $\infty$-categories are the fibrant objects of the domain and codomain of $i_1^*$, respectively.
} $W_{\bullet,0} \to X_{\bullet,0}$. The proposition follows by noting the equality of simplicial sets
	\eqnn
	\Ex(\cC)_\bullet = \fun^0_{\simeq}(\sd(\Delta^\bullet),\cC)^{\sim}.
	\eqnd
\end{proof}

We claimed the following while proving Proposition~\ref{prop. Exsim(C) is C}:

\begin{lemma}\label{lemma. Exeq is fibrant}
If  $\cC$ is an $\infty$-category, then $\fun^\ast_{\simeq}(\sd(\Delta^\bullet),\cC)^{\sim}$ is Reedy fibrant.
\end{lemma}

The proof of Lemma~\ref{lemma. Exeq is fibrant} will rely on another lemma:

\begin{lemma}\label{lemma. cofibrations induce fibrations of mapping spaces}
Let $X$ and $Y$ be $\infty$-categories and $f: X \to Y$ a cofibration (meaning $f_n: X_n \to Y_n$ is an injection for all $n \geq 0$). Then for any $\infty$-category $\cC$, the induced map of Kan complexes
	\eqnn
	\fun(Y,\cC)^\sim \to 
	\fun(X,\cC)^\sim 
	\eqnd 
is a Kan fibration.
\end{lemma}

\begin{proof}
This is a special case of Lemma~3.1.3.6 of~\cite{htt}. In the notation and language of loc. cit. (which we use throughout this proof), one sets  $S = \Delta^0$ and $Z = \cC$. Note that because $X$ and $Y$ are $\infty$-categories, any map of simplicial sets sends the naturally marked edges of $X$ and $Y$ (i.e., the equivalences in $X$ and $Y$) to the naturally marked edges (the equivalences) of $\cC$. Thus $\Map^\flat(X,\cC) = \hom_{\sset}(X,\cC)$, and $\Map^\sharp(X,\cC)$ is the largest Kan complex inside $\hom_{\sset}(X,\cC) = \fun(X,\cC)$ (Remark~3.1.3.1 of ibid.). 

(There are other proofs using the fact that the Joyal model structure for $\infty$-categories is Cartesian closed -- see Proposition~4.10 of~\cite{joyal-tierney}.)
\end{proof}

\begin{proof}[Proof of Lemma~\ref{lemma. Exeq is fibrant}.]
It is straightforward to verify that Reedy fibrancy is equivalent to the following statement: For every $n \geq 0$, the map of simplicial sets
	\eqn\label{eqn. Kan fibration for reedy fibrancy}
	\fun^\ast_{\simeq}(\sd(\Delta^n),\cC)^{\sim} \to 
	\fun^\ast_{\simeq}(\sd(\del \Delta^n),\cC)^{\sim}
	\eqnd
is a Kan fibration. We prove this now.

Note that the map $\sd(\del \Delta^n) \to \sd(\Delta^n)$ is a monomorphism. We also see $\sd(\del \Delta^n)$ is an $\infty$-category. (For example, the monomorphism identifies $\sd(\del \Delta^n)$ as the full subcategory of $\sd(\Delta^n)$ obtained by excluding the terminal object $[n]$.)

It follows from Lemma~\ref{lemma. cofibrations induce fibrations of mapping spaces} that the induced map on functor spaces
	$
	\fun^\ast(\sd(\Delta^n),\cC)^{\sim} \to 
	\fun^\ast(\sd(\del \Delta^n),\cC)^{\sim}
	$
is a Kan fibration. Thus the map~\eqref{eqn. Kan fibration for reedy fibrancy} is also a Kan fibration (as it is the restriction of a Kan fibration to certain connected components).
\end{proof}

\begin{prop}
\label{prop. Ex is homotopical}
The functor $\Ex$, from the category of simplicial sets to itself, admits a natural promotion to a functor of $\sset$-enriched categories. 

(Concretely, this means that for any two simplicial sets $\cC$, $\cD$ and a map $\cC \times \Delta^m \to \cD$, we can create a map of simplicial sets $\Ex(\cC) \times \Delta^m \to \Ex(\cD)$, natural in the three variables $\cC,\cD,$ and $\Delta^m$.)

Likewise, restricting to the full subcategory of $\infty$-categories, $\Ex_{\simeq}$ may be promoted to a functor of simplicially enriched categories.
\end{prop}

\begin{remark}
Informally: Proposition~\ref{prop. Ex is homotopical} tells us that $\Ex_{\simeq}$ is a functor that respects natural transformations of functors and higher homotopies of functors.
\end{remark}

\begin{proof}
We have the following sequence of natural (in the $\Delta^k \in \Delta^{\op}$ variable) maps:
	\begin{align}
	\hom_{\sset}(\Delta^k, \Ex(\cC) \times \Delta^m)
	& \cong 	
	\hom_{\sset}(\subdiv(\Delta^k),\cC)
	\times
	\hom_{\sset}(\Delta^k, \Delta^m) \nonumber \\
	& \to 	
	\hom_{\sset}(\subdiv(\Delta^k),\cC)
	\times
	\hom_{\sset}(\subdiv(\Delta^k), \Delta^m) \nonumber \\
	& \cong	\hom_{\sset}(\subdiv(\Delta^k), \cC \times \Delta^m)\nonumber \\
	& 	\cong \hom_{\sset}(\Delta^k, \Ex(\cC \times \Delta^m))\nonumber 
	\end{align}
The one non-bijection is given by pulling back along the map $\max: \sd(\Delta^k) = N(\cP'([k])) \to N([k]) = \Delta^k$. (See Remark~\ref{remark. max is a localization}.) All bijections are  justified by universal properties of products, or by adjunction/definition. By Yoneda's Lemma, we thus have a map of simplicial sets
	\eqn\label{eqn. ExC Delta to ExCDelta}
	\Ex(\cC) \times \Delta^m \to \Ex(\cC \times \Delta^m). 
	\eqnd
Because $\Ex$ is a functor, we have a map (natural in all variables) which takes any map $H: \cC \times \Delta^m \to \cD$ to a map $\Ex(H): \Ex(\cC \times \Delta^m) \to \Ex(\cD)$.  Composing~\eqref{eqn. ExC Delta to ExCDelta} with $\Ex(H)$, we obtain a map of simplicial sets
	\eqnn
	\Ex(\cC) \times \Delta^m \to \Ex(\cD)
	\eqnd
natural in all variables.

The final claim follows by noting that  all maps involved preserve simplices satisfying the max-localizing condition. 
\end{proof}

\subsection{Making a homotopy}
\label{section. non-strict unitality}

Here we explicate the existence of a homotopy we use to prove our localization theorem (Theorem~\ref{theorem. localization}). 
Throughout, we will utilize standard notation and results from the theory of marked simplicial sets. We refer the reader to Section~3.1 of~\cite{htt} for details, where the only relevance for us (in the notation of ibid.)  is the setting where the base simplicial set $S = \Delta^0$ is a point.

\begin{remark}[Marked simplicial sets model localizations]
\label{remark. marked simplicial sets model localizations}
Let us onramp the unfamiliar reader. When $\cD$ is an $\infty$-category, a marked simplicial set $(\cD,\cE)$ is input data for localizing $\cD$ along a collection of edges $\cE$. Indeed, in the model category of marked simplicial sets, a fibrant replacement for $(\cD,\cE)$ produces a localization of $\cD$ along $\cE$. 

This last claim follows as a combination of~\cite[Proposition~3.1.4.1]{htt} which shows that any fibrant replacement is an $\infty$-category marked by its equivalences, of~\cite[Proposition~3.1.3.7]{htt} which by model-categorical nonsense guarantees any marked simplicial set $(\cD,\cE)$ admits a fibrant replacement, and of~\cite[Proposition~3.1.3.3]{htt} which states that any fibrant replacement map is a Cartesian equivalence, co-representing the same spaces of functors as $(\cD,\cE)$.
\end{remark}

\begin{notation}[$\cC^\natural$ and $\cD^\sharp$, Definition~3.1.1.8 and Section~3.1 of~\cite{htt}]
Let $\cC$ be an $\infty$-category. $\cC^\natural$ indicates the marked simplicial set $(\cC,\cE)$ with the marking $\cE$ given by the equivalences in $\cC$. For any simplicial set $\cD$, we let $\cD^\sharp$ denote the marked simplicial set where every edge of $\cD$ is marked.
\end{notation}

\begin{notation}[$\map^\sharp$, see~3.1.3 of~\cite{htt}]
\label{notation. map sharp}
Let $(X,\cE)$ be a marked simplicial set and $Y$ an $\infty$-category. We let
	\eqnn
	\map^\sharp( (X,\cE),Y^\natural)
	\eqnd
denote the Kan complex where a $k$-simplex is a map of marked simplicial sets $f: (\Delta^k)^\sharp \times (X,\cE) \to Y^\natural$. Explicitly, $f$ is a $k$-simplex of $\map^\sharp(X,Y)$ if and only if, for every edge $e$ of $\Delta^k$ and every marked edge $e'$  of $X$, the image $f(e,e')$ is an equivalence in $Y$. 
\end{notation}
 
\begin{notation}[$X_+$]
Let $X$ be a simplicial set.
We let $X_+$ be the free simplicial set generated by the semisimplicial set underlying $X$.
\end{notation}

\begin{remark}
Informally, $X_+$ is obtained by adjoining degeneracies freely to each simplex of $X$. It is most conveniently modeled as a Kan extension of $X$ (restricted to $\Delta_{\inj}$) along the inclusion $\Delta_{\inj} \into \Delta$.

In particular, the set of edges 
	\eqn\label{eqn. splitting of edges in C+}
	(X_+)_1 \cong X_1 \coprod X_0
	\eqnd
decomposes as a disjoint union of the edges arising in $X$ (including all the degenerate edges of $X$), and of all the degenerate edges in $X_+$, which is naturally identified with a copy of the set of vertices of $X$.
\end{remark}

\begin{example}
Let $\mathsf{C}$ be a category (in the usual, non-$\infty$-categorical sense) where only identity morphisms admit left or right inverses. For example, $\mathsf C$ may be a poset. 

We let $\mathsf{C}^{\nonunital}$ denote the non-unital category obtained by removing all identity morphisms from $\mathsf{C}$. Note $\mathsf{C}^{\nonunital}$ still has a well-defined composition thanks to the assumption that $\mathsf{C}$ has no inverses.

One then has a semisimplicial set $N_{\semi}(\mathsf{C}^{\nonunital})$ given by the semisimplicial nerve of $\mathsf{C}^{\nonunital}$; concretely, $N_{\semi}(\mathsf{C}^{\nonunital})$ is the functor $\Delta_{\inj}^{\op} \to \sets$ sending $[k]$ to the set of functors $[k]^{\nonunital} \to \mathsf{C}^{\nonunital}$. Equivalently, $N_{\semi}(\mathsf{C}^{\nonunital})$ is the semisimplicial subset of $N(\mathsf{C})$ obtained by removing all degenerate simplices.

Then the natural map
	\eqnn
	(N_{\semi}(\mathsf{C}^{\nonunital}))_+ 
	\to 
	N(\mathsf{C})
	\eqnd
from the free simplicial set on $N_{\semi}(\mathsf{C}^{\nonunital})$ to the nerve of $\mathsf{C}$ is an isomorphism of simplicial sets. 

Treating the poset $[k]$ as a category, we see a subexample: The standard $k$-simplex $\Delta^k \cong N([k])$ is in fact a free simplicial set; it is generated by the semisimplicial set represented by $[k]$ in $\Delta_{\inj}$; this semisimplicial set is otherwise written as $N_{\semi}([k]^{\nonunital})$. 

Likewise, $\subdiv(\Delta^k)$ is the nerve of a poset (the poset of non-empty subsets of $[k]$). Hence we see that $\subdiv(\Delta^k)$ is also a free simplicial set. 
\end{example}

\begin{notation}\label{notation. plus E marked}
Let
	\eqnn
	\cE = s_0^{X}(X_0) \coprod X_0 \subset X_1 \coprod X_0 \cong (X_+)_1
	\eqnd
be the set of edges in $X_+$ that are identified with the degenerate edges in $X$, or are degenerate in $X_+$. (In the above equation, the last subset inclusion passes through the identification~\eqref{eqn. splitting of edges in C+}.) The pair
	\eqnn
	(X_+, \cE)
	\eqnd
is a marked simplicial set.
\end{notation}

As suggested by our language of ``free'' simplicial set, there is a ``free-forget'' adjunction between the category of simplicial sets and the category of semisimplicial sets. The functor sending a semisimplicial set $\cD$ to the free simplicial set generated by $\cD$ is the left adjoint; accordingly, one has a counit map $X_+ \to X$. 
 
\begin{notation}
Let $X$ be an $\infty$-category and let $X_+ \to X$ be the counit of the adjunction. We naturally obtain a map of marked simplicial sets
	\eqnn
	\epsilon: (X_+,\cE) \to X^\natural.
	\eqnd
\end{notation}

\begin{remark}\label{remark. epsilon is cartesian equivalence}
The main result of~\cite{tanaka-non-strict}, which we have cited as Theorem~\ref{theorem. hiro non strict}, implies that---when $X$ is an $\infty$-category---$\epsilon$ is a Cartesian equivalence of marked simplicial sets. That is, the induced map
	\eqnn
	\map^\sharp(X^\natural, -)
	\to
	\map^\sharp( (X_+,\cE),-)
	\eqnd
is a homotopy equivalence of Kan complexes for any $\infty$-category plugged into the target $-$. (This informally states that $X^\natural$ and $(X_+,\cE)$ corepresent the same functor when tested only against $\infty$-categories.) This is of course equivalent to the formal statement that $X$ is the localization of (an $\infty$-category replacement of) $X_+$ along the edges $\cE$ (see Remark~\ref{remark. marked simplicial sets model localizations}) but the explicit combinatorics of marked simplicial sets brings added power.
\end{remark}

In this work, we will find ourselves in the following formal setting: We will have two simplicial sets $X$ and $Y$, and a map of semisimplicial sets
	\eqnn
	h_{\semi}: X \to Y.
	\eqnd
The geometry will allow us to conclude that $h_{\semi}$ sends degenerate edges of $X$ to equivalences in $Y$. Now, by adjunction, $h_{\semi}$ induces a map of simplicial sets
	\eqnn
	(h_{\semi})_+: X_+ \to Y.
	\eqnd
Because $h_{\semi}$ sends degeneracies of $X$ to equivalences in $Y$, we see that $(h_{\semi})_+$ is a map of marked simplicial sets with domain $(X_+,\cE)$; by Remark~\ref{remark. epsilon is cartesian equivalence}, we are guaranteed a unique (up to contractible choice) map of simplicial sets
	\eqn\label{eqn. h from h semi}
	h: X \to Y
	\eqnd
equipped with a homotopy-commuting diagram
	\eqn\label{eqn. h semi to h}
	\xymatrix{
	X_+ \ar[r]^{(h_{\semi})_+} \ar[d]_{\epsilon} & Y \\
	X \ar[ur]_{h} & .
	}
	\eqnd
\begin{notation}[$H$]
\label{notation. H homotopy}
Concretely, this triangle encodes a map of simplicial sets 
	\eqnn
	H: \Delta^1 \times X_+ \to Y
	\eqnd
whose restriction to $\Delta^{\{0\}} \times X_+$ is equal to $(h_{\semi})_+$, to $\Delta^{\{1\}} \times X_+$ is equal to $h \circ \epsilon$, and for which every edge of the form $(e,e')$ is sent to an equivalence in $Y$ (see Notation~\ref{notation. map sharp}). 
\end{notation}

\begin{remark}\label{remark. alpha semi is easier to work with}
This homotopy $H$ empowers us to do the following. $h$ is a priori completely abstract; on the other hand, for every simplex $f : \Delta^k \to X$ of $X$, consider the natural simplex $f_+: \Delta^k \to X_+$ satisfying $\epsilon f_+ = f$. ($f_+$ is the image of $f$ under the unit map $X \to X_+$ of semisimplicial sets.) $H$ allows us to  homotope $h(f)$ to $(h_{\semi})_+(f_+) = h_{\semi}(f)$ in a way respecting faces.
In fact, we will see that $H$ allows us to replace arguments involving $h$ with those involving the more explicit $h_\semi$.
\end{remark}

Our proof of Theorem~\ref{theorem. localization} will depend on the following Lemma. To state it, let $W$, $X$ and $Y$ be $\infty$-categories and fix functors $i: W \to Y$ and $j: W \to X$. Fix also a semisimplicial set map $h_{\semi}: X \to Y$ sending degenerate edges of $X$ to equivalences in $Y$. (This in particular gives rise to a simplicial set map $h$ as in~\eqref{eqn. h from h semi}.) The Lemma allows us to construct a homotopy-coherent triangle
	\eqnn
	\xymatrix{
	W \ar[r]^j \ar[dr]_i & X \ar[d]^h \\
	& Y.
	}
	\eqnd

\begin{remark}
Note that the case $W=X$ and $j=\id_X$ 
\end{remark}

\begin{lemma}\label{lemma. making a homotopy}
Suppose that for every simplex $f: \Delta^k \to W$, one may exhibit a map 
	\eqnn
	G_f : \Delta^1 \times \Delta^k \to Y
	\eqnd
in such a way that the collection $\{G_f\}_{k\geq 0, f \in W_k}$ satisfies the following:
	\enum[(i)]
	\item\label{item. Gf homotopes what we want} ($G_f$ homotopes between the maps we want.) For each $f$, $G_f|_{\Delta^{\{0\}} \times \Delta^k} = i(f)$ 
		and 
		$G_f|_{\Delta^{\{1\}} \times \Delta^k} = h_{\semi}\circ j(f)$. Moreover, every edge of the form $(e, e')$ with $e$ an arbitrary edge of $\Delta^1$, and $e'$ degenerate in $\Delta^k$, is sent to an equivalence in $Y$.
	\item\label{item. Gf respects face maps} ($G_f$ respects face maps.) For each $f$ and every $0 \leq i \leq k$, $G_f|_{\Delta^1 \times \del_i \Delta^k} = G_{\del_i f}$. 
	\enumd
Then there exists a homotopy between $i$ and $h \circ j$.
\end{lemma}

\begin{proof}
Because the $G_f$ respect faces, the collection $\{G_f\}$ glues together to define a map of simplicial sets
	\eqnn
	G: \Delta^1 \times W_+ \to Y.
	\eqnd
Letting $\epsilon_{W} : W_+ \to W$ denote the counit, we note that $G$ restricted to $\Delta^{\{0\}} \times W_+$ agrees with $i \circ \epsilon_{W}$.  
On the other hand, let $j_+: W_+ \to X_+$ denote the map of simplicial sets induced by applying the $+$ functor to $j$. Then the restriction of $G$ to  $\Delta^{\{1\}} \times W_+$ agrees with $(h_{\semi})_+ \circ j_+$ because each $G_f$ restricts to $h_{\semi} \circ j(f)$ by hypothesis. 

Having spelled out the consequences of Conditions~\eqref{item. Gf homotopes what we want} and~\eqref{item. Gf respects face maps}, the rest of the proof is formal. The definition of $H$ (Notation~\ref{notation. H homotopy}) tells us
	\eqnn
	H|_{\Delta^{\{0\}} \times X_+} \circ j_+
	=
	(h_{\semi})_+ \circ j_+
	\eqnd
so we may concatenate the two homotopies $G$ and $H$ as follows:
	\eqnn
	G \bigcup_{(h_{\semi})_+ \circ j_+} \left( H \circ (\id_{\Delta^1} \times j_+) \right) : \Lambda^2_1 \times W_+ \to Y.
	\eqnd
Because $G$ and $H$ are homotopies (meaning the non-degenerate edges of $\Lambda^2_1$ are taken to equivalences in $Y$), this defines a map
	\eqnn
	\Lambda^2_1 \to \map^\sharp((W_+,\cE), Y^{\natural})
	\eqnd
where $(W_+,\cE)$ is the marked simplicial set defined in Notation~\ref{notation. plus E marked}. 
Because $Y$ is an $\infty$-category, the codomain $\map^\sharp((W_+,\cE), Y^{\natural})$ is a Kan complex, so we may fill the horn $\Lambda^2_1$; the resulting map
	\eqnn
	\Delta^1 \cong \del_1\Delta^2 = \Delta^{\{0,2\}} \subset \Delta^2 \to \map^\sharp((W_+,\cE),Y^{\natural})
	\eqnd
exhibits a homotopy from $i \circ \epsilon_{W}$ to $h \circ \epsilon \circ j_+ = h \circ j \circ \epsilon_{W}$. (This last equality uses that $j$ is a map of simplicial sets.) By Remark~\ref{remark. epsilon is cartesian equivalence}, we conclude there is a homotopy between $i$ and $h \circ j$, as desired.
\end{proof}

\subsection{Understanding subdivisions of subdivisions}
In Section~\ref{section. proof of M to m} we will confront the subdivision of a subdivision of a simplex. So we collect a convenient combinatorial description.

\begin{notation}
We let 	
	\eqnn
	\cP'_{\linear}(Q)
	\eqnd
be the subset of the power set of $Q$ consisting of non-empty subsets of $Q$ that are linear under the poset relation of $Q$. We treat $\cP'_{\linear}(Q)$ as a poset via the inclusion relation of subsets.
\end{notation}

We leave the proof of the following to the reader. 

\begin{prop}
Let $Q$ be a poset.
	\eqnn
	\subdiv(Q) \cong N(\cP'_{\linear}(Q)).
	\eqnd
\end{prop}

In particular, we note that $\subdiv(Q)$ is (the nerve of) a poset, so we have a recursive description of iterated subdivisions of any poset.

\begin{example}
\label{example. subdivision of subdivision}
Fix the $I$-simplex $\Delta^I \cong N(I)$. Then a vertex of $\subdiv(\subdiv(I))$ is the data of a strictly increasing chain
	\eqnn
	\vec A:= A_0 \subset \ldots \subset A_k 
	\eqnd
of non-empty subsets of $I$ -- i.e., an injective map of posets
	\eqnn
	[k] \xrightarrow{\vec A}  \cP'(I).
	\eqnd
An $l$-simplex of $\subdiv(\subdiv(I))$ is the data of a diagram of posets
	\eqnn
	[k_0] \into [k_1] \into \ldots \into [k_l] \xrightarrow{\vec A} \cP'(I).
	\eqnd
Such a simplex is degenerate precisely when at least one of the inclusions $[k_i] \into [k_{i+1}]$ is an isomorphism.
\end{example}

\clearpage
\section{Simplices and collarings}
 
Much of homotopy theory is organized by the combinatorics of simplices and the ability to fill horns. However, we encounter two issues in trying to corral sectorial data into the usual simplicial framework. First, there is no natural sense in which a Liouville structure on $M \times T^*\Delta^n$ ``restricts'' to a Liouville structure on $M \times T^*\del_i \Delta^n$ without choosing normal coordinates to the face $\del_i \Delta^n \subset \Delta^n$. Second, we would also like to naturally embed stabilizations of $M$ and of $M \times T^*\del_i \Delta^n$ into $M \times T^* \Delta^n$. For this, we must choose once and for all a {\em collaring} on our simplices, which allow us to determine compatible normal coordinates along all faces of all simplices.

But introducing collarings creates a degeneracy wrinkle: 
While it is common in most situations to observe that collarings can be chosen compatibly with face maps, standard degeneracy maps may not obviously respect collarings. So Theorems~\ref{theorem. steimle} and~\ref{theorem. hiro non strict} come to the rescue. Thanks to them, we may use collars to encode only semisimplicial objects, and then a posteriori promote semisimplicial data to simplicial data.

This sort of gymnastics is common when producing $\infty$-categories of geometric origins~\cite{tanaka-pairing}.

Here, we make explicit the existence of the necessary collarings. 
Now, our particular construction of collarings (Proposition~\ref{prop. collaring convention exists}) forces us to declare the following unconventional convention:

\subsection{An unconventional convention}
\label{section. -1 convention for simplex}

\begin{notation}\label{notation. I simplex}
In all following sections, for any finite, non-empty, linearly ordered set $I$, the (non-standard) {\em $I$-simplex} is the topological space
	\eqn\label{eqn. simplex}
	\Delta^I := \{ (\sum_{i\in I} x_i = 1) \& (\forall i \in I ,\, x_i \geq -1) \} \subset \RR^{I}.
	\eqnd
\end{notation}
 
The usual definition of the standard $I$-simplex requires that all coordinate be non-negative. Allowing for negative values allows for us to define collarings in Construction~\ref{construction. simplex collars}. 

The usual inclusions of coordinate hyperplanes induce inclusions of lower-dimensinal simplices as faces of higher-dimensional simplices, and hence define a semisimplicial collection of smooth manifolds with corners. 

\begin{remark}
By induction, one may choose homeomorphisms between the non-standard $k$-simplices in this work to the standard $k$-simplices in such a way that all face maps are respected. Fixing such a choice, one obtains a semisimplicial isomorphism between all natural semisimplicial sets defined in terms of our non-standard simplices, and in terms of the standard simplices. (For example, the singular complex of maps $\Delta^k \to \cE$ for some topological space $\cE$.) 

In short, the particular non-standard definition of the $k$-simplex has no deep philosophical value, and we would strongly encourage the reader to just pretend all simplices are standard anyway. We imagine a collaring convention exists surely for the standard simplices; we simply didn't construct one.
\end{remark}

\subsection{Collaring conventions}

\begin{defn}\label{defn. collaring convention}
Consider the data of:
\begin{itemize}
\item For every pair of linearly ordered, finite, non-empty sets $I$, $J$, and for every injective, order-preserving map $a: I \to J$, a smooth, open embedding
	\eqnn
	\eta = \eta_a: \Delta^I \times [-1,\epsilon)^{J \setminus a(I)} \to \Delta^J
	\eqnd
whose restriction to $\Delta^I \times \{0\}$ is the standard simplicial inclusion induced by $a$.
\end{itemize}
The data $\{\eta_a\}_{(I,J,a)}$ (a choice of $\eta_a$ for every $a: I \to J$) are called a  {\em collaring convention} if:
For every pair of injective, order-preserving maps $a: I \to J, b: J \to K$, the following diagram commutes:
	\eqnn
	\xymatrix{
	\Delta^I \times [-1,\epsilon)^{K \setminus ba(I)} \ar[drr]_{\eta_{b \circ a}} \ar@{=}[r]
		& \left(\Delta^I \times [-1,\epsilon)^{J \setminus a(I)}\right) \times [-1,\epsilon)^{K \setminus b(J)} \ar[r]^-{\eta_a \times \id}
		& \Delta^J \times [-1,\epsilon)^{K \setminus b(J)} \ar[d]^{\eta_{b}}
		\\
	&&\Delta^K
	}
	\eqnd
\end{defn}

\begin{remark}
The interval $[-1,\epsilon)$ is chosen to be compatible with our stabilization operation. If our stabilization were defined by $-\times T^*[a,b]$, the collaring would have domain $\Delta^I \times [a,b+\epsilon)^{c}$. 

Note also that the definition of collaring guarantees that $\eta$ is a map respecting the natural stratifications.
\end{remark}

\begin{warning}
The strict commutativity of the diagram in Definition~\ref{defn. collaring convention} is a subtle but important point. Indeed, the most ``natural'' way to collar $\Delta^I$ inside $\Delta^J$ may be by taking join coordinates to realize $\Delta^I \star \Delta^{J \setminus I} \cong \Delta^J$. However, join constructions are only associative up to natural isomorphism.
\end{warning}

For any finite set $I$ and any subset $J$, we construct inclusions
\[
  \Phi_{I \subset J} \colon \Delta^I \times [-1,0]^{J \setminus I} \into \Delta^J
\]
as follows.

\begin{construction}
  \label{construction. simplex collars}
  First, note that is is easy to build piecewise smooth collars $\phi_{I \subset J}$. Indeed, given $x \in \Delta^I$, we can define $\phi_{I \subset J}(x,t) = \phi_{S_x}(x,t)$, where $S_x = \{i \mid x_i > 0\}$ and 
  \begin{equation}\nonumber
    \left(\phi_S(x, t)\right)_j = \begin{cases}
      t_j & j \not\in I \\
      x_j & j \in I \setminus S \\
      b_S(x,t)x_j & j \in S.
    \end{cases}
  \end{equation}
  Here, $b_S(x,t) = \frac{1 - \sum t_j - \sum_{i \not\in S}x_i}{\sum_{i \in S} x_i}$ is the unique scalar for which $\phi_{S_x}(x,t)$ lands in $\Delta^J$. The main task is thus to smooth this map while ensuring that the result is still coherent and is still an embedding.

  To that end, suppose that we had a permutation-equivariant partition of unity $\left\{g_S\right\}_{S \subset J}$ on $\Delta^J$ with the following two properties.
    \begin{enumerate}[i]
    \item\label{it:simplex_partition_support} If $g_S(x) > 0$, then $x_s > 0$ for all $s \in S$.
    \item\label{it:simplex_partition_negative_indepencence} $g_S(x)$ depends only on the positive components of $x$. In particular, its dependence on the negative components factors through their sum.
  \end{enumerate}
  Consider a path $\mathbf t \colon[0,1] \to [-1,0]^{J \setminus I}$. For each $\tau \in [0,1]$ and $S \subset I$, we get an embedding
  \[
    \phi_S(\bullet, \mathbf t(\tau)) \colon \Delta^I \cap \supp(g_S) \into \Delta^I_{\mathbf t(\tau)},
  \]
  where $\Delta^I_{\mathbf t(\tau)} = \left\{x \in \Delta^J \mid x_j = \mathbf t_j(\tau)\right\}$. We can combine these into a smooth vector field along $\Delta^I_{\mathbf t(\tau)}$ by taking the weighted sum
  \eqn\label{eqn. bullet formula}
    \xi(\tau) = \sum_{S \subset I} g_S\cdot\frac{\partial}{\partial\tau} \phi_S(\bullet, \mathbf t(\tau)).
  \eqnd
  (Note that a given $x\in\Delta^I_{\mathbf t(\tau)}$ will come from different values of $\bullet$ for different $S$. See Remark~\ref{remark. bullet formula}.) Define the collar $\Phi_{I \subset J}(\bullet, \mathbf t(1))$ to be the time $1$ flow of this vector field $\xi$. Then $\Phi_{I \subset J}$ is coherent with respect to face maps by condition \eqref{it:simplex_partition_support}, and it is independent of the path $\mathbf t$ because both $g_S$ and $\phi_S$ factor through $\mathbf t \mapsto \sum \mathbf t_j$. 

  It thus suffices to construct the partion of unity $\left\{g_S\right\}$. Let us begin by fixing some notation. Given real numbers $a < b$, we'll write $\kappa_a^b \colon \RR \to [0,1]$ for a nondecreasing cutoff function with $\kappa_a^b(a) = 0$ and $\kappa_a^b(b) = 1$. For a subset $I \subset J$ and a point $x \in \Delta^J$, we'll write $x_I = \sum_{i \in I} x_i$.

  For $S \subset J$ of size $n > 0$, define
  \[
    \kappa_S(x) := \kappa_{1-\frac{1}{2^{n-1}}}^{1-\frac{1}{2^n}}(x_S) \prod_{s \subset S}\kappa_{\frac{1}{2^{n+1}}}^{\frac{1}{2^n}}(x_s),
  \]
  which is nonzero if and only if $x_S$ is big and all $x_s$ are not too small. These functions are positive in the right places and have the correct dependence on the signs of the components of $x$, but it is possible for them to overlap and sum to a quantity larger than $1$. To repair this, set
  \[
    f_S := \kappa_S \cdot \prod_{\substack{S' \ne S \\ \lvert S' \rvert = n}}\left( 1 - \kappa_{S'} \right),
  \]
  which satisfies $\sum_{\lvert S \rvert = n} f_S \le 1$. Finally, we can inductively define our partition of unity by $g_\emptyset = 0$ and
  \[
    g_S = f_S \cdot \Biggl( 1 - \sum_{\lvert S' \rvert < n} g_{S'} \Biggr).
  \]

  Let us check the necessary properties. Permutation-equivariance and properties \eqref{it:simplex_partition_support} and \eqref{it:simplex_partition_negative_indepencence} are automatic from the construction of the $\kappa_S$, so the main task is to show that $\{g_S\}$ is indeed a partition of unity. To that end, consider a point $x \in \Delta^J$, set $S = S_x$ to be the set of positive components, and set $n = \lvert S \rvert$.

  If $\kappa_S(x) = 1$, then we are done, since any other $S'$ of size $n$ must index a negative component, so $\kappa_{S'}$ must vanish. This implies that $f_S(x) = 1$, and so $\sum_{\lvert S' \rvert \le n} g_{S'}(x) = 1$.

  If $\kappa_S(x) < 1$, then since $x_S \ge 1 > 1 - \frac{1}{2^n}$ there must be some component $s_1 \in S$ so that $x_{s_1} < \frac{1}{2^n}$. That means $x_{S \setminus \{x_1\}} > 1-x_{s_1} > 1 - \frac{1}{2^{n-1}}$, and we can repeat the argument with $S \setminus \{x_1\}$ (replacing ``negative component'' in the previous paragraph with ``component smaller than $\frac{1}{2^{n-1+1}} = \frac{1}{2^n}$''). The argument terminates for some nonempty $S \setminus \{x_1, \dotsc, x_m\}$, since otherwise we would have
  \[
    x_S < \frac1{2^n} + \frac{1}{2^{n-1}} + \dotsb + \frac{1}{2^1},
  \]
  contradicting $x_S \ge1$.
\end{construction}

\begin{remark}
\label{remark. bullet formula}
Let us write out the precise (but very cluttered) meaning of~\eqref{eqn. bullet formula}. For $x\in\Delta^I_{\mathbf{t}(\tau)}$, we have $x = \phi_S\left(\bigl[\phi_S(\bullet,\mathbf{t}(\tau))\bigr]^{-1}(x), \mathbf{t}(\tau)\right)$. The weighted sum is
  \[
    \sum_{S \subset I} g_S(x, \mathbf t(\tau)) \cdot
    D_2\phi_S\left(\bigl[\phi_S(\bullet,\mathbf{t}(\tau))\bigr]^{-1}(x), \mathbf{t}(\tau)\right)\circ\frac{\partial\mathbf t}{\partial\tau},
  \]
  where $D_2$ is the partial derivative with respect to the second component.
\end{remark}

Thus, we have shown

\begin{prop}
\label{prop. collaring convention exists}
A collaring convention exists.
\end{prop}

\begin{choice}\label{choice. collaring}
From hereon, we fix a collaring convention.
\end{choice}

\begin{remark}
A collaring convention renders every simplex $S = \Delta^k$ a collared manifold-with-corners in the sense of Definition~\ref{defn. collaring of S}. The convention further provides the data showing that the collarings are compatible with the usual face maps.
\end{remark}

\begin{remark}
Thanks to a choice of collaring convention, we may ask geometric structure on $M \times T^*\Delta^k$ known to {\em split} with respect to the collaring near the boundary---that is, we may demand that structures be a direct product of a structure on $M \times T^*\Delta^A$, and on $T^*[-1,\epsilon)^i$.  (Here, $A$ is some subset of $[k]$.) 
\end{remark}

\subsection{Linearity, locally}
At the heart of smooth approximation theorems for $\infty$-dimensional smooth spaces is the ability to reduce local problems to linear ones. (See for example~\cite{oh-tanaka-smooth-approximation} and~\cite{kihara-smooth-homotopy}.) More precisely, suppose we have a continuous collection in an infinite-dimensional space, and we seek to deform it to a smooth collection. A tried-and-true strategy is to write the collection locally as a continuous collection of functions into a linear object (for examples, as a collection of sections of some vector bundle, e.g., as a collection of real-valued functions) and then use the classical smooth approximation theorem for functions to deform these functions into a smooth collection.

We now prove that $\embtop_{\liou}$ is amenable to this strategy.

\begin{prop}
\label{prop. families of Liouvilles are families of Hamiltonians}
Let $S$ be a manifold (possibly with corners) and let $j: S \to \embtop_{\liou}(M,N)$ be a  continuous and collared map. ($j$ is not assumed smooth.) Then, for every $s_0 \in S$, there exist \enum[(i)]
\item some neighborhood $U$ with $s_0 \in U$ and  
\item a continuous function $H: U \times N \to \RR$, whose restriction to each $\{s\} \times N \subset U \times N$ is smooth,
\enumd
so that for all $s \in U$, $j_{s} = \flow_1^{X_{H_{s}}} \circ j_{s_0}$. In other words, $j_{s}$ is obtained from $j_{s_0}$ by post-composing with the time-1 flow of some time-independent but $s$-dependent family of Hamiltonians.
\end{prop}

\begin{proof}
Fixing $s_0 \in S$, we may as well assume that $M \subset N$ and that $j_{s_0}$ is the canonical inclusion of the subset. Thus the graph of $j_{s_0}$ may be thought of as a subset of the diagonal of $N \times N$, which in particular has a normal neighborhood identified with an open subset $V \subset T^*N$.  Under this identification, the image of $j_{s_0}$ is identified with the graph of the zero 1-form of $N$ (restricted to $M$).

We now claim that for a small-enough neighborhood $U$ with $s_0 \in U \subset S$, and for all $s \in U$, the graph of $j_{s}$ is contained in $V$. Note that if $M$ and $N$ are compact, this is a consequence of the weak $C^\infty$-topology on smooth mapping spaces. But both $M$ and $N$ are non-compact. So invoke the continuity of $j$ with respect to the colimit topology of $\embtop_{\liou}(M,N)$; then the standard yoga of shrinking of $U$, choosing a single $K$ as in Remark~\ref{remark. sectorial embeddings are eventually conical}, and using the definition of the weak $C^\infty$-topology, we may assume that for all $s \in U$, (i) the graph of $j_s(K)$ is contained in $V$, and (ii) $j_s^*\lambda^N = \lambda^M + dh_s$ where the support of $h_s$ is inside $K$.

It is now a standard fact of symplectic geometry that, interpreting the image of $j_s(K)$ as the graph of a 1-form, we may conclude that the image of each $j_s(K)$ is the graph of an {\em exact} 1-form (Remark~\ref{remark. exact maps are exact 1 form graphs}). 

These exact 1-forms are $C^0$-dependent on $s \in S$ by hypothesis, so by integrating, we may in fact choose a single continuous function $g: U \times K \to \RR$ for which the graph of $d_M(g|_{\{s\} \times K})$ is equal to the graph of $j_s(K)$ (when both are interpreted as subsets of $V \subset T^*N$).

(The following works on every path-connected component of $K$, so we may assume $K$ is path-connected and fix a base point $x_0 \in K$. Given that $j_s(K) = graph(\alpha_s)$ and $\alpha_s$ is exact, choose a smooth function $g_s$ so that $dg_s = \alpha_s$ and $g_s(x_0) = 0$. This recovers $g_s(x) = \int_\gamma \alpha_s$ by choosing any smooth $\gamma$ with endpoints $x_0$ and $x$. Because $\alpha_s$ depends continuously on $s$ in the $C^\infty$-topology, $g(s,x):=g_s(x)$ is a continuous function.)

Of course, the graph of an exact 1-form $\alpha = dg$ is the result of flowing the zero-section for time 1 along the Hamiltonian defined by $g$. 

For each $s$, having fixed the value of $g_s$ on $K$, there is a unique extension of $g_s$ to $\widetilde g_s$ on $M$ compatible with that Liouville flow of $N$ outside $K$ (so that, in particular, the time-1 flow $j_s$ commutes with the Liouville flow outside of $K$). Then $H(x,s) = \widetilde g_s(x)$ is smooth in $x$ and continuous in $(s,x)$; and we are finished.
\end{proof}

\begin{remark}
\label{remark. exact maps are exact 1 form graphs}
 
Let us explain why an exact symplectic embedding near the identity is the graph (in the cotangent bundle) of an exact 1-form. Endow $M \times M$ with the 1-form $\lambda \oplus -\lambda$. On the other hand, because the diagonal $\Delta_M \subset M \times M$ is a Lagrangian, we have a Weinstein neighborhood $T^*\Delta_M \supset U \into M \times M$, so we have another Liouville structure in a neighborhood of $\Delta_M$ given by pushing forward ${\bf p}d{\bf q}|_U$. Now note that along the diagonal $\Delta_M$,  $(\lambda \oplus -\lambda) - {\bf p}d{\bf q} = 0$; in particular, the two Liouville forms differ by an exact form along the diagonal. Because we may choose $U$ to deformation retract to $\Delta_M$, we conclude the difference $(\lambda \oplus -\lambda) - {\bf p}d{\bf q}$ is exact on all of $U$. 

Now, let us use that the graph $\Graph(j)$ of an exact symplectic embedding $j: M \to M$ is an exact Lagrangian for the form $\lambda \oplus -\lambda$. (Since $j^*\lambda = \lambda + dg$ for some $g$, the restriction of $\lambda \oplus - \lambda$ to $\Graph(j)$ is also $dg$.)  By the previous paragraph, it follows that $\Graph(j)$ is an exact Lagrangian in $U \subset T^*\Delta_M$ as well. That $j$ is near to $\id_M$ means that $\Graph(j)$, treated as a brane in $T^*\Delta_M$, is the graph of a 1-form, while the exactness of the brane implies it is the graph of an {\em exact} 1-form. 
\end{remark}

\subsection{From continuous to smooth}
\label{section. continuous to smooth}
We collect some facts that allow us to pass between {\em continuous} families of smooth maps and {\em smooth} families of smooth maps. We will use this trick repeatedly in showing that horns can be filled in $\lioudelta$ later on (Theorem~\ref{theorem. weak kan})---by first constructing {\em continuous} data filling a horn (guaranteed by general properties of topological spaces like $\embtop_\liou(M,N)$), then homotoping this data to be smooth (because all data in $\lioudelta$ are required to be smooth).

\begin{remark}[From continuous to continuous and collared]
\label{remark. continuous to collared}
Fix a topological space $\cE$ and a continuous map $f: \Delta^k \to \cE$. Recall that $f$ is called {\em collared} if the value of $f$ is independent of the chosen coordinates normal to the faces and corners of $\Delta^k$ (Definition~\ref{defn. collared}). In particular, choose a continuous homotopy from $\Delta^k$ to itself which is equal to the identity along the boundary strata, but collapses the normal coordinates. (Such a homotopy is necessarily not smooth.) Pre-composing $f$ by this homotopy, we see that $f$ is  continuously homotopic rel boundary to a collared map.

Of course,  one may inductively choose these homotopies for the 1-simplex, then the 2-simplex, and so forth, so that these homotopies respect face maps (i.e., respect restriction to boundary faces).

Now consider the inclusion---of the semisimplicial set of all {\em collared} continuous maps in a fixed space $\cE$---into the semisimplicial set of all (not necessarily collared) continuous maps to $\cE$. By Steimle's theorem (Theorem~\ref{theorem. steimle}) the domain may be promoted to be a simplicial set, and by Theorem~\ref{theorem. hiro non strict}, this semisimplicial set map may be promoted to a map of simplicial sets up to homotopy. In particular, the previous paragraphs show that the induced map of simplicial sets is a weak homotopy equivalence.
\end{remark}

Our goal for Section~\ref{section. continuous to smooth} is to show that we can continuously homotope a continuous and collared map into a smooth and collared map (Proposition~\ref{prop. continuous to smooth}). We start with a Lemma, which crucially uses that we can convert small families of sectorial embeddings into linear data (Proposition~\ref{prop. families of Liouvilles are families of Hamiltonians}).

\begin{lemma}\label{lemma. smooth along faces to smooth along interior}
Suppose $j: \Delta^k \to \embtop_{\liou}(M,N)$ is a continuous map that is collared along the boundary strata of $\Delta^k$, and is smooth when restricted along any face of $\Delta^k$. Then there exists a homotopy of $j$, rel boundary, to a smooth and collared and continuous\footnote{For the terminology behind imposing both smoothness and continuity, see Remark~\ref{remark. smooth and continuous}.} map from $\Delta^k$ to $\embtop_{\liou}(M,N)$. 
\end{lemma}

\begin{proof}
First, recall that we may locally model $\embtop_{\liou}(M,N)$ as consisting of linear data as in Proposition~\ref{prop. families of Liouvilles are families of Hamiltonians} which---using the notation of loc. cit.---we now apply to the case $S = \Delta^k$. 

Choose a fine enough simplicial subdivision $\DD$ of $\Delta^k$ so that for every simplex $\sigma$ in the subdivision, the open star\footnote{Given a triangulation, recall that the open star of a subsimplex $\sigma$ is defined to be the union of the interior of all the simplices that contain $\sigma$ as a face (including $\sigma$ itself).} of $\sigma$ is contained in a neighborhood $U$ as in Proposition~\ref{prop. families of Liouvilles are families of Hamiltonians}. 

We construct the homotopy inductively, starting with the index $n=-1$, and setting $j^{(-1)} = j$. Assume $j^{(n)}$ has been given, with the property that in a neighborhood of any $n$-simplex of $\DD$, $j^{(n)}$ is smooth.

For every $(n+1)$-simplex $\sigma$ in $\DD$ and its open star $U_\sigma$, fix some point $s_0 \in \sigma$ so that we may treat the data of $j^{(n)}|_{U_\sigma}$ as a function $H_\sigma: U_\sigma \times N \to \RR$, so that $j^{(n)}(s) = \flow_1^{X_{H_{\sigma}}} \circ j^{(n)}_{s_0}$ for all $s \in U_\sigma$. As in Proposition~\ref{prop. families of Liouvilles are families of Hamiltonians}, $H_\sigma$ is continuous, but is smooth upon restriction to $\{s\} \times N$ for any $s \in U_\sigma$. By usual smooth approximation, we may continuously homotope $H_\sigma$ to another function $G_\sigma: U_\sigma \times N \to \RR$ so that 
\enum[(i)]
\item $G_\sigma$ is {\em smooth} in a tiny neighborhood $V_\sigma \subset\subset U_\sigma$ containing $\sigma$, meaning the induced $V_\sigma \times N \to \RR$ is smooth,
\item $G_\sigma$ is equal to $H_\sigma$ outside of $V_\sigma$, and on neighborhoods of any $m$-simplex for $m < n$. (More precisely, the homotopy can be chosen to be constant outside of $V_\sigma$ and on neighborhoods of all $m$-simplices.)  
\item For every $s \in U_\sigma$, the time-1 Hamiltonian flow of $G_\sigma(s,-): N \to \RR$ is a sectorial embedding for all $s \in U_\sigma$.
\item\label{item. continuous to smooth is filtered-continuous} Let $K \subset M$ be the compact set outside which $j^{(n)}|_{U_\sigma}$ is strictly compatible with $\lambda$. Then the sectorial embeddings induced by $G_\sigma(s,-)$ are also strictly compatible with $\lambda$ outside $K$.
\enumd 
We have a collection of functions $\{G_\sigma\}_{\sigma \in \DD_{n+1}}$, and by flowing by time 1, we in turn obtain a collection of maps 
	\eqnn
	\{j^{(n+1)}_\sigma: U_\sigma \to \embtop_{\liou}(M,N)\}_{\sigma \in \DD_{n+1}}.
	\eqnd
The fact that the value of $G_\sigma$ and $H_\sigma$ agree outside $V_\sigma$, and on neighborhoods of all lower-dimensional simplices, means that $j^{(n+1)}_\sigma$ agrees with $j^{(n)}$ in these regions; in particular, the collection $\{j^{(n+1)}_\sigma\}_{\sigma \in \DD_{n+1}}$ (and the collection of homotopies from $j^{(n)}|_{U_\sigma}$ to $j^{(n+1)}|_{U_\sigma}$) glues together to form a single map (and a single homotopy from $j^{(n)}$ to)
	\eqnn
	j^{(n+1)}: \Delta^k \to \embtop_{\liou}.
	\eqnd
Again because $j^{(n+1)}$ does not alter the value of $j^{(n)}$ near boundaries of open stars, it in particular does not change the value of $j^{(n)}$ near the boundary of $\Delta^k$ itself. This proves the ``rel boundary'' claim for every homotopy, and in particular, that each $j^{(n+1)}$ is collared. The continuity of $j^{(n+1)}$ is a result of Condition~\ref{item. continuous to smooth is filtered-continuous}.

After our final inductive step, we have arrived at a smooth map $j^{(k)}: \Delta^k \to \embtop_{\liou}(M,N)$ satisfying all the claims of the Lemma.
\end{proof}

\begin{prop}[From continuous and collared to smooth and collared]
\label{prop. continuous to smooth}
If $j: \Delta^k \to \embtop_{\liou}(M,N)$ is continuous, $j$ can be homotoped continuously to a collared and smooth and continuous map.

Moreover, suppose that the restriction of $j$ to some collection of faces of $\Delta^k$ is smooth and collared. Then $j$ may be homotoped {\em rel this collection of faces} to a collared and smooth and continuous  map.
\end{prop}

\begin{proof}
By Remark~\ref{remark. continuous to collared} we may suppose that $j$ is both continuous and collared. By induction on dimension of the faces, we may assume that the restrictions $j|_{\Delta^A}$ to the relevant faces of $\Delta^k$ are smooth, collared, and continuous by Lemma~\ref{lemma. smooth along faces to smooth along interior}. We are finished.
\end{proof}

We have the following:

\begin{prop}
\label{prop. coll emb spaces are equivalent}
Let $\emb_{\liou}(M,N)$ be the semisimplicial set given by (forgetting the degeneracies of) $\sing(\embtop_{\liou}(M,N))$ (Notation~\ref{notation. emb liou}). Let 
	\eqnn
	\emb_{\liou}^{\coll}(M,N)
	\eqnd
denote the sub-semisimplicial set consisting only of the collared continuous simplices, and
	\eqnn
	\emb_{\liou}^{\coll,C^\infty}(M,N)
	\eqnd
be the further sub-semisimplicial set consisting of {\em smooth} and collared and continuous simplices. Then each semisimplicial set may be promoted to be a Kan complex, and further, the inclusions of semisimplicial sets induce homotopy equivalences of Kan complexes.
\end{prop}

\begin{proof}
It is well-known that $\sing$ of any topological space is a Kan complex, so it is obvious that $\emb_{\liou}(M,N)$ may be promoted to be a Kan complex.

For $\emb_{\liou}^{\coll}(M,N)$, we may easily verify the hypotheses of Steimle's theorem (Theorem~\ref{theorem. steimle}). For example, for any $0$-simplex of the semisimplicial set,  the constant edge at that vertex realizes an idempotent edge. That this edge is a self-equivalence is verified by filling the horn in question by a non-collared simplex, then homotoping the non-collared simplex to a collared simplex rel the horn (using similar techniques as Remark~\ref{remark. continuous to collared}). The end result of this homotopy is the filler we seek. So $\emb_{\liou}^{\coll}(M,N)$ satisfies the hypotheses of Steimle's Theorem.

For $\emb_{\liou}^{\coll,C^\infty}(M,N)$, the constant edges again are idempotents (filled by constant 2-simplices). To see that these edges admit fillers of the relevant horns, use the previous paragraph to fill the horn with a collared and continuous simplex, then use Proposition~\ref{prop. continuous to smooth} rel the horn to homotope this filler to a collared, smooth, and continuous simplex. 

Finally, the inclusions of semisimplicial sets obviously respect the constant edges, so the hypotheses of Theorem~\ref{theorem. hiro non strict} are satisfied, and we obtain maps of simplicial sets
	\eqn\label{eqn. emb spaces equivalences}
	\emb_{\liou}^{\coll,C^\infty}(M,N)
	\to
	\emb_{\liou}^{\coll}(M,N) \to
	\emb_{\liou}(M,N).
	\eqnd
These maps are bijections on vertices. On higher simplices, Theorem~\ref{theorem. hiro non strict} guarantees that these simplicial maps respect the original semisimplicial maps up to homotopy. In particular, by repeated application of Lemma~\ref{lemma. smooth along faces to smooth along interior} and Remark~\ref{remark. continuous to collared}, we see that these inclusions all induce isomorphisms on all homotopy groups. This completes the proof.
\end{proof}

\subsection{The homotopy equivalence of different deformation embedding spaces}
\label{section. different deformation spaces}
Recall that we have defined five different spaces of morphisms between Liouville sectors. They are organized in~\eqref{eqn. diagram of embedding spaces}. The following proves they are all homotopy equivalent.

\begin{theorem}
\label{theorem. embedding spaces are deformation embedding spaces}
The right-pointing arrows in~\eqref{eqn. diagram of embedding spaces} are weak homotopy equivalences.
\end{theorem}

We will need the following version of Moser's theorem in proving Lemma~\ref{lemma. big maps shrink}.
\begin{lemma}
  \label{lem:moser_in_fam}

  Let \[\boldsymbol\lambda = (T_b, \lambda_{t,b}) \colon B \to \deflioumoore(N)\] 
    be a family of  
bordered (and not necessarily compactly supported) deformations of $N$ parametrized by some compact manifold-with-corners $B$, and let $E \subset N$ be a Liouville-invariant neighborhood of $\partial N$ on which $\boldsymbol\lambda$ is trivial. Then there is a family $\psi_{t,b} \colon \RR \times B \to \diff(N)$ which (i) restricts to $\id_E$, (ii) agrees with $\id_N$ for $t\le0$, (iii) is $t$-independent for $t \ge T_b$, and such that (iv) $\psi_{t,b}^*\lambda^N = \lambda_{t,b} + dh_{t,b}$, where $h_{t,b}$ is uniformly compactly supported.
\end{lemma}

\begin{proof}
  The proof is essentially the same as for the usual Moser theorem for Liouville manifolds \cite{cieliebak-eliashberg}, so we give only a sketch.
  The key point is that, by definition of $\bL(N)$, we can find smoothly varying choices of symplectization coordinates $r_{t,b}$, which induce smoothly varying contact structures on $\partial^\infty N$. Applying Gray's stability theorem in a $B$-family, we may immediately assume that $\lambda_{t,b} = \lambda^N$ outside a uniform compact set. The remaining isotopy is given by the flow of $\xi_{t,b}$, where $\imath_{\xi_{t,b}}d\lambda_{t,b} = \frac{d}{dt}\lambda_{t,b}$.
\end{proof}

\begin{lemma}
  \label{lemma. big maps shrink}
  Let $B$ be a compact manifold-with-corners, and let $C \subset B$ be a compact submanifold-with-corners which is a smooth deformation retract of a neighborhood of itself (e.g. a closed submanifold or union of disjoint boundary faces). Then any smooth family 
  \eqnn
  \Phi \colon (B, C) \to \left(\embtop_{\liou,\operatorname{Moore}}^{\defliou}(M,N), \embtop_{\liou,\operatorname{Moore}}^{\defliou,\cmpct}(M,N) \right)
  \eqnd
  is smoothly homotopic rel $C$ to to a family $B \to \embtop_{\liou,\operatorname{Moore}}^{\defliou,\cmpct}(M,N) $.
\end{lemma}
\begin{proof}[Proof of Lemma~\ref{lemma. big maps shrink}]
  After reparametrizing by a partial retraction, we may assume that $\Phi(b)$ lands in the compactly-supported-deformation-embedding space $ \embtop_{\liou,\operatorname{Moore}}^{\defliou,\cmpct}(M,N) $ on some regular neighborhood $U$ of $C$. We will explain how to homotope $\Phi$ on the complement $B^c = B \setminus U$; the extension to $U$ is obtained by interpolating the duration over which we perform each step of the construction (in a way that it only occurs in the region where the previous step was fully performed).

  Write $N_{t,b} = (N, \lambda_{t,b})$ for the family of Liouville sectors and $\Phi_\strict$ for the family of underlying strict embeddings $M \into N_{T_b,b}$. By post-composing $\Phi_\strict$ with a family of shrinking isotopies, we may assume that the image of $\Phi_\strict$ avoids $\partial N$, and by strictness this means that it avoids a $\lambda_{T_b,b}$-Liouville-invariant neighborhood of $\partial N$ for all $b \in B^c$. Since all deformations are bordered and $B$ is compact, this neighborhood can be made uniform in $B$, and in fact it can be taken to be $(\pi^{rc})^{-1}[0,1)$ for some corner-rounding projection $\pi^{rc}$.

  Since most of our remaining work will take place in this neighborhood, we will write $\pi^{rc}_{t,b}$ for the function $\pi^{rc}$ thought of as a corner-rounding projection on $N_{t,b}$. Write $F_{t,b}$ for the fiber of $\pi^{rc}_{t,b}$. By Lemma \ref{lem:moser_in_fam}, there is a smooth family of symplectomorphisms $\psi_{t,b}\colon F_{t,b} \cong F_{0,b}$ so that $\psi_{t,b}^*\lambda_{0,b} = \lambda_{t,b} + dh_{t,b}$, where $h_{t,b}$ is uniformly compactly supported. Crossing with $T^*(0,1)$, extend $\psi_{t,b}$ to a symplectomorphism
  \[
    \Psi_{t,b} \colon \left(\pi^{rc}_{t,b}\right)^{-1}T^*(0,1) \to \left(\pi^{rc}_{0,b}\right)^{-1}T^*(0,1).
  \]
  
  To turn off the deformation of $\partial N$ while preserving the embedding $\Phi_\strict$, we will need to reparametrize the deformation parameter $t$ as follows. For $\tau \in [0,1]$, write $s(t, \tau)$ for a concave function which is constant for $t \ge \frac1\tau - 1$ and has $s(t, \tau)=t$ for $t \le \frac1\tau - 2$. Then, for any $\tau$, $N_{s(t, \tau), b}$ is $t$-indepentent for $t \ge T_b+1$, so we begin by moving the deformation time $T_b$ to $T_b+1$. Next, perform the following $\tau$-parametrized modifications simultaneously.
  \begin{enumerate}[(i)]
    \item Replace $N_{t, b}$ with $N_{s(t, \tau), b}$ on the collar $\bigl( \pi^{rc}_{s(t, \tau), b} \bigr)^{-1}T^*[0,\frac14)$, so that at the end of the modification $\partial N_{s(t, 1),b} \equiv \partial N$.
    \item Replace the Liouville form $\lambda_{s(t, \tau),b}$ on $\bigl(\pi^{rc}_{s(t, \tau),b}\bigr)^{-1}T^*\left[\frac14,\frac12\right]$ with the movie $\lambda_{s(t, \tau),b} + d\tilde h_{t,\tau,b}$, where $\tilde h_{t,\tau,b}$ interpolates between $0$ near $q=\frac14$ and $h_{s(t,\tau),b} - (\psi_{t,b}^{-1}\psi_{s(t, \tau),b})^*h_{t,b}$ near $q=\frac12$.
    \item Replace the Liouville form on $\bigl(\pi^{rc}_{s(t, \tau),b}\bigr)^{-1}T^*\left[\frac12,1\right)$ with $\tilde\Psi_{t,\tau,b}^*\lambda_{t,b}$, where $\tilde\Psi_{t,\tau,b}$ is a diffeomorphism interpolating (by integrating an interpolating vector field) between $\Psi_{t,b}^{-1}\Psi_{s(t, \tau),b}$ near $q=\frac12$ and $\id_N$ for $q\ge\frac34$.
  \end{enumerate}

  After the above modification, we are in a setting where the family of deformation $N_{t,b}$ is trivial in the Liouville-invariant neighborhood $\bigl(\pi^{rc}\bigr)^{-1}T^*[0,\frac14)$ of $\partial N$, meaning that we have retracted our family of bordered deformation morphisms to a family of interior deformation morphisms. The passage from interior to compactly supported is another application of Lemma \ref{lem:moser_in_fam}.
\end{proof}

\begin{proof}[Proof of Theorem~\ref{theorem. embedding spaces are deformation embedding spaces}.]
Taking $B$ and $C$ to be a cylinder and a sphere, respectively, Lemma~\ref{lemma. big maps shrink} shows that the smooth homotopy groups of $\embtop_{\liou,\operatorname{Moore}}^{\defliou}(M,N)$ and $\embtop_{\liou,\operatorname{Moore}}^{\defliou,\cmpct}(M,N)$ are isomorphic. 

Using methods similar to the proof of Proposition~\ref{prop. coll emb spaces are equivalent}, we in fact have equivalences among all of the following:
\begin{itemize}
\item The Kan complex of continuous simplices in $\embtop_{\liou,\operatorname{Moore}}^{\defliou}(M,N)$ and $\embtop_{\liou,\operatorname{Moore}}^{\defliou,\cmpct}(M,N)$.
\item The Kan complex of collared simplices in $\embtop_{\liou,\operatorname{Moore}}^{\defliou}(M,N)$ and $\embtop_{\liou,\operatorname{Moore}}^{\defliou,\cmpct}(M,N)$.
\item The semisimplicial set (made into a Kan complex by Steimle's theorem) of smooth and collared and continuous simplices in $\embtop_{\liou,\operatorname{Moore}}^{\defliou}(M,N)$ and $\embtop_{\liou,\operatorname{Moore}}^{\defliou,\cmpct}(M,N)$.
\end{itemize}
(The first of these shows the weak homotopy equivalence, but we record the other equivalences for later use.)

Now, that two other edges in the square of~\eqref{eqn. diagram of embedding spaces} are weak homotopy equivalences implies that the bottom horizontal arrow must also be a weak homotopy equivalence. And again, we have a equivalences of the following Kan complexes:
\begin{itemize}
\item The Kan complex of continuous simplices in $\embtop_{\liou}^{\defliou}(M,N)$ and $\embtop_{\liou}^{\defliou,\cmpct}(M,N)$.
\item The Kan complex of collared simplices in $\embtop_{\liou}^{\defliou}(M,N)$ and $\embtop_{\liou}^{\defliou,\cmpct}(M,N)$.
\item The semisimplicial set (made into a Kan complex by Steimle's theorem) of smooth and collared and continuous simplices in $\embtop_{\liou}^{\defliou}(M,N)$ and $\embtop_{\liou}^{\defliou,\cmpct}(M,N)$.
\end{itemize}
\end{proof}

\begin{remark}
  One can construct a space of not-necessarily-exact, not-necessarily-bordered Liouville deformations still satisfying Condition~\ref{item. tameness}. With minor additions to the proof of Lemma \ref{lemma. big maps shrink}, the above arguments show that the resulting spaces of generalized deformation embeddings are homotopy equivalent to the standard ones.
\end{remark}

Lemma~\ref{lemma. big maps shrink} has another significant consequence.

\begin{defn}\label{defn. deformation equivalence}
Fix Liouville sectors $M$ and $N$. We say that a (bordered, but not necessarily compactly supported) deformation embedding $(f, \{\lambda^N_t\}_{t \in \Delta^1})$ is a {\em deformation equivalence} if there exists another (bordered, but not necessarily compactly supported) deformation embedding
	\eqnn
	(g, \{\lambda^M_t\}_{t \in \Delta^1})
	\eqnd
so that
\enum
\item The pair $(gf, \widetilde{g_*(\{\lambda^N_t\})} \star \{\lambda^M_t\})$ is homotopic to $\id_M$ and the constant deformation of $\lambda^M$ -- here, ``homotopic'' means ``through a smooth family of (bordered, but not necessarily compactly supported) deformation embeddings.'' In particular, $gf$ is smoothly isotopic to $\id_M$.

\item Likewise, we demand the pair $(fg, \widetilde{f_*(\{\lambda^M_t\})} \star \{\lambda^N_t\})$ is homotopic to $id_N$ with the constant $\lambda^N$ deformation, through a smooth family of (bordered, but not necessarily compactly supported) deformation embeddings.
\enumd
\end{defn}

\begin{remark}
Let us explain the notation in Definition~\ref{defn. deformation equivalence}. First, $\widetilde{g_*(\{\lambda^N_t\}_{t})}$ is an extension of the deformation on the image of $N$ to a deformation on all of $M$. 

Next, the notation $\lambda'_t \star \lambda_t$ denotes a family of Liouville forms obtained through concatenation -- first through the family $\lambda_t$, then through the family $\lambda'_t$. In particular, one may think of $\widetilde{g_*(\{\lambda^N_t\})} \star \{\lambda^M_t\}$ as a family parametrized by the horn $\Delta^1 \cup \Delta^1 \cong \Lambda^2_1$ comprising two edges in a 2-simplex. 
Then the homotopy required in each item of Definition~\ref{defn. deformation equivalence} ``fills in'' this horn to a 2-simplex $\Delta^2$, where the remaining edge in $\Delta^2$ is occupied by a constant family.

See also Remark~\ref{remark. why composition of families is annoying}.
\end{remark}

Definition~\ref{defn. deformation equivalence} is the notion of equivalence in any $\infty$-category whose morphisms are deformation embeddings (through deformations that are bordered but not necessarily compactly supported). Thus, it seems to be a far weaker relation to be deformation equivalent, than to be sectorially equivalent (Definition~\ref{defn. sectorial equivalence}). However:

\begin{corollary}\label{corollary. deformation equiv is sectorial equiv}
  Deformation-equivalent Liouville sectors are sectorially equivalent.
\end{corollary}
\begin{proof}
Lemma \ref{lemma. big maps shrink} supplies us with morphisms in each direction. The composition of these morphisms are homotopic to identity in $\embtop_{\liou}^{\defliou}$, so by Theorem~\ref{theorem. embedding spaces are deformation embedding spaces} they are homotopic to the identity in $\embtop_{\liou}^{\defliou,\cmpct}$.
\end{proof}

\clearpage
\section{\texorpdfstring{$\lioudelta$}{Liou Delta} and \texorpdfstring{$\lioudeltadef$}{Liou Delta Def}}
\label{section. lioudelta}
Informally, movies of maps (Definition~\ref{defn. movie of maps}) arise when we convert a family of morphisms into the data of a single morphism. Unlike the usual definition of homotopy, which traditionally converts a family of maps $X \to Y$ into a map obtained by replacing the domain (with $X \times [0,1]$), we may replace a family of embeddings with a single strict morphism by thickening both the domain and codomain to remain in the world of codimension zero embeddings. 

This allows us to construct semisimplicial sets $\lioudelta$ and $\lioudeltadef$ whose $n$-simplices are modeled by a diagram of strict sectorial embeddings in the shape of the barycentric subdivision of the $n$-simplex (Definition~\ref{defn. lioudelta}).  The difference between the two is that $\lioudeltadef$ allows for higher simplices to encode all bordered and collared deformations, while $\lioudelta \subset \lioudeltadef$ demands that deformations further be compactly supported.

The present section constructs $\lioudelta$ and $\lioudeltadef$ and proves some of their basic properties. We will later see that both may be promoted to be an $\infty$-category (\ref{section. lioudelta is oo cat}). 

\subsection{Collaring movies of deformation embeddings}

\begin{notation}\label{notation. A fold product of intervals}
For $A$ a set, we let the $A$-fold product $[0,1]^A$ denote as usual the collection of functions from $A$ to $[0,1]$. When $A$ is a finite set, this collection is naturally a topological space, and in fact, a smooth manifold with corners in the usual way. We will apply this notation when the base is any interval -- e.g., $[0,\epsilon)^A$.

We caution that $\Delta^I$, for $I$ a finite, non-empty, linearly ordered set, is the $I$-simplex; it is not an $I$-fold product. (See Notation~\ref{notation. I simplex}.)
\end{notation}

\begin{remark}\label{remark. faces of DeltaI objections}
Fix a $\Delta^I$-movie $M_I \times T^*\Delta^I$ (Definition~\ref{defn. movie object}). Let $A \subset I$ and let $\Delta^{A} \subset \Delta^I$ be the corresponding face.
Because the movie is assumed collared, we may write $M_I \times T^*\Delta^I$ as 
	$\left(M_I \times T^*\Delta^{A}\right) \tensor T^*[-1,\epsilon_{A}]^{I\setminus A}$ in a neighborhood of $M_I \times T^*\Delta^{A}$. 
	
	This decomposition uniquely determines a Liouville structure on $M_I \times T^*\Delta^{A}$ as well, and it follows that $M_I \times T^*\Delta^{A}$ is a $\Delta^{A}$-movie.
\end{remark}

In the following definition, all deformation embeddings are assumed to be bordered (Definition~\ref{defn. various deformation embeddings}) and collared (Definition~\ref{defn. collared}) so that we may safely speak of movies (Definition~\ref{defn. movie object}).

\begin{defn}\label{defn. collaring movie}
Fix an injective, order-preserving map $a: A \to I$, along with a $\Delta^I$-movie $M_I \times T^*\Delta^I $ and a $\Delta^{A}$-movie $M_{A} \times T^*\Delta^{A}$. A sectorial embedding
	\eqn\label{eqn. collared movie}
	\phi: (M_{A} \times T^*\Delta^{A}) \tensor T^*[-1,\epsilon]^{I \setminus A} 
	\to
	M_I \times T^*\Delta^I
	\eqnd
is called a {\em collaring movie} of deformation embeddings if
\enum
	\item For some $\epsilon_A$ with $0 < \epsilon_A < \epsilon$, there exists a map $\widetilde y$ such that $\phi$ can be written as a composition
	\eqnn
	\xymatrix{
	(M_{A} \times T^*\Delta^{A}) \tensor T^*[-1,\epsilon_A]^{I \setminus A}
	\ar[r]^-{\widetilde{y}\tensor \id} &
	(M_{I} \times T^*\Delta^{A}) \tensor T^*[-1,\epsilon_A]^{I \setminus A}
	\ar[rr]^-{\id_{M_{I}} \times T^*\eta_{a}}
	&& M_I \times T^*\Delta^I	 
	}
	\eqnd
	where $\eta_{a}$ is the collaring from Choice~\ref{choice. collaring} (restricted to $[-1,\epsilon_A] \subset [-1,\epsilon]$), and $T^*\eta_{a}$ is the induced smooth map on cotangent bundles. 
(Note that if $\widetilde{y}$ exists, it is unique.)
	\item $\widetilde{y}$ is a movie associated to some family of deformation embeddings 
		\eqn\label{eqn. y family from simplex}
		\{y^s: M_{A} \to M_I\}_{s \in \Delta^{A}}.
		\eqnd
		(See Definition~\ref{defn. movie of maps}.)
	\item Finally, we demand $\widetilde{y}$ itself is collared. This means that for any subset $A' \subset A$ (with induced injection $a': A' \to I$), there exists some $\epsilon_{A'}$ with $0<\epsilon_{A'} < \epsilon$ and a map 
	\eqnn
	\widetilde{y}_{a'}: M_{A} \times T^*\Delta^{A'} \to M_I \times T^*\Delta^{A'}
	\eqnd
so that $\widetilde{y}$---when restricted to a neighborhood of $\Delta^{A'}\times[-1,\epsilon_{A'})^{A \setminus A'} \subset \Delta^{A}$---can be written as a composition
	\eqn
	\label{eqn. family of embeddings}
	\xymatrix{
	(M_{A} \times T^*\Delta^{A'}) \tensor T^*[-1,\epsilon_{A'}]^{I \setminus A'}
	\ar[rr]^--{\widetilde{y}_{a'}\tensor \id} &&
	(M_{I} \times T^*\Delta^{A'}) \tensor T^*[-1,\epsilon_{A'}]^{I \setminus A'}
	\ar[d]^-{\id_{M_{I}} \times T^*\eta_{a'}} \\
	&& M_I \times T^*\Delta^{I}	 .
	}
	\eqnd
(Equivalently, the family $\{y^s\}_{s \in \Delta^{A}}$ is collared near every face $\Delta^{A'}$ of $\Delta^{A}$.)
	\enumd
\end{defn}

\begin{remark}\label{remark. movies give movies}
Suppose that we are given a collaring movie of deformation embeddings~\eqref{eqn. collared movie}, and that we further know that $M_I \times T^*\Delta^I$ and $M_{A} \times T^*\Delta^{A}$ are both movies of {\em compactly supported} deformations (Definition~\ref{defn. movie object}). Then $\phi$ defines a continuous, smooth, and collared map from $\Delta^{A}$ to the space of sectorial embeddings from $M_{A}$ to $M_{I}$ (Remark~\ref{remark. movies give smooth maps to embedding spaces}). 
\end{remark}

\subsection{\texorpdfstring{$\lioudeltadef$}{LiouDeltaDef} (as a semisimplicial set)}

\begin{defn}\label{defn. lioudeltadef}
We define a semisimplicial set $\lioudeltadef$ as follows.
An $I$-simplex of $\lioudeltadef$ is the data of:
	\begin{itemize}
	\item For every non-empty subset ${A} \subset I$, a $\Delta^{A}$-movie $M_{A} \times T^*\Delta^{A}$. (Note that $A$ may equal $I$.)
	\item For every proper subset inclusion ${A'} \subset{A}$ of non-empty subsets of $I$, a strict sectorial embedding
		\eqnn
		\phi_{{A'} \subset {A}}: M_{{A'}}  \times T^*\Delta^{A'} \tensor T^*[-1,\epsilon_{A,A'}]^{{A} \setminus {A'}}
		\to
		M_{{A}} \times T^*\Delta^{A}
		\eqnd
	such that $\phi_{{A'} \subset {A}}$ is a collaring movie of deformation embeddings (Definition~\ref{defn. collaring movie}).\footnote{Note that $\phi_{A' \subset A}$ thus defines a family of smooth embeddings $M_A' \to M_A$ parametrized by $\Delta^{A'}$. Also note that $[-1,\epsilon]^{A\setminus A'}$ is a cube of dimension $\#(A \setminus A')$ (Notation~\ref{notation. A fold product of intervals}).}
	\end{itemize}
satisfying the following conditions.

\newenvironment{lioudelta-props}{
	  \renewcommand*{\theenumi}{(LD\arabic{enumi})}
	  \renewcommand*{\labelenumi}{(LD\arabic{enumi})}
	  \enumerate
	}{
	  \endenumerate
}

\begin{lioudelta-props}
		\item\label{item. lioudelta composition} For every string of proper inclusions ${A''} \subset {A'} \subset {A}$, we have
	\eqnn
	\phi_{{A''} \subset {A}} = \phi_{{A'} \subset {A}} \circ (\phi_{{A''} \subset {A'}} \times \id_{T^*[-1,\epsilon_{A',A''}]^{{A'} \setminus {A''}}})
	\eqnd
	where both sides are defined.
		\item\label{item. lioudelta is max localizing} If $\max A = \max A'$, then the function $M_{A'} \to M_{A}$  is a strict isomorphism of Liouville sectors. (Here, the function is the one determined by the collaring near the vertex $\Delta^{ \{\max A\}} \subset \Delta^{A'} \subset \Delta^{A} \subset \Delta^{I}$. We also note that the Liouville forms on both domain and codomain are those determined by the $\Delta$-movie structures near the vertex $\max A$.)
\end{lioudelta-props}

\begin{remark}
Note that given an $I$-simplex of $\lioudeltadef$, for all non-empty subsets $A,A' \neq I$, we have that $\dim M_A = \dim M_{A'}$. (By definition of dimension of a manifold, the empty manifold is a manifold of all dimensions.)
\end{remark}

\begin{remark}
The collaring on our movies---and the collarings on all the simplices of $\lioudeltadef$---guarantees that our movies are sectors, as we saw in Proposition~\ref{prop. movies are sectors}. Collaring by cubes of the form $[-1,\epsilon]^N$ for $\epsilon>0$ will serve a further purpose in proving Theorem~\ref{theorem. localization}; namely, the construction in Remark~\ref{remark. lioudelta is a sub of Ex lioustr} is justified by the fact our collarings take place on neighborhoods of the form $[-1,\epsilon]^N$ for $\epsilon>0$.
\end{remark}

The face maps of $\lioudeltadef$ are what one would expect: Given an injection $\iota: J \to I$ of linearly ordered sets and an $I$-simplex $(\{M_{A} \times T^*\Delta^{A}\}_{\emptyset \neq A\subset I},\{\phi_{{A'} \subset {A}}\})$, the associated $J$-simplex is pulled back along $\iota$. More precisely, it is given by:
	\eqnn
	(\{M_{\iota(B)} \times T^*\Delta^{B}\}_{\emptyset \neq B\subset J},\{\phi_{{\iota(B')} \subset {\iota(B)}}|_{M_{\iota(B)} \times T^*\Delta^{B}} \}).
	\eqnd
\end{defn}

\begin{remark}\label{remark. y compose}
Fix an element $s \in \Delta^{A''}$. Then the composition requirement~\ref{item. lioudelta composition} guarantees that
	\eqnn
	y^s_{A' \subset A} 
	\circ
	y^s_{A'' \subset A'}
	=
	y^s_{A'' \subset A}.  
	\eqnd
Here, $\{y^s_{A'' \subset A}: M_{A''} \to M_A\}_{s \in \Delta^{A''}}$ is the family of deformation embeddings determined by the collaring movies $\phi$. (See~\eqref{eqn. family of embeddings}.) Note that in this family, the Liouville structures on $M_{A}$ and $M_{A''}$ are $s$-dependent.
\end{remark}

\subsection{A 1-simplex}
\label{section. 1 simplex of lioudelta}

Let us explicate the definition of a 1-simplex; see also Figure~\ref{figure. 1-simplex-lioudelta}. In this section, we will use the symbol $\epsilon = \epsilon_A$ to mean some positive real number, and not indicate the dependency on the choice of a particular set $A$.

Following the notation of Definition~\ref{defn. lioudeltadef},
we take $I$ to be the poset $[1] = \{0, 1\}$. Recall that the standard notation $\Delta^I$ in this case is often abbreviated to $\Delta^1$ -- that is, $\Delta^{\{0,1\}}$ is the standard 1- simplex. 

Then to give an $I$-simplex (i.e., a 1-simplex) is to give the data of 
\begin{itemize}
\item  Three Liouville sectors 
$M_{\{0\}}$,
$M_{\{1\}}$, and
$M_{\{0,1\}} \times T^*\Delta^{\{0,1\}}$. 

\item A strict sectorial embedding
	\begin{equation}\nonumber
	\phi_{ \{0\} \subset \{0,1\}} :
	M_{\{0\}} \tensor T^*[-1,\epsilon]^{\{1\}} \to M_{\{0,1\}} \times T^*\Delta^{\{0,1\}}
	\end{equation}
	and
\item A strict sectorial embedding
	\begin{equation}\nonumber
	\phi_{ \{1\} \subset \{0,1\}} :
	M_{\{1\}} \tensor T^*[-1,\epsilon]^{\{0\}} \to M_{\{0,1\}} \times T^*\Delta^{\{0,1\}}.
	\end{equation}.
\end{itemize}

\begin{figure}[ht]
    \begin{equation}\nonumber
			\xy
			\xyimport(8,8)(0,0){\includegraphics[width=4in]{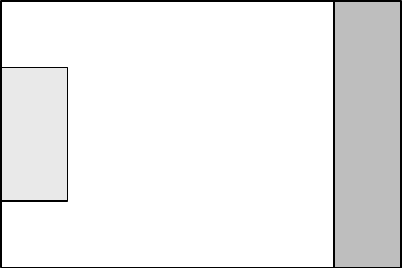}}
				,(7.3,4)*{M_{\{1\}}}	
				,(0.7,4)*{M_{\{0\}}}
			\endxy
    \end{equation}
	\begin{figurelabel}
    \label{figure. 1-simplex-lioudelta}
The largest rectangle in this image represents $M_{\{0,1\}} \times T^*\Delta^1$, with the cotangent direction omitted. The horizontal direction represents the $\Delta^1$ direction, and the vertical direction represents the $M_{\{0,1\}}$ direction. We depict a copy of $M_{\{0\}} \times [-1,\epsilon]_s$, in light gray, embedding into $M_{\{0,1\}} \times \Delta^1$, where the $M_{\{0\}}$ component is mapped independently of $s$. Likewise we depict a copy of $M_{\{1\}} \times [-1,\epsilon]_s$ embedded into $M_{\{0,1\}}$ where the $M_{\{1\}}$ component is mapped independently of of $s$; that $M_{\{1\}}$ fills the entire vertical component is meant to connote that $M_{\{1\}}$ maps diffeomorphically to $M_{\{0,1\}}$. In contrast, that $M_{\{0\}}$ occupies only a proper subset of the vertical direction shows that $y_0$ \eqref{eqn. y_0 of 1 simplex} need not be a diffeomorphism.
	\end{figurelabel}
\end{figure}

These data must satisfy the following condition to be a 1-simplex in 
$\lioudeltadef$:

\begin{itemize}
\item 
$M_{\{0,1\}} \times T^*\Delta^{\{0,1\}}$ must be a $\Delta^{\{0,1\}}$-movie (Definition~\ref{defn. movie object}). In particular, it must have a Liouville sector structure of the form $\theta_0 \oplus pdq$ along ``$M_{\{0,1\}}$ times $T^*$ of a neighborhood of the 0th vertex of $\Delta^{\{0,1\}}$,'' and likewise of the form $\theta_1 \oplus pdq$ in a neighborhood of ``$M_{\{0,1\}}$ times $T^*$ of a neighborhood of the 1st vertex of $\Delta^{\{0,1\}}$.'' Note $\theta_0$ and $\theta_1$ are (not necessarily equal) Liouville structure on $M_{\{0,1\}}$.  

\item 
By the condition that each $\phi$ be a collaring movie of deformation embeddings, we conclude that $\phi_{ \{1\} \subset \{0,1\}}$ is simply a direct product of functions
	\eqn\label{eqn. y_0 of 1 simplex}
	\phi_{ \{i\} \subset \{0,1\}}
	=
	y_i \times T^*\eta
	\eqnd
where $y_i$ is some strict sectorial embedding from $M_{\{i\}}$ to $M_{\{0,1\}}$ (with $M_{\{0,1\}}$ endowed with the Liouville structure $\theta_i$) and $\eta = \eta_{ \{i\} \subset \{0,1\}}$ is a collaring we have already chosen (Choice~\ref{choice. collaring}). 

\item Because $I = \{0,1\}$, there are no proper inclusions $A'' \subset A' \subset A$, so condition~\ref{item. lioudelta composition}
is empty.

\item By Condition~\ref{item. lioudelta is max localizing}, 
the map $y_1: M_{\{1\}} \to M_{\{0,1\}}$ must be a strict {\em isomorphism} (that is, a diffeomorphism respecting Liouville structures on the nose). 
In particular, using $y_1$, one may think of the movie $M_{\{0,1\}} \times T^*\Delta^1$ as a movie of Liouville structures on $M_{\{1\}}$.
\end{itemize} 

\begin{remark}
\label{remark. 1-simplex is family of structures and an embedding}
Thus, a 1-simplex in $\lioudeltadef$ may be informally thought of as the data of two sectors $M_{\{0\}}$ and $M_{\{1\}}$, equipped with 
(i) a smooth codimension zero embedding $M_{\{0\}} \into M_{\{1\}}$ and
(ii) a (bordered and collared) deformation of the Liouville structure on $M_{\{1\}}$ which renders the embedding a strict sectorial embedding.
\end{remark}

\subsection{A 2-simplex}
Now we illustrate what it means to be a 2-simplex in $\lioudeltadef$. We set $I = [2] = \{0< 1 < 2\}$. As before,  we use the symbol $\epsilon = \epsilon_A$ to mean some positive real number, and not indicate the dependency on the choice of a particular set $A$.

The data of a 2-simplex of $\lioudeltadef$ involves

\begin{itemize}
\item Liouville sectors
 \begin{align}
 M_{\{0\}} \cong M_{\{0\}} \times T^*\Delta^{\{0\}}, & &
 M_{\{1\}} \cong M_{\{1\}} \times T^*\Delta^{\{1\}}, & &
 M_{\{2\}} \cong M_{\{2\}} \times T^*\Delta^{\{2\}} \nonumber \\
 M_{\{0,1\}} \times T^*\Delta^{\{0,1\}} , & &
 M_{\{0,2\}} \times T^*\Delta^{\{0,2\}} , & &
 M_{\{1,2\}} \times T^*\Delta^{\{1,2\}} , \nonumber \\
   &&M_{\{0,1,2\}}  \times T^*\Delta^{\{0,1,2\}}. & & \label{eqn. sectors of 2-simplex}
 \end{align} 
 Each of these are movies; for example, $M_{\{0,2\}} \times T^*\Delta^{\{0,2\}} $ is a $\Delta^{\{0,2\}}$-movie, and in particular induces a Liouville structure on  $M_{\{0,2\}}$ near vertex 0 of $\Delta^2$, and a Liouville structure vertex 2 of $\Delta^2$. Likewise, the Liouville sector structure on $M_{\{0,1,2\}}  \times T^*\Delta^{\{0,1,2\}}$ is induced from some $\Delta^{\{0,1,2\}}$-parameter family of Liouville structures on $M_{\{0,1,2\}}$. 
Using the same reasoning as in Remark~\ref{remark. 1-simplex is family of structures and an embedding}, one may think of $M_{\{0,1,2\}} \times T^*\Delta^{\{0,1,2\}}$ as encoding a $\Delta^2$-family of (bordered and collared) deformations of the Liouville structure on $M_{\{2\}}$.
 
 \item For every cardinality 2 subset $\{i,j\}$ of $[2]$ (where $i$ may be less than, or greater than, $j$), a $\Delta^{\{i,j\}}$-simplex of $\lioudeltadef$. Section~\ref{section. 1 simplex of lioudelta} already described what data this entails -- Liouville sectors of the form $M_{\{i\}}$ and $M_{\{i,j\}} \times T^*\Delta^{\{i,j\}}$, along with strict sectorial embeddings
 	\eqnn
	M_{\{i\}} \tensor T^*[-1,\epsilon]^{\{j\}}
	\to
	M_{\{i,j\}} \times T^*\Delta^{\{i,j\}}
	\leftarrow
	M_{\{j\}} \tensor T^*[-1,\epsilon]^{\{i\}}
	\eqnd
of the form 
	\eqnn
	y_{\{i\} \subset \{i,j\}} \times T^*\eta_{\{i\}\subset\{i,j\}},
	\qquad
	y_{\{j\} \subset \{i,j\}} \times T^*\eta_{\{j\}\subset\{i,j\}}
	\eqnd
again for $\eta$ given by collaring conventions (Choice~\ref{choice. collaring}). As an example, the manifold $M_{\{1\}}$ is thus equipped with two sectorial embeddings
	\eqnn
		y_{\{1\} \subset \{1,2\}} : M_{\{1\}} \to  M_{\{1,2\}} \qquad
		\text{and}
		\qquad
		y_{\{1\} \subset \{0,1\}} : M_{\{1\}} \to  M_{\{0,1\}} .
	\eqnd

 \item Again for every cardinality 2 subset $\{i,j\}$ of $[2]$, a strict sectorial embedding
 	\eqnn
	\phi_{\{i,j\} \subset [2]}: 
	M_{\{i,j\}} \times T^*\Delta^{\{i,j\}} \tensor T^*[-1,\epsilon]^{[2] \setminus\{i,j\}}
	\to
	M_{\{0,1,2\}}  \times T^*\Delta^{\{0,1,2\}}.
	\eqnd
Note that the domain is isomorphic to $
	M_{\{i,j\}} \times T^*\Delta^1 \tensor T^*[-1,\epsilon]$. Further, Definition~\ref{defn. lioudeltadef} demands that $\phi_{\{i,j\} \subset [2]}$ be a collaring movie of deformation embeddings; in particular, $
	\phi_{\{i,j\} \subset [2]}$ encodes a $\Delta^{\{i,j\}}$-parametrized family of smooth embeddings 
		\eqnn
		\{y^s : M_{\{i,j\}} \to M_{\{0,1,2\}}\}_{s \in \Delta^{\{i,j\}}}.
		\eqnd
	See Figure~\ref{figure. 2-simplex-lioudelta}.
\end{itemize}

These maps must satisfy the composition requirement~\ref{item. lioudelta composition} (so that, for example, $\phi_{\{0\} \subset [2]}$ may be recovered from $\phi_{\{0\} \subset [1]}$ and $\phi_{[1] \subset [2]}$) and the max-localizing requirement~\ref{item. lioudelta is max localizing} (so that, for example, $\phi_{\{0,2\} \subset \{0,1,2\}}$ induces an isomorphism from $M_{\{0,2\}}$ to $M_{\{0,1,2\}}$). 

\begin{figure}[ht]
    \begin{equation}\nonumber
			\xy
			\xyimport(8,8)(0,0){\includegraphics[width=4in]{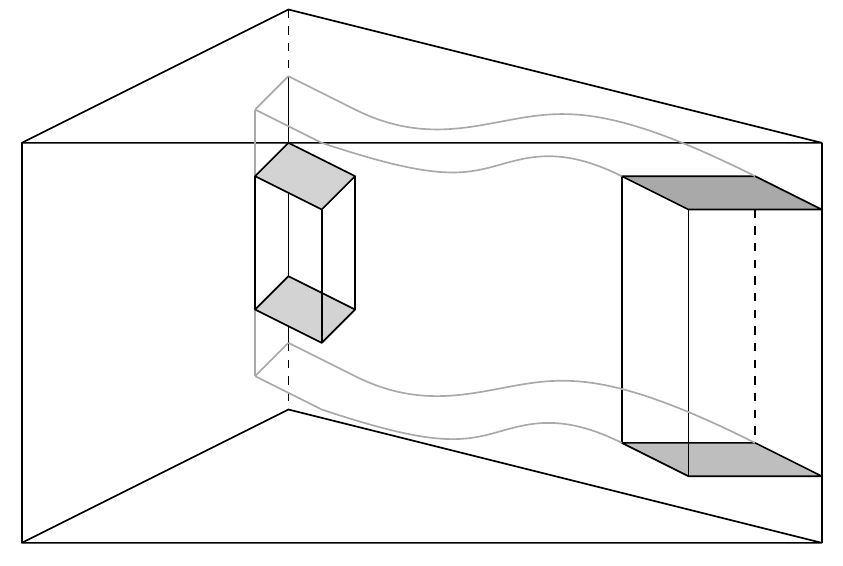}}
				,(6.9,3.2)*{M_{\{1\}}}
				,(2.8,4.4)*{M_{\{0\}}}
			\endxy
    \end{equation}
    \label{figure. 2-simplex-lioudelta}
    
	\begin{figurelabel}
Depicted is a drawing of a triangular prism modeling the manifold $M_{\{0,1,2\}} \times T^*\Delta^2$. The ``base'' directions (i.e., the directions of the triangle) indicate the directions of $\Delta^2$. The vertical direction depicts the $M_{\{0,1,2\}}$ direction, and the cotangent directions of $T^*\Delta^2$ are not indicated. 
The rectangular prism with the light grey top and bottom represents the image of $\phi_{\{0\} \subset [2]}$; the image is isomorphic to a copy of $M_{\{0\}} \times T^*(\Delta^1 \times [-1,\epsilon])$. Likewise, the rectangular prism with the dark grey top and bottom represents the image of $\phi_{\{1\} \subset [2]}$. 
Finally, the curvy grey figure represents the image of $\phi_{\{0,1\} \subset [2]}$ -- i.e., the image of an embedding of $M_{\{0,1\}} \times T^*(\Delta^1 \times [-1,\epsilon])$ into $M_{\{0,1,2\}} \times T^*\Delta^2$. The curviness connotes that $\phi_{\{0,1\} \subset [2]}$ potentially parametrizes a non-constant $\Delta^1$-family of embeddings $y^2$ of $M_{\{0,1\}}$ into $M_{\{0,1,2\}}$. 
	\end{figurelabel}
\end{figure}

\clearpage

\subsection{\texorpdfstring{$\lioudelta$}{Liou Delta} (as a semisimplicial set)}
\label{section. lioudelta properties}

\begin{defn}
\label{defn. lioudelta}
We define $\lioudelta$ to be the subsemisimplicial set of $\lioudeltadef$ consisting only of those movies that are associated to compactly supported deformations.
\end{defn}

\begin{remark}
The consequence of requiring that all movies in sight be compactly supported is that the maps $\{\widetilde{y}_s\}_{s \in \Delta^A}$ are now families of sectorial embeddings for a {\em fixed} Liouville structure on $M$ and $N$ (because any two of $\lambda_s,\lambda_{s'}$ are equivalent). See Remark~\ref{remark. movies give smooth maps to embedding spaces}.
\end{remark}

\begin{remark}
 
The simplices of $\lioudelta$ will thus have concretely interpretable geometric data. As we will see,
\begin{itemize}
\item A 1-simplex in $\lioudelta$ will give rise to the data of a (not necessarily strict) sectorial embedding from $M_0$ to $M_1$ (Example~\ref{example. edges in lioudelta are non-strict embeddings}).
\item A 2-simplex will give rise to a diagram of sectorial embeddings commuting up to isotopy through sectorial embeddings (Proposition~\ref{prop. 2-simplex is an isotopy}).
\end{itemize}
Conversely,
\begin{itemize}
\item Any (not necessarily strict) sectorial embedding from $M$ to $N$ gives rise to a (non-unique, but determined by contractible data) 1-simplex with initial vertex $M$ and terminal vertex $N$ (Construction~\ref{construction. edges in lioudelta}).
\item Any isotopy through sectorial embeddings $f, g: M \to N$ will give rise to a (non-unique, but determined up to contractible data) 2-simplex (Example~\ref{example. isotopies are 2 simplices}).
\end{itemize}

There is a similar interpretation of higher simplices; this is encapsulated in Theorem~\ref{theorem. EmbLiou is homLiouDelta}.

We remark that the ``converse'' constructions are postponed until Section~\ref{section. lioudelta geometry}, as we will not have established that $\lioudelta$ admits degeneracy maps until then. The delay is worth it: Having degeneracy maps available will allow us to articulate what we mean by homotopies of morphisms in $\lioudelta$ precisely. Because the converse constructions require choices, we wait for the availability of homotopies to be able to assuage the reader that these choices are not consequential (up to homotopy).
\end{remark}

\begin{remark}
Fix an $I$-simplex of $\lioudelta$ and fix an element $a \in I$ along with a subset $A \subset I$, $a \in A$. Then the inclusion  $\{a\} \subset A$ gives the data of a map
	\eqnn
	\phi_{\{a\} \subset {A}} : M_{ \{a\} } \tensor T^*[-1,\epsilon]^{A} \to M_{A} \times T^*\Delta^{A}.
	\eqnd
By the collaring assumption, this in fact determines a smooth embedding
	\eqn\label{eqn. liouville embeddings y}
	y_{\{a\} \subset A} : M_{\{a\}} \to M_{A}
	\eqnd
and this is a strict sectorial embedding for the Liouville structure on $M_{A}$ determined near the vertex $\Delta^{\{a\}} \subset \Delta^{A}$. 
\end{remark}

\begin{remark}
\label{remark. M_A structure}
When further $a = \max A$, Condition~\ref{item. lioudelta is max localizing} states that $y_{\{a\} \subset A}$ is a Liouville isomorphism:
	\eqn\label{eqn. max isomorphisms for lioudelta}
	y_{\{\max A\} \subset A} : M_{\{\max A\}} \xrightarrow{\cong} M_{A}
	\eqnd
and this isomorphism is strict if $M_{A}$ is endowed with this Liouville structure determined near the vertex $\Delta^{\{a\}} \subset \Delta^{A}$. 

For this reason, when we write $M_A$, we will often consider the manifold $M_A$ to be equipped with the Liouville structure determined near the vertex $\Delta^{\max A}$ of $\Delta^A$.
\end{remark}

\begin{remark}
Given the data of all the families $\{\widetilde{y}^s_{A' \subset A}\}_{s \in \Delta^{A'}}$, and of Liouville structure on each $M_A$ for $A \subset I$, the data of an $I$-simplex of $\lioudelta$ is determined completely by the additional data of compactly supported smooth functions $h_{A} : M_{A} \times \Delta^{A} \to \RR$.

Because the space of $h_{A}$ is contractible, one should imagine that the families 
$\{\widetilde{y}^s_{A' \subset A}\}_{s \in \Delta^{A'}}$ are the homotopically relevant ``meat'' of the information conveyed by a simplex of $\lioudelta$.
\end{remark}

\begin{example}\label{example. edges in lioudelta are non-strict embeddings}
Fix an $I$-simplex of $\lioudelta$ with $|I| \geq 2$, and fix two elements $a< b$ in $I$. Let us study the associated 1-simplex of $\lioudelta$. We have the data of maps
	\eqnn
	M_{\{a\}} \tensor T^*[-1,\epsilon_a] \to M_{\{a,b\}} \times T^*\Delta^{\{a,b\}} \leftarrow M_{\{b\}} \tensor T^*[-1,\epsilon_b].
	\eqnd
The main difference with a 1-simplex in $\lioudeltadef$ (as studied in Section~\ref{section. 1 simplex of lioudelta}) is that the movie $M_{\{a,b\}} \times T^*\Delta^{\{a,b\}}$ is a movie through {\em compactly supported} deformations of the Liouville structure.

These maps in turn determine (not necessarily strict) sectorial embeddings
	\eqnn
	y_{\{a\} \subset \{a,b\}} : M_{\{a\}} \to M_{\{a,b\}},
	\qquad
	y_{\{b\} \subset \{a,b\}}: M_{\{b\}} \to M_{\{a,b\}}
	\eqnd
as in \eqref{eqn. liouville embeddings y}. In fact, by endowing $M_{\{a,b\}}$ with the Liouville structure determined near the vertex $\Delta^{\{b\}} \subset \Delta^{\{a,b\}}$  (Remark~\ref{remark. M_A structure}), the map from $M_{\{b\}}$ is a strict Liouville isomorphism as in~\eqref{eqn. max isomorphisms for lioudelta}. By composing with the inverse, we thus obtain a (not necessarily strict) sectorial embedding
	\eqn\label{eqn.  y ab}
	y_{a < b} = (y_{\{b\} \subset \{a,b\}})^{-1} \circ y_{\{a\} \subset \{a,b\}} : M_{\{a\}} \to M_{\{b\}}. 
	\eqnd
Thus, though $\lioudelta$ has diagrams built out of strict sectorial embeddings, we see that we can detect non-strict sectorial embeddings as edges in $\lioudelta$.
\end{example}

Here is the first hint that $\lioudelta$ detects the topology of Liouville embedding spaces:

\begin{prop}\label{prop. 2-simplex is an isotopy}
Fix a (semisimplicial) 2-simplex $\Delta^2 \to \lioudelta$. Letting $0,1,2$ denote the vertices of $\Delta^2$, consider the (not necessarily strict) sectorial embeddings $y_{0<1}, y_{1<2}, y_{0<2}$ as defined in~\eqref{eqn.  y ab}.

Then the composition $y_{1<2} \circ y_{0<1}$ is isotopic, through (not necessarily strict) sectorial embeddings, to $y_{0<2}$.
\end{prop}

\begin{proof}
The 2-simplex gives rise to the following commutative diagram\footnote{Though we have drawn the diagram rectilinearly, the reader may reshape the diagram to witness that the diagram has the shape of the barycentric subdivision of the 2-simplex.} of smooth maps (none of which are codimension zero embeddings):
	\eqnn
	\xymatrix{
	M_{\{0\}} \times \Delta^{\{0\}}
	\ar[rr]^{y_{\{0\}\subset\{0,1\}} \times \iota}
	\ar[dd]^{y_{\{0\}\subset\{0,2\}} \times \iota}
	&& 	M_{\{0,1\}} \times \Delta^{\{0,1\}}
		\ar[d]^{(y^s_{\{0,1\}\subset\{0,1,2\}},s)}
	&& \ar[ll]_{y_{\{1\}\subset\{0,1\}} \times \iota}
		M_{\{1\}} \times \Delta^{\{1\}} 
		\ar[d]^{y_{\{1\}\subset\{1,2\}} \times \iota}
	\\
	&& 	M_{\{0,1,2\}} \times \Delta^{\{0,1,2\}}
	&& 	\ar[ll]^{(y^s_{\{1,2\}\subset\{0,1,2\}},s)} 
		M_{\{1,2\}} \times \Delta^{\{1,2\}}
	\\
	M_{\{0,2\}} \times \Delta^{\{0,2\}}
		\ar[urr]_-{(y^s_{\{0,2\}\subset\{0,1,2\}},s)} 
	&&
	&& \ar[llll]^{y_{\{2\}\subset\{0,2\}} \times \iota}
		M_{\{2\}} \times \Delta^{\{2\}}
		\ar[u]_{y_{\{2\}\subset\{1,2\}} \times \iota}
	}
	\eqnd
where we have used the notation from~\eqref{eqn. liouville embeddings y} and \eqref{eqn. y family from simplex}, and the $\iota$ are the obvious inclusions of faces of simplices. The diagonal map takes an element $(x,s) \in M_{\{0,2\}} \times \Delta^{\{0,2\}}$ and maps it to the element $(y^s_{\{0,2\}\subset\{0,1,2\}}(x),s) \in M_{\{0,1,2\}} \times \Delta^{\{0,1,2\}}$.

Moreover, by the definition of a simplex in $\lioudelta$, we know that the $y^s$ form smooth families of (not necessarily strict) sectorial embeddings. So, for example, the middle vertical map encodes an isotopy
	\eqnn
	\{y^s_{\{0,1\} \subset \{0,1,2\}} : M_{\{0,1\}} \to M_{\{0,1,2\}} \}_{s \in \Delta^{\{0,1\}}}
	\eqnd
from the sectorial embedding encoded at the vertex $s = \Delta^{\{0\}}$ to the sectorial embedding encoded at the vertex $s = \Delta^{\{1\}}$. In the equations that follow, we will abuse notation and write $s=0, 1, 2$ to mean the vertices $\Delta^{\{0\}}, 
\Delta^{\{1\}}, 
\Delta^{\{2\}}$ of the 2-simplex $\Delta^{\{0,1,2\}}$. We also let $\sim$ mean ``is smoothly isotopic to, through (not necessarily strict) sectorial embeddings.'' Then we have:
\begin{align}	
	y^{s=2}_{ \{1,2\} \subset \{0,1,2\}} 
		\circ y_{ \{2\} \subset \{1,2\}}
		\circ y_{1 < 2}
		\circ y_{0 < 1}
	& \sim 
	y^{s=1}_{ \{1,2\} \subset \{0,1,2\}} 
		\circ y_{ \{2\} \subset \{1,2\}}
		\circ y_{1 < 2}
		\circ y_{0 < 1}\nonumber \\
	& = 
	y^{s=1}_{ \{1,2\} \subset \{0,1,2\}} 
		\circ y_{ \{1\} \subset \{1,2\}}
		\circ y_{0 < 1}\nonumber \\
	& = 
	y^{s=1}_{ \{0,1\} \subset \{0,1,2\}} 
		\circ y_{ \{1\} \subset \{0,1\}}
		\circ (y_{ \{1\} \subset \{0,1\}})^{-1}
		\circ y_{ \{0\} \subset \{0,1\}}\nonumber \\
	& = 
	y^{s=1}_{ \{0,1\} \subset \{0,1,2\}} 
		\circ y_{ \{0\} \subset \{0,1\}}\nonumber \\
	& \sim 
	y^{s=0}_{ \{0,1\} \subset \{0,1,2\}} 
		\circ y_{ \{0\} \subset \{0,1\}}\nonumber \\
	& = 
	y^{s=0}_{ \{0,2\} \subset \{0,1,2\}} 
		\circ y_{ \{0\} \subset \{0,2\}}\nonumber \\
	& \sim
	y^{s=2}_{ \{0,2\} \subset \{0,1,2\}} 
		\circ y_{ \{0\} \subset \{0,2\}}. \nonumber
\end{align}
We are finished by post-composing with the inverse diffeomorphism to 
	\eqnn
	y^{s=2}_{ \{1,2\} \subset \{0,1,2\}} 
		\circ y_{ \{2\} \subset \{1,2\}}
	=
	y^{s=2}_{ \{0,2\} \subset \{0,1,2\}} 
		\circ y_{ \{2\} \subset \{0,2\}}
	.
	\eqnd
\end{proof}

\subsection{Change of bases to max-constant simplices}
\label{section. max constant simplices}
\begin{defn}[Max-constant simplices]
Fix an $I$-simplex of $\lioudelta$. Let us say that the simplex is {\em max-constant} (as opposed to just max-localizing) if for every subset $A \subset I$,
	\enum
	\item $M_A = M_{\{\max A\}}$, and
	\item The sectorial isomorphism $y_{\{\max A\} \subset A} : M_{\{\max A\}} \to M_A$---see~\eqref{eqn. max isomorphisms for lioudelta}---is the identity map. 
	\enumd
\end{defn}

There is a natural way to ``change bases'' to take the data of an $I$-simplex of $\lioudelta$ and render it a max-constant simplex. Given a chain of subsets $\emptyset \neq A \subset B \subset I$, consider the map
	\eqnn
	y_{A \subset B}: \Delta^A \to \embliou(M_A, M_B)
	\eqnd
given by the original $I$-simplex. This induces a map
	\eqnn
	y'_{A \subset B}: \Delta^A \to \embliou(M_{\max A},M_{\max B})
	\eqnd
by conjugating by the isomorphisms $M_{\max A} \cong M_A$ and $M_{B} \cong M_{\max B}$ determined by the original $I$-simplex. It is straightforward to verify that the collection of maps $\{y'_{A \subset B}\}_{\emptyset \neq A \subset B \subset I}$ defines a max-constant $I$-simplex.

This change of basis has the advantage that a collection of maps $M_A \to M_B$ and $M_{A'} \to M_{B'}$ may---if $\max A = \max A'$ and $\max B = \max B'$---be considered as a collection of maps $M_{\max A} \to M_{\max B}$ in a single mapping space (i.e., with fixed domain and codomain). We will utilize this when proving that $\lioudelta$ satisfies the weak Kan condition (see Construction~\ref{construction. tilde alpha on jth face}, for example).

\clearpage
\section{The weak Kan condition}
\label{section. weak kan}

Our goal is to now prove:
\begin{theorem}\label{theorem. weak kan}
$\lioudelta$ satisfies the (semisimplicial) weak Kan condition. 
\end{theorem}

We will, in parallel, prove also the following:

\begin{theorem}\label{theorem. weak kan def}
$\lioudeltadef$ satisfies the (semisimplicial) weak Kan condition. 
\end{theorem}

\begin{remark}
The underlying philosophy of establishing the weak Kan property for both semisimplicial sets is identical, but there are technical differences. For $\lioudelta$, the only deformations appearing in movies are compactly supported, so all the technical difficulty is in constructing particular families of sectorial embeddings. (The construction of the compactly supported deformations is simple because the space of compactly supported functions is contractible.) For $\lioudeltadef$, horn-filling further requires us to extend families of (collared, bordered, but otherwise arbitrary) Liouville structures from an image of an isotopy to the rest of the codomain. Here, we make extensive use of shrinking isotopies (Section~\ref{section. shrinking}). These allow us to rearrange families to only alter the Liouville structures away from the boundary of the image, and hence to extend these families to the entire codomain.
\end{remark}

\begin{notation}[$\alpha$ and $\widetilde{\alpha}$]\label{notation. alpha}
So fix a map 
	\eqnn
	\alpha: \Lambda^n_j \to \lioudelta \qquad \text{ (or $\lioudeltadef$)}
	\eqnd
of semsimplicial sets, with $0<j<n$. We must extend $\alpha$ to a map 
	\eqnn
	\tilde \alpha: \Delta^n \to \lioudelta \qquad \text{ (or $\lioudeltadef$)}
	\eqnd
of semisimplicial sets. 
\end{notation}

\begin{notation}[$\widetilde{\embtop_{\liou}}$]
\label{notation. tilde emb}
By the compact-support condition on movies, one may interpret the data of an $I$-simplex in $\lioudelta$ as giving rise to maps
	\eqnn
	\Delta^B \to \embtop_{\liou}(M_B,M_A)
	\eqnd
for each $B \subset A \subset I$ (Remark~\ref{remark. movies give movies}). This is a map to the space of all sectorial embeddings. Because the identity function defines a sectorial isomorphism $(M_A,\lambda^A_s) \to (M_A,\lambda^A_{s'})$ for any $s,s' \in \Delta^B$ (and likewise for $M_B$), this space is independent of $s$.

Let us now define a space $\widetilde{\embtop_{\liou}}(M_B,M_A)$ as the total space of a fibration
	\eqn\label{eqn. section tilde emb}
	\widetilde{\embtop_{\liou}}(M_B,M_A) \to \Delta^B
	\eqnd
whose fiber above $s \in \Delta^B$ is the space $\embtop_{\liou}((M_B,\lambda^B_s),(M_A,\lambda^A_s))$.

An $I$-simplex in $\lioudeltadef$ defines a section of this fibration.
\end{notation}

\begin{remark}
Thus, to produce the simplices in $\lioudeltadef$, we will have to produce maps from $\Delta^B$ to $\widetilde{\embtop_{\liou}}$, while to produce simplices in $\lioudelta$, we must also produce such maps, but we may identify the image with $\embtop_{\liou}$. 

In our arguments below, to preserve generality, we will thus often speak of maps to $\widetilde{\embtop_{\liou}}$, with the understanding that all such maps are sections of the fibration~\eqref{eqn. section tilde emb}.
\end{remark}

\begin{notation}[The values of $\alpha$ and $\tilde \alpha$]
\label{notation. tilde alpha values}
To give $\tilde\alpha$, we must make explicit $\tilde\alpha(A)$---that is, we must specify the movies $M_A \times T^*\Delta^A$---for every $A \subset [n]$. Of course, given $\alpha$, this amounts to specifying the movies $\tilde\alpha([n])$ and $\tilde \alpha([n] \setminus \{j\})$.

Moreover, for every $B \subset A$, we must specify collaring movies $\phi_{B \subset A}$. By Remark~\ref{remark. movies give movies}, the data of such collaring movies is equivalent to the data of a function from $\Delta^B$ to the space of smooth function from $M_B$ to $M_A$ (where each function at $s \in \Delta^B$ is a sectorial embedding with respect to the Liouville structures on $M_B$ and $M_A$ at $s$). We will denote this smooth function by $\widetilde \alpha(B \subset A)$.

Likewise, for every $B \subset A$ with $\Delta^A$ a face of the horn $\Lambda^n_j$, $\alpha$ defines a map
	\eqnn
	\alpha(B \subset A) : \Delta^B \to \widetilde{\embtop_{\liou}}(M_B,M_A)(M_B,M_A).
	\eqnd
\end{notation}

\subsection{For \texorpdfstring{$n=2$}{n equals 2}}
\label{section. n=2 horn filling}
We begin with the case $n=2$ (and hence $j=1$). Unlike the higher-$n$ case, we here require no isotopy extension arguments. We begin with the $\lioudelta$ case.

\begin{construction}[$\tilde \alpha$ when $n=2$, for $\lioudelta$.]\label{construction. alpha n = 2}
The data of $\alpha: \Lambda^n_1 \to \lioudelta$ determine smooth, proper codimension zero embeddings
	\eqn\label{eqn. constructing alpha 2 filler}
	M_{\{0\}}\to M_{\{0,1\}} \xleftarrow{\cong}M_{\{1\}} \to M_{\{1,2\}} \xleftarrow{\cong} M_{\{2\}}
	\eqnd
Each $M_{\{i,j\}}$ is equipped with a compactly supported deformation of Liouville structures; when given the Liouville structure near $\max\{i,j\}$, the leftward pointing arrows are isomorphisms of Liouville sectors as in~\eqref{eqn. max isomorphisms for lioudelta}. To fill $\alpha$, let us declare 
	\eqn\label{eqn. tilde alpha for n 2}
	M_{\{0,1,2\}} = M_{\{2\}}
	\eqnd
and define the smooth maps
	\eqnn
	y^s_{\{0,1\} \subset \{0,1,2\}}: M_{\{0,1\}} \to M_{\{0,1,2\}},
	\qquad
	s \in \Delta^{\{0,1\}}
	\eqnd 
as the composition
	\eqnn
	M_{\{0,1\}} \xrightarrow{\cong} M_1 \to M_{\{1,2\}} \xrightarrow{\cong} M_{\{2\}} = M_{\{0,1,2\}}.
	\eqnd
(The rightward pointing isomorphisms are inverses of the isomorphisms in~\eqref{eqn. constructing alpha 2 filler}. Note also that this is independent of $s \in \Delta^{\{0,1\}}$.) 
We would like to extend this to a strict sectorial embedding
	\eqn\label{eqn. example M01 embedding}
	M_{\{0,1\}} \times T^*\Delta^{\{0,1\}} \to M_{\{0,1,2\}} \times T^*\Delta^{\{0,1\}}
	\eqnd
where the domain of~\eqref{eqn. example M01 embedding} is given the Liouville structure specified by $\alpha$ for the $\Delta$-movie $M_{\{1,2\}} \times T^*\Delta^{\{1,2\}}$, and we must specify the Liouville structure for the codomain. To do this, we note that because the map $y^s_{\{0,1\} \subset \{0,1,2\}}$ for $s = \Delta^{\{1\}}$ is a Liouville isomorphism, one can choose a $\Delta^{\{0,1\}}$-dependent family of compactly supported functions $\tilde h_s$ on $M_{\{0,1,2\}}$ 
	\eqn\label{eqn. h for 01}
	\tilde h : M_{\{0,1,2\}} \times \Delta^{\{0,1\}} \to \RR,
	\qquad
	(x,s) \mapsto \tilde h_s(x)
	\eqnd
so that, endowing $M_{\{0,1,2\}} = M_{\{2\}}$ with the Liouville structure $\lambda^{M_{\{2\}}} + d \tilde h_s$---and giving $M_{\{0,1,2\}} \times T^*\Delta^{\{0,1\}}$ the induced Liouville structure---\eqref{eqn. example M01 embedding} becomes a sectorial embedding.

Likewise, we define an $s$-independent family of smooth maps
	\eqnn
	y^s_{\{1,2\} \subset \{0,1,2\}}: M_{\{1,2\}} \to M_{\{0,1,2\}}, 
	\qquad
	s \in \Delta^{\{1,2\}}
	\eqnd 
by the composition
	$
	M_{\{1,2\}} \xrightarrow{\cong} M_{\{2\}}=M_{\{0,1,2\}}.
	$
As before, by specifying an appropriate family of compactly supported functions $\tilde h_s: M_{\{2\}} \to \RR$ for $t \in \Delta^{\{1,2\}}$
	\eqn\label{eqn. h for 12}
	\tilde h : M_{\{0,1,2\}} \times \Delta^{\{1,2\}} \to \RR,
	\qquad
	(x,s) \mapsto \tilde h_s(x)
	\eqnd
we obtain a sectorial embedding
	\eqn\label{eqn. example M12 embedding}
	M_{\{1,2\}} \times T^*\Delta^{\{1,2\}} \to M_{\{0,1,2\}} \times T^*\Delta^{\{1,2\}}.
	\eqnd
Now, because all the structures given by $\alpha$ are collared, one can smoothly extend the function
	\eqnn
	\tilde h: M \times (\Delta^{\{0,1\}} \bigcup_{\Delta^{\{1\}}} \Delta^{\{1,2\}}) \to \RR
	\eqnd
(which is obtained by gluing together ~\eqref{eqn. h for 01} and~\eqref{eqn. h for 12})
to a function
	\eqn\label{eqn. h M 012}
	\tilde h: M \times \Delta^{\{0,1,2\}} \to \RR,
	\qquad
	(x,s) \mapsto \tilde h_s(x),
	\qquad (s \in \Delta^{\{0,1,2\}})
	\eqnd
for which $\tilde h$ is collared, and for which each $\tilde h_s$ is compactly supported. We then declare the $\Delta$-movie
	\eqn\label{eqn. M012 definition}
	M_{\{0,1,2\}} \times T^*\Delta^{\{0,1,2\}}
	:= M_{\{2\}} \times T^*\Delta^{\{0,1,2\}}
	\eqnd
to be given the Liouville structure induced by adding $d$ of~\eqref{eqn. h M 012} to the Liouville structure of $M_{\{2\}}$. 

\eqref{eqn. M012 definition} defines the value of $\tilde \alpha$ on $[2] = \{0,1,2\}$. It is also clear that the maps
~\eqref{eqn. example M01 embedding} and~\eqref{eqn. example M12 embedding} extend, by collaring, to specify maps $\phi_{\{1,2\} \subset \{0,1,2\}}$ and $\phi_{\{0,1\} \subset \{0,1,2\}}$. It remains to define $\tilde \alpha$ on $\{0,2\}$ and the associated $\phi$.
\begin{itemize}
\item We declare  (Notation~\ref{notation. tilde alpha values})
	\eqn\label{eqn. M02}
	\tilde \alpha (\{0,2\}) = M_{\{0,2\}} \times T^*\Delta^{\{0,2\}} := M_{\{2\}} \times T^*\Delta^{\{0,2\}}
	\eqnd
with the Liouville form inherited from restricting the form on~\eqref{eqn. M012 definition} along the edge $\Delta^{\{0,2\}} \subset \Delta^{\{0,1,2\}}$. (This restriction is well-defined thanks to the collared choice of~\eqref{eqn. h M 012}.) 
\item We declare (see Notation~\ref{notation. tilde alpha values}) $\tilde \alpha ( \{2\} \subset \{0,2\})$ to be (extended from) the identity map $M_{\{2\}} \to M_{\{2\}}$. Note that this is a sectorial embedding because by the collaring of $\tilde h$, the Liouville form on $M_{\{0,1,2\}} \times T^*\Delta^2$ agrees with that of $M_{\{2\}}$ near the vertex $\Delta^{\{2\}}$. 
\item We declare $\tilde \alpha ( \{0\} \subset \{0,2\})$  to be (extended from) the composition of smooth maps
	\eqnn
	M_{\{0\}} \to M_{\{0,1\}} \xrightarrow{\cong} M_{\{1\}} \to M_{\{1,2\}} \xrightarrow{\cong} M_{\{2\}} = M_{\{0,2\}}.
	\eqnd
(Every map here was already given by $\alpha$.) By design of $\tilde h$~\eqref{eqn. h M 012} and by our definition of the Liouville structure on~\eqref{eqn. M02}, this is a strict sectorial embedding.
\end{itemize}
We have defined $\tilde \alpha$. The data satisfy~\ref{item. lioudelta composition} and~\ref{item. lioudelta is max localizing} by construction, and clearly extend $\alpha$. 
\end{construction}

\begin{remark}
The above construction fails when $j=0$ and $n=2$ (and likewise for $j=2$). To see this, fix three sectors $M_0, M_1, M_2$ and strict sectorial embeddings $y_{0 < 2}: M_0 \to M_2, y_{1 < 2}: M_1 \to M_2$ for which $y_{0<2}$ is not isotopic to any map factoring through $y_{1<2}$. (This defines a horn by setting, for example, $M_{\{i,j\}} \times T^*\Delta^{\{i,j|} := M_{\{ \max\{i,j\}\}} \tensor T^*\Delta^{\{i,j\}}$.) 

Proposition~\ref{prop. 2-simplex is an isotopy} guarantees that this horn cannot be filled. 
\end{remark}

\begin{construction}[$\tilde \alpha$ when $n=2$, for $\lioudeltadef$.]
\label{construction. alpha n = 2 def}
We follow the notations of~\eqref{eqn. constructing alpha 2 filler} and~\eqref{eqn. tilde alpha for n 2}. Then, by pulling back along under the isomorphism $M_{\{2\}} \xrightarrow{\cong}M_{\{1,2\}}$, one obtained a collared family of deformations on $M_{\{0,1,2\}}$ along the edge $\Delta^{\{1,2\}}$. We must now extend this collared deformation to all of $\Delta^{\{0,1,2\}}$. 

We begin by defining it along the edge $\Delta^{\{0,1\}}$. Here, we are already given a deformation of $M_{\{0,1\}} \cong M_{\{1\}}$, and $M_{\{1\}}$ is equipped with a strict sectorial embedding into $M_{\{0,1,2\}}$ (strict with respect t the Liouville structure of $M_{\{0,1,2\}}$ at $s = \Delta^{\{1\}}$). By applying a shrinking isotopy to $M_{\{1\}}$ (and hence to its image inside $M_{\{0,1,2\}}$) as necessary, we obtain a family of embeddings of $M_{\{1\}}  \cong M_{\{0,1\}}$ into $M_{\{0,1,2\}}$, parametrized by $s \in \Delta^{\{0,1\}}$. The deformations of Liouville structures on $M_{\{0,1\}}$ pushforward, and these pushforwards may be extended to become deformations of Liouvuille structures on $M_{\{0,1,2\}}$ again thanks to the shrinking. (See Remark~\ref{remark. extending deformations}.)

So we have defined a deformation of Liouville structures for $M_{\{0,1,2\}}$ parametrized by the horn $\Delta^{\{0,1\}} \bigcup \Delta^{\{1,2\}}$. Of course the space of deformations is a space, so we may continuously extend this to all of $\Delta^{\{0,1,2\}}$; the usual smooth-approximation trick applies as in Section~\ref{section. continuous to smooth} (because the space of deformations is locally modeled by the space of (very small) exact 1-forms), and we may thus ensure that we have a smooth, collared, and continuous $\Delta^{\{0,1,2\}}$-parametrized deformation of Liouville structures on $M_{\{0,1,2\}}$. 

Now it is straightforward to define the embeddings $M_{\{0\}} \to M_{\{0,1,2\}}$ at $s=\Delta^{\{0\}}$, and declaring $M_{\{0,2\}} = M_{\{0,1,2\}}$ and restricting the family of Liouville structures on $M_{\{0,1,2\}}$ to one on $M_{\{0,2\}}$ along the edge $\Delta^{\{0,2\}}$.
\end{construction}

\begin{remark}
\label{remark. extending deformations}
Let us elaborate on how to extend families of deformations from a subsector to a large subsector, after a suitable shrinking isotopy of the subsector. We explain in the context of Construction~\ref{construction. alpha n = 2 def}, but the remarks here apply more widely.

Fix the image $A$ of $M_{\{1\}}$ inside $M_{\{0,1,2\}}$, induced by the inclusion $M_{\{1\}} \into M_{\{1,2\}}$ at $s = \Delta^{\{1\}}$. We may identify $A$ with $M_{\{0,1\}}$ due the the isomorphism $M_{\{1\}} \cong M_{\{0,1\}}$ at $s = \Delta^{\{1\}}$. For every $s$ in the edge $\Delta^{\{0,1\}}$, we have a Liouville structure $\lambda_s$ on $M_{\{0,1\}}$, which is collared by assumption. This induces a new Liouville structure on $A$, by choosing a shrinking of $A$ (Section~\ref{section. shrinking}) to have image $A' \subset A$, and defining a Liouville structure $\lambda^A_s$ on $A$ which equals the pushforward of $\lambda_s$ onto $A'$, and parametrizes between $\lambda_s$ and $\lambda_{\Delta^{\{1\}}}$ along the region $A \setminus A'$. Now, because $\lambda^A_s$ agrees with $\lambda^{M_{\{0,1,2\}}}$ in a neighborhood of $\del A$, we may extend $\lambda^A_s$ to equal $\lambda^M_{\{0,1,2\}}$ outside of $A$.
\end{remark}

\subsection{For \texorpdfstring{$n \geq 3$}{n at least 3}, filling most of the simplex}
Fix $\alpha: \Lambda^n_j \to \lioudelta$ (or to $\lioudeltadef$) as in Notation~\ref{notation. alpha}. Roughly speaking, the construction of $\widetilde{\alpha}$ may be divided into two steps: (a) Defining the maps $\widetilde{\alpha}(A \subset [n])$ for those $A$ with $\max A = n$, then (b) Defining the map $\widetilde{\alpha}(A \subset [n])$ for $A = [n-1]$. We establish (a) in the present section, using a strategy illustrated in Table~\ref{table. horn-filling example outline} for the case of $n=3$ and $j=1$. We leave (b) to Sections~\ref{section. b 1} and~\ref{section. b 2}.

\newcommand{\FigureTildeAlphaB}{
    \begin{tikzpicture}[line join = round, line cap = round]
        \coordinate [label=above:1] (1) at (0,{sqrt(2)},0);
        \coordinate [label=left:2] (2) at ({-.5*sqrt(3)},0,-.5);
        \coordinate [label=below:0] (0) at (0,0,1);
        \coordinate [label=right:3] (3) at ({.5*sqrt(3)},0,-.5);
        \begin{scope}[decoration={markings,mark=at position 0.5 with {\arrow{to}}}]
        \node[] at (0) {$\bullet$};
        \node[] at (2) {$\bullet$};
        \node[] at (1) {$\bullet$};
        \node[] at (3) {$\bullet$};
         \draw[] (0)--(3);
         \draw[] (2)--(3);
         \draw[] (1)--(3);
        \end{scope}
    \end{tikzpicture}
}

\newcommand{\FigureTildeAlphaBPrime}{
    \begin{tikzpicture}[line join = round, line cap = round]
        \coordinate [label=above:1] (1) at (0,{sqrt(2)},0);
        \coordinate [label=left:2] (2) at ({-.5*sqrt(3)},0,-.5);
        \coordinate [label=below:0] (0) at (0,0,1);
        \coordinate [label=right:3] (3) at ({.5*sqrt(3)},0,-.5);
        \begin{scope}[decoration={markings,mark=at position 0.5 with {\arrow{to}}}]
        \node[] at (0) {$\bullet$};
        \node[] at (2) {$\bullet$};
        \node[] at (1) {$\bullet$};
        \node[] at (3) {$\bullet$};
        \draw[fill=gray,fill opacity=.5] (1)--(2)--(3)--cycle;
        \draw[fill=gray,fill opacity=.5] (0)--(1)--(3)--cycle;
         \draw[] (0)--(3);
         \draw[] (2)--(3);
         \draw[] (1)--(3);
        \end{scope}
    \end{tikzpicture}
}

\newcommand{\FigureTildeAlphaBPrimeLast}{
    \begin{tikzpicture}[line join = round, line cap = round]
        \coordinate [label=above:1] (1) at (0,{sqrt(2)},0);
        \coordinate [label=left:2] (2) at ({-.5*sqrt(3)},0,-.5);
        \coordinate [label=below:0] (0) at (0,0,1);
        \coordinate [label=right:3] (3) at ({.5*sqrt(3)},0,-.5);
        \begin{scope}[decoration={markings,mark=at position 0.5 with {\arrow{to}}}]
        \node[] at (0) {$\bullet$};
        \node[] at (2) {$\bullet$};
        \node[] at (1) {$\bullet$};
        \node[] at (3) {$\bullet$};
        \draw[fill=gray,fill opacity=.5] (1)--(2)--(3)--cycle;
        \draw[fill=gray,fill opacity=.5] (0)--(1)--(3)--cycle;
        \draw[fill=gray,fill opacity=.5] (0)--(3)--(2)--cycle;
         \draw[] (0)--(3);
         \draw[] (2)--(3);
         \draw[] (1)--(3);
        \end{scope}
    \end{tikzpicture}
}

\newcommand{\FigureTildeAlphaBHorn}{
    \begin{tikzpicture}[line join = round, line cap = round]
        \coordinate [label=above:1] (1) at (1.1,0.3,0);
        \coordinate [label=left:2] (2) at (0.3,1.8,0);
        \coordinate [label=below:0] (0) at (0,0,0);
        \coordinate [label=right:3] (3) at (1.95,0,0);
        \begin{scope}[decoration={markings,mark=at position 0.5 with {\arrow{to}}}]
        \node[] at (0) {$\bullet$};
        \node[] at (2) {$\bullet$};
        \node[] at (1) {$\bullet$};
        \node[] at (3) {$\bullet$};
         \draw[] (0)--(3);
         \draw[] (2)--(3);
         \draw[] (1)--(3);
        \end{scope}
    \end{tikzpicture}
}

\newcommand{\FigureTildeAlphaBPrimeHorn}{

    \begin{tikzpicture}[line join = round, line cap = round]
    
        \coordinate [label=above:1] (1) at (1.1,0.3,0);
        \coordinate [label=left:2] (2) at (0.3,1.8,0);
        \coordinate [label=below:0] (0) at (0,0,0);
        \coordinate [label=right:3] (3) at (1.95,0,0);
        \begin{scope}[decoration={markings,mark=at position 0.5 with {\arrow{to}}}]
        \node[] at (0) {$\bullet$};
        \node[] at (2) {$\bullet$};
        \node[] at (1) {$\bullet$};
        \node[] at (3) {$\bullet$};
        \draw[fill=gray,fill opacity=.5] (1)--(2)--(3)--cycle;
        \draw[fill=gray,fill opacity=.5] (0)--(1)--(3)--cycle;
         \draw[] (0)--(3);
         \draw[] (2)--(3);
         \draw[] (1)--(3);
        \end{scope}
    \end{tikzpicture}

}

\newcommand{\FigureTildeAlphaBPrimeLastHorn}{
    \begin{tikzpicture}[line join = round, line cap = round]
    
        \coordinate [label=above:1] (1) at (1.1,0.3,0);
        \coordinate [label=left:2] (2) at (0.3,1.8,0);
        \coordinate [label=below:0] (0) at (0,0,0);
        \coordinate [label=right:3] (3) at (1.95,0,0);
        \begin{scope}[decoration={markings,mark=at position 0.5 with {\arrow{to}}}]
        \node[] at (0) {$\bullet$};
        \node[] at (2) {$\bullet$};
        \node[] at (1) {$\bullet$};
        \node[] at (3) {$\bullet$};
        \draw[fill=gray,fill opacity=.5] (1)--(2)--(3)--cycle;
        \draw[fill=gray,fill opacity=.5] (0)--(1)--(3)--cycle;
        \draw[fill=gray,fill opacity=.5] (0)--(1)--(2)--cycle;
         \draw[] (0)--(3);
         \draw[] (2)--(3);
         \draw[] (1)--(3);
        \end{scope}
    \end{tikzpicture}
    \qquad
    \begin{tikzpicture}[line join = round, line cap = round]
        \coordinate [label=left:2] (2) at (0.3,1.8,0);
        \coordinate [label=below:0] (0) at (0,0,0);
        \coordinate [label=right:3] (3) at (1.95,0,0);
        \begin{scope}[decoration={markings,mark=at position 0.5 with {\arrow{to}}}]
        \node[] at (0) {$\bullet$};
        \node[] at (2) {$\bullet$};
        \node[] at (3) {$\bullet$};
        \draw[fill=gray,fill opacity=.5] (2)--(0)--(3)--cycle;
        \end{scope}
    \end{tikzpicture}
}

\begin{table}[htp]
\begin{center}
\begin{tabular}{c|c|c}
Construction~\ref{construction. tilde alpha B} & Construction~\ref{construction. B' of tilde alpha} & Construction~\ref{construction. tilde alpha on jth face} \\
\FigureTildeAlphaB & \FigureTildeAlphaBPrime & \FigureTildeAlphaBPrimeLast \\
\FigureTildeAlphaBHorn & \FigureTildeAlphaBPrimeHorn & \FigureTildeAlphaBPrimeLastHorn \\
\end{tabular}
\end{center}
\caption{Partially filling the horn $\Lambda^3_1$ by accomplishing (a). The top row of images emphasizes the filling of the 3-simplex; the bottom row of images emphasizes the filling of the face opposite the vertex 1, using the decomposition of this face induced by projecting the horn. The subsets $B''$ of Construction~\ref{construction. tilde alpha B} are the edges containing the 3rd vertex. 
The subsets $B'$ of Construction~\ref{construction. B' of tilde alpha} are the faces of $\Delta^3$ that are not the 1st face, nor the 0th face. Construction~\ref{construction. tilde alpha on jth face} defines a piecewise smooth map from the $1$st face, then smooths it out to give a smooth map from the $1$st face. After this, to complete the horn-filling, one need only (b) define data along the face $\Delta^{\{0,1,2\}}$---and more generally, on the face opposite the $n$th vertex.
}
\label{table. horn-filling example outline}
\end{table}

\begin{construction}[Of $\widetilde{\alpha}$, Part I]
\label{construction. tilde alpha [n]}
We set $M_{[n]} := M_{\{n\}}$.
\end{construction}

\begin{remark}\label{remark. inductive step choice of M_[n]}
More generally, we may choose a sectorial diffeomorphism from $M_{\{n\}}$ to some other Liouville sector $Y$, and declare $M_{[n]}$ to be given by $Y$. This freedom will be important for a later inductive step, so that in fact we can choose $M_{[n]}$ to be a pre-given $Y$.
\end{remark}

\begin{choice}[A family $\sigma_{t,u}$ of shinking isotopies]
\label{choice. sigma}
Consider the geometric realization $|\Lambda^n_j|$ of the $j$th $n$-horn. We may up to homeomorphism write $|\Lambda^n_j|$ as a cone/join
	\eqnn
	|\Lambda^n_j| \cong 
	\left(
	\left(
	[0,1] \times \del(\del_j\Delta^n) 
	\right)
	/_{\sim}
	\right)
	\cong
	|\Delta^{\{n\}}| \star \del(\del_j\Delta^n)
	\eqnd
from the $n$th vertex, to the boundary of the $j$th face. Writing the coordinates of $[0,1] \times \del(\del_j\Delta^n)$ as pairs $(t,u)$, we now choose a $u$-family of $t$-parametrized shrinking isotopies
	\eqn
	\sigma_{t,u}: [0,1] \times \del(\del_j\Delta^n) \to \embtop_{\liou}(M_{[n]},M_{[n]})
	\eqnd
satisfying the following:
	\enum
	\item $\sigma_{t,u}$ is the identity for all small $t$ (and hence descends to a map from the join),
	\item $\sigma_{t,u}$ is, along every top-dimensional face of $|\Lambda^n_j|$, collared and smooth.
	\item\label{item. sigma collaring} Finally, we put a restriction on the collaring region of $\sigma_{t,u}$. We demand that $\sigma_{t,u}$ is collared (by the identity) whenever\footnote{See Remark~\ref{remark. collaring of sigma}.} the maps $\widetilde{y}^s$ determined by $\alpha$ are collared near the vertex $|\Delta^{\{n\}}|$, but everywhere else, $\sigma_{t,u}: M_{[n]} \to M_{[n]}$ has image bounded away from the sectorial boundary of $M_{[n]}$. 
	\enumd
\end{choice}

\begin{remark}
\label{remark. collaring of sigma}
We make precise the collaring restriction~\ref{item. sigma collaring} on $\sigma_{t,u}$. By definition, $\alpha$ determines collared families $\widetilde{y}^s_{I'' \subset I'}$, so we know that every pair $I'' \subset I'$---with $n \in I'$ and $I'$ determining some simplex of $\Lambda^n_j$---defines some region $R_{I'' \subset I'} \subset |\Delta^n|$ (i) containing the $n$th vertex and (ii) along which $\widetilde{y}^s$ is independent of $s \in R_{I'' \subset I'}$. (Each region is open in the face $\Delta^{I''}$, but not necessarily in $\Lambda^n_j$.) Let $R \subset |\Delta^n|$ denote the union of all these regions, where the union is indexed by pairs $I'' \subset I'$.

We know that the region $R$ contains a subset $R_{\min}$, which is the union, indexed by $A$ for which $\Delta^{\{n\}} \subset \Delta^A \subset \Lambda^n_j$, of the regions $[0,1]^{I' \setminus A}$ under the image of the collarings $\eta$. (These $A$-dependent regions are each obtained by setting $\epsilon_A = 0$ in~\eqref{eqn. family of embeddings}.) This $R_{\min}$ represents the closed subset containing the $n$th vertex along which everything in sight must be collared. We are then guaranteed that---because each $\epsilon_A>0$---$R$ contains some open neighborhood $U$ of $R_{\min}$. (This $U$ is open and $R$, and may be chosen so that its closure is contained in the interior of $R$.)

So here is the precise statement: We demand that the family $\sigma_{t,u}$ is constant along $R_{\min}$, and that for each $(t,u)$ whose image is not in $R$, $\sigma_{t,u}: M_{[n]} \to M_{[n]}$ does not intersect the sectorial boundary of $M_{[n]}$ (hence, by properness, has image bounded away from the sectorial boundary).
\end{remark}

\begin{construction}[Of $\widetilde{\alpha}$, Part II]
\label{construction. tilde alpha B}
Let $B'' \subset [n]$ be a subset of cardinality $n-1$ containing $n$.
We now construct maps 
	\eqn\label{eqn. tilde alpha B'' B'}
	\widetilde{\alpha}_{B'',B'}:\Delta^{B''} \to \widetilde{\embtop_{\liou}}(M_{B'},M_{[n]})
	\eqnd
for all $B'' \subset B' \subset [n]$ with $B'$ a top-dimensional simplex of $\Lambda^n_j$.\footnote{In other words, $B'$ is a proper subset of $[n]$ of cardinality $n$, and $B' \cup \{j\} \neq [n]$.} 
The goal is to choose~\eqref{eqn. tilde alpha B'' B'} so that the  composite
	\eqn\label{eqn. composition B' B''}
	\Delta^{B''} \xrightarrow{\widetilde{\alpha}_{B'',B'} \times \alpha(B'' \subset B')}
	\widetilde{\embtop_{\liou}}(M_{B'},M_{[n]}) \times \widetilde{\embtop_{\liou}}(M_{B''},M_{B'})
	\xrightarrow{\circ}
	\widetilde{\embtop_{\liou}}(M_{B''},M_{[n]})
	\eqnd
is independent of the choice of $B'$. (Given $B''$, there are at most two $B'$ that fit into a string of proper inclusions $B'' \subset B' \subset [n]$.)

For this, consider the diffeomorphism
	\eqnn
	(\widetilde{y}^{\{n\}}_{\{n\} \subset B'})^{-1}: M_{B'} \to M_{\{n\}} = M_{[n]}
	\eqnd
determined by $\alpha$, and the composition
	\eqn\label{eqn. tilde alpha B''}
	\xymatrix{
	\Delta^{B''} \ar[drrrr]
	\ar[rrrr]^-{\sigma_{t,u} \times (\widetilde{y}^{\{n\}}_{\{n\} \subset B'})^{-1} \times \alpha(B'' \subset B')}
	&&&&\widetilde{\embtop_{\liou}}(M_{[n]},M_{[n]}) \times \widetilde{\embtop_{\liou}}(M_{B'},M_{[n]}) \times \widetilde{\embtop_{\liou}}(M_{B''},M_{B'}) \ar[d]^{\circ}\\
	&&&&\widetilde{\embtop_{\liou}}(M_{B''},M_{[n]})
	}
	\eqnd
where $\sigma_{t,u}$ is the family of isotopies from Choice~\ref{choice. sigma}.
We define this composition to be the value of $\widetilde{\alpha}$ on the inclusion $B'' \subset [n]$; that is, this composition defines
	\eqnn
	\widetilde{\alpha}(B'' \subset [n]).
	\eqnd

(To make sense of this in the case of $\lioudeltadef$, we must also of course construct deformations of Liouville structures on $M_{[n]}$ parametrized by $\Delta^{B''}$. Once we have defined the maps $\widetilde{\alpha}(B'' \subset [n])$, it is straightforward to construct such a deformation by extending from the images of $M_{B''}$. See Remark~\ref{remark. extending deformations}. On the other hand, in the case of $\lioudelta$, where everything is compactly supported, we have defined a map $\Delta^{B''} \to \embtop_{\liou}(M_{B''},M_{[n]})$ on the nose without any need for choices of deformations of the codomain.)

This is a $\Delta^{B''}$-family of (not necessarily strict) sectorial embeddings of $M_{B''}$ into $M_{[n]}$, which near the $n$th vertex of $\Delta^{B''}$ is constant, and is given by
	\eqnn
	(\widetilde{y}^{\{n\}}_{\{n\} \subset B''})^{-1}: M_{B''} \to M_{\{n\}} = M_{[n]}.
	\eqnd
By Isotopy Extension (Proposition~\ref{prop. isotopy extension 2}), there also exists a family of maps
	$
	\widetilde{\alpha}_{B'',B'}
	$
as in~\eqref{eqn. tilde alpha B'' B'} so that the composition~\eqref{eqn. composition B' B''} equals the map $\widetilde{\alpha}(B'' \subset [n])$ given by the composition in~\eqref{eqn. tilde alpha B''}.

We fix such a choice.  Of course, by the usual yoga, we may assume everything in sight is collared and smooth and continuous. To summarize, we have now defined $\widetilde{\alpha}$ on the inclusion $B'' \subset [n]$ for every $B'' \subset [n]$ of cardinality $n-1$ containing $n$; and also defined maps as in~\eqref{eqn. tilde alpha B'' B'}.
\end{construction}

\begin{construction}[Of $\widetilde{\alpha}$, Part III]
\label{construction. B' of tilde alpha}
Now fix $B'\subset[n]$ to be a subset of cardinality $n$ for which $B' \bigcup \{j\} \neq [n]$ and $B' \bigcup \{n\} \neq [n]$. By Construction~\ref{construction. tilde alpha B}, we have a map
	\eqnn
	\bigcup_{B'' \subset B'} \widetilde{\alpha}_{B'',B'}: \bigcup_{B'' \subset B'} \Delta^{B''} \to \widetilde{\embtop_{\liou}}(M_{B'},M_{[n]})
	\eqnd 
where the index runs over all $B'' \subset B'$ for which $B''$ has cardinality $n-1$ and contains $n$. Recognizing the union of $\Delta^{B''}$ as the horn $\Lambda^{B'}_n$ obtained by deleting from $\Delta^{B'}$ the face opposite the vertex $\Delta^{\{n\}}$, we thus have a map
	\eqn\label{eqn. lambda B' step}
	\Lambda^{B'}_n \to \widetilde{\embtop_{\liou}}(M_{B'},M_{[n]}).
	\eqnd
It is smooth, continuous, and collared one very simplex of $\Lambda^{B'}_n$.

Now, the space of Liouville structures on $M_{[n]}$ is a space, and in particular, horns may be filled. Thus we may extend the $\Lambda^{B'}_n$-family of Liouville structures on $M_{[n]}$ implicit in~\eqref{eqn. lambda B' step} to a simplex $\Delta^{B'}$ worth of Liouville structures. As usual, we may choose this family to be collared, smooth, and continuous.

By the usual lifting properties of fibrations, we may extend~\eqref{eqn. lambda B' step} to a continuous map
	\eqn\label{eqn. wilde alpha on B'}
	\widetilde{\alpha}(B' \subset [n]): \Delta^{B'}\to \widetilde{\embtop_{\liou}}(M_{B'},M_{[n]}).
	\eqnd
Then, altering this map if necessary, we may assume~\eqref{eqn. wilde alpha on B'} is actually smooth and collared and continuous by Proposition~\ref{prop. continuous to smooth}. 

To summarize, we have now defined $\widetilde{\alpha}$ on the inclusions $B' \subset [n]$ for all facets $\Delta^{B'} \subset \Delta^n$ for which $\Delta^{B'} \neq \del_n \Delta^n$, and $\Delta^{B'} \neq \del_j \Delta^n$. 
\end{construction}

\begin{construction}[Of $\widetilde{\alpha}$, Part IV]
\label{construction. tilde alpha on jth face}
We have a natural isomorphism of simplicial sets
	\eqnn
	\bigcup_{B'} \del_n \Delta^{B'} \cong \Lambda^{n-1}_j.
	\eqnd
The lefthand side's union is indexed over those $B' \subset [n]$ for which $B' \neq [n] \setminus \{j\}$, and for which $B' \neq [n-1]$. $\del_n \Delta^{B'}$ is the simplicial set given by the face opposite the vertex $\Delta^{\{n\}} \subset \Delta^{B'}$. By Section~\ref{section. max constant simplices}, and because we know $\max B' = n$ for all $B'$, we may change bases and assume that the maps~\eqref{eqn. wilde alpha on B'} define sections
	\eqn\label{eqn. B' maps to Emb n n}
	\Delta^{B'} \to \widetilde{\embtop_{\liou}}(M_{[n]},M_{[n]})
	\eqnd
for all choices of $B'$. By construction, it follows that these maps agree along intersections of the various $\Delta^{B'}$, so they glue to give a single map
	\eqnn
	\bigcup_{B'} \del_n \Delta^{B'} \cong \Lambda^{n-1}_j \to \widetilde{\embtop_{\liou}}(M_{[n]},M_{[n]}).
	\eqnd
Again by extending the $\Lambda^{n-1}_j$-family of Liouville structures to a $\Delta^{n-1}$-family, then using the fact that $\widetilde{\embtop_{\liou}}$ is a Kan fibration over the space of Liouville structures, we may fill the horn to yield a map
	\eqn\label{eqn. fill Delta n-1 map to Emb n n}
	\Delta^{n-1} \to \widetilde{\embtop_{\liou}}(M_{[n]},M_{[n]})
	\eqnd
where we have used Remark~\ref{remark. extending deformations} again to know what we mean by Liouville structures on $M_{[n]}$ over $s \in \Delta^{n-1}$. We may as usual assume everything in sight is smooth, collared, and continuous.

By gluing~\eqref{eqn. fill Delta n-1 map to Emb n n} and~\eqref{eqn. B' maps to Emb n n} we obtain a map
	\eqn\label{eqn. piecewise map to Emb n n}
	\left(\bigcup_{B'} \Delta^{B'}\right) \bigcup_{\Lambda^{n-1}_j} \Delta^{n-1} \to  \widetilde{\embtop_{\liou}}(M_{[n]},M_{[n]}).
	\eqnd
We now note that, at the level of geometric realizations, there is a piecewise smooth homeomorphism\footnote{
This can be seen, for example, by projecting the horn $\Lambda^n_j$ to the face $|\del_j \Delta^n|$ along a vector normal to $|\del_j \Delta^n|$ and passing through the vertex $|\Delta^{\{j\}}|$. }
	\eqnn
	\left(\bigcup_{B'} |\Delta^{B'}| \right) \bigcup_{|\Lambda^{n-1}_j|} |\Delta^{n-1}|
	\cong |\del_j \Delta^n|.
	\eqnd
Thus, the map~\eqref{eqn. piecewise map to Emb n n} induces a piecewise smooth map from $|\del_j \Delta^n|$ to $\widetilde{\embtop_{\liou}}(M_{[n]},M_{[n]})$. Importantly, the map~\eqref{eqn. piecewise map to Emb n n} is collared along all the boundaries and corners; thus, Proposition~\ref{prop. continuous to smooth} allow us to continuously homotope the piecewise smooth map to an honest smooth map, and collared along the boundary of $|\del_j \Delta^n|$. Thus, we have a map
	\eqnn
	\del_j \Delta^n \to \widetilde{\embtop_{\liou}}(M_{[n]},M_{[n]})
	\eqnd
that is smooth and continuous and collared according to our collaring conventions. By undoing our change of basis, we obtain a map
	\eqn\label{eqn. tilde alpha on jth face}
	\widetilde{\alpha}([n] \setminus\{j\} \subset [n]): \del_j \Delta^n \to \widetilde{\embtop_{\liou}}(M_{[n] \setminus\{j\} },M_{[n]})
	\eqnd
(End of Construction~\ref{construction. tilde alpha on jth face}.)
\end{construction}

To this point, we have constructed $\widetilde{\alpha}(B \subset [n])$ for all facets $\Delta^B \subset \Delta^n$ that contain the vertex $\Delta^{\{n\}}$. It remains to define $\widetilde{\alpha}$ on the face opposite the $n$th vertex. Unraveling the definition of a simplex of $\lioudeltadef$, and again assuming that both $\alpha$ and $\widetilde{\alpha}$ as max-constant for simplicity (Section~\ref{section. max constant simplices}), this means we must produce a map
	\eqn\label{eqn. tilde alpha on nth face}
	\widetilde{\alpha}([n-1] \subset [n]): \Delta^{n-1} \to \widetilde{\embtop_{\liou}}(M_{\{n-1\}}, M_{\{n\}}).
	\eqnd
This will be the goal of Sections~\ref{section. b 1} and~\ref{section. b 2}. 

\subsection{\texorpdfstring{$n \geq 3$}{n at least 3} and \texorpdfstring{$j \neq n-1$}{j at most n-1}}
\label{section. b 1}
Assume $j \neq n-1$. Again we assume that all simplices of $\alpha$ are max-constant (Section~\ref{section. max constant simplices}). Then $\alpha$ assigns 
	\enum
	\item to the inclusion $[n-2] \subset [n] \setminus \{n-1\}$ a map
    	\eqnn
    	\alpha([n-2] \subset [n] \setminus \{n-1\}): \Delta^{n-2} \to \widetilde{\embtop_{\liou}}(M_{\{n-2\}}, M_{\{n\}})
    	\eqnd
	(This is the point at which we are using that $j \neq n-1$.) By post-composing with (the restriction to $\Delta^{n-2} \subset \Delta^{n-1}$ of) the map $\widetilde{\alpha}([n] \setminus \{n-1\} \subset [n])$, we obtain a map
		\eqn\label{eqn. j not n-1, a}
		\Delta^{n-2} \to \widetilde{\embtop_{\liou}}(M_{\{n-2\}},M_{\{n\}}).
		\eqnd
	\item to the inclusion $[n-2] \subset [n-1]$ a map
    	\eqn\label{eqn. j not n-1, b}
    	\alpha([n-2] \subset [n-1]): \Delta^{n-2} \to \widetilde{\embtop_{\liou}}(M_{\{n-2\}},M_{\{n-1\}}),
    	\eqnd
	\item and, for any $A \subset [n]$ with $\max A = n-1$, and for any $i \neq j$, to the inclusion $A \subset [n] \setminus \{i\}$ a map
		\eqnn
		\alpha(A \subset [n] \setminus \{i\}): \Delta^A \to \widetilde{\embtop_{\liou}}(M_{\{n-1\}}, M_{\{n\}}).
		\eqnd
	By Construction~\ref{construction. tilde alpha B}, if we post-compose with $\widetilde{\alpha}([n]\setminus \{i\} \subset [n])$, the induced map from $\Delta^A$ to $\widetilde{\embtop_{\liou}}(M_{\{n-1\}}, M_{\{n\}})$ is independent of the choice of $i$.	In particular, by gluing together those $A$ for which $\Delta^{A \setminus \{n-1\}} \subset \del \Delta^{n-2}$, we obtain a map
		\eqn\label{eqn. j not n-1, c}
		\del \Delta^{n-2} \to \widetilde{\embtop_{\liou}}(M_{\{n-1\}},M_{\{n\}}).
		\eqnd
	\enumd
By construction, we are guaranteed that the composition of~\eqref{eqn. j not n-1, b} with~\eqref{eqn. j not n-1, c} equals the map~\eqref{eqn. j not n-1, a} when restricted to $\del \Delta^{n-2}$.
	
By isotopy extension, we may thus extend~\eqref{eqn. j not n-1, c} to a map
	\eqn\label{eqn. tilde alpha on n-2 simplex}
	\Delta^{n-2} \to \widetilde{\embtop_{\liou}}(M_{\{n-1\}},M_{\{n\}}).
	\eqnd
Now, by again composing the data of (i) $\alpha([n-1] \setminus \{i\} \subset [n] \setminus \{i\})$ for $i \neq j, n-1$, with (ii) $\widetilde{\alpha}([n] \setminus \{i\} \subset [n])$, we obtain maps
	\eqn\label{eqn. j not n-1, e}
	\Delta^{[n-1] \setminus \{i\}} \to \widetilde{\embtop_{\liou}}(M_{\{n-1\}}, M_{\{n\}}),
	\qquad
	i\neq j, n.
	\eqnd
By construction, the maps~\eqref{eqn. tilde alpha on n-2 simplex} and~\eqref{eqn. j not n-1, e} agree on their overlaps, so we have a map
	\eqnn
	\Lambda^{n-1}_j \to \widetilde{\embtop_{\liou}}(M_{\{n-1\}},M_{\{n\}}).
	\eqnd
Because $\widetilde{\embtop_{\liou}}(M_{\{n-1\}},M_{\{n\}})$ is a Kan complex, we may fill this horn; and again Proposition~\ref{prop. continuous to smooth} allows us to deform this filler to be smooth and continuous and collared. The result is a map
	\eqnn
	\widetilde{\alpha}([n-1] \subset [n]): \Delta^{n-1} \to \widetilde{\embtop_{\liou}}(M_{\{n-1\}},M_{\{n\}})
	\eqnd
which, as we have indicated by the notation, defines $\widetilde{\alpha}([n-1] \subset [n])$. This finishes our construction of the desired map~\eqref{eqn. tilde alpha on nth face} when $j \neq n-1$. (By construction, all the composition requirements are met, so $\widetilde{\alpha}$ indeed defines a simplex of $\lioudelta$.)

\subsection{\texorpdfstring{$n \geq 3$}{n at least 3} and \texorpdfstring{$j = n-1$}{j equals n-1}}
\label{section. b 2}
This is an easier case than the previous one; as it turns out, we have fewer constraints because $\alpha$ does not make any assignment to the inclusion $[n-2] \subset [n] \setminus \{n-1\}$. Indeed, by post-composing the maps
	\eqnn
	\alpha([n-1] \setminus \{i\} \subset [n] \setminus \{i\}) : \Delta^{[n-1] \setminus \{i\}} \to \widetilde{\embtop_{\liou}}(M_{\{n-1\}},M_{\{n\}})
	\eqnd
with (the restrictions of) the maps
	\eqnn
	\widetilde{\alpha}([n] \setminus \{i\} \subset [n]): \del_i \Delta^n \to \widetilde{\embtop_{\liou}}(M_{\{n\}},M_{\{n\}}),
	\qquad
	i \neq n-1, n
	\eqnd
from Construction~\ref{construction. B' of tilde alpha}, we have maps
	\eqnn
	\Delta^{[n-1] \setminus \{i\}} \to \widetilde{\embtop_{\liou}}(M_{\{n-1\}},M_{\{n\}}),
	\qquad i \neq n-1
	\eqnd
which agree on the overlaps. In other words, we obtain a single map
	\eqnn
	\Lambda^{n-1}_{n-1} \to \widetilde{\embtop_{\liou}}(M_{\{n-1\}},M_{\{n\}}).
	\eqnd
We may again use the Kan property of the codomain to extend this to a map from $\Delta^{n-1}$, and we may deform this map to be smooth and collared. We declare the resulting map
	\eqnn
	\Delta^{n-1} \to \widetilde{\embtop_{\liou}}(M_{\{n-1\}},M_{\{n\}})
	\eqnd
to be the desired map~\eqref{eqn. tilde alpha on nth face}. It is straightforward, while tedious, to check that all composition requirements are met.

\subsection{Proof of Theorem~\ref{theorem. weak kan} and~\ref{theorem. weak kan def}}
The $n=2$ case was accomplished in Section~\ref{section. n=2 horn filling}.
For $n \geq 3$, unwinding the definition of $\tilde \alpha$, the horn is filled when we produce maps~\eqref{eqn. wilde alpha on B'},
\eqref{eqn. tilde alpha on jth face}, and
\eqref{eqn. tilde alpha on nth face}, satisfying the compatibilities of agreeing along faces and agreeing with $\alpha$. This was accomplished by construction. 

\subsection{\texorpdfstring{$\lioudeltadefstab$}{Stabilized Liou Delta Def} is weak Kan}
Simplices of $\lioudeltadef$ encode sectorial embeddings $M_B \times T^*\Delta^B \times T^*[-1,\epsilon) \to M \times T^*\Delta^A$. All structures are preserved by taking the direct product with $T^*[0,1]$ on the left, so we have a functor
	\eqnn
	T^*[0,1] \times - : \lioudeltadef \to \lioudeltadef,
	\qquad
	M \mapsto T^*[0,1] \times M.
	\eqnd
We thus define the semisimplicial set $\lioudeltadefstab$ as an increasing union, as in~\eqref{eqn. lioudeltadefstab}.

\begin{corollary}\label{cor. lioudeltadef stab is weak Kan}
The semisimplicial set $\lioudeltadefstab$ satisfies the (semisimplicial) weak Kan condition.
\end{corollary}

\begin{proof}
Any (semisimplicial) map of $\Lambda^n_j$ into a filtered colimit will factor through a finite stage of the filtration. Thus the corollary reduces to filling horns $\Lambda^n_i \to \lioudeltadef$, which is Theorem~\ref{theorem. weak kan def}.
\end{proof}

\begin{remark}
Here we see a key difference between $\lioudelta$ and $\lioudeltadef$. For the former, $T^*[0,1] \times M$ is not a functor, as the compact-support condition for deformations is not preserved under product with (non-compact) $T^*[0,1]$. On the other hand, we saw in Section~\ref{section. lioudelta properties} that $\lioudelta$ admits straightforward comparisons to the usual theory of sectorial embeddings and their isotopies; it is more geometrically tractable than $\lioudeltadef$. But it does not stabilize easily.
\end{remark}

\clearpage

\section{\texorpdfstring{$\lioudelta$}{Liou Delta} as an \texorpdfstring{$\infty$}{infinity}-category}
\subsection{Idempotent equivalences and constant simplices}

\begin{lemma}\label{lemma. idempotents self equivalences}
Every object $M \in \lioudelta, \lioudeltadef$ admits an idempotent equivalence---i.e., an edge from $M$ to $M$ that is an idempotent equivalence (Definition~\ref{defn. idempotent equivalence}).
\end{lemma}

\begin{proof}
Given an object $M$, define a 1-simplex $e_M$ by declaring
	\eqnn
	M_{\{0,1\}}  \times T^*\Delta^1 := M \tensor T^*\Delta^1
	\eqnd
and set the maps $y_{\{i\} \subset [1]} := \id_M$. We claim this is an idempotent self-equivalence. It is obviously idempotent by witnessing the 2-simplex whose $\Delta$-movies, for any $A \subset [2]$, are given by
	\eqnn
	M_{\{A\}} \times T^*\Delta^A := M \tensor T^*\Delta^A
	\eqnd
and whose morphisms $y_{\{A \subset B\}}^s := \id_M$ are $s$-independent. This exhibits a 2-simplex with all three faces equaling $e_M$, so this shows that $e_M$ is idempotent.

(a) of Definition~\ref{defn. idempotent equivalence}: By assumption, a horn $\alpha$ gives us identifications $M_{\{n\}} = M_{\{n-1\}} = M$. Thus, by symmetry, $\alpha$ is equivalent to the data of a map $\Lambda^n_{n-1} \to \lioudelta$ (or $\lioudeltadef$). Concretely: Consider the bijection $\sigma: [n] \to [n]$ that swaps $n-1$ and $n$, leaving all other elements alone. (This is not order-preserving.) We have an induced map on subsets by $A \mapsto \sigma(A)$. Then we can define a map
	\eqnn
	\alpha^\sigma: \Lambda^n_{n-1} \to \lioudelta,
	\qquad
	A \mapsto \alpha(\sigma(A)) = M_{\sigma(A)},
	\qquad
	\phi_{A \subset B} \mapsto \phi_{\sigma(A) \subset \sigma(B)}.
	\eqnd
Because the collaring functions are symmetric with respect to permutation actions, these $\phi$ are all collared. It is clear that $\alpha^\sigma$ still sends sequences of inclusions to compositions, the assignment satisfies~\ref{item. lioudelta composition}. And \ref{item. lioudelta is max localizing} is still satisfied because $M_{\{n\}} = M_{\{n-1\}}$. 
 Now, because $\alpha^\sigma$ is an inner horn, it may be filled by a map $\widetilde{\alpha^\sigma}$ from $\Delta^n$ (Theorems~\ref{theorem. weak kan} and~\ref{theorem. weak kan def}). By pre-composing the data of $\widetilde{\alpha^\sigma}$ with the automorphism of $T^*\Delta^n$ induced by $\alpha$, we obtain a simplex $\widetilde{\alpha}$ filling $\alpha$.

The proof for (b)  of Definition~\ref{defn. idempotent equivalence} is similar.
\end{proof}

\begin{cor}\label{cor. lioudelta stab has idempotents}
Every object of the semisimplicial set $\lioudeltadefstab$ admits an idempotent equivalence.
\end{cor}

\begin{proof}
Any map of $\Lambda^n_n$ or $\Lambda^n_0$ to $\lioudeltadefstab$ factors through some finite stage of the filtered colimit. 
\end{proof}

\begin{defn}[Constant simplices]
\label{defn. constant simplices in lioudelta}
Given any object (i.e., vertex) $M$ of $\lioudelta$ (or $\lioudeltadef$), we define the associated {\em constant $I$-simplex} as follows:
\begin{itemize}
\item For every subset $I' \subset I$, we declare the $\Delta^{I'}$-movie to be $M \tensor T^*\Delta^{I'}.$
\item We declare, for every inclusion $I'' \subset I'$, the collared movie $\phi_{I'' \subset I'}$ to be the collared movie associated to the constant movie of the identity function. In other words, we set $\phi_{I'' \subset I'} = \id_{M} \times T^*\eta_{I'' \subset I'}$. 
\end{itemize}
\end{defn}

\begin{remark}
The idempotent edge from
Lemma~\ref{lemma. idempotents self equivalences}
is precisely the constant edge associated to the object $M$. 
\end{remark}

\begin{remark}\label{remark. constant simplices of lioudelta form a simplicial set}
The collection of constant simplices defines a semisimplicial subset of $\lioudelta$ with the same underlying set of vertices.

Moreover, the collection of constant simplices may obviously be promoted to a simplicial set isomorphic to a disjoint union of $\Delta^0$.
\end{remark}

\subsection{Proof of Theorem~\ref{theorem. liou delta is oo-cat}}
\label{section. lioudelta is oo cat}

We treat categories as simplicial sets (and semisimplicial sets) by taking nerves (and forgetting degeneracies).
Recall the category $\lioustr$.

\begin{construction}\label{construction. j from lioustr to lioudelta}
We now define a map of semisimplicial sets
	\eqnn
	\lioustr \to \lioudelta
	\eqnd
(and hence a map $\lioustr \to \lioudeltadef$ by recalling that $\lioudelta$ is a sub-semsimplicial-set of $\lioudeltadef$).
Recall that an $n$-simplex of (the nerve of) $\lioustr$ is determined by (i) Liouville sectors $(M_i, \lambda_i)$ for $i=0,1,\ldots,n$, and (ii) for every $i<j$, a strict embedding 
	\eqn\label{eqn. morphisms in lioustr}
	f_{i,j}: M_i \to M_j
	\eqnd
satisfying $f_{i,k} = f_{j,k} \circ f_{i,j}$. 
Given such data, we define an $n$-simplex of $\lioudelta$ as follows.
\enum
	\item For any subset $I' \subset [n]$, 
		\eqnn
		M_{I'} \times T^*\Delta^{I'}
		:=
		M_{\max I'} \tensor T^*\Delta^{I'}.
		\eqnd
	\item For any $a: I'' \subset I'$, we define
		\eqnn
		\phi_{I'' \subset I'} = f_{\max I'', \max I'} \times T^*\eta_{a}.
		\eqnd
\enumd
\end{construction}

\begin{remark}
This construction is somewhat easier to conceptualize after passing to stabilized versions of $\lioustr$ and $\lioudelta$ -- see Examples~\ref{example. j 1 simplex} and~\ref{example. j 2 simplex}.
\end{remark}

\begin{theorem}\label{theorem. lioudelta is an infinity-cat}
The semisimplicial set $\lioudelta$ may be promoted to be a simplicial set satisfying the following properties:
\enum[(a)]
\item The face maps are unchanged.
\item The map of semisimplicial sets $\lioustr \to \lioudelta$ (Construction~\ref{construction. j from lioustr to lioudelta}) is a map of simplicial sets.
\item\label{item. constant degeneracies are constant} The degeneracies of constant simplices are constant simplices (Definition~\ref{defn. constant simplices in lioudelta}).
\item $\lioudelta$ is an $\infty$-category.
\enumd

\end{theorem}

\begin{notation}
From now on, we also let $\lioudelta$ denote the simplicial set guaranteed by Theorem~\ref{theorem. lioudelta is an infinity-cat}.  This abuse of notation will not cause too much confusion---often, the important structure will be whether a {\em map} to/from $\lioudelta$ is semisimplicial or simplicial. 
\end{notation}

\begin{remark}
The last item of Theorem~\ref{theorem. lioudelta is an infinity-cat} is Theorem~\ref{theorem. liou delta is oo-cat} from the introduction. Theorem~\ref{theorem. lioudelta is an infinity-cat} elaborates on the various properties that the $\infty$-category $\lioudelta$ satisfies.
\end{remark}

\begin{proof}[Proof of Theorem~\ref{theorem. lioudelta is an infinity-cat}.]
By Theorem~\ref{theorem. weak kan}, $\lioudelta$ satisfies the semisimplicial weak Kan condition. By Lemma~\ref{lemma. idempotents self equivalences}, every object admits an idempotent self-equivalence. 
Thus Theorem~\ref{theorem. steimle} shows that we can satisfy (b) and (c) by taking $A$ (see the notation in Theorem~\ref{theorem. steimle}) to be the the image of $\lioustr$ (which also contains all constant simplices as in Remark~\ref{remark. constant simplices of lioudelta form a simplicial set}). We finally note that, indeed, Construction~\ref{construction. j from lioustr to lioudelta} is compatible with the idempotent self-equivalences from the proof of Lemma~\ref{lemma. idempotents self equivalences}.
\end{proof}

\begin{theorem}\label{theorem. lioudeltadef is an infinity-cat}
The semisimplicial set $\lioudeltadef$ may be promoted to be a simplicial set satisfying the following properties:
\enum[(a)]
\item The face maps are unchanged.
\item\label{item. lioudelta maps to lioudeltadef} The map of semisimplicial sets $\lioudelta \to \lioudeltadef$ is a map of simplicial sets when $\lioudelta$ is given the simplicial set structure from Theorem~\ref{theorem. lioudelta is an infinity-cat}. (In particular, the degeneracies of constant simplices are constant simplices.)
\item $\lioudeltadef$ is an $\infty$-category.
\enumd
\end{theorem}

\begin{proof}
Identical to the proof of Theorem~\ref{theorem. lioudelta is an infinity-cat}, except one takes $A$ to be $\lioudelta$.
\end{proof}

\subsection{Some of the geometry of \texorpdfstring{$\lioudelta$}{Liou Delta}}
\label{section. lioudelta geometry}
Concrete geometric constructions in $\lioudelta$ can be given $\infty$-categorical interpretations. We have already seen in Proposition~\ref{prop. 2-simplex is an isotopy} that 2-simplices in $\lioudelta$ encode isotopies (through sectorial embeddings) of sectorial embeddings.

And in Example~\ref{example. edges in lioudelta are non-strict embeddings}, we saw that 1-simplices encode (not necessarily strict) sectorial embeddings. We now record the converse construction.

\begin{remark}
 
We could have introduced the following construction back in Section~\ref{section. lioudelta properties}, but we have waited until we established the $\infty$-categorical structure (namely, degeneracies) of $\lioudelta$. The reason is so that we can then state Proposition~\ref {prop. homotopic maps are homotopic}, which requires a notion of homotopy to state precisely (and hence, in turn, requires us to know the degenerate edges of $\lioudelta$).
\end{remark}

\begin{construction}\label{construction. edges in lioudelta}
Let $f: M_0 \to M_1$ be a sectorial embedding (Definition~\ref{defn. sectorial embedding}) so that $f^*\lambda_1 = \lambda_0 + dh$ for some compactly supported $h$. Then for every 
\begin{itemize}
\item smooth, compactly supported function $\tilde h: M_1 \to \RR$ extending $f_* h$, and  
\item smooth, collared, weakly increasing function $\chi: \Delta^1 \to \RR_{\geq 0}$ interpolating from 1 to 0 (so that $\chi$ equals 1 at the initial vertex of $\Delta^1$, and equals 0 at the terminal vertex),
\end{itemize}
one obtains an edge of $\lioudelta$ by declaring
	\eqn\label{eqn. edge construction}
	M_{01} \times T^*\Delta^1= (M_1 \times T^*\Delta^1 , \lambda_1 \oplus pdq + d(\chi(q)\tilde h)),
	\eqnd
and
\begin{itemize}
\item $(M_0 \times T^*[-1,0] \to  M_{01} \times T^*\Delta^1) = f \times T^*\eta_{0 \subset 01}$,
\item  $(M_1 \times T^*[-1,0] \to  M_{01} \times T^*\Delta^1 )= \id \times T^*\eta_{1 \subset 01}$.
\end{itemize}
\end{construction}

\begin{prop}\label{prop. homotopic maps are homotopic}
Fix a map $f: M_0 \to M_1$ for which $f^*\lambda_1 = \lambda_0 + dh$ for some compactly supported function $h: M_0 \to \RR$.
\enum[(a)]
\item \label{item. choices of tilde h and chi} For any two choices of $\tilde h$ and $\chi$, Construction~\ref{construction. edges in lioudelta} results in homotopic edges of $\lioudelta$. (Informally: Construction~\ref{construction. edges in lioudelta} is independent of the choices of $\chi,\tilde h$ up to homotopy.)
\item\label{item. homotopic maps are homotopic}
Further, if $f$ is smoothly isotopic through sectorial embeddings to $f'$, then Construction~\ref{construction. edges in lioudelta} for $f$ and $f'$ result in homotopic edges in $\lioudelta$.  (Informally: Construction~\ref{construction. edges in lioudelta} is compatible with the relation of isotopy.)
\item\label{item. composition} Given another sectorial embedding $g: M_1 \to M_2$, there exists a triangle in $\lioudelta$ whose 02, 01, and 12 edges are associated to $g \circ f$, $f$, and $g$, respectively. (Informally: Construction~\ref{construction. edges in lioudelta} respects composition up to homotopy.)
\item\label{item. equivalences are equivalences}
In particular, if $f$ is a Liouville equivalence, then the edge associated to $f$ is an equivalence in $\lioudelta$.
\enumd
\end{prop}

\begin{proof}
\eqref{item. choices of tilde h and chi} It suffices to construct a 2-simplex in $\lioudelta$ for which the edge from 0 to 1 is a degenerate edge for $M_0$, the edge from 1 to 2 is constructed using $\chi$ and $h$, and the edge from 0 to 2 is constructed using some other choices $\chi'$ and $h'$.

Let $M_{01} \times T^*\Delta^1$ denote the Liouville sector constructed out of $\chi$ and $\tilde h$~\eqref{eqn. edge construction}. We let $(M_{01} \times T^*\Delta^1)'$ denote the Liouville sector constructed out of $\chi'$ and $\tilde h'$. By definition, we have smooth embeddings
	\eqnn
	(M_{01} \times T^*\Delta^1 \times T^*[-1,0]) \to M_{1} \times T^*\Delta^2
	\leftarrow (M_{01} \times T^*\Delta^1 \times T^*[-1,0])'
	\eqnd
each given by $\id \times T^*\eta$ (where the rightward pointing $\eta$ is induced by the inclusion $\{1,2\} \subset \{0,1,2\}$ and the leftward pointing $\eta$ is induced by the inclusion $\{0,2\} \subset \{0,1,2\}$). Because everything in sight is collared, we may extend (the pushforwards of) both $\chi\tilde h$ and $\chi'\tilde h'$ to a smooth, compactly supported function $\tilde H: M_1 \times \Delta^2 \to \RR$; moreover, by the smooth contractibility of the space of collared and compactly supported smooth functions, we may further demand that $\tilde H$ is collared by a function $h_{01} : M_1 \times \Delta^1 \to \RR$ along the edge from 0 to 1 for which $f^* h_{01} = h$. By declaring the inclusion $M_0 \times T^*\Delta^1 \to M_{1} \times T^*\Delta^2$ to be given by $f \times T^*\eta_{01 \subset 012}$, the totality of the data in this paragraph exhibits the desired two-simplex.

\eqref{item. homotopic maps are homotopic} Let $f_s$ denote a homotopy between two sectorial embeddings $f = f_0$ and $f' = f_1$. Without loss of generality, we may assume that $f_s$ is parametrized by $s \in \Delta^1$, and that $f_s$ is collared (and in particular, $s$-independent near the endpoints of $\Delta^1$). 

We may construct a strict sectorial embedding
	\eqnn
	\tilde f: M_0 \tensor T^*\Delta^1
	\to
	M_1 \times T^*\Delta^1
	\eqnd
where the Liouville structure on the codomain depends on the family $f_s$, and is given by $\lambda_1 \oplus pdq + H$ for some smooth, compactly supported function $H: M_1 \times \Delta^1 \to \RR$. By identifying $\Delta^1$ with the 2nd face of $\Delta^2$ (i.e., with the face spanning the 0th and 1st vertices) and again by the collaring condition of everything in sight, we may extend the trio of $H$, $\chi \tilde h$ and $\chi' \tilde h'$ to a single smooth, compactly supported function $\tilde H: M_1 \times \Delta^2 \to \RR$. (Here, $\chi \tilde h$ is pushed forward along the inclusion of the 12 face, constant along the collaring, and likewise $\chi' \tilde h'$ is pushed forward along the inclusion of the 02 face.) The totality of the data in this paragraph exhibits a 2-simplex collared by $\id_{M_0}$ (along the 01 edge) and the edges associated to $f$ and $f'$ from Construction~\ref{construction. edges in lioudelta}; in other words, we have constructed a homotopy from $f$ to $f'$.

\eqref{item. composition} We define $M_{012}=M_2$ as a Liouville sector, and define a smooth embedding
	\eqnn
	M_{01} \times T^*\Delta^1 := M_1 \times T^*\Delta^1 \to M_2 \times T^*\Delta^2
	\eqnd
by $g \times T^*\eta_{01 \subset 012}$. By our collaring assumptions, the three pushforwards
	\eqnn
	(g \times \eta_{01 \subset 012})* \chi_{01}\tilde h_{01},
	\qquad
	(\id_{M_2} \times \eta_{12 \subset 012})_* \chi_{12}\tilde h_{12},
	\qquad
	(\id_{M_2} \times \eta_{02 \subset 012})_* \chi_{02} \tilde h_{02}	
	\eqnd
(which are smooth functions defined on subsets of $M_2 \times \Delta^2$) agree along their overlaps; we can choose a smooth, compactly supported function  $H: M_2 \times \Delta^2 \to \RR$ extending all three. This defines our Liouville sector $M_{012} \times T^*\Delta^2 = (M_2 \times T^*\Delta^2, \lambda_2 \oplus pdq + dH)$, and the three face inclusions are defined by
	\eqnn
	g \times T^*\eta_{01 \subset 012}
	\qquad
	\id_{M_2} \times T^*\eta_{12 \subset 012}
	\qquad
	\id_{M_2} \times T^*\eta_{02 \subset 012}.
	\eqnd

\eqref{item. equivalences are equivalences} This is obvious in light of \eqref{item. homotopic maps are homotopic} and \eqref{item. composition}.
\end{proof}

\begin{remark}
When every sectorial embedding in sight is {\em strict}, we may choose every extension to be constant in the simplicial directions; by declaring every $h$ and $H$ to be zero (and hence removing the dependence on $\chi$) the proof of Proposition~\ref{prop. homotopic maps are homotopic} in fact recovers the functor of $\infty$-categories from $N(\lioustr)$ to $\lioudelta$ from Theorem~\ref{theorem. lioudelta is an infinity-cat}. 
\end{remark}

\begin{remark}
Indeed, one can think of an edge in $\lioudelta$ as a choice of embedding $f: M_0 \to M_1$, along with an interpolation of $\lambda_1$ to an extension of $f_* \lambda_0$ along compactly supported functions $h_t: M_1 \to \RR$. The rough idea is that the forgetful map $(f,\{h_t\}) \mapsto f$ has contractible fibers (the space of $h$ is convex) so that the space of edges $(f,h_t)$ in $\lioudelta$ is the space of (not necessarily strict) Liouville embeddings $f$.
\end{remark}

\begin{example}[Isotopies]
\label{example. isotopies are 2 simplices}
The converse to Proposition~\ref{prop. 2-simplex is an isotopy} is also true---any isotopy gives rise to a 2-simplex in $\lioudelta$. 

Let us illustrate this in a simple example: Given two Liouville sectors $M$ and $N$, and a smooth isotopy of sectorial embeddings $\{f_t: M \to N\}_{t \in [0,1]}$, we will produce a 2-simplex in $\lioudelta$ that ``encodes'' this data (in a way depending on some choices).

We choose
	\eqnn
	M_{\{0\}} =
	M_{\{0,1\}} =
	M_{\{1\}} =
	M
	\eqnd
with the maps $M_{\{i\}}\tensor T^*[-1,\epsilon] \to M_{\{0,1\}} \times T^*\Delta^{\{0,1\}}$ given by $\id_M \times T^*\eta_{\{i\} \subset \{0,1\}}$ for $i=0,1$. We also choose
	\eqnn
	M_{\{2\}} =
	M_{\{0,2\}} =
	M_{\{1,2\}} =
	M_{\{0,1,2\}} =
	N
	\eqnd
choosing all relevant maps among these to be given by $\id_N \times T^*\eta$.

Finally, let us convert $\{f_t\}_{t \in [0,1]}$ to a collared family $\{\widetilde{y}^s: M \to N\}_{s \in \Delta^{\{0,1\}}}$ by choosing some orientation-respecting $\phi: \Delta^{\{0,1\}} \to [0,1]$ that sends some open neighborhood of the initial vertex $\Delta^{\{0\}}$ to $0 \in [0,1]$, and some open neighborhood of the vertex $\Delta^{\{1\}}$ to $1 \in [0,1]$.

Then we are finished by 
\begin{itemize}
\item declaring the map 
	$y^0_{\{0\} \subset \{0,2\}}
	M = M_{\{0\}} \to M_{\{0,2\}} = N
	$
to be the map $f_0$, 
\item declaring the map  $y^1_{\{1\} \subset \{1,2\}}: 	M = M_{\{1\}} \to M_{\{1,2\}} = N $ to be the map $f_1$,
\item declaring the family $\{y^s_{\{0,1\} \subset \{0,1,2\}}: M_{\{0,1\}} \to M_{\{0,1,2\}}\}_{s \in \Delta^{\{0,1\}}}$ to be $y^s_{\{0,1\} \subset \{0,1,2\}} = f_{\phi(s)}$, and
\item   Choosing a collared, compactly supported smooth function $N \times \Delta^2 \to \RR$ so that all embeddings are strict (similar to Construction~\ref{construction. edges in lioudelta}).
\end{itemize}
We may informally draw this 2-simplex as follows:
	\eqn\label{eqn. 2-simplex isotopy}
	\xymatrix{
	&& M_2 \ar[dr]^{\sim} \ar[dl]_{\sim} \ar[d]^{\sim}&& \\
	& M_{02} \ar@{=>}[r]^{\sim} & M_{012} & M_{12} \ar@{=>}[l]_{\sim} & \\
	M_0\ar[rr] \ar[ur] \ar[urr] && M_{01} \ar@{=>}[u] && M_1\ar[ll]_{\sim} \ar[ull] \ar[ul]
	}
	=
	\xymatrix{
	&& N \ar[dr]^{=} \ar[dl]_{=} \ar[d]^{=}&& \\
	& N \ar@{=>}[r]^{=} & N & N\ar@{=>}[l]_{=} & \\
	M \ar[rr]^{=} \ar[ur]^{f_0} \ar[urr] && M  \ar@{=>}[u]_{f_t} && M \ar[ll]_{=} \ar[ull] \ar[ul]_{f_1}
	}
	\eqnd
Above, the double arrows $\implies$ indicate the data of a family of morphisms, and arrows labeled by an equality sign indicate the identity morphism, or a constant family of identity morphisms.

\end{example}

\begin{remark}
There is likewise a way to construct a 2-simplex out of the data of $\{f_t\}$ for which the 1st and 2nd vertices are given by $N$, while only the 0th vertex is given by $M$. (This is in contrast to Example~\ref{example. isotopies are 2 simplices}, whose 2-simplex has the property that the 1st vertex is given by $M$.) Such a construction involves extending the isotopy $f_t$ to an isotopy of $N$ to itself. This is possible thanks to isotopy extension (Proposition~\ref{prop. isotopy extension}). 

We leave the details to the reader, but point out the following: If $\lioudelta$ is to be promoted to be an $\infty$-category, then there should be no appreciable homotopic difference between the space of 2-simplices where the $0,1$ edge is constant, and the space of 2-simplices whose $1,2$ edge is constant. We see here the importance of isotopy extension in even having a hope for this to be true.
\end{remark}

\clearpage
\section{Mapping spaces of \texorpdfstring{$\lioudelta$}{Liou Delta}}
\label{section. lioudelta homs}
We have already seen that edges in $\lioudelta$ encode (not necessarily strict) sectorial embeddings, and that triangles in $\lioudelta$ encode isotopies among (not necessarily strict) sectorial embeddings. See Example~\ref{example. edges in lioudelta are non-strict embeddings} and Proposition~\ref{prop. 2-simplex is an isotopy}.

Using higher-dimensional analogues of the construction in Example~\ref{example. isotopies are 2 simplices}, we now prove that morphism spaces of $\lioudelta$ are homotopy equivalent to the spaces of (not necessarily strict) sectorial embeddings:

\begin{theorem}
\label{theorem. EmbLiou is homLiouDelta}
Fix two Liouville sectors $M$ and $N$. There exists a homotopy equivalence of simplicial sets
	\eqnn
	\emb_{\liou}(M,N) \simeq \hom^R_{\lioudelta}(M,N).
	\eqnd
Here, $\hom^R_{\lioudelta}(M,N)$ is the Kan complex modeling the space of morphisms from $M$ to $N$ (Notation~\ref{notation. right morphism space} below), and $\emb_{\liou}(M,N)$ is the singular complex of the space of sectorial embeddings (Notation~\ref{notation. emb liou}).
\end{theorem}

\subsection{The right mapping spaces}
Now that we have established $\lioudelta$ as an $\infty$-category (Theorem~\ref{theorem. lioudelta is an infinity-cat}), we may speak of its morphism spaces.

\begin{notation}
\label{notation. right morphism space}
We recall the ``space\footnote{See the discussion around Proposition~1.2.2.3 of~\cite{htt}.} of right morphisms'' for $\lioudelta$, 
	\eqnn
	\hom^R_{\lioudelta}(M,N).
	\eqnd
$\hom^R_{\lioudelta}(M,N)$ is the simplicial set whose $k$-simplices are given by maps $\Delta^{k+1} \to \lioudelta$ whose restriction to the vertex $\Delta^{\{k+1\}}$ is given by $N$ and whose restriction to $\Delta^{\{0,\ldots,k\}}$ is the degenerate simplex at $M$.
\end{notation}

\begin{example}\label{example. f in hom^R}
Fix a (not necessarily exact) sectorial embedding $f: M \to N$. Then one may construct a 1-simplex in $\lioudelta$ as in Example~\ref{example. edges in lioudelta are non-strict embeddings}. In particular, these $f$ are examples of zero-simplices in $\hom^R_{\lioudelta}(M,N)$.
\end{example}

\begin{example}\label{example. simplex in hom^R}
Fix $f$ as in Example~\ref{example. f in hom^R}. Now suppose that we have a $k-1$-simplex $\gamma$ in $\hom^R_{\lioudelta}(M,N)$ for which the boundary $\del \Delta^{k-1}$ is sent to the degenerate simplex at $f$. (By definition of simplicial homotopy groups, such a thing defines an element  $[\gamma] \in \pi_{k-1}(\hom^R_{\lioudelta}(M,N),f)$.) Let us parse  this data.

$\gamma$, by definition, is the data of a $k$-simplex in $\lioudelta$ whose $k$th face is the degenerate simplex at the object $M$, and whose other faces are the degenerate simplices $s_0 \circ \ldots \circ s_0 (f)$ at the morphism $f$. Thus, $\gamma$ is determined by
\begin{itemize}
\item A choice of Liouville sector $Q = \gamma([k])$ (together with Liouville structure on $Q \times T^*\Delta^k$),
\item For every $k \in I \subset [k]$, choice of map $\gamma(I \subset [k]): \Delta^I \to \embliou(N,Q)$ (so that the associated movie defines a strict embedding of $N \times T^*\Delta^I \times [-1,\epsilon]^{[k] \setminus I}$ into $Q \times T^*\Delta^k$), and
\item For every $I \subset [k-1]$, a choice of map $\gamma(I \subset [k]): \Delta^I \to \embliou(M,Q)$ (with the analogous strict embedding requirement as above),
\end{itemize}
satisfying the composition requirements. Because of the composition requirements, the above data may be summarized. $\gamma$ is determined by $Q$ (together with Liouville structure on $Q \times T^*\Delta^k$), along with 
\enum[(i)]
\item A map
	\eqnn
	\Phi_N^Q: \Lambda^k_k \to \embliou(N,Q),
	\eqnd
	and
\item a map
	\eqnn
	\Phi_M^Q: \Delta^{k-1} \to \embliou(M,Q)
	\eqnd
\enumd
satisfying the equation
	\eqnn
	\Phi_N^Q|_{\del \Delta^{k-1}} \circ f = \Phi_M^Q|_{\del \Delta^{k-1}}
	\eqnd
where $f$ is treated as a constant map $\del \Delta^{k-1} \to \embliou(M,N)$. (Note also that $\del \Delta^{k-1}$ is a subsimplicial set of $\Lambda^k_k$.)

By definition of $\lioudelta$, every $\Phi$ above, when restricted to any subsimplex, must be smooth and collared and continuous. 

Note also that $N$ is necessarily diffeomorphic to $Q$ by~\ref{item. lioudelta is max localizing}; the diffeomorphism is given by evaluating the map $\Phi^Q_N$ at the maximal vertex of $\Lambda^k_k$. In particular, note that $\Phi^Q_N$ encodes a (contractible) family of embeddings all isotopic to the just-mentioned diffeomorphism.
\end{example}

\begin{notation}[$\hom'_{\lioudelta}(M,N)$]
\label{notation. hom'}
Let 
	\eqn\label{eqn. hom' to homR semi}
	\hom'_{\lioudelta}(M,N) \subset \hom^R_{\lioudelta}(M,N)
	\eqnd
denote the semisimplicial subset of those $(k+1)$-simplices in $\lioudelta$ satisfying the following:
\begin{itemize}
\item If $I' \subset [k+1]$ does not contain $k+1$, then the $\Delta^{I'}$-movie is given by $M \tensor T^*\Delta^{I'}$.
\item If $I' \subset [k+1]$ does contain $k+1$, then the corresponding $\Delta^{I'}$-movie has underlying manifold $N \times T^*\Delta^{I'}$. (As a sector, this object need not necessarily be a tensor.)
\item If $I'' \subset I'$ and neither contain $k+1$, then the associated map $\phi_{I'' \subset I'}$ is given by $\id_M \times T^*\eta_{I'' \subset I'}$.
\item Likewise, if $I'' \subset I'$ and both contain $k+1$, then the map $\phi_{I'' \subset I'}$ is given by $\id_N \times T^*\eta_{I'' \subset I'}$.
\end{itemize}
Note that if $I''$ does not contain $k+1$ but $I'$ does, then $\phi_{I'' \subset I'}$ encodes some $\Delta^{I''}$-family of sectorial embeddings from $M$ to $N$.
\end{notation}

\begin{notation}
Let $\widetilde{\emblioucoll}(M,N)$ denote the semisimplicial set where a $k$-simplex is a pair of
\enum
\item A collared, smooth map $f: \Delta^k \to \emb_{\liou}(M,N)$, and
\item A $\Delta^{k+1}$-parametrized family of compactly supported smooth functions $\{h^s: N \to \RR\}_{s \in \Delta^{k+1}}$. 

Notice this determines a $\Delta^{k+1}$-movie $N \times T^*\Delta^{k+1}$ associated to $\{h^s\}$. We demand that the pair $(f, \{h^s\})$ satisfies the following: 
	\enum
		\item The family $\{h^s\}$ is collared.
       	\item For every subset $I \subset [k]$, there exists a (necessarily unique) strict sectorial embedding
        	\eqnn
        	M \tensor T^*(\Delta^{I} \times [-1,0]^{[k+1] \setminus I})
        	\to
        	N \times T^*\Delta^{k+1}
        	\eqnd 
        exhibiting a movie associated to $f|_{\Delta^I}$, and
        \item For every $s \in \Delta^{k+1}$ near the vertex $\Delta^{\{k+1\}}$, $h^s \equiv 0$.
     \enumd
\enumd
\end{notation}

\begin{remark}
\label{remark. tile emb is hom'}
There is a natural isomorphism of semisimplicial sets between
	$
	\widetilde{\emblioucoll}(M,N)
	$
	and
	$
	\hom'_{\lioudelta}(M,N).
	$
For example, given a pair $(f, \{h^s\})$, the restrictions of $f$ to the faces of $\Delta^{k}$ determine maps $\phi_{I'' \subset I'}$, while the restrictions of $\{h^s\}$ to the faces of $\Delta^{k+1}$ determine the Liouville structures on the various $\Delta$-movies.
\end{remark}

\begin{remark}
\label{remark. tilde emb is Kan}
We have a forgetful map of semisimplicial sets. 
	\eqn\label{eqn. tilde emb to emb}
	\widetilde{\emblioucoll}(M,N)
	\to
	\emblioucoll(M,N)
	\eqnd
Using the same methods as
Proposition~\ref{prop. coll emb spaces are equivalent} (and noting the convexity of the space of compactly supported functions $h$), we see that $\widetilde{\emblioucoll}(M,N)$ admits idempotent autoequivalences for all its vertices. By utilizing Theorem~\ref{theorem. steimle} again, we conclude that~\eqref{eqn. tilde emb to emb} results in a map of {\em simplicial} sets, and both the domain and codomain are Kan complexes.
\end{remark}

\begin{prop}
\label{prop. tilde emb to emb is equivalence}
The map~\eqref{eqn. tilde emb to emb} is a homotopy equivalence of simplicial sets.
\end{prop}

\begin{proof}
The map is a Kan fibration because we can relatively fill horns simply by choosing collared extensions of $\{h^s\}_{s \in \Lambda^n_j}$ to $\{h^s\}_{s \in \Delta^n}$. The fibers are contractible because of the convexity of the space of compactly supported functions. 
\end{proof}

\subsection{\texorpdfstring{$\hom'$}{hom'} is equivalent to \texorpdfstring{$\hom^R$}{hom R}}
Combining Remarks~\ref{remark. tile emb is hom'} and~\ref{remark. tilde emb is Kan}, we conclude:

\begin{prop}
$\hom'_{\lioudelta}(M,N)$ admits a simplicial structure lifting the face maps inherited from $\hom^R_{\lioudelta}(M,N)$.
\end{prop}

Now, the inclusion~\eqref{eqn. hom' to homR semi} is a map of semisimplicial sets, but may not be a map of simplicial sets. This is because the degeneracy maps of $\hom'_{\lioudelta}(M,N)$ are induced from $\widetilde{\emblioucoll}(M,N)$, while the degeneracy maps of $\hom^R_{\lioudelta}(M,N)$ are induced from the degeneracy maps of $\lioudelta$.

However, it is straightforward to see that, under the semisimplicial set map~\eqref{eqn. hom' to homR semi}, degenerate edges of $\hom'_{\lioudelta}(M,N)$ are sent to equivalences in $\hom^R_{\lioudelta}(M,N)$. Theorem~\ref{theorem. hiro non strict} guarantees that the semisimplicial set map may then be deformed to a {\em simplicial} set map
	\eqn\label{eqn. simplicial set map hom' to homR}
	\hom'_{\lioudelta}(M,N) \to
	\hom^R_{\lioudelta}(M,N)
	\eqnd
with the same assignment on objects as~\eqref{eqn. hom' to homR semi}, and with the property that the images of~\eqref{eqn. simplicial set map hom' to homR} are naturally homotopic to the images of~\eqref{eqn. hom' to homR semi}. (We used this result in Construction~\ref{construction. alpha} as well.)

We now prove:

\begin{lemma}
\label{lemma. hom' is homR}
The map~\eqref{eqn. simplicial set map hom' to homR} is a homotopy equivalence of simplicial sets.
\end{lemma}

Because both sides of~\eqref{eqn. simplicial set map hom' to homR} are Kan complexes, it suffices to show that the induced maps on $\pi_0$, and on $\pi_k$ (along all connected components), are bijections.

\begin{proof}[Proof of surjection on $\pi_0$]
Let us show that the inclusion $\hom'_{\lioudelta}(M,N) \subset \hom^R_{\lioudelta}(M,N)$ is a surjection on $\pi_0$. 

For this, fix an arbitrary vertex $F$ of $\hom^R_{\lioudelta}(M,N)$. By definition of an edge of $\lioudelta$, $F$ produces the data of a $\Delta^1$-movie of the form $X \times T^*\Delta^{1}$, along with two morphisms
	\eqnn
	f: M \to (X, \lambda^0)
	\qquad
	\phi: N \to (X, \lambda^1)
	\eqnd
where we have emphasized the Liouville form $\lambda^t$ of $X_{01}$, where $\lambda^t$ is the form on $X_{01}$ at time $t \in \Delta^1 \cong [0,1]$. By definition of 1-simplex in $\lioudelta$, we know that $\phi$ is an isomorphism of Liouville sectors. (In particular, a diffeomorphism.)

We produce a 2-simplex in $\lioudelta$ which we illlustrate informally using the following diagram:
	\eqn\label{eqn. 2-simplex hom pi0 onto}
	\xymatrix{
	&& N = N_2 \ar[dr]^{\id} \ar[dl]_{\phi} \ar[d]^{\id}&& \\
	& X \ar@{=>}[r]^{\phi^{-1}}_{\phi^{-1}} & N & N \ar@{=>}[l]_{\id} & \\
	M = M_0\ar[rr]^{\id} \ar[ur]^{f} \ar[urr] && M \ar@{=>}[u]^{\phi^{-1}f}_{\phi^{-1}f} && M = M_1\ar[ll]_{\id} \ar[ull] \ar[ul]_{\phi^{-1}f}
	}
	\eqnd
Note that the $\{0,1\}$ edge (the bottom edge) is a constant 1-simplex at the object $M$. In particular, the 2-simplex indeed depicts an edge in $\hom^R_{\lioudelta}(M,N)$. The $\{0,2\}$ edge of the diagram is the given vertex $F$ of $\hom^R_{\lioudelta}(M,N)$. The $\{1,2\}$ edge of the diagram is the vertex in $\hom'_{\lioudelta}(M,N)$ to which we claim the given vertex is homotopic.

We first describe the underlying smooth maps. All the ``double'' arrows of the form $\implies$ would, in general, depict a 1-simplex family of morphisms. Here, they are all constant families given by the morphisms as labeled in the diagram. (So for example, we have a constant 1-simplex family of morphisms from $X$ to $N$ given by $\phi^{-1}$.)

To finish, we describe the Liouville structures on the $\Delta$-movies 
$X \times T^* \Delta^{\{0,2\}}$,
$N \times T^* \Delta^{\{1,2\}}$, and
$N \times T^* \Delta^{\{0,1,2\}}$.
The Liouville structure on the first is the one specified by $F$. On $N \times T^*\Delta^{\{1,2\}}$, we choose a family of Liouville structures on $N$ which equals the given structure on $N$ near the vertex $\Delta^{\{2\}}$, and equals the structure induced by $\phi^{-1}$ near the vertex $\Delta^{\{1\}}$. 

Finally, given the maps above, the Liouville structure on $N \times T^* \Delta^{\{0,1,2\}}$ is uniquely specified along the boundary of $\del \Delta^{\{0,1,2\}}$ by requiring that all morphisms in sight be strict. Now, the space of compactly supported smooth functions on $N$ is convex (hence contractible) so we may choose any filling of this family of Liouville structures to all of $\Delta^{\{0,1,2\}}$ to exhibit a $\Delta^{\{0,1,2\}}$-movie structure on $N \times T^*\Delta^{\{0,1,2\}}$. 
\end{proof}

\subsection{Proof of injection on homotopy groups}

\begin{proof}
We choose some element $f\in\hom'(M,N)$ and show that the map of homotopy groups based at $f$
	\eqnn
	\pi_{k-1}(\hom'_{\lioudelta}(M,N),f) \to \pi_{k-1}(\hom^R_{\lioudelta}(M,N),f),
	\qquad
	k \geq 2
	\eqnd
is an injection. Similar techniques as what follows also demonstrate that the induced map on $\pi_0$ is an injection.

\begin{remark}[Parsing the hypotheses]
Fix $\gamma_0,\gamma_1$ representing elements on $\pi_{k-1}(\hom'_{\lioudelta}(M,N),f)$ and assume their images are homotopic (rel $f$) in $\hom^R_{\lioudelta}(M,N)$. By definition of simplicial homotopy groups, and by the definition of $\hom^R$, we conclude there exists a simplex
	\eqnn
	F: \Delta^{k+1} \to \lioudelta
	\eqnd
satisfying the usual boundary conditions. By definition of simplices in $\lioudelta$, $F$ in particular gives rise to maps
	\begin{align}
	\Phi_M^N: &\del \Delta^k \to \embliou(M,N) \nonumber \\
	\Phi_N^Q: &\Lambda^{k+1}_{k+1} \to \embliou(N,Q) \nonumber \\
	\Phi_M^Q: &\Delta^k \to \embliou(M,Q) \label{eqn. Phi maps in injective pi hom}
	\end{align}
which are all piecewise smooth and collared, and which satisfy
	\eqnn
	\Phi_N^Q |_{\del \Delta^k} \circ \Phi_M^N = \Phi_M^Q|_{\del \Delta^k}.
	\eqnd
Because $F$ exhibits a homotopy (rel $f$) from $\gamma_0$ to $\gamma_1$, we observe
	\eqn\label{eqn. Phi boundary conditions}
	\Phi^M_N|_{\del_i\Delta^k} =
	\begin{cases}
	\gamma_0 \circ (\del_0\Delta^k \cong \Delta^{k-1}) & i = 0 \\
	\gamma_1 \circ (\del_1\Delta^k \cong \Delta^{k-1}) & i = 1 \\
	\underline{f} & i \neq 0, 1
	\end{cases}
	\eqnd
where $\underline{f}$ is the degenerate $(k-1)$-simplex mapping to $\embliou(M,N)$ with value $f$, and the precompositions are by the unique simplicial isomorphisms.
Finally, $F$ also specifies a Liouville isomorphism
	\eqn\label{eqn. phi isomorphism N to Q}
	\phi: N \to Q
	\qquad
	\text{ at the vertex $\Delta^{\{k+1\}}$.}
	\eqnd
\end{remark}

Now that we have parsed the hypotheses, our task is to construct a smoothly collared map
	\eqn\label{eqn. Gamma homotopy injection}
	\Gamma: \Delta^k \to \embliou(M,N)
	\eqnd
satisfying the boundary conditions~\eqref{eqn. Phi boundary conditions} (with $\Gamma|_{\del_i \Delta^k}$ replacing $\Phi_M^N|_{\del_i\Delta^k}$). We refer the reader to Figure~\ref{figure. Gamma drawings} for a visual guide to the construction of $\Gamma$, which we now make precise.

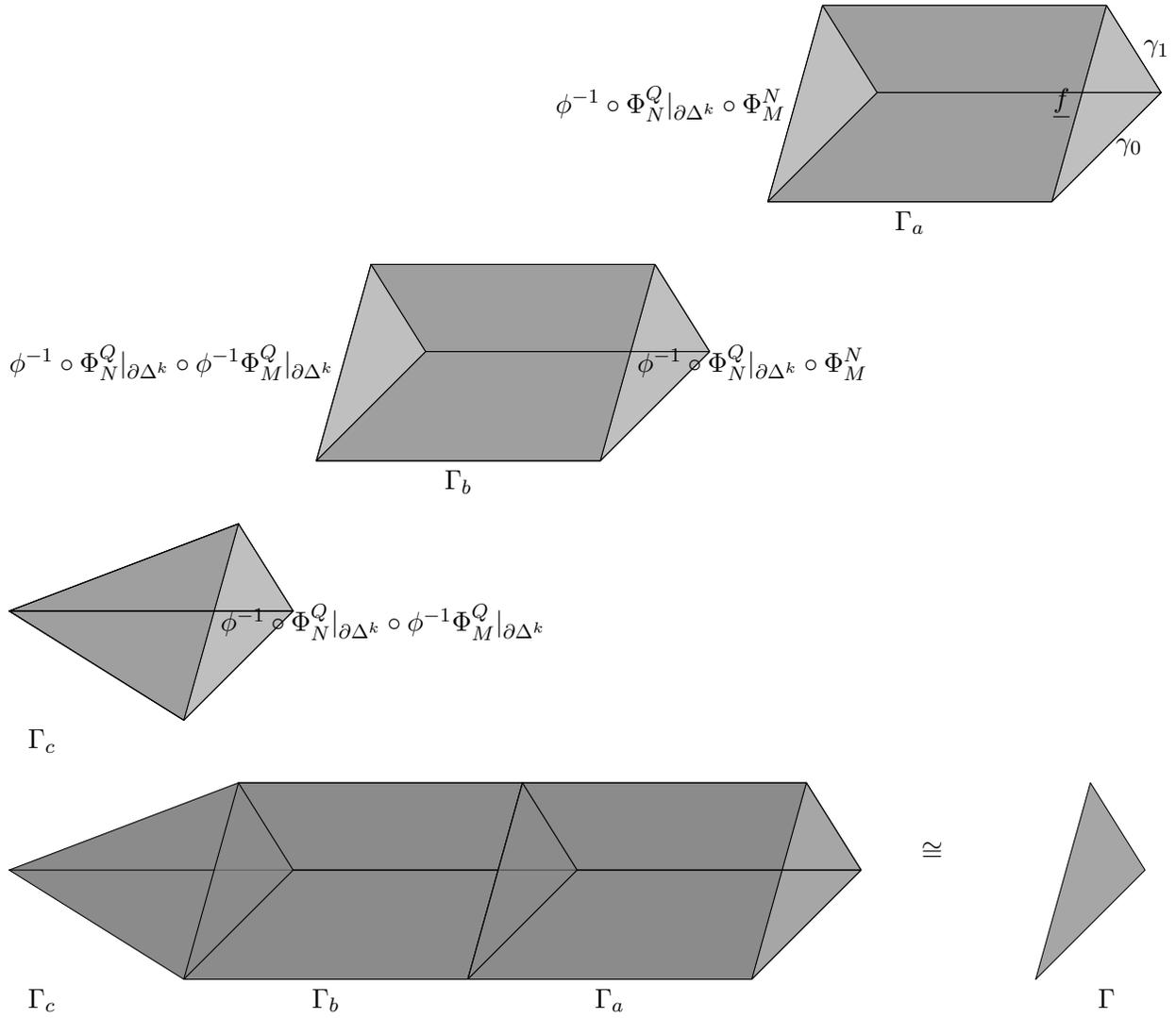
\begin{figure}
    \qquad\qquad
    \qquad\qquad
    \qquad\qquad
    \qquad\qquad
    \qquad\qquad
    \begin{tikzpicture}[line join = round, line cap = round]
        \coordinate [label=above:] (B0) at (0,0,4);
        \coordinate [label=above:] (B1) at (0,0,0);
        \coordinate [label=above:] (B2) at (0,2,2);
        \coordinate [label=above:] (F0) at (4,0,4);
        \coordinate [label=above:] (F1) at (4,0,0);
        \coordinate [label=above:] (F2) at (4,2,2);
        \begin{scope}[decoration={markings,mark=at position 0.5 with {\arrow{to}}}]
        \draw[] (B1)--(B2)--(B0)--cycle;
        \draw[] (F1)--(F2)--(F0)--cycle;
        \fill[opacity=0.5,gray] (F1)--(F2)--(B2)--(B1)--cycle;
        \fill[opacity=0.5,gray] (F1)--(F0)--(B0)--(B1)--cycle;
        \fill[opacity=0.5,gray] (F2)--(F0)--(B0)--(B2)--cycle;
		\draw[] (F1)--(F2)--(B2)--(B1)--cycle;
        \draw[] (F1)--(F0)--(B0)--(B1)--cycle;
        \draw[] (F2)--(F0)--(B0)--(B2)--cycle;
        \end{scope}
        \coordinate [label=below:$\Gamma_a$] (Gamma_a) at (2,0,4);
        \coordinate [label=right:$\gamma_1$] (del1) at (4,1,1);
        \coordinate [label=right:$\gamma_0$] (del0) at (4,0,2);
        \coordinate [label=left:$\underline{f}$] (del_i) at (4,1,3);
        \coordinate [label=left:$\phi^{-1}\circ \Phi_N^Q|_{\del \Delta^k} \circ \Phi_M^N$] (del_i) at (0,1,3);
    \end{tikzpicture}
    
    \begin{tikzpicture}[line join = round, line cap = round]
        \coordinate [label=above:] (B0) at (0,0,4);
        \coordinate [label=above:] (B1) at (0,0,0);
        \coordinate [label=above:] (B2) at (0,2,2);
        \coordinate [label=above:] (F0) at (4,0,4);
        \coordinate [label=above:] (F1) at (4,0,0);
        \coordinate [label=above:] (F2) at (4,2,2);
        \begin{scope}[decoration={markings,mark=at position 0.5 with {\arrow{to}}}]
        \draw[] (B1)--(B2)--(B0)--cycle;
        \draw[] (F1)--(F2)--(F0)--cycle;
        \fill[opacity=0.5,gray] (F1)--(F2)--(B2)--(B1)--cycle;
        \fill[opacity=0.5,gray] (F1)--(F0)--(B0)--(B1)--cycle;
        \fill[opacity=0.5,gray] (F2)--(F0)--(B0)--(B2)--cycle;
		\draw[] (F1)--(F2)--(B2)--(B1)--cycle;
        \draw[] (F1)--(F0)--(B0)--(B1)--cycle;
        \draw[] (F2)--(F0)--(B0)--(B2)--cycle;
        \end{scope}
        \coordinate [label=below:$\Gamma_b$] (Gamma_b) at (2,0,4);
        \coordinate [label=left:$\phi^{-1}\circ \Phi_N^Q|_{\del \Delta^k} \circ \phi^{-1}\Phi_M^Q|_{\del \Delta^k}$] (del_i) at (0,1,3);
        \coordinate [label=right:$\phi^{-1}\circ \Phi_N^Q|_{\del \Delta^k} \circ \Phi_M^N$] (del_i) at (4,1,3);
    \end{tikzpicture}
    
    \begin{tikzpicture}[line join = round, line cap = round]
        \coordinate [label=above:] (B) at (0,0,0);
        \coordinate [label=above:] (F0) at (4,0,4);
        \coordinate [label=above:] (F1) at (4,0,0);
        \coordinate [label=above:] (F2) at (4,2,2);
        \begin{scope}[decoration={markings,mark=at position 0.5 with {\arrow{to}}}]
        \draw[] (F1)--(F2)--(F0)--cycle;
        \draw[] (B)--(F0);
        \draw[] (B)--(F1);
        \draw[] (B)--(F2);
        \fill[opacity=0.5,gray] (F1)--(F2)--(B)--cycle;
        \fill[opacity=0.5,gray] (F2)--(F0)--(B)--cycle;
        \fill[opacity=0.5,gray] (F1)--(F0)--(B)--cycle;
		\draw[] (F1)--(F2)--(B)--cycle;
        \draw[] (F2)--(F0)--(B)--cycle;
        \draw[] (F1)--(F0)--(B)--cycle;
        \end{scope}
        \coordinate [label=below:$\Gamma_c$] (Gamma_c) at (2,0,4);
        \coordinate [label=right:$\phi^{-1}\circ \Phi_N^Q|_{\del \Delta^k} \circ \phi^{-1}\Phi_M^Q|_{\del \Delta^k}$] (del_i) at (4,1,3);
    \end{tikzpicture}

    \begin{tikzpicture}[line join = round, line cap = round]
        \coordinate [label=above:] (B) at (0,0,0);
        \coordinate [label=above:] (F0) at (4,0,4);
        \coordinate [label=above:] (F1) at (4,0,0);
        \coordinate [label=above:] (F2) at (4,2,2);
        \coordinate [label=above:] (C0) at (4,0,4);
        \coordinate [label=above:] (C1) at (4,0,0);
        \coordinate [label=above:] (C2) at (4,2,2);
        \coordinate [label=above:] (D0) at (8,0,4);
        \coordinate [label=above:] (D1) at (8,0,0);
        \coordinate [label=above:] (D2) at (8,2,2);
        \coordinate [label=above:] (E0) at (8,0,4);
        \coordinate [label=above:] (E1) at (8,0,0);
        \coordinate [label=above:] (E2) at (8,2,2);
        \coordinate [label=above:] (EE0) at (12,0,4);
        \coordinate [label=above:] (EE1) at (12,0,0);
        \coordinate [label=above:] (EE2) at (12,2,2);
        \coordinate [label=above:] (FF0) at (16,0,4);
        \coordinate [label=above:] (FF1) at (16,0,0);
        \coordinate [label=above:] (FF2) at (16,2,2);
        \begin{scope}[decoration={markings,mark=at position 0.5 with {\arrow{to}}}]
        \draw[] (F1)--(F2)--(F0)--cycle;
        \fill[opacity=0.7,gray] (F1)--(F2)--(B)--cycle;
        \fill[opacity=0.7,gray] (F2)--(F0)--(B)--cycle;
        \fill[opacity=0.7,gray] (F1)--(F0)--(B)--cycle;
        \draw[] (B)--(F0);
        \draw[] (B)--(F1);
        \draw[] (B)--(F2);
        \fill[opacity=0.7,gray] (D1)--(D2)--(C2)--(C1)--cycle;
        \fill[opacity=0.7,gray] (D1)--(D0)--(C0)--(C1)--cycle;
        \fill[opacity=0.7,gray] (D2)--(D0)--(C0)--(C2)--cycle;
		\draw[] (D1)--(D2)--(C2)--(C1)--cycle;
        \draw[] (D1)--(D0)--(C0)--(C1)--cycle;
        \draw[] (D2)--(D0)--(C0)--(C2)--cycle;
        \draw[] (E1)--(E2)--(E0)--cycle;
        \draw[] (EE1)--(EE2)--(EE0)--cycle;
        \fill[opacity=0.7,gray] (EE1)--(EE2)--(E2)--(E1)--cycle;
        \fill[opacity=0.7,gray] (EE1)--(EE0)--(E0)--(E1)--cycle;
        \fill[opacity=0.7,gray] (EE2)--(EE0)--(E0)--(E2)--cycle;
		\draw[] (EE1)--(EE2)--(E2)--(E1)--cycle;
        \draw[] (EE1)--(EE0)--(E0)--(E1)--cycle;
        \draw[] (EE2)--(EE0)--(E0)--(E2)--cycle;
        \coordinate [label=above:$\cong$] (cong) at (13,0,0);
        \fill[opacity=0.7,gray] (FF1)--(FF2)--(FF0)--cycle;
        \draw[] (FF1)--(FF2)--(FF0)--cycle;
        \end{scope}
        \coordinate [label=below:$\Gamma_c$] (Gamma_c) at (2,0,4);
        \coordinate [label=below:$\Gamma_b$] (Gamma_b) at (6,0,4);
        \coordinate [label=below:$\Gamma_a$] (Gamma_a) at (10,0,4);
        \coordinate [label=below:$\Gamma$] (Gamma) at (17,0,4);
    \end{tikzpicture}
    
\caption{An informal drawing, for $k=2$, of the maps $\Gamma_a, \Gamma_b, \Gamma_c,$ and $\Gamma$ in proving the injectivity on homotopy groups.}
\label{figure. Gamma drawings}
\end{figure}

Consider the maps
	\begin{align}
	\underline{\Phi_M^N} : \del \Delta^k \times \Delta^1
		\to \del \Delta^k
		\xrightarrow{\Phi_M^N} \embliou(M,N) \nonumber \\
	\underline{\Phi_N^Q} : \del \Delta^k \times \Delta^1
		\to \del \Lambda^{k+1}_{k+1}
		\xrightarrow{\Phi_N^Q} \embliou(N,Q) \nonumber 
	\end{align}
where the unlabeled maps are the projections to the $\del \Delta^k$ factor; realizing the horn as a quotient of a cylinder along $\del \Delta^k \times \Delta^{\{1\}}$; and to a point, respectively. Composition defines a map
	\eqnn
	\Gamma_a := {\phi^{-1}} \circ \underline{\Phi_N^Q} \circ \underline{\Phi_M^N}
	: 
	\del \Delta^k \times \Delta^1 \to \embliou(M,N).
	\eqnd
Because $\Phi_N^Q$ is identical to $\phi$~\eqref{eqn. phi isomorphism N to Q} near the vertex $\Delta^{\{k+1\}}$, $\Gamma_a$ has the property that its restriction to $\del \Delta^k \cong \del \Delta^k \times \Delta^{\{1\}}$ satisfies the boundary conditions~\eqref{eqn. Phi boundary conditions}.

We now ``cap off'' $\Gamma_a$ along $\del \Delta^k \times \Delta^{\{0\}}$. For this, choose a piecewise smooth homeomorphism $p: |\Lambda^{k+1}_{k+1}| \xrightarrow{\cong} |\Delta^k|$, for example by projection along a vector normal to $|\Delta^k|$ and passing through the vertex $|\Delta^{\{k+1\}}|$. We choose a collared smoothing $(\Phi_N^Q)'$ of the composition $\Phi_N^Q \circ p^{-1}$, and define
	\eqnn
	\Gamma_c :=  {\phi^{-1}} \circ (\Phi_N^Q)' \circ {\phi^{-1}} \circ \Phi_M^Q
	:
	|\Delta^k| \to \embliou(N,Q).
	\eqnd
To continue the cap-off process, we note that
	\eqn\label{eqn. Gamma a and Gamma c boundaries}
	\Gamma_a|_{\del \Delta^k \times \Delta^{\{0\}}} = {\phi^{-1}} \circ \Phi_N^Q|_{\del \Delta^k} \circ \Phi_M^N
	\qquad \text{and} \qquad
	\Gamma_c|_{\del \Delta^k} =  {\phi^{-1}} \circ \Phi_N^Q|_{\del \Delta^k} \circ {\phi^{-1}} \circ \Phi_M^Q|_{\del \Delta^k}
	\eqnd
where the latter equality follows from the collaring assumption on $(\Phi_N^Q)'$.
In Lemma~\ref{lemma. Phi commutes with phi} below, we show that these two are smoothly homotopic as maps $|\del \Delta^k| \to \embliou(M,N)$. Thus we conclude the existence of a smooth map
	\eqnn
	\Gamma_b: |\del \Delta^k| \times |\Delta^1| \to \embliou(M,N)
	\eqnd
realizing this smooth homotopy. Because everything in sight is assumed collared, we obtain a gluing
	\eqnn
	\Gamma_c \bigcup \Gamma_b \bigcup \Gamma_a
	:
	|\Delta^k| \bigcup \left(
	|\del \Delta^k| \times |\Delta^1|
	\right)
	\bigcup
	|\del \Delta^k \times \Delta^1|
	\to
	\embliou(M,N).
	\eqnd
Now choose a diffeomorphism
	\eqnn
	|\Delta^k| \xrightarrow{\cong} 	|\Delta^k| \bigcup \left(
	|\del \Delta^k| \times |\Delta^1|
	\right)
	\bigcup
	|\del \Delta^k \times \Delta^1|
	\eqnd
preserving the collaring of $\Gamma_a$ near the boundary of the domain $|\Delta^k|$. We define $\Gamma$~\eqref{eqn. Gamma homotopy injection} to be the composition of this diffeomorphism with $\Gamma_c \bigcup \Gamma_b \bigcup \Gamma_a$. We see that $\Gamma$ satisfies the boundary conditions~\eqref{eqn. Phi boundary conditions} because $\Gamma_a$ does.

We conclude by noting that, given $\Gamma$, one now is free to choose a Liouville structure on $N \times T^*\Delta^{k+1}$ rendering $\Gamma$ a family of strict embeddings.
\end{proof}

\begin{lemma}
\label{lemma. Phi commutes with phi}
There exists a smooth homotopy
	$
	{\phi^{-1}} \circ \Phi_N^Q|_{\del \Delta^k} \sim \Phi_N^Q|_{\del \Delta^k} \circ {\phi^{-1}}
	$
among maps from $\del \Delta^k$ to $\embliou(M,N)$. In particular, there exists a homotopy between the two maps in~\eqref{eqn. Gamma a and Gamma c boundaries}.
\end{lemma}

\begin{proof}
We have the obvious projection $\del \Delta^k \times \Delta^1 \to \Lambda^{k+1}_{k+1}$, again by collapsing $\del \Delta^k \times \Delta^{\{1\}}$. Composing the projection with $\Phi_N^Q$ from~\eqref{eqn. Phi maps in injective pi hom}, we thus witness a homotopy from $\Phi_N^Q|_{\del \Delta^k}$ to ${{\phi}}$ (by abuse of notation, this stands for the constant family $\del \Delta^k$-parametrized family with value $\phi$). We thus have the string of homotopies
	\eqn\label{eqn. phi and Phi homotopy commute}
	{\phi^{-1}} \circ \Phi_N^Q|_{\del \Delta^k}
	\sim
	{\phi^{-1}}{{\phi}}
	= {{\id_N}}
	=
	{{\phi}}{\phi^{-1}}
	\sim
	\Phi_N^Q|_{\del \Delta^k}
	\circ {\phi^{-1}}.
	\eqnd
We may prove the second claim as follows. The following are all homotopies and equalities as maps from $\del \Delta^k$ to $\embliou(M,N)$:
	\begin{align}
	\Gamma_c|_{\del \Delta^k} 
	&=  {\phi^{-1}} \circ \Phi_N^Q|_{\del \Delta^k} \circ {\phi^{-1}} \circ \Phi_M^Q|_{\del \Delta^k} \nonumber \\
	&\sim
	{{\id_N}}  \circ {\phi^{-1}} \circ \Phi_M^Q|_{\del \Delta^k}\nonumber \\
	&={\phi^{-1}} \circ \Phi_N^Q|_{\del \Delta^k} \circ \Phi_M^N|_{\del \Delta^k}\nonumber \\
	&=\Gamma_a|_{\del \Delta^k \times \Delta^{\{0\}}}  \nonumber
	\end{align}
The lone homotopy is from~\eqref{eqn. phi and Phi homotopy commute}, while the non-trivial equality is by the composition assumption of being a simplex in $\lioudelta$. (See~\ref{item. lioudelta composition}.) 
\end{proof}

\begin{remark}
Because $\Phi_N^Q|_{\del \Delta^k} $ is homotopic, as a family, to ${{\phi}}$, the above lemma is just a homotopical version of the (obvious) statement that any isomorphism commutes with its inverse.
\end{remark}

\subsection{Proof of surjection on homotopy groups}

In the present section, we will use $|\Delta^k|$ to mean the topological space, and $\Delta^k$ to mean the simplicial set. (Both are rightly called $k$-simplices.)

\begin{proof}
Let $\gamma_0$ be a representative of an element of $\pi_{k-1}(\hom^R(M,N),f)$. We may assume that $f$ is in the image of the map $\hom'_{\lioudelta}(M,N) \to \hom_{\lioudelta}^R(M,N)$ by the surjectivity on $\pi_0$.

We must exhibit an element $\gamma_1$ in $\hom'_{\lioudelta}(M,N)$, together with a simplex exhibiting a homotopy from $\gamma_0$ to $\gamma_1$. Before we do this, let us parse the hypotheses.

\begin{remark}[Parsing the data of $\gamma_0$.]
As usual, we will ignore the data of the Liouville structures on the $\Delta$-movies. By definition, the data of $\gamma_0$ is a $k$-simplex in $\lioudelta$ whose $k$th face is the degenerate $(k-1)$-simplex associated to $M$, and whose other faces are the degenerate simplices $s_0 \circ \ldots \circ s_0(f)$ associated to $f$. So, by Example~\ref{example. simplex in hom^R},  $\gamma_0$ gives rise to the following data:
\enum[(i)]
	\item A Liouville sector $Q = \gamma_0([k+1])$.
	\item A map
		\eqn\label{eqn. pi surjection PhiNQ}
		\Phi_N^Q: \Lambda^k_k \to \embliou(N,Q)
		\eqnd
	which, at the vertex $\Delta^{\{k\}}$ specifies a diffeomorphism
		\eqn\label{eqn. pi surjection phi}
		\phi: N \to Q
		\eqnd
		and
	\item A map
		\eqn\label{eqn. pi surjection PhiMQ}
		\Phi_M^Q : \Delta^{k-1} = \del_k\Delta^k \to \embliou(M,Q)
		\eqnd
\enumd
satisfying the condition
	\eqn\label{eqn. pi surjection PhiMQ = PhiNQ o f}
	\Phi_M^Q|_{\del \Delta^{k-1}} = \Phi_N^Q|_{\del \Delta^{k-1}} \circ f
	\eqnd
where $f$ is treated as the constant map $\del\Delta^{k-1} \to \embliou(M,N)$.
\end{remark}

How do we exhibit $\gamma_1$? By definition, $\gamma_1$ must be a $k$-simplex in $\lioudelta(M,N)$ satisfying the following boundary conditions:
\enum[(a)]
	\item The $k$th face is the degenerate $(k-1)$-simplex generated by the object $M$,
	\item All other faces are degenerate $(k-1)$-simplices $s_0\circ\ldots\circ s_0(f)$generated by the morphism $f$. 
	\item Moreover, because $\gamma_1$ is in $\hom'_{\lioudelta}(M,N)$, we conclude 
		$
		\gamma_1([k]) = N
		$
	and $\gamma_1 (I \subset [k]) \cong \id_N$ for any $I \subset [k]$ with $k \in I$. So the data of $\gamma_1$ is specified entirely by a map
		\eqnn
		\Psi_M^N: \gamma_1([k-1] \subset [k]) : \Delta^{k-1} \to \emblioucoll(M,N)
		\eqnd
	(along with some Liouville structure on $N \times T^*\Delta^k$).
\enumd
So let us produce $\Psi_M^N$. As usual, the extra data of a Liouville structure on $N \times T^*\Delta^k$ is an afterthought; we ignore it.

\begin{construction}[$\Psi_M^N$]
Note that we have a map
	\eqnn
	\overline{\Phi_N^Q \circ f}: \Delta^1 \times \del \Delta^{k-1} \to \del \Delta^{k-1} \xrightarrow{\Phi_N^Q|_{\del \Delta^{k-1}} \circ f} \embliou(M,Q)
	\eqnd
where the first arrow is the projection forgetting the $\Delta^1$ factor. Because of~\eqref{eqn. pi surjection PhiMQ = PhiNQ o f}, and by collaring, we may glue this map with $\Phi_M^Q$ to obtain a single map
	\eqnn
	\overline{\Phi_N^Q \circ f} \bigcup \Phi_M^Q: \Delta^1 \times \del \Delta^{k-1} \bigcup_{\del \Delta^{k-1}} \Delta^{k-1} \to \embliou(M,Q).
	\eqnd 
Now choose a diffeomorphism
	\eqnn
	j: |\Delta^{k-1}| \cong |\Delta^1 \times \del \Delta^{k-1}|
	\eqnd
respecting the collars. We define
	\eqnn
	\Psi_M^N:= \phi^{-1} \circ \left( \overline{\Phi_N^Q \circ f} \bigcup \Phi_M^Q \right) \circ j
	\eqnd
where $\phi^{-1}$ is the inverse to~\eqref{eqn. pi surjection phi}.
\end{construction}

\begin{remark}
By construction, $\Psi^M_N$ is equal to $f$ near the boundary of $|\Delta^{k-1}|$. Moreover, $\Psi^M_N$ satisfies the following:
	\eqn\label{eqn. PsiMN interior equals PhiMQ}
	\text{In some interior region of $|\Delta^{k-1}|$, $\Psi^M_N$ encodes (up to reparametrization) the family  $\phi^{-1} \circ \Phi_M^Q$.}
	\eqnd
\end{remark}

So now we construct a $(k+1)$-simplex $H$ in $\lioudelta$ realizing a homotopy from $\gamma_0$ to $\gamma_1$ in $\hom^R_{\lioudelta}(M,N)$. As it turns out, we will be able to define $H$ so that $H([k+1]) = Q$ and so that the embedding
	\eqnn
	H([k+1] \setminus \{0\} \subset [k+1]) : \del_0\Delta^{k+1} \to \embliou(Q,Q)
	\eqnd
is the constant map to $\id_Q$. Then, unwinding the definitions, we must only specify the following data to specify $H$:
\begin{itemize}
\item $\Gamma_N^Q: \del_1 \Delta^{k+1} \to \embliou(N,Q)$
\item $(\Xi_N^Q)_i: \del_i \Delta^{k+1} \to \embliou(N,Q)$ for $2 \leq i \leq k$
\item $\Gamma_M^Q: \Delta^k = \del_{k+1}\Delta^{k+1} \to \embliou(M,Q)$
\end{itemize}
(along with a Liouville structure on $Q \times T^*\Delta^{k+1}$, which we again ignore because one may be constructed straightforwardly).

\begin{construction}[$\Gamma_N^Q$]
Because $\embliou(N,Q)$ is a Kan complex, we have a horn-filler
	\eqnn
	\xymatrix{
	\Lambda^k_k \ar[rr]^-{\Phi_N^Q \eqref{eqn. pi surjection PhiNQ}} \ar[d]
	&& \embliou(N,Q) \\
	\Delta^k \ar@{-->}[urr]^{\widetilde{\Phi_N^Q}}_{\exists}.
	}
	\eqnd
Again by Proposition~\ref{prop. continuous to smooth}, we may guarantee that the filler $\widetilde{\Phi_N^Q}$ is collared. In fact, if we choose this filler using a ``project-to-$\del_k\Delta^k$-and-interpolate'' construction, we may guarantee that the value of $\widetilde{\Phi_N^Q}$ has the following boundary behavior along $\del_k \Delta^k$:
	\eqn\label{eqn. tilde PhiNQ region with phi}
	\text{There is some region in the interior of $|\del_k \Delta^k| \cong |\Delta^{k-1}|$ where $\widetilde{\Phi_N^Q}$ is identical to $\phi$.}
	\eqnd
We choose $\widetilde{\Phi_N^Q}$ so that this region~\eqref{eqn. tilde PhiNQ region with phi} contains the region referenced in~\eqref{eqn. PsiMN interior equals PhiMQ}. 
We let $\Gamma_N^Q$ be the composition
	\eqnn
	\del_1 \Delta^{k+1} \xrightarrow{\cong} \Delta^k \xrightarrow{\widetilde{\Phi_N^Q}} \embliou(N,Q)
	\eqnd
where the first arrow is the unique isomorphism of simplicial sets.
\end{construction}

\begin{construction}[$(\Xi_N^Q)_i$]
For every $2 \leq i \leq k$, we define the $(\Xi_N^Q)_i$ to be the composition
	\eqnn
	\xymatrix{
	\del_i \Delta^{k+1}
	\cong
	\Delta^k
	\ar[r]^{\sigma_0}
	& \Delta^{k-1}
	\cong \del_i \Delta^k
	\ar[r]^{\Phi_N^Q}
	& \embliou(N,Q)
	}
	\eqnd
where $\sigma_0$ is the codegeneracy map, induced by the surjection $[k] \to [k-1]$ sending both $0,1 \in [k]$ to $0 \in [k-1]$.
\end{construction}

\begin{construction}[$\Gamma_M^Q$]
Based on our previous choices, the value of $\Gamma_M^Q: \Delta^k \to \embliou(M,Q)$ is completely constrained on $\del \Delta^k$. For example, on $\del_1 \Delta^k$, the composition requirement for a $(k+1)$-simplex in $\lioudelta$ demands that we have
	\eqnn
	\Gamma_M^Q|_{\del_1 \Delta^k} = \Gamma_N^Q|_{\del_1 \Delta^k} \circ \Phi_M^N|_{\del_1 \Delta^k \cong \Delta^{k-1}}
	\eqnd
We make an important observation. Because the regions of~\eqref{eqn. tilde PhiNQ region with phi} was chosen to contain the region of~\eqref{eqn. PsiMN interior equals PhiMQ}, one can choose a diffeomorphism
	\eqn
		|\del_1 \Delta^k| \cong |\Delta^{k-1}| \bigcup |\del \Delta^{k-1} \times \Delta^1|
	\eqnd
so that
	\eqnn
	\text{$\Gamma_M^Q|_{\del_1 \Delta^k}$ encodes a homotopy $h$ from $\Phi_M^Q|_{\del \Delta^{k-1}}$ to itself.}
	\eqnd
In particular, by choosing a reverse homotopy $\overline{h}$, there is a map
	\eqnn
	H'_1: \del \Delta^{k-1} \times \Delta^2 \to \embliou(M,Q)
	\iff
	\Delta^2 \to \hom_{\sset}(\del \Delta^{k-1}, \embliou(M,Q))
	\eqnd
where the 2-simplex has boundary conditions $h$, $\overline{h}$, and a constant map along the edges $\Delta^{\{0,1\}}$, $\Delta^{\{1,2\}}$, and $\Delta^{\{0,2\}}$, respectively. Informally, the two-simplex thus encodes a homotopy from the concatenation $\overline{h} \sharp h$ to the constant homotopy at $\Phi_M^Q|_{\del \Delta^{k-1}}$.

Now let $H'_2$ be the composition
	\eqnn
	H'_2: \Delta^{k-1} \times \Delta^1 \to \Delta^{k-1} \xrightarrow{\Phi_M^Q}\embliou(M,Q)
	\eqnd
and choose a piecewise smooth, appropriately collared diffeomorphism
	\eqn\label{eqn. l PL diff}
	l: |\Delta^{k-1} \times \Delta^1| \bigcup |\del \Delta^{k-1} \times \Delta^2|
	\cong
	|\Delta^{k-1} \times \Delta^1|
	\eqnd
where the union 
	$|\Delta^k \times \Delta^{\{0\}}| \bigcup |\del \Delta^k \times \Delta^{\{0,1\}}|$
of the domain is mapped to the face 
	$|\Delta^k \times \Delta^{\{0\}}|$ 
of the codomain. (See Figure~\ref{figure. cylinder collapse}.) 
\begin{figure}
	\eqnn
    \begin{tikzpicture}[line join = round, line cap = round]
        \coordinate [label=above:] (B0) at (0,0,4);
        \coordinate [label=above:] (B1) at (0,0,0);
        \coordinate [label=above:] (B2) at (0,2,2);
        \coordinate [label=above:] (F0) at (4,0,4);
        \coordinate [label=above:] (F1) at (4,0,0);
        \coordinate [label=above:] (F2) at (4,2,2);
        \coordinate [label=above:] (G) at (0,0.67,2);
        \coordinate [label=above:] (G0) at (0,0.34,3);
        \coordinate [label=above:] (G1) at (0,0.34,1);
        \coordinate [label=above:] (G2) at (0,1.34,2);
        \begin{scope}[decoration={markings,mark=at position 0.5 with {\arrow{to}}}]
        \draw[] (G1)--(G2)--(G0)--cycle;
        \draw[] (G0)--(F0);
        \draw[] (G1)--(F1);
        \draw[] (G2)--(F2);
        \draw[] (B1)--(B2)--(B0)--cycle;
        \draw[] (F1)--(F2)--(F0)--cycle;
        \fill[opacity=0.5,gray] (F1)--(F2)--(B2)--(B1)--cycle;
        \fill[opacity=0.5,gray] (F1)--(F0)--(B0)--(B1)--cycle;
        \fill[opacity=0.5,gray] (F2)--(F0)--(B0)--(B2)--cycle;
		\draw[] (F1)--(F2)--(B2)--(B1)--cycle;
        \draw[] (F1)--(F0)--(B0)--(B1)--cycle;
        \draw[] (F2)--(F0)--(B0)--(B2)--cycle;
        \end{scope}
        \coordinate [label=below:$|\Delta^{k-1}\times \Delta^1| \cup |\del \Delta^{k-1} \times \Delta^2| \cong |\Delta^{k-1} \times \Delta^1|$] (Gamma_b) at (2,0,4);
    \end{tikzpicture}
  	\eqnd
	\caption{The piecewise smooth isomorphism~\eqref{eqn. l PL diff}}\label{figure. cylinder collapse}
\end{figure}
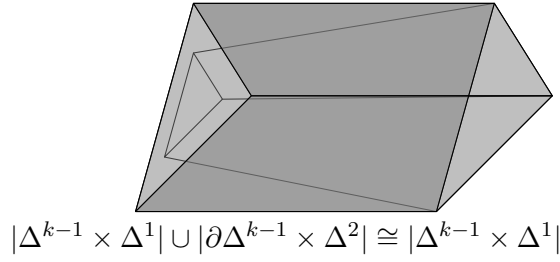
Now we note that the map
	\eqn\label{eqn. H' pi surjection}
	\left(H'_2 \bigcup H'_1\right)\circ l^{-1} : |\Delta^{k-1} \times \Delta^1| \to \embtop_{\liou}(M,Q)
	\eqnd
is constant along the $\Delta^1$-direction of $|\del_0 \Delta^{k-1} \times \Delta^1|$. Consequently, \eqref{eqn. H' pi surjection} factors through a quotient
		\eqnn
		|\Delta^{k-1} \times \Delta^1| \to |\Delta^k|
		\eqnd
(for example, by writing $|\Delta^k|$ as a cone on $|\Delta^{k-1}|$); we may again apply Proposition~\ref{prop. continuous to smooth} to smooth and collar the resulting map. We define the resulting factorization to be
		\eqnn
		\Gamma_M^Q: \Delta^k \to \emb_{\liou}(M,Q).
		\eqnd
We leave to the reader the details of verifying that all compositional requirements are met.
\end{construction}
\end{proof}

\subsection{Proof of Theorems~\ref{theorem. EmbLiou is homLiouDelta} and~\ref{theorem. hom liou delta is hom liou}.}
\begin{proof}
We have the following homotopy equivalences of Kan complexes:
	\eqnn
	\xymatrix{
	\widetilde{\emblioucoll}(M,N) 
	\ar[rr]^{\text{Rem~\ref{remark. tile emb is hom'}}} \ar[d]_{\text{Prop~\ref{prop. tilde emb to emb is equivalence}}}
	&& \hom'_{\lioudelta}(M,N)\ar[rr]^{\text{Lem~\ref{lemma. hom' is homR}}}
	&& \hom^R_{\lioudelta}(M,N)
	\\
	\emblioucoll(M,N) \ar[rr]_{\text{Prop~\ref{prop. coll emb spaces are equivalent}}}
 	&& \emb_{\liou}(M,N).
	}
	\eqnd
\end{proof}

\begin{remark}
We point out one subtlety, which is that the simplicial set structure used in Lemma~\ref{lemma. hom' is homR} is not necessarily the simplicial set structure used in Proposition~\ref{prop. tilde emb to emb is equivalence}; more specifically, the isomorphism of semisimplicial sets in Remark~\ref{remark. tile emb is hom'} need not respect degeneracy maps.

However, if two Kan complexes have the same underlying semisimplicial sets (i.e., with differing degeneracy maps but with equal face maps), they are automatically homotopy equivalent.\footnote{See for example Theorem~1.4 of~\cite{steimle}, which proves a more general result---two $\infty$-categories with the same underlying semisimplicial sets are equivalent.} So the reader should interpret the arrow labeled by Remark~\ref{remark. tile emb is hom'} not as ``the'' map constructed in the remark, but as a homotopy equivalence guaranteed by it.
\end{remark}

\subsection{An informal argument}
\label{section. lioudelta equivalence sketch}
In Theorem~\ref{theorem. EmbLiou is homLiouDelta}, we proved that $\lioudelta$ has morphism spaces that are homotopy equivalent to the usual spaces of sectorial embeddings. One can do better. Let us give an informal argument, using the language of constructible sheaves, as to why $\lioudelta$ is equivalent to the usual $\infty$-category of sectors with spaces of sectorial embeddings. (The only difference between this ``better'' result and Theorem~\ref{theorem. EmbLiou is homLiouDelta} is the exhibition of a {\em functor} resulting in the equivalences of morphism spaces guaranteed in Theorem~\ref{theorem. EmbLiou is homLiouDelta}.)

Fix an $I$-simplex of $\lioudelta$.
For all $\emptyset \neq A \subset I$, If one forgets the sectorial structures on $M_{A} \times T^*\Delta^{A}$ (and only remembers the sectorial structure on $M_{A}$), one obtains families of sectorial embeddings $\Delta^{A} \to \embliou(M_{A},M_{B}).$ We note that the assignment
\begin{itemize}
\item To every subsimplex $\Delta^A$ of $\Delta^I$, an object $M_A$, and
\item To every inclusion $\Delta^A \subset \Delta^B$, a family of maps $\Delta^A \to \embliou(M_A,M_B)$
\end{itemize}
determines a constructible sheaf on $\Delta^I$, where $\Delta^I$ is stratified in the usual way. 

\begin{remark}
The stratification is as follows. Given a point $x \in \Delta^I \subset \RR^I$, $x$ is assigned to the set of those $i \in I$ for which $x_i>0$. This assignment from $x$ to subsets of $I$ defines a stratification on $\Delta^I$ by endowing the power set of $I$ the Alexandroff topology. 

For more on this perspective on stratified spaces, we refer the reader to~\cite[Appendix A]{higher-algebra}, \cite{aft-1} and~\cite{tanaka-paracyclic}.
\end{remark}

Moreover, because we demand that our $I$-simplices have a max-localizing property, it is formal that this constructible sheaf is constructible when considered as a sheaf on $\Delta^I$ with the ``max stratification'' (where now $\Delta^I$ is stratified by the poset $I$ itself, and $x$ is sent to the element $\max \{i, x_i>0\}$). It is easy to see that the $\infty$-category of constructible sheaves on $\Delta^I$ with the max stratification is equivalent to the $\infty$-category of functors from $I$ to the $\infty$-category in which the sheaf is taking values.

Put another way, we see that every $I$-simplex of $\lioudelta$ gives rise to an $I$-simplex in the usual $\infty$-category of Liouville sectors\footnote{This may be obtained, for example, by topologically enriching the category of Liouville sectors, then taking the homotopy coherent nerve.}. Thus, by varying the $I$-variable, $\lioudelta$ admits a map to the $\infty$-category of Liouville sectors; this assignment is redundant in that different Liouville structures on $M_A \times T^*\Delta^A$ (with a fixed structure near $\max A$) may be sent to the same simplices; but the space of these Liouville structures (with fixed structure near $\max A$) is contractible because, by assumption, these deformations are all compactly supported.

Let us make one remark: $\lioudelta$ does not see all constructible sheaves on the simplex. This is because we demand that along max-localizing inclusions, the morphism $M_A \to M_B$ be a {\em diffeomorphism} along the $\max A = \max B$ vertex, as opposed to just a sectorial equivalence. However, it is straightforward to verify that, for any constructible sheaf consisting of max-localizing morphisms that are sectorial equivalences (as opposed  to sectorial  isomorphisms), one can construct a constructible sheaf on a higher-dimensional simplex parametrizing to a simplex whose max-localizing morphisms are all sectorial {\em isomorphisms} along the relevant vertices.

This is an informal argument. We included our informal argument to give the reader peace of mind that, indeed, $\lioudelta$ has a moral reason to be equivalent to the $\infty$-category of Liouville sectors. The combinatorics to make the argument rigorous will be left for another work.

\clearpage
\section{\texorpdfstring{$\lioudelta$}{Liou Delta} and \texorpdfstring{$\lioudeltadef$}{Liou Delta Def} are equivalent}

\begin{notation}[$\emb_{\liou}^{\defliou}(M,N)$]
Fix Liouville sectors $M$ and $N$. We let $\emb_{\liou}^{\defliou}(M,N)$ denote the singular complex of the space of deformation embeddings $\embtop_{\liou}^{\defliou}(M,N)$.
\end{notation}

We have the following analogue of Theorem~\ref{theorem. EmbLiou is homLiouDelta}:

\begin{theorem}
\label{theorem. emblioudef is hom^R}
There is a homotopy equivalence of simplicial sets
	\eqnn
	\emblioudef(M,N) \simeq \hom_{\lioudeltadef}(M,N).
	\eqnd
\end{theorem}

\begin{proof}
The proof is nearly identical to that of Theorem~\ref{theorem. EmbLiou is homLiouDelta}; the same strategy to extending (non-compactly-supported) families of Liouville structure of an embedded image to the rest of the codomain as in Remark~\ref{remark. extending deformations} applies. For the reader's sake, we do not reproduce the arguments.
\end{proof}

Finally, we have:

\begin{theorem}\label{theorem. lioudelta is lioudeltadef}
The inclusion $\lioudelta \to \lioudeltadef$ (which is a map of $\infty$-categories by Theorem~\ref{theorem. lioudeltadef is an infinity-cat})  is an equivalence of $\infty$-categories.
\end{theorem}

\begin{proof}
The inclusion is a bijection on objects, so it remains to show that it induces a homotopy equivalence of mapping spaces. For this, fix two sectors $M$ and $N$ and note the equivalences
	\eqnn
	\xymatrix{
	\hom^R_{\lioudelta}(M,N) \ar[r]^-{Th~\ref{theorem. EmbLiou is homLiouDelta}}_{\sim}
	&\embliou(M,N) \ar[rr]^-{Th~\ref{theorem. embedding spaces are deformation embedding spaces}}_{\sim}
	&&\emblioudef(M,N) \ar[r]^-{Th~\ref{theorem. emblioudef is hom^R}}_-{\sim}
	&\hom^R_{\lioudeltadef}(M,N).
	}	
	\eqnd
It is a tedious but straightforward verification that, at the level of $\pi_0$ and of homotopy groups, the composition of these equivalences equals the map induced by the inclusion $\lioudelta \to \lioudeltadef$.
\end{proof}

\clearpage
\section{Localization (the proof of Theorem~\ref{theorem. localization})}\label{section. proof of localization}

In this section, we will prove Theorem~\ref{theorem. localization}---that $\lioudeltadefstab $ is equivalent to the localization $ \lioustrstab[(\eqs^\dd)^{-1}]$. 
 
Recall the category of stabilized sectors $\lioustrstab$ from Notation~\ref{notation. lioustr}. There, in defining $\lioustrstab$, we stabilized $\lioustr$ by the object $T^*[0,1]$. Because $T^*[0,1]$ and $T^*[-1,0]$ are obviously isomorphic objects (for example, induced by the translation diffeomorphism from $[0,1]$ to $[-1,0]$), we adopt a different model here. (See Remark~\ref{remark. why stabilize by -1} for why prefer $[-1,0]$ over $[0,1]$). Stabilizing by either interval obviously yields equivalent $\infty$-categories, so we abuse notation:

\begin{notation}
\label{notation. lioustrstab 2}
 
We define
	\eqn\label{eqn. lioustr defn 2}
	\lioustrstab := \colim \left( \lioustr \xrightarrow{T^*[-1,0] \times -}  \lioustr \xrightarrow{T^*[-1,0] \times -} \ldots \right)
	\eqnd 
and	\eqn\label{eqn. lioudeltadefstab 2}
	\lioudeltadefstab := \colim \left(\lioudeltadef \xrightarrow{T^*[-1,0] \times - }
	\lioudeltadef \xrightarrow{T^*[-1,0] \times - } \ldots \right).
	\eqnd
\end{notation}

\begin{remark}
The techniques of this section come down to combinatorial arguments enabling Constructions~\ref{construction. Gf for alpha left inverse} and~\ref{construction. GGf for alpha right inverse}. 
Indeed, it should be true that for any $\infty$-category $\cC$ with a notion of ``simplex object'', and with a class of morphisms $\eqs$ that pick out inclusions of faces of a simplex, the methods here carry over to show that the localization $\cC[\eqs^{-1}]$ is equivalent to an $\infty$-category $\cC_\Delta$ built out of certain barycentric subdivision diagrams in $\cC$.
\end{remark}

\subsection{Defining \texorpdfstring{$\beta$}{beta}}

\begin{construction}[$j$]
\label{construction. j from lioustrstab to lioudeltadefstab}
Recall the map $\lioustr \to\lioudeltadef$ from Construction~\ref{construction. j from lioustr to lioudelta}; it is in fact a map of $\infty$-categories (i.e., simplicial sets) by combining Theorems~\ref{theorem. lioudelta is an infinity-cat} and~\ref{theorem. lioudeltadef is an infinity-cat}. This map is compatible with taking direct product with a fixed sector; in particular, we may pass to the stabilizations to obtain a map of $\infty$-categories
	\eqnn
	j: \lioustrstab \to \lioudeltadefstab.
	\eqnd
\end{construction}

\begin{example}[$j$ on 1-simplices]
\label{example. j 1 simplex}
Fix a 1-simplex $f_{01}: X_0 \to X_1$ in $\lioustrstab$. Then $j(f_{01})$ is a 1-simplex in $\lioudeltadefstab$ whose underlying diagram may be drawn as follows:
	\eqnn
	\xymatrix{
	X_0 \cong X_0 \times T^*[-1,0] \ar[rr]^-{f_{01} \times T^*\eta}
	&& X_1 \times T^*\Delta^1 
	&&  X_1 \times T^*[-1,0]  \cong X_1 \ar[ll]_-{\id_{X_1} \times T^*\eta}.
	}
	\eqnd
To reduce clutter, we will suppress the $T^*[-1,0]$ factors of stabilizations, and the $\id_{X_1}$ factors of functions, to depict the above diagram as
	\eqnn
	\xymatrix{
	X_0   \ar[rr]^-{f_{01} \times T^*\eta}
	&& X_1 \times T^*\Delta^1 
	&&  X_1  \ar[ll]_-{  T^*\eta}.
	}
	\eqnd
\end{example}

\begin{example}[$j$ on 2-simplices]
\label{example. j 2 simplex}
Fix a 2-simplex in $\lioustrstab$ -- this is the data of sectors $X_0, X_1, X_2$ and morphisms $f_{01},f_{12},f_{02}$ with $f_{02} = f_{12} \circ f_{01}$ (up to stabilization). Then $j$ of this 2-simplex may be depicted as the following barycentric subdivision diagram:
	\eqnn
	\xymatrix{
	&&&& X_2 \ar[drr]^{T^*\eta} \ar[dll]_{T^*\eta} \ar[d]^{T^*\eta} \\
	&& X_2 \times T^*\Delta^1 \ar[rr]^{T^*\eta}  && X_2 \times T^*\Delta^2 && X_2 \times T^*\Delta^1 \ar[ll]_{T^*\eta}  \\
	X_0 \ar[rrrr]_{f_{01} \times T^*\eta} \ar[urr]^{f_{02} \times T^*\eta} \ar[urrrr]_{f_{02} \times T^*\eta} &&&& X_1 \times T^*\Delta^1 \ar[u]^{f_{12} \times T^*\eta} &&&& X_1 \ar[llll]^{T^*\eta} \ar[ullll]^{f_{12} \times T^*\eta} \ar[ull]_{f_{12} \times T^*\eta}
	}
	\eqnd
Note we are suppressing the identifications $X_0 = X_0 \times T^*[-1,0] = X_0 \times T^*[-1,0]^2$. We are also suppressing subscripts on $\eta$ -- that $\eta \circ \eta = \eta$ should be interpreted as $\eta_{J \subset K} \circ \eta_{I \subset J} = \eta_{I \subset K}$, for example.
\end{example}

\begin{prop}
\label{prop. j factors through i}
Let $\eqs \subset \lioustrstab$ denote those strict (stabilized) sectorial embeddings which are also equivalences of Liouville sectors (Definition~\ref{defn. sectorial equivalence}). Then the functor $j$ (Construction~\ref{construction. j from lioustrstab to lioudeltadefstab}) factors, up to homotopy, through the localization $ \lioustrstab[(\eqs^\dd)^{-1}]$. 
\end{prop}

\begin{proof}
Any equivalence of Liouville sectors is rendered homotopy-invertible in $\lioudelta$ (Proposition~\ref{prop. homotopic maps are homotopic}), and hence in $\lioudeltadef$ (Theorem~\ref{theorem. lioudeltadef is an infinity-cat}\eqref{item. lioudelta maps to lioudeltadef}). Further, any morphism which is homotopy-invertible in a finite stage (such as $\lioudeltadef$) of an increasing union of $\infty$-categories  (such as $\lioudeltadefstab$) remains homotopy-invertible in the union. Thus the claim follows by the universal property of localizations.
\end{proof}

\begin{notation}[$\iota, j, \beta$]\label{notation. iota j beta}
By Proposition~\ref{prop. j factors through i}, letting $\iota$ denote the localization map, we have a homotopy-commuting diagram of $\infty$-categories
    \eqn\label{eqn. beta in general}
        \xymatrix{
       \lioustrstab \ar[rr]^j \ar[d]_{\iota}
    		&& \lioudeltadefstab  \\
         \lioustrstab[(\eqs^\dd)^{-1}] \ar[urr]_{\beta} 
        }
    \eqnd
where $\beta$ is a morphism guaranteed by the universal property of localizations.
\end{notation}

Theorem~\ref{theorem. localization} is the claim that $\beta$ is an equivalence. To see $\beta$ is an equivalence, we will exhibit an equivalence $\alpha$ for which $\alpha \circ \beta$ is also an equivalence.

\subsection{Defining  \texorpdfstring{$\alpha$}{alpha}}

\begin{remark}\label{remark. lioudelta to exliou}
Let $\subdiv(I)$ be the barycentric subdivision of $I$ (thought of as a simplicial set). The data of an $I$-simplex of $\lioudelta$ determines a functor
	\eqnn
	\subdiv(I) \to \lioustrstab
	\eqnd
which assigns to a subset $I' \subset I$ the object 
	\eqnn
	M_{I'} \times T^*\Delta^{I'} =  T^*[-1,0]^n \tensor M_{I'} \times T^*\Delta^{I'}
	\cong M_{I'} \times T^*\Delta^{I'} \tensor T^*[-1,0]^n
	\eqnd 
of $\lioustrstab$.

Likewise, the data of an $I$-simplex of $\lioudeltadef$ determines a functor $\subdiv(I) \to \lioustrstab$.
\end{remark}

\begin{remark}
\label{remark. why stabilize by -1}
 
Remark~\ref{remark. lioudelta to exliou} is the reason that we stabilize by $T^*[-1,0]$ instead of $T^*[0,1]$; our collarings $\eta$ were by $[-1,0]$. In turn, we chose to collar by $[-1,0]$ because of the particulars of Construction~\ref{construction. simplex collars}.
If the collarings were given by embeddings of $[a,b]^n$, then we would stabilize by $T^*[a,b]$. 
\end{remark}

\begin{remark}\label{remark. lioudelta is a sub of Ex lioustr}
Remark~\ref{remark. lioudelta to exliou} exhibits maps of semisimplicial sets
	\eqnn
	\lioudelta \to \Ex(\lioustrstab),
	\qquad
	\lioudeltadef \to \Ex(\lioustrstab).
	\eqnd
(By abuse of notation, the codomain is the semisimplicial set obtained by forgetting the degeneracy maps of the simplicial set $\Ex(\lioustrstab)$.) 
In fact, 
we may identify the stabilization $\lioudeltadefstab$ with a semisimplicial subset of $\Ex(\lioustrstab)$, so we obtain a map
	\eqn\label{eqn.  lioudelta is a sub of Ex lioustr}
	\lioudeltadefstab \to \Ex(\lioustrstab)
	\eqnd
of semisimplicial sets. We caution, however, that the image of $\lioudeltadefstab $ is not a simplicial subset. This is because the image inside $\Ex(\lioustrstab)$ is not closed under degeneracy maps.
\end{remark}

\begin{construction}[$\alpha$ and $\alpha_{\semi}$]
\label{construction. alpha}
By post-composing~\eqref{eqn.  lioudelta is a sub of Ex lioustr} with $\Ex(\iota)$, we  obtain a map of semisimplicial sets
	$
	\lioudeltadefstab \to \Ex( \lioustrstab[(\eqs^\dd)^{-1}]).
	$	
Condition~\ref{item. lioudelta is max localizing} guarantees that this map factors through $\Ex_{\simeq}(  \lioustrstab[(\eqs^\dd)^{-1}])$---this is because if $y: M \to N$ is a strict sectorial isomorphism, then any sectorial embedding of the form
	$
	y \tensor T^*\eta: M \tensor T^*[-1,0]^k \to N \tensor T^*\Delta^k
	$
is an equivalence of Liouville sectors (Example~\ref{example. from isotopies to equivalences}). And more generally, any map of the form $\phi_{A' \subset A}$ is also a sectorial equivalence when $\max A' = \max A$. See Remark~\ref{remark. max phi in Liou is equivalence}.

Thus, we have a map of semisimplicial sets
	\eqn\label{eqn. alpha semi}
	\alpha_{\semi}: \lioudeltadefstab \to \Ex_{\simeq}(  \lioustrstab[(\eqs^\dd)^{-1}]).
	\eqnd
On the other hand, the degeneracy maps of the simplicial set $\lioudeltadefstab$ are such that the map $j: \lioustrstab \to \lioudeltadefstab$ is a simplicial set map (Construction~\ref{construction. j from lioustrstab to lioudeltadefstab}). This guarantees that the semisimplicial set map ~\eqref{eqn. alpha semi} sends degenerate edges of the domain to equivalences (in fact, degenerate edges) in the codomain.
So Theorem~\ref{theorem. hiro non strict} guarantees the existence of a map of {\em simplicial} sets
	\eqnn
	\alpha:
	\lioudeltadefstab
	\to
	\Ex_{\simeq}( \lioustrstab[(\eqs^\dd)^{-1}]).
	\eqnd 
There is moreover a homotopy (unique up to contractible choice) from $(\alpha_{\semi})_+$ to $\alpha \circ \epsilon$, as in~\eqref{eqn. h semi to h} for $h = \alpha$.
\end{construction}

\begin{remark}
The map $\alpha_{\semi}$ in~\eqref{eqn. alpha semi} is explicit. By definition, a $k$-simplex in $\lioudeltadefstab$ is a diagram in the shape of a barycentric subdivision of $k$-simplex; a $k$-simplex in the codomain is also some diagram in the shape of a barycentric subdivision. The construction above is nothing more than a simple check that every such $k$-simplex in the domain is an example of a $k$-simplex in the codomain. 

Moreover, every statement about $\alpha$ will be deduced by providing explicit constructions involving $\alpha_{\semi}$. 
\end{remark}

\subsection{\texorpdfstring{$\alpha$}{alpha} as a left inverse}
\label{section. alpha left inverse}
\begin{lemma}\label{lemma. alpha j homotopic to m iota}
The composition
	\eqnn
	\lioustrstab \xrightarrow{j}
	\lioudeltadefstab
	\xrightarrow{\alpha}
	\Ex_{\simeq}( \lioustrstab[(\eqs^\dd)^{-1}])
	\eqnd
is homotopic to the composition
	\eqnn
	\lioustrstab \xrightarrow{\iota}
	 \lioustrstab[(\eqs^\dd)^{-1}]
	\xrightarrow{m}
	\Ex_{\simeq}( \lioustrstab[(\eqs^\dd)^{-1}]).
	\eqnd
\end{lemma}

Assuming Lemma~\ref{lemma. alpha j homotopic to m iota} for the moment, we have the following consequence:

\begin{lemma}\label{lemma. alpha-beta equivalence}
$\alpha \circ \beta$ is an equivalence of $\infty$-categories.
\end{lemma}

\begin{proof}
We have a diagram in the $\infty$-category of $\infty$-categories as follows:
	\eqnn
    \xymatrix{
    \lioustrstab \ar[d]^{\iota} \ar[r]^-{j}
		& \lioudeltadefstab \ar[d]^{\alpha} 
		\\
     \lioustrstab[(\eqs^\dd)^{-1}] \ar[r]^-{\alpha\circ\beta} \ar[ur]^{\beta}
    	& \Ex_{\simeq}( \lioustrstab[(\eqs^\dd)^{-1}]).
    }
	\eqnd
The lower-right triangle commutes on the nose as a diagram of simplicial sets, while the upper-left triangle (Notation~\ref{notation. iota j beta}) commutes up to homotopy specified by the universal property of localization.

In what follows, we write $\sim$ to mean ``is homotopic to:''
	\eqnn
	m \circ \iota
	\sim \alpha \circ j 
	\sim \alpha \circ \beta \circ \iota \nonumber.
	\eqnd
The first homotopy is by Lemma~\ref{lemma. alpha j homotopic to m iota}. The next homotopy follows because we know that $j$ is homotopic to $\beta \circ \iota$ by definition (Notation~\ref{notation. iota j beta}). 
By the universal property of localizations, we see that $\alpha \circ \beta$ is homotopic to $m$. We conclude by noting that $m$ is an equivalence (Proposition~\ref{prop. Exsim(C) is C}).
\end{proof}

The rest of Section~\ref{section. alpha left inverse} is devoted to proving Lemma~\ref{lemma. alpha j homotopic to m iota}.
As mentioned in Remark~\ref{remark. alpha semi is easier to work with}, the map $\alpha$ is rather inexplicit, while $\alpha_{\semi}$ has a known simplex-by-simplex description.
So with the goal of utilizing Lemma~\ref{lemma. making a homotopy}, we construct a map
	\eqnn
	G_f: \Delta^1 \times \Delta^k \to \Ex_{\simeq}( \lioustrstab[(\eqs^\dd)^{-1}])
	\eqnd
for every $k$-simplex $f : \Delta^k \to \lioustr^{\dd}$, satisfying the hypotheses of Lemma~\ref{lemma. making a homotopy}.

\begin{notation}[$A$]
\label{notation. A for Delta 1 times Delta k}
Recall that a top-dimensional simplex of $\Delta^1 \times \Delta^k$ is uniquely determined by any of the following equivalent data:
	\enum
	\item A length $k+1$ chain in the product poset $[k] \times [1]$. By noting that any such chain must be of the form 
		\eqn\label{eqn. top simplex in a cylinder}
		(0,0) < (0,1) < \ldots < (0,j) < (1,j) < (1,j+1) < \ldots < (1,k)
		\eqnd
	we see that the data of this chain is in turn equivalent to:
	\item A choice of element $j \in [k]$. In turn, a choice of $j$ determines:
	\item A choice of convex subset $A \subset [k]$ for which $\min A = 0$. Here, convexity means if $p,q \in A$ then $[p,q] \subset A$. Any such $A$ is of the form $A = \{i \, | \, 0 \leq i \leq j\}$ for some $j \in [n]$.
	\enumd
\end{notation}

\begin{construction}[$G_f$]
\label{construction. Gf for alpha left inverse}
Fix a $k$-simplex $f$ in $\lioustr^{\dd}$ -- this fixes, up to stabilization by finitely many dimensions, a choice of Liouville sector 
	\eqnn
	X_i
	\eqnd
for each $i \in [k$],  and strict sectorial embeddings 
	\eqnn
	f_{i,j} : X_i \to X_j
	\eqnd
for each $i < j$, satisfying $f_{j,h}f_{i,j}=f_{i,h}$. (In the case $i=j$, we have $f_{ii} = \id_{X_i}$.) We fix also a top-dimensional simplex $\Delta^{k+1} \xrightarrow{A} \Delta^1 \times \Delta^k$. Note that $A$ determines the surjection 
	\eqnn
	s_j: [k+1] \to [k]
	\eqnd
via the composition $[k+1] \xrightarrow{A} [1] \times [k] \to [k]$ (see Notation~\ref{notation. A for Delta 1 times Delta k}).

Given the data of $(A,f)$, we associate:
\begin{itemize}
	\item To every non-empty subset $I \subset [k+1]$, the (stabilized class of the) Liouville sector
		\eqnn
		X_{\max s_j(I)} \times T^*\Delta^{I/A}
		\eqnd
	where $I/A$ refers to the linear poset obtained from $I$ by declaring any two elements in $I \cap A$ to be equivalent. (There is an induced poset structure on this quotient set by virtue of $A$ being convex.) The simplex $\Delta^{I/A}$ is defined for any linear poset $I/A$ as in~\eqref{eqn. simplex}
.
	\item To every inclusion $I \subset J \subset [k+1]$ with $I$ non-empty, the strict sectorial embedding
		\eqn\label{eqn. Gf tricky bit}
		f_{\max s_j(I), \max s_j(J)} \times T^*\eta
		: 
		X_{\max s_j(I)}  \times T^*\Delta^{I/A} \times T^*[-1,0]^{ (J/A) \setminus (I/A) }
		\to 
		X_{ \max s_j(J)} \times T^*\Delta^{J/A}.
		\eqnd
	Here, $\eta$ is the collaring associated to the inclusion $I/A \into J/A$ (see Definition~\ref{defn. collaring convention} and Choice~\ref{choice. collaring}).
	\end{itemize}
We leave it to the reader to verify that for every $(A,f)$, the above data define a map 
	\eqnn
	G_f|_A: \subdiv(\Delta^{k+1}) \to \lioustrstab
	\eqnd
and hence a map to $ \lioustrstab[(\eqs^\dd)^{-1}]$. It is somewhat abusive to write $G_f|_A$ at this stage as we have not defined $G_f$ yet; but we justify this notation now: One straightforwardly  checks that, by fixing $f$ and varying $A$, the $G_f|_A$ glue together to define a single map
	\eqnn
	\colim_A (G_f|_A):  \colim_{A} \Delta^{k+1} \cong \Delta^1 \times \Delta^k \to \Ex_{\simeq}( \lioustrstab[(\eqs^\dd)^{-1}]).
	\eqnd
And we define
	\eqnn
	G_f := \colim_A (G_f|_A).
	\eqnd
\end{construction}

\begin{example}
In the case $k=0$ (so $f$ is just a choice of object $X$ in $\lioustrstab$), $G_f$ is the map
	\eqnn
	\Delta^1 \times \Delta^0 \to \Ex_{\simeq}( \lioustrstab[(\eqs^\dd)^{-1}])
	\eqnd
encoded by a diagram $\subdiv(\Delta^1) \to \lioustrstab$ we depict as follows:
	\eqnn
	\xymatrix{
	X \times T^*[-1,0]^{ \{(0,0)\}} \ar[rr]^-{\id_X \times T^*\eta} 
	&& X \times T^*\Delta^{ \{(0,0), (1,0) \}} \cong X \times T^*\Delta^1
	&& X \times T^*[-1,0]^{ \{(1,0)\}} \ar[ll]_-{\id_X \times T^*\eta}
	}
	\eqnd
For sake of space, and because we are in the stabilization $\lioustrstab$, we will simplify the above diagram and not explicitly encode the relevant posets or stabilizing intervals, and also not indicating the identity components of $\id_X \times T^*\eta$:
	\eqnn
	\xymatrix{
	X   \ar[rr]^-{T^*\eta} 
	&&  X \times T^*\Delta^1
	&& X  \ar[ll]_-{T^*\eta}
	}
	\eqnd
\end{example}

\begin{example}
In the case $k=1$ (so $f$ is a choice of strict sectorial embedding $f_{01}: X_0 \to X_1$ up to stabilization), $G_f$ is a map
	\eqnn
	\Delta^1 \times \Delta^1 \to \Ex_{\simeq}( \lioustrstab[(\eqs^\dd)^{-1}])
	\eqnd
encoded by a diagram $\subdiv(\Delta^2) \bigcup \subdiv(\Delta^2)\to \lioustrstab$ whose two halves we informally depict as
	\eqnn
	\xymatrix{
	X_0   \ar[r]^-{T^*\eta}  \ar[d]_{f_{01} \times T^*\eta}
	&  X_0 \times T^*\Delta^1 \ar[d]_{f_{01} \times T^*\eta}
	&&& X_0   \ar[lll]_{T^*\eta} \ar[dll]^-{f_{01} \times T^*\eta} \\
	X_1 \times T^*\Delta^1 \ar[r]_{T^*\eta} 
	& X_1 \times T^*\Delta^2 
	& X_1 \times T^*\Delta^1 \ar[l]_{T^*\eta}\\
	X_1 \ar[u]^{T^*\eta}   \ar[urr]_{T^*\eta} 
	}
	\eqnd
and
	\eqnn
	\xymatrix{
	&&&& X_0 \ar[dll]_{f_{01} \times T^*\eta}  \ar[d]^{f_{01}}\\
	&& X_1 \times T^*\Delta^1 \ar[r]_-{\cong} 
	& X_1 \times T^*\Delta^1
	& X_1 \ar[l]_-{T^*\eta}\\
	X_1 \ar[urr]^-{T^*\eta}  \ar[rrr]_{T^*\eta}
	&&& X_1 \times T^*\Delta^1 \ar[u]^-{\cong}
	& X_1 \ar[u]_-{\cong} \ar[l]^-{T^*\eta}
	}
	\eqnd
We note that the arrows labeled by $\cong$ really are isomorphisms in $\lioustrstab$; this is because in~\eqref{eqn. Gf tricky bit}, it may be that $I \subset_{\neq} J$ but that $I/A \cong J/A$. 
\end{example}

\begin{proof}[Proof of Lemma~\ref{lemma. alpha j homotopic to m iota}.]
We will apply Lemma~\ref{lemma. making a homotopy}. In terms of the notation there, we will set
	\eqnn
	W = \lioustrstab,
	\qquad
	X = \lioudeltadefstab,
	\qquad
	Y = \Ex_{\simeq}(\lioustrstab[(\eqs^{\dd})^{-1}]),
	\qquad
	i = m \circ \iota,
	\qquad
	h_{\semi} = \alpha_{\semi}
	\eqnd
(and the map $j$ in the Lemma is the map $j$ we are considering from Contruction~\ref{construction. j from lioustrstab to lioudeltadefstab}).
We must verify that the collection $\{G_f\}_{f : \Delta^k \to \lioustrstab, k \geq 0}$ satisfies the hypotheses of Lemma~\ref{lemma. making a homotopy}. To this end, fix $f: \Delta^k \to \lioustrstab$.

We first check condition~\eqref{item. Gf homotopes what we want} of Lemma~\ref{lemma. making a homotopy}. 
The $k$-simplex $\Delta^{\{0\}}\times\Delta^k \subset \Delta^1 \times \Delta^k$ is the $(k+1)$st face of the simplex determined by $A=[k]$. Construction~\ref{construction. Gf for alpha left inverse} assigns the $k$-simplex
	\eqn\label{eqn. Gf check on 0th vertex}
	G_f|_{\Delta^{\{0\}} \times \Delta^k} : \Delta^k \cong \Delta^{\{0\}} \times \Delta^k \to \Ex_{\simeq}(\lioustrstab[(\eqs^{\dd})^{-1}])
	\eqnd
given by a map
	\eqnn
	\subdiv(\Delta^k) \to \lioustrstab[(\eqs^{\dd})^{-1}]
	\eqnd
described as follows:
	\begin{itemize}
	\item To every $\emptyset \neq I \subset [k]$, the (stabilized class of the) Liouville sector $X_{\max I}$.
	\item To every $\emptyset \neq I \subset J \subset [k]$, the  (stabilized class of the) strict sectorial embedding $f_{\max I, \max J} : X_{\max I} \to X_{\max J}$.  Note that when $\max I = \max J$, this is simply the identity morphism.
	\end{itemize}
Thus the map~\eqref{eqn. Gf check on 0th vertex} is precisely $m \circ \iota (f)$. 

Let us next understand the map
	\eqn\label{eqn. Gf check on 1st vertex}
	G_f|_{\Delta^{\{1\}} \times \Delta^k} : \Delta^k \cong \Delta^{\{1\}} \times \Delta^k \to \Ex_{\simeq}(\lioustrstab[(\eqs^{\dd})^{-1}])
	\eqnd
(which corresponds to the choice $A =[0]$). For convenience, we will identify any subset $I \subset \{1,\ldots,k+1\}$ with a subset $I' \subset [k]$ by sending $i \mapsto i - 1$. Then~\eqref{eqn. Gf check on 1st vertex} is a diagram in the shape of a barycentric subdivision one may describe as follows:
	\begin{itemize}
	\item To every $\emptyset \neq I \subset \{1,\ldots,k+1\}$, the (stabilized class of the) Liouville sector $X_{\max I'} \times \Delta^{I'}$.
	\item To every $\emptyset \neq I \subset J \subset \{1,\ldots,k+1\}$, the  (stabilized class of the) strict sectorial embedding 
		\eqnn
		f_{\max I', \max J'} \times T^*\eta : X_{\max I'} \times T^*\Delta^{I'} \times T^*[-1,0]^{J' \setminus I'} \to X_{\max J'} \times T^*{\Delta^{J'}}.
		\eqnd
	\end{itemize}
Thus the map~\eqref{eqn. Gf check on 1st vertex} is precisely the barycentric subdivision diagram encoded by $j(f)$ (see Constructions~\ref{construction. j from lioustr to lioudelta} and~\ref{construction. j from lioustrstab to lioudeltadefstab}). So $G_f$, restricted to $\Delta^{\{1\}} \times \Delta^k$, recovers $\alpha_{\semi} \circ j(f)$. 

We have finished checking condition~\eqref{item. Gf homotopes what we want} of Lemma~\ref{lemma. making a homotopy}.

We now check condition~\eqref{item. Gf respects face maps}. Fix $0 \leq i \leq k$.  Note that the top-dimensional simplices of $\Delta^1 \times \del_i \Delta^k$ are indexed by the subsets
	\eqnn
	\del_i A := A \setminus \{i\} = A \cap ([k] \setminus \{i\}),
	\qquad
	i \neq \max A.
	\eqnd
Thus, to check the condition $G_f|_{\Delta^1 \times \del_i \Delta^k} = G_{d_i f}$, it suffices to check---for every $A = \{0,1,\ldots,\max A\} \subset [k]$ with $\max A \neq i$---that
	\eqn\label{eqn. checking Gf respects faces}
	(\del_{A,i})^* G_{f,A} = G_{\del_i f, \del_i A}. 
	\eqnd
Here, $\del_{A,i}$ is the map fitting into the following commutative diagram:
	\eqnn
	\xymatrix{
	\Delta^{k+1} \ar[r]^-A & \Delta^1 \times \Delta^k \\
	\Delta^k \ar[u]_{\del_{A,i}} \ar[r]^-{\del_i A} & \Delta^1 \times \del_i \Delta^k \ar[u].
	}
	\eqnd
Parsing the definitions, we see that the lefthand side of~\eqref{eqn. checking Gf respects faces} is a map $\subdiv(\Delta^k) \to \lioustrstab$ that assigns, to every subset $\emptyset \neq I \subset [k]$, the object
	\eqnn
	X_{\max( s_{\max \del_i A}(I))} \times T^*\Delta^{I/\del_i A}.
	\eqnd
The righthand side of~\eqref{eqn. checking Gf respects faces} assigns to each $I$ the exact same object. A similar check of morphisms associated to $I \subset J$ shows that the barycentric subdivision diagrams defined by both sides of~\eqref{eqn. checking Gf respects faces} are identical; this finishes the check of condition~\eqref{item. Gf respects face maps}, so we conclude by applying Lemma~\ref{lemma. making a homotopy}.
\end{proof}

\subsection{\texorpdfstring{$\alpha$}{alpha} as a right inverse}
\label{section. alpha is right inverse}
The purpose of this section is to prove the following:

\begin{lemma}\label{lemma. main lemma for Ex}
The map 
	\eqnn
	\Ex_{\simeq}(\beta) \circ \alpha: \lioudeltadefstab \to \Ex_{\simeq}(\lioudeltadefstab)
	\eqnd
is homotopic to  the map $m: \lioudeltadefstab \to \Ex_{\simeq}(\lioudeltadefstab)$ from Proposition~\ref{prop. Exsim(C) is C}.
In particular, $\Ex_{\simeq}(\beta) \circ \alpha$ is an equivalence.
\end{lemma}

The proof requires some set-up.

\begin{notation}[$\eta$]
\label{notation. eta}
For this section only, we will denote the natural transformation $\Ex_{\simeq} \into \Ex$ by $\eta$. (This is a natural transformation between two functors from the category of simplicial sets to itself.)
\end{notation}

\begin{notation}[$M$ and $M_+$]
Because every simplex in $\lioudeltadefstab$ encodes a diagram in $\lioustrstab$ in the shape of a barycentric subdivision, we easily obtained the map~\eqref{eqn.  lioudelta is a sub of Ex lioustr} of semisimplicial sets. Post-composing with $\Ex(j)$ (Constructin~\ref{construction. j from lioustrstab to lioudeltadefstab}), we obtain a map from $\lioudeltadefstab$ to $\Ex(\lioudeltadefstab)$. This composition factors through $\eta$, and we notate the factoring morphism by $M$ to obtain a commuting rectangle of semisimplicial sets as follows:
    \eqnn
	\xymatrix{
	\lioudeltadefstab \ar[rrrr]^{M} 
	\ar[d]^{\eqref{eqn.  lioudelta is a sub of Ex lioustr}} 
	&&&& \Ex_{\simeq}(\lioudeltadefstab) \ar[d]_{\eta} \\
	 \Ex(\lioustrstab) \ar[rrrr]^{\Ex(j)}
		&&&& \Ex(\lioudeltadefstab)
	}
    \eqnd
Let us justify the factorization through $\Ex_{\simeq}(\lioudeltadefstab)$. By construction, the map \eqref{eqn.  lioudelta is a sub of Ex lioustr} hits only those simplices in $\Ex(\lioustrstab)$ whose max-localizing edges are sectorial equivalences (Condition~\ref{item. lioudelta is max localizing}). And---applying $j$---such edges are homotopy-invertible in $\lioudeltadef$ (combine Proposition~\ref{prop. homotopic maps are homotopic}\eqref{item. equivalences are equivalences} with Theorem~\ref{theorem. lioudelta is lioudeltadef}). 

Finally, while $M$ and \eqref{eqn.  lioudelta is a sub of Ex lioustr} do not respect degeneracy maps, by adjunction they induce maps of simplicial sets with domain $\lioudeltadefstab_+$. Calling these maps $M_+$ and $\eqref{eqn.  lioudelta is a sub of Ex lioustr}_+$ respectively, we obtain a commutative diagram of simplicial sets:
    \eqn\label{eqn. defining M eta}
	\xymatrix{
	\lioudeltadefstab_+ \ar[rrrr]^{M_+} 
	\ar[d]^{\eqref{eqn.  lioudelta is a sub of Ex lioustr}_+} 
	&&&& \Ex_{\simeq}(\lioudeltadefstab) \ar[d]_{\eta} \\
	 \Ex(\lioustrstab) \ar[rrrr]^{\Ex(j)}
		&&&& \Ex(\lioudeltadefstab)
	}
    \eqnd
\end{notation}

We now have the notation to state the main combinatorial lemma.

\begin{lemma}\label{lemma. M is homotopic to m}
The composition 
	\eqnn
	\xymatrix{
	\lioudeltadefstab_+
	\ar[r]^-\epsilon
	& \lioudeltadefstab \ar[r]^-{m}
	& \Ex_{\simeq}(\lioudeltadefstab )
	}
	\eqnd
is homotopic to $M_+$. That is, there exists an edge in
	\eqnn
	\map^{\sharp}( (\lioudeltadefstab_+,\cE), \Ex_{\simeq}(\lioudeltadefstab))
	\eqnd
from $m \circ \epsilon$ to $M_+$.
\end{lemma}

\begin{proof}[Proof of Lemma~\ref{lemma. main lemma for Ex}, assuming Lemma~\ref{lemma. M is homotopic to m}.]
Consider the following diagram of simplicial sets.
	\eqnn
	\xymatrix{
	 \lioudeltadefstab_+ \ar[rrrr]^{M_+} 
	\ar[d]^{\eqref{eqn.  lioudelta is a sub of Ex lioustr}_+} 
	\ar@/_70pt/[ddd]_{(\alpha_{\semi})_+} 
	&&&& \Ex_{\simeq}(\lioudeltadefstab) \ar[d]_{\eta} \\
	 \Ex(\lioustrstab) \ar[rrrr]^{\Ex(j)} \ar[d]^{\Ex(\iota)}
		&&&& \Ex(\lioudeltadefstab)	\\
	 \Ex(\lioustrstab[(\eqs^{\dd})^{-1}] \ar[urrrr]_{\Ex(\beta)} \\
	\Ex_{\simeq}(\lioustrstab[(\eqs^{\dd})^{-1}]) \ar[rrrr]_{\Ex_{\simeq}(\beta)} \ar[u]^{\eta}
	&&&& \Ex_{\simeq}(\lioudeltadefstab) \ar[uu]^{\eta}
	}
	\eqnd
There are four cells in the diagram, all of which commute on the nose except for the triangle.
\begin{itemize}
\item The homotopy-coherent triangle in the middle is obtained by applying $\Ex$ to the diagram~\eqref{eqn. beta in general} -- note that Proposition~\ref{prop. Ex is homotopical} shows us that $\Ex$ transports homotopy-coherent diagrams to homotopy-coherent diagrams. Thus there is a homotopy of maps of simplicial sets
	\eqnn
	\Ex(\beta) \circ \Ex(\iota) \circ \eqref{eqn.  lioudelta is a sub of Ex lioustr}_+
	\sim 
	\Ex(j) \circ \eqref{eqn.  lioudelta is a sub of Ex lioustr}_+.
	\eqnd
\item On the other hand, $(\alpha_{\semi})_+$ factors $\Ex(\iota) \circ \eqref{eqn.  lioudelta is a sub of Ex lioustr}_+$ (as indicated by the curved diagram on the left) so the above homotopy restricts to a homotopy of maps of simplicial sets
	\eqnn
	\Ex(\beta) \circ \eta \circ (\alpha_{\semi})_+
		\sim 
	\Ex(j) \circ \eqref{eqn.  lioudelta is a sub of Ex lioustr}_+.
	\eqnd
\item The top rectangle is~\eqref{eqn. defining M eta}.

\item The bottom trapezoid is the naturality of $\eta$ (Notation~\ref{notation. eta}).

By using both the rectangle and the trapezoid, the above homotopy becomes
	\eqn\label{eqn. post eta homotopy}
	\eta \circ \Ex_{\simeq}(\beta) \circ (\alpha_{\semi})_+
	\sim
	\eta \circ M_+.
	\eqnd
\end{itemize}
Now, the homotopy encoded in the triangle~\eqref{eqn. beta in general} is a homotopy of functors between $\infty$-categories; in particular, the induced homotopy~\eqref{eqn. post eta homotopy} is a map $\lioudeltadefstab_+ \times \Delta^1 \to \Ex(\lioudeltadefstab)$ whose image lies entirely in $\Ex_{\simeq}(\lioudeltadefstab)$. So~\eqref{eqn. post eta homotopy} restricts to a homotopy
	\eqnn
	\Ex_{\simeq}(\beta) \circ (\alpha_{\semi})_+
	\sim
	M_+
	\eqnd
respecting the marked edges of $(\lioudeltadefstab_+,\cE)$.
By Lemma~\ref{lemma. M is homotopic to m}, $M_+$ is homotopic to $m \circ \epsilon$. Combining with the homotopy $(\alpha_{\semi})_+ \sim \alpha \circ \epsilon$ (Construction~\ref{construction. alpha}), we conclude
	\eqnn
	\Ex_{\simeq}(\beta) \circ \alpha \circ \epsilon 
	\sim
	m \circ \epsilon.
	\eqnd
Because precomposition with $\epsilon$ induces a homotopy equivalence of mapping spaces (Remark~\ref{remark. epsilon is cartesian equivalence}), we conclude that $\Ex_{\simeq}(\beta) \circ \alpha$ is homotopic to $m$, as desired.
\end{proof}

\subsection{Proof of Lemma~\ref{lemma. M is homotopic to m}}
\label{section. proof of M to m}

One may establish Lemma~\ref{lemma. M is homotopic to m} via the strategy used to prove Lemma~\ref{lemma. alpha j homotopic to m iota}. Because the techniques are similar, we will simply exposit an understanding of the maps $M$ and $m$ and outline the construction; what we do not spell out in detail is the construction of the maps $\GG_{f, k, \vec A}^\vee$ in~\eqref{eqn. G fka check}.

\begin{example}[A description of $M$]
Fix an $l$-simplex $f: \Delta^l \to \lioudeltadefstab$. Such a simplex is equivalent to the data of a diagram in the shape of $\subdiv(\Delta^l)$ which assigns
\begin{itemize}
	\item To any non-empty subset $A \subset [l]$, a Liouville sector $X_A \times T^*\Delta^A$, and
	\item To any inclusion $A \subset B$, a morphism
		\eqnn
		f_{A \subset B} \tensor T^*\eta_{A \subset B} : X_A \times T^*\Delta^A \tensor T^*[-1,0]^{B \setminus A} \to X_B \times T^*\Delta^B.
		\eqnd
\end{itemize}

Then, unwinding the definitions, $M(f): \Delta^l \to \Ex_{\simeq}(\lioudeltadefstab)$ encodes a diagram
	\eqnn
	\subdiv(\subdiv(\Delta^l)) \to \lioustrstab
	\eqnd
which assigns the following. (See Example~\ref{example. subdivision of subdivision} for the notation.)
\begin{itemize}
	\item To every pair $(k, A_0 \subset \ldots \subset A_k)$ where $k \geq 0$ and $A_k \subset [l]$, the sector
		\eqnn
		(X_{A_k} \times T^*\Delta^{A_k} ) \tensor T^*\Delta^k
		\eqnd
	where the sectorial structure in the parentheses is that specified by $f$. Indeed, the object above can be identified as what $j$ assigns to the barycenter of the simplex in $\subdiv(\Delta^l)$ determined by the chain $A_0 \subset \ldots \subset A_k$.
	\item To every inclusion $\delta: [k] \into [k']$ exhibiting $A_0 \subset \ldots \subset A_k$ as a subchain of $A'_{0} \subset \ldots \subset A'_{k'}$, the morphism
		\eqnn
		f_{A_k \subset A'_{k'}} \tensor T^*\eta_{A_k \subset A'_{k'}} \tensor T^*(\delta \times \eta_{\delta}).
		\eqnd
	Concretely, this is the sectorial inclusion
		\eqnn
		X_{A_k} \times T^*\Delta^{A_k} \tensor T^*[-1,0]^{A'_{k'} \setminus A_k} \tensor T^*(\Delta^k \times [-1,0]^{[k'] \setminus \delta([k])})
		\to
		(X_{A'_{k'}} \times T^*\Delta^{A'_{k'}} ) \tensor T^*\Delta^{k'}.
		\eqnd
\end{itemize}
Note that $M$ requires no knowledge of the degeneracy maps of $\lioudeltadefstab$.
\end{example}

We now move onto understanding $m$. 
Let us first set some notation that will help us describe $m: \cC \to \Ex(\cC)$ for a general simplicial set $m$. 

\begin{notation}[$\sigma_{\vec A}$ and $\delta_{\vec A}$]
\label{notation. sigma delta A}
Fix an $l$-simplex $f: \Delta^l \to \cC$. The data of $m(f): \Delta^l \to \Ex(\cC)$ is equivalent to a map
	\eqnn
	\subdiv(\Delta^l) \to \cC.
	\eqnd
We will understand this map by understanding what it does on every $k$-simplex $\Delta^k \to \subdiv(\Delta^l)$.

First, note that a $k$-simplex in $\subdiv(\Delta^l)$ is the data of a chain
	\eqnn
	\vec A = A_0 \subset \ldots \subset A_k \subset [l].
	\eqnd
Define the map
	\eqnn
	\max \vec A : [k] \to [l],
	\qquad
	i \mapsto \max A_i.
	\eqnd
This is the composition
	\eqnn
	[k] \xrightarrow{\vec A} \cP'([l]) \xrightarrow{\max} [l].
	\eqnd
Then $\max \vec A$ factors uniquely as a surjection $\sigma$ followed by an injection $\delta$:
	\eqnn
	[k] \xrightarrow{\sigma_{\vec A}}
	[l'] \cong \Image(\max \vec A)
	\xrightarrow{\delta_{\vec A}}
	[l].
	\eqnd
(The subscript $\vec A$ is to remind us of the dependence on $\vec A$.) Then the composite
	\eqnn
	\Delta^k \to \subdiv(\Delta^l) \to \cC
	\eqnd
is the simplex
	\eqnn
	(\sigma_{\vec A})^*(\delta_{\vec A})^* f
	\eqnd
(where the pullback along $\sigma$ and $\delta$ are defined using the simplicial set structure of $\cC$).
\end{notation}

\begin{example}[A description of $m$]
Now we explicate $m: \lioudeltadefstab \to \Ex_{\simeq}(\lioudeltadefstab)$. Fix an $l$-simplex $f: \Delta^l \to \lioudeltadefstab$. 

As indicated in Notation~\ref{notation. sigma delta A}, it is hard to give an explicit definition of $m(f)$ when $\max \vec A$ is not an injection, for the degeneracy maps of $\lioudeltadefstab$ are, in general, not concretely given (and instead are known to exist by abstract results). We only know the degeneracy maps of a constant barycentric diagram (Definition~\ref{defn. constant simplices in lioudelta})---these degenerate simplices are also constant by Theorem~\ref{theorem. lioudeltadef is an infinity-cat}\eqref{item. lioudelta maps to lioudeltadef}. So let us do the best we can and describe those $k$-simplices in $\subdiv(\Delta^l)$ defining $m(f)$ that do not involve the abstractly defined degeneracies:
\begin{itemize}
	\item We note that $\delta_{\vec A}^* f$ is a constant barycentric diagram when the restriction of $f: \Delta^l \to \lioudeltadefstab$ to the image of $\max \vec A$ is constant. In this case, the simplex $(\sigma_{\vec A})^* (\delta_{\vec A})^* f$ is also constant. In particular, the map $\Delta^k \to \subdiv(\Delta^l) \to \lioudeltadefstab$ is a diagram that is encoded by a single sector $X$. The diagram $\Delta^k \to \lioudeltadefstab$ sends any $\emptyset \neq A \subset [k]$ to the sector $X \tensor T^*\Delta^A$ and all $A \subset B$ to morphisms of the form $\id \tensor T^*\eta_{A \subset B}$.
	\item The other case we can describe is when $\delta_{\vec A}$ is a bijection. This is the case when $\max \vec A: [k] \to [l]$ is an injection---i.e., when the chain $A_0 \subset \ldots \subset A_k$ satisfies the property that $\max A_i$ is strictly less than $\max A_{i+1}$. In this case, the chain $\max A_0 < \max A_1 < \ldots < \max A_k$ defines a subsimplex of $\Delta^l$, and $(\sigma_{\vec A})^* (\delta_{\vec A})^* f$ is simply the restriction of $f$ to this subsimplex.
\end{itemize} 
\end{example}

\begin{example}[$M$ and $m$ on a 1-simplex]
Fix a 1-simplex $f: \Delta^1 \to \lioudeltadefstab$. This is the data of a diagram of (stabilized) Liouville sectors as follows:
	\eqnn
	\xymatrix{
	X_0 \ar[rr]^-{f_{0 \subset 01} \times T^*\eta}
	&& X_{01} \times T^*\Delta^1
	&& X_1 \ar[ll]_-{f_{1 \subset 01} \times T^*\eta}.
	}
	\eqnd
Then $M(f)$ is a 1-simplex in $\Ex_{\simeq}(\lioudeltadefstab)$ that we obtain by applying the map $j$ (Construction~\ref{construction. j from lioustrstab to lioudeltadefstab}) to each edge in the above diagram. The result is a diagram of the shape $\subdiv(\subdiv(\Delta^1))$ in $\lioustrstab$, which we draw as follows:
	\eqnn
	\resizebox{\displaywidth}{!}{
	\xymatrix{
	X_0 \ar[rrr]^-{(f_{0 \subset 01} \times T^*\eta) \tensor T^*\eta}
	&&& (X_{01} \times T^*\Delta^1) \tensor T^*\Delta^1
	&& \ar@{-->}[ll]_-{\id \tensor T^*\eta} X_{01} \times T^*\Delta^1
	  \ar@{-->}[rr]^-{\id \tensor T^*\eta}
	&& (X_{01} \times T^*\Delta^1) \tensor T^*\Delta^1
	&&& X_1 \ar[lll]_-{(f_{1 \subset 01} \times T^*\eta) \tensor T^*\eta}.
	}
	}
	\eqnd
The dashes are to conform to the notation in Figure~\ref{figure. M on 2 simplex}.

On the other hand, $m(f)$ can be drawn as follows:
	\eqnn
	\resizebox{\displaywidth}{!}{
	\xymatrix{
	X_0 \ar[rrr]^-{f_{0 \subset 01} \times T^*\eta}
	&&& X_{01} \times T^*\Delta^1
	&& 	\ar[ll]_-{f_{1 \subset 01} \times T^*\eta} 
		X_1
	  	\ar@{-->}[rr]^-{\id \tensor T^*\eta}
	&& X_1 \times T^*\Delta^1 
	&&& X_1 \ar@{-->}[lll]_-{\id \tensor T^*\eta}.
	}
	}
	\eqnd
\end{example}

\begin{example}[$M$ and $m$ on a 2-simplex]
Fix a 2-simplex $f: \Delta^2 \to \lioudeltadefstab$. Both $M(f)$ and $m(f)$ define a diagram $\subdiv(\subdiv(\Delta^2) \to \lioustrstab$, which we depict in Figure~\ref{figure. M on 2 simplex}.
\end{example}

\clearpage

\begin{figure}[ht]
	
	\eqnn
	\resizebox{\displaywidth}{!}{
	\xymatrix{
	&&&&
	&&&&
	&&&&
	&&&&	X_2 \ar[ddddrrrr] \ar[ddddllll] \ar[dddd]
	\\ \\ \\ \\ 
	&&&&
	&&&&
	&&&&	X_{02} \times T^*\Delta^1 \tensor T^*\Delta^1
	&&&&	X_{012} \times T^*\Delta^2 \tensor T^*\Delta^1
	&&&&	X_{12} \times T^*\Delta^1 \tensor T^*\Delta^1
	\\ \\ 
	&&&&
	&&&&
	&&&&	
	&& X_{012} \times T^*\Delta^2 \tensor T^*\Delta^2 &&	
	&& X_{012} \times T^*\Delta^2 \tensor T^*\Delta^2 &&	
	\\ \\ 
	&&&&
	&&&&	X_{02} \times T^*\Delta^1 \ar@{-->}[uuuurrrr]\ar@{-->}[ddddllll] \ar[rrrr]
	&&&&	X_{012} \times T^*\Delta^2 \tensor T^*\Delta^1
	&&&&	X_{012} \times T^*\Delta^2 
		\ar@{-->}[uuuu] \ar@{-->}[dddd] \ar@{-->}[rrrr] \ar@{-->}[llll] \ar@{-->}[ddddllll] \ar@{-->}[ddddrrrr]
	&&&&	X_{012} \times T^*\Delta^2 \tensor T^*\Delta^1
	&&&&	X_{12} \times T^*\Delta^1  \ar@{-->}[uuuullll]\ar@{-->}[ddddrrrr] \ar[llll]
	\\ \\ \\ \\ 
	&&&&	X_{02} \times T^*\Delta^1 \tensor T^*\Delta^1
	&&&&	X_{012} \times T^*\Delta^2 \tensor T^*\Delta^2
	&&&&	X_{012} \times T^*\Delta^2 \tensor T^*\Delta^1
	&&&&	X_{012} \times T^*\Delta^2 \tensor T^*\Delta^1
	&&&&	X_{012} \times T^*\Delta^2 \tensor T^*\Delta^1
	&&&&	X_{012} \times T^*\Delta^2 \tensor T^*\Delta^2
	&&&&	X_{12} \times T^*\Delta^1 \tensor T^*\Delta^1
	\\ \\ 
	&&&& 
	&&&& 
	&&&& 
	&& X_{012} \times T^*\Delta^2 \tensor T^*\Delta^2 
	&& 
	&& X_{012} \times T^*\Delta^2 \tensor T^*\Delta^2 
	\\ \\ 
			X_0 \ar[rrrruuuu] \ar[rrrrrrrr] \ar[rrrrrrrrrrrruuuu]
	&&&&
	&&&&	X_{01} \times T^*\Delta^1 \tensor T^*\Delta^1
	&&&&
	&&&&	X_{01} \times T^*\Delta^1 \ar@{-->}[rrrrrrrr] \ar@{-->}[llllllll] \ar[uuuu]
	&&&&
	&&&&	X_{01} \times T^*\Delta^1 \tensor T^*\Delta^1
	&&&&
	&&&&	X_1 \ar[lllluuuu] \ar[llllllll] \ar[uuuullllllllllll]
	}
	}
	\eqnd

	\eqnn
	\resizebox{\displaywidth}{!}{
	\xymatrix{
	&&&&
	&&&&
	&&&&
	&&&&	X_2 \ar@{-->}[ddddrrrr] \ar@{-->}[ddddllll] \ar@{-->}[dddd]
	\\ \\ \\ \\ 
	&&&&
	&&&&
	&&&&	X_2 \times T^*\Delta^1
	&&&&	X_2 \times T^*\Delta^1
	&&&&	X_2 \times T^*\Delta^1
	\\ \\ 
	&&&&
	&&&&
	&&&&	
	&& X_2 \times T^*\Delta^2 &&	
	&& X_2 \times T^*\Delta^2 &&	
	\\ \\ 
	&&&&
	&&&&	X_2 \ar@{-->}[uuuurrrr]\ar[ddddllll] \ar@{-->}[rrrr]
	&&&&	X_2 \times T^*\Delta^1
	&&&&	X_2 \ar@{-->}[uuuu] \ar[dddd] \ar@{-->}[rrrr] \ar@{-->}[llll] \ar[ddddllll] \ar[ddddrrrr]
	&&&&	X_2 \times T^*\Delta^1
	&&&&	X_2 \ar@{-->}[uuuullll]\ar[ddddrrrr] \ar@{-->}[llll]
	\\ \\ \\ \\ 
	&&&&	 X_{02} \times T^*\Delta^1
	&&&&	 s
	&&&&	 X_{02} \times T^*\Delta^1
	&&&&	 X_{12} \times T^*\Delta^1
	&&&&	 X_{12} \times T^*\Delta^1
	&&&&	 s
	&&&&	 X_{12} \times T^*\Delta^1
	\\ \\ 
	&&&& &&&& &&&& && X_{012} \times T^*\Delta^2 && && s 
	\\ \\ 
			 X_0 \ar[rrrruuuu] \ar[rrrrrrrr] \ar[rrrrrrrrrrrruuuu]
	&&&&
	&&&&	 X_{01} \times T^*\Delta^1
	&&&&
	&&&&	 X_1 \ar@{-->}[rrrrrrrr] \ar[llllllll] \ar[uuuu]
	&&&&
	&&&&	 X_1 \times T^*\Delta^1
	&&&&
	&&&&	 X_1  \ar[lllluuuu] \ar@{-->}[llllllll] \ar[uuuullllllllllll]
	}
	}
	\eqnd
	\begin{figurelabel}
	\label{figure. M on 2 simplex} 
	\label{figure. m on 2 simplex}
	On top is the diagram $\subdiv(\subdiv(\Delta^2)) \to \lioustrstab$ given by $M$ of a 2-simplex. The dashed arrows are all of the form $\id \tensor T^*\eta$. Not indicated are the morphisms to the objects $X_{012} \times T^*\Delta^2 \tensor T^*\Delta^2$, to reduce clutter. 

	On the bottom is the diagram given by $m$ of a 2-simplex. The dashed arrows are all of the form $\id \tensor T^*\eta$. Simply labeled as ``$s$'' are centers of barycentric subdivision diagrams that are degenerate simplices in $\lioudeltadefstab$. (The only degenerate simplices that are explicitly indicated are the constant simplices.) 
	
	Despite the kinks in the edges, we hope the reader can detect the outline of the $\subdiv(\Delta^2)$ of which we are taking the subdivision.
	\end{figurelabel}
\end{figure}

\clearpage

\begin{construction}[$\GG_f$]
\label{construction. GGf for alpha right inverse}
We have fixed an $l$-simplex $f: \Delta^l \to \lioudeltadefstab$. 
Fix a $k$-simplex $\vec A: \Delta^k \to \subdiv(\Delta^l)$. 
The definition of $m(f)$ on this $k$-simplex involves using degeneracy maps in $\lioudeltadefstab$. Because these degeneracy maps are maps whose existence was concluded formally, they are rather inexplicit. As a result, we will construct $\GG_f$ inductively, by the cardinality of the fibers of the map
	\eqnn
	\Delta^k \xrightarrow{\vec A} \subdiv(\Delta^l) \xrightarrow{\max} \Delta^l.
	\eqnd
\end{construction}

Fix an $l$-simplex $f: \Delta^l \to \lioudeltadefstab$. We may understand the simplices
	\eqnn
	M(f), m(f): \Delta^l \to \Ex_{\simeq}(\lioudeltadefstab)
	\eqnd
by their adjoint maps $\subdiv(\Delta^l) \to \lioudeltadefstab$. To understand these maps, let us note
	\eqnn
	\subdiv(\Delta^l) \cong \bigcup_{k, \vec A} \Delta^k.
	\eqnd
(This is a formula expressing that a $k$-simplex in a barycentric subdivision is the same thing as an increasing length-$k$ chain $\vec A$ of subsimplices.) For each length $k$ chain $\vec A = A_0 \subset \ldots A_k$, suppose one can construct maps
	\eqn\label{eqn. G fka check}
	\xymatrix{
	\Delta^1 \times \Delta^k \ar[rrr]^{\GG_{f, k, \vec A}^\vee}
	&&& \lioudelta
	}
	\eqnd
that glue together to define a map
	\eqn\label{eqn. G check}
	\GG^\vee_f:
	\Delta^1 \times \subdiv(\Delta^l) \cong 
	\Delta^1 \times \left(\bigcup_{(k,\vec A)} \Delta^k \right) 
	\to
	\lioudelta.
	\eqnd
As indicated, we let $\GG^\vee_f$ denote the composite map (of the isomorphism above, together with the gluings of $\GG^\vee_{f,k,\vec A}$). 
We then have 
	\eqn\label{eqn. collapse map}
	\xymatrix{
	\Delta^1 \times \subdiv(\Delta^l)  \ar[rrr]^{\GG^\vee_f}
	&&& \lioudelta \\
	\subdiv(\Delta^1) \times \subdiv(\Delta^l) \ar[u]^{\max_{\Delta^1} \times \id_{\subdiv(\Delta^l)}} \\
	\subdiv(\Delta^1 \times \Delta^l) \ar[u]^{\text{collapse}}.
	}
	\eqnd
By adjunction, we thus obtain a map $\Delta^l \times \Delta^1 \to \Ex(\lioudeltadefstab)$. A routine check on edges demonstrates that this map factors through $\Ex_{\simeq}(\lioudeltadefstab)$, so we obtain the desired map
	\eqnn
	\GG_f: \Delta^1 \times \Delta^l \to \Ex_{\simeq}(\lioudeltadefstab).
	\eqnd

\begin{remark}
Given any two simplicial sets $X$ and $Y$, there is a collapse map given by the dashed arrow below, defined as the composition of the solid arrows. 
	\eqnn	
	\resizebox{\displaywidth}{!}{
	\xymatrix{
	\subdiv(X) \times \subdiv(Y)	
	&	\ar[l]_-{\cong} \left(\colim\limits_{\Delta^{k'} \to X} \subdiv(\Delta^{k'}) \right) \times \subdiv(Y)	
	&	\ar[l]_-{\cong} \colim\limits_{\Delta^{k'} \to X} \subdiv(\Delta^{k'}) \times \subdiv(Y)
	&	\ar[l]_-{\cong} \colim\limits_{\Delta^{k'} \to X} \subdiv(\Delta^{k'}) \times \left( \colim\limits_{\Delta^{k''} \to Y}  \subdiv(\Delta^{k''}) \right)
	&	\ar[l]_-{\cong} \colim\limits_{\Delta^{k'} \to X} \colim\limits_{\Delta^{k''} \to Y} \subdiv(\Delta^{k'}) \times \subdiv(\Delta^{k''})
	\\
	&&&& \colim\limits_{\Delta^k \to X \times Y} \subdiv(\Delta^k) \times \subdiv(\Delta^k) \ar[u]
	\\
	\subdiv(X \times Y) \ar@{-->}[uu]
	&&&&  \ar[llll]_{\cong} \colim\limits_{\Delta^k \to X \times Y} \subdiv(\Delta^k) \ar[u]
	}
	}
	\eqnd
Every isomorphism in the above diagram is by definition of $\subdiv$ of a simplicial set, and by the fact that direct products of simplicial sets commutes with colimits (because, for example, the category of simplicial sets is Cartesian closed). 

The two vertical arrows in the right are diagonal maps -- the lower-right upward map is induced by the diagonal map of simplicial sets $W \to W\times W$ in the case $W = \subdiv(\Delta^k)$, while the upper-right upward map is the inclusion of the diagram of simplices $\{\Delta^k \to X \times Y\}$ into the diagram of simplices $\{\Delta^{k'} \to X, \Delta^{k''} \to Y\}$ (by including into the subdiagram where $k' = k''$).

The collapse map in~\eqref{eqn. collapse map} is the special case of $X = \Delta^1,Y=\Delta^l$.
\end{remark}

\begin{example}
Let $l=1$. Then the diagram $\Delta^1 \times \subdiv(\Delta^l)$ may be drawn as follows:
	\eqn\label{eqn. delta 1 x subdiv delta 1}
	\resizebox{\displaywidth}{!}{
	\xymatrix{
	\bullet \ar[rrrrr]
	&&& 
	&& 
	\bullet 
	&&
	&&& \bullet \ar[lllll]
	\\
	& & 
	\\
	\bullet \ar[rrrrr]  \ar[uu]   \ar[uurrrrr]
	&&& 
	&& 	
		\bullet \ar[uu] 
	&& 
	&&& \bullet \ar[lllll]  \ar[uu] \ar[uulllll] 
	}
	}
	\eqnd
Fix a simplex $f: \Delta^l \to \lioudeltadefstab$. 
The map $\GG^\vee_f: \Delta^1 \times \subdiv(\Delta^l) \to \lioudeltadefstab$ from~\eqref{eqn. G check} represents a diagram $\sub(\Delta^1 \times \subdiv(\Delta^l)) \to \lioustrstab$; we can draw the parts of this diagram corresponding to the arrows in~\eqref{eqn. delta 1 x subdiv delta 1} (and not, say, the 2-simplices) as follows:
	\eqnn
	\resizebox{\displaywidth}{!}{
	\xymatrix{
	X_0 \ar[rrr]^-{(f_{0 \subset 01} \times T^*\eta) \tensor T^*\eta} \ar[d] 
	&&& (X_{01} \times T^*\Delta^1) \tensor T^*\Delta^1
	&& \ar@{-->}[ll]_-{\id \tensor T^*\eta} 
		\ar[dll]
		X_{01}  \times T^*\Delta^1
	  \ar@{-->}[rr]^-{\id \tensor T^*\eta} 
	   \ar[d]
	   \ar[drr] 
	&& (X_{01} \times T^*\Delta^1) \tensor T^*\Delta^1
	&&& X_1 \ar[lll]_-{(f_{1 \subset 01} \times T^*\eta) \tensor T^*\eta}. \ar[d]
	\\
	X_0 \tensor T^*\Delta^1 
	&&& (X_{01} \times T^*\Delta^1) \tensor T^*\Delta^1
	&&  X_{01} \times T^*\Delta^1  \tensor T^*\Delta^1 
	&& (X_{01} \times T^*\Delta^1) \tensor T^*\Delta^1
	&&& X_1 \tensor T^*\Delta^1  
	\\
	X_0 \ar[rrr]^-{f_{0 \subset 01} \times T^*\eta} \ar[u]  \ar[urrr]
	&&& X_{01} \times T^*\Delta^1
	&& 	\ar[ll]_-{f_{1 \subset 01} \times T^*\eta} 
		X_1
	  	\ar@{-->}[rr]^-{\id \tensor T^*\eta} \ar[u]
	&& X_1 \times T^*\Delta^1 
	&&& X_1 \ar@{-->}[lll]_-{\id \tensor T^*\eta}. \ar[u]\ar[ulll]
	}
	}
	\eqnd
\end{example}

\subsection{Completing the proof of Theorem~\ref{theorem. localization}}

We have the following corollary of Lemmas~\ref{lemma. alpha-beta equivalence} and~\ref{lemma. main lemma for Ex}.

\begin{cor}\label{cor. beta is equivalence}
The maps $\alpha$ and $\beta$ are equivalences of $\infty$-categories.
\end{cor}

\begin{proof}
It is clear that $\alpha$ is essentially surjective. Lemma~\ref{lemma. main lemma for Ex} shows that $\alpha$ induces an injection on $\pi_*\hom$. Since Lemma~\ref{lemma. alpha-beta equivalence} says $\alpha \circ \beta$ is an equivalence, we see $\alpha$ induces a surjection on $\pi_*\hom$. Thus $\alpha$ is an equivalence of $\infty$-categories. Since $\alpha$ and $\alpha\circ\beta$ are both equivalences, we conclude that $\beta$ is an equivalence as well. 
\end{proof}

\begin{proof}[Proof of Theorem~\ref{theorem. localization}]
$\beta$ from Notation~\ref{notation. iota j beta} is the map induced by the universal property of localization. Now invoke Corollary~\ref{cor. beta is equivalence}.
\end{proof}

\clearpage
\section{The symmetric monoidal structure}
The goal of this section is to construct the symmetric monoidal structure on $ \lioustrstab[(\eqs^\dd)^{-1}]$.

\subsection{Stabilizing commutes with localization}\label{section. stabilizing and localization}
Let $\cC$ be an arbitrary $\infty$-category, and fix a collection of morphisms $S \subset \cC$. We denote by $\cC[S^{-1}]$ the localization. 

\begin{notation}
We let $\dd: \cC \to \cC$ be an endofunctor and define
	\eqnn
	\cC^{\dd}:= \colim \left(\cC \xrightarrow{\dd} \cC \xrightarrow{\dd} \ldots \right)
	\eqnd
Assuming that $\dd(S) \subset S$, we define
	\eqnn
	\cC[S^{-1}]^{\dd} := \colim \left(\cC[S^{-1}] \xrightarrow{\dd} \cC[S^{-1}] \xrightarrow{\dd} \ldots \right).
	\eqnd
\end{notation}

Then:

\begin{lemma}\label{lemma. increasing union of localizations is localization}
The natural functor $\cC^{\dd} \to \cC[S^{-1}]^{\dd}$ is a localization of $\cC^{\dd}$. In fact, letting $a_j: \cC \to \cC^{\dd}$ denote the natural morphism from the $j$th copy of $\cC$ appearing in the colimit, and letting $S^{\dd} \subset \cC^{\dd}$ denote the collection of morphisms 
	\eqnn
	S^{\dd} = \bigcup_j a_j(S)
	\eqnd
the map $\cC^{\dd} \to \cC[S^{-1}]^{\dd}$ realizes an equivalence
	\eqnn
	\cC^{\dd}[(S^{\dd})^{-1}] \xrightarrow{\sim} \cC[S^{-1}]^{\dd}
	\eqnd
\end{lemma}

\begin{proof}
We $\mathcal{W}\inftycat$ be the $\infty$-category of $\infty$-categories marked by systems. Concretely, an object is a pair $(\cC,S)$ of an $\infty$-category $\cC$ together and a collection of edges $S$ in $\cC$ containing all equivalences of $\cC$, and closed under homotopy and composition. (For more details, see Construction~4.1.3.1 of~\cite{higher-algebra}.) In what follows, we also let $\cC^{\sim}$ denote the largest $\infty$-groupoid contained in $\cC$---it is the Kan complex generated by all objects of $\cC$ and by all equivalences of $\cC$.
The functor
	\eqnn
	\mathcal{W}\inftycat \to \mathcal{W}\inftycat,
	\qquad 
	(\cC,S) \mapsto ( \cC[S^{-1}], \cC[S^{-1}]^{\sim})
	\eqnd
---which sends the pair $(\cC,S)$ to the localization $\cC[S^{-1}]$, marked by all equivalences of the localization---may be modeled as a fibrant replacement functor in the (left proper, combinatorial) model category of marked simplicial sets. The image of this localization functor is equivalent to the $\infty$-category of $\infty$-categories;  the functor  is in fact an idempotent left adjoint to the inclusion of the $\infty$-category of $\infty$-categories. (See Proposition~4.1.3.2 of~\cite{higher-algebra}.) So it preserves colimits. We conclude by noting that
	\eqnn
	(\cC^{\dd}, S^{\dd}) = \colim \left( (\cC,S) \xrightarrow{\dd} (\cC,S) \xrightarrow{\dd} \ldots \right)
	\eqnd
as marked simplicial sets.
\end{proof}

\begin{corollary}\label{cor. stab loc is loc stab}
We have an equivalence 
	$
	\lioustrstab[(\eqs^{\dd})^{-1}] \xrightarrow{\sim} \lioustr[\eqs^{-1}]^{\dd}.
	$
\end{corollary}

\subsection{Symmetric monoidal localization}
Let $\cC^\tensor$ be a symmetric monoidal $\infty$-category with underlying $\infty$-category $\cC$. Fix also a collection of morphisms $S \subset \cC$.

\begin{defn}
We say that $S$ is {\em compatible} with a symmetric monoidal structure $\tensor$ if for all  objects $M \in \cC$ and all morphisms $f \in S$, we have that $\id_M \tensor f$ and $f \tensor \id_M$ are both in $S$. 
\end{defn}

\begin{prop}\label{prop. localization inherits symmetric monoidal structure}
Suppose $S$ is compatible with $\tensor$.
Then there is an induced symmetric monoidal structure on $\cC[S^{-1}]$ so that the functor $\cC \to \cC[S^{-1}]$ may be promoted to a symmetric monoidal functor.

In fact, more is true: for any symmetric monoidal $\infty$-category $\cD^{\tensor}$, the $\infty$-category of symmetric monoidal functors $\fun^\tensor(\cC[S^{-1}]^{\tensor},\cD^{\tensor})$ is identified with the full subcategory of $\fun^\tensor(\cC^{\tensor},\cD^{\tensor})$ sending morphisms in $S$ to equivalences in $\cD$. This identification is given by composing with the symmetric monoidal functor $\cC^{\tensor} \to \cC[S^{-1}]^{\tensor}$ 
\end{prop}

\begin{proof}
This is a consequence of Proposition 4.1.3.4 of~\cite{higher-algebra}.
\end{proof}

\subsection{Symmetric monoidal structures on stabilizations}

Let $\cC^\tensor$ be a symmetric monoidal $\infty$-category with underlying $\infty$-category $\cC$, and an object $T$. One may contemplate whether the colimit
	\eqn\label{eqn. colimit T inverse}
	\colim\left(
	\cC \xrightarrow{T \tensor -}
	\cC \xrightarrow{T \tensor -}\ldots
	\right)
	\eqnd
---which may\footnote{We will not use the notation $\cC[T^{-1}]$ in the rest of this work, to avoid confusion with notation such as $\eqs^{-1}$. The philosophy of our work centers around compositionally inverting morphisms, not multiplicatively inverting objects.} rightfully be called $\cC[T^{-1}]$--- admits a symmetric monoidal structure for which $T$ becomes an invertible object. (That is, for which there exists some object $U$ such that $T \tensor U$ is equivalent to the tensor unit.)

Inspired by observations of Voevodsky~\cite[Section~4]{voevodsky-icm}, and following ideas of To\"en-Vezzosi~\cite{toen-vezzosi-hag2}, Robalo gives the following criterion.

\begin{prop}[Corollary~4.24 of~\cite{robalo-non-commutative-motives-i}]\label{prop. symmetric T}
Let $\cC^\tensor$ be a symmetric monoidal $\infty$-category with underlying $\infty$-category $\cC$.  Let $T \in \cC$ and let
	\eqnn
	\tau: T \tensor T \tensor T \to T\tensor T \tensor T
	\eqnd
denote the cyclic permutation map $(123)$ determined by the symmetric monoidal structure. If $\tau$ is homotopic to the identity of $T \tensor T \tensor T$, then 
	\enum
	\item The colimit~\eqref{eqn. colimit T inverse} inherits a symmetric monoidal structure from $\cC$ for which the functor
		\eqnn
		j: \cC \to \colim\left(
	\cC \xrightarrow{T \tensor -}
	\cC \xrightarrow{T \tensor -}\ldots
	\right)
		\eqnd
	can be made symmetric monoidal.
	\item Moreover, $j$ is universal among symmetric monoidal functors that send $T$ to an invertible object.
	\enumd
\end{prop}

\begin{proof}
For the reader's convenience, let us recall that any symmetric monoidal $\infty$-category is a commutative (that is, $E_\infty$) algebra in the $\infty$-category of $\infty$-categories (where the monoidal structure is direct product of $\infty$-categories). It thus makes sense to speak of modules over any symmetric monoidal $\infty$-category. The original reference is Section~4.5 of~\cite{higher-algebra}; we also recommend Section~3 of~\cite{robalo-non-commutative-motives-i}.

We follow the notation of \cite{robalo-non-commutative-motives-i}. In Proposition 4.1 of ibid., the universal symmetric monoidal $\infty$-category obtained from $\cC$ by inverting $T$ is constructed---it is denoted by $\cL^\tensor_{(\cC^{\tensor},T)}(\cC^\tensor)$. In fact, $\cL^\tensor_{(\cC^{\tensor},T)}$ is more generally a left adjoint taking any commutative $\cC^\tensor$-algebra $A$ to another $\cC^\tensor$-algebra $\cL^\tensor_{(\cC^{\tensor},T)}(A)$ for which the map $\cC \to A$ sends $T$ to an invertible object.  

The same construction works for $\cC^{\tensor}$-modules: Pulling back along the canonical map $\cC^\tensor \to \cL^\tensor_{(\cC^{\tensor},T)}(\cC^\tensor)$, one witnesses a fully faithful embedding from the $\infty$-category of $\cL^\tensor_{(\cC^{\tensor},T)}(\cC^\tensor)$-modules to that of $\cC^\tensor$-modules. The left adjoint produces, out of any $\cC^\tensor$-module $M$, the universal $\cC^\tensor$-module 
	\eqnn
	M \to \cL_{(\cC^{\tensor},T)}(M)
	\eqnd
rendering $T \tensor-$ an autoequivalence (Proposition~4.2 of ibid.). We note that---by the usual passage between algebras and modules---in the special case that $M = \cC$ (where the module action is the canonical one of $\cC^\tensor$ acting on itself), there is a natural equivalence
	\eqn\label{eqn. non tensor L}
	\cL_{(\cC^{\tensor},T)}^{\tensor}(\cC^\tensor)
	\simeq
	\cL_{(\cC^{\tensor},T)}(\cC).
	\eqnd
	
We shall now finally use the assumption on $\tau$ being homotopic to $\id_{T^{\tensor 3}}$.
Given a module $M$ over $\cC^\tensor$, the colimit
	\eqnn
	\stab_T(M):= \colim\left(
	M \xrightarrow{T \tensor -}
	M \xrightarrow{T \tensor -} \ldots
	\right)
	\eqnd
is another $\cC^\tensor$-module with the property that the endofunctor $T \tensor -$ is an equivalence (Proposition~4.21 of ibid.; this requires $\tau \sim \id_{T^{\tensor 3}}$). Thus, we obtain a canonical map 
	\eqn\label{eqn. cL to stab}
	\cL_{(\cC^{\tensor},T)}(M) \to \stab_T(M).
	\eqnd
We claim this map is an equivalence; this is more or less the exact same proof as in Corollary~4.24 of ibid. and we reproduce it here for the reader's convenience.

By adjunction, we have a factorization $\cL_{(\cC^{\tensor},T)}(M) \to \cL_{(\cC^{\tensor},T)}(\stab_T M) \to \stab_T(M)$. Because $\cL_{(\cC^{\tensor},T)}$ is a left adjoint, it commutes with colimits, hence $\stab_T\cL_{(\cC^{\tensor},T)} \simeq \cL_{(\cC^{\tensor},T)} \stab_T$. This results in a factorization
	\eqnn
	\cL_{(\cC^{\tensor},T)}(M) 
	\to \stab_T\cL_{(\cC^{\tensor},T)}(M) 
	\xrightarrow{\sim} \cL_{(\cC^{\tensor},T)}(\stab_T (M)) 
	\to \stab_T(M).
	\eqnd
The last arrow is an equivalence because $\stab_T M$ is already a module for whom $T \tensor -$ is an equivalence. The first arrow is an equivalence because $T\tensor -$ is an autoequivalence of $\cL_{(\cC^{\tensor},T)}(M)$, meaning the sequence colimit $\stab_T$ recovers $\cL_{(\cC^{\tensor},T)}(M)$ itself.

Combining the equivalences~\eqref{eqn. non tensor L} and~\eqref{eqn. cL to stab}, and by invoking the known universal property for the lefthand side of~\eqref{eqn. non tensor L}, we are finished.
\end{proof}

\subsection{The symmetric monoidal structure}\label{section. construction of symmetric monoidal structure}
We will prove the following lemma shortly; it is at the interface of geometry and localizations:

\begin{lemma}\label{lemma. main monoidal lemma}
Consider the object $T^*[0,1] \times T^*[0,1]  \times T^*[0,1]$ in $\lioustr$ (and hence in $\lioustr[\eqs^{-1}]$). Denote by $T^*\tau$ the natural automorphism given by the cyclic order 3 permutation $(123)$ of coordinates:
	\eqnn
	\tau: \left( (q_1,p_1),
	(q_2,p_2),
	(q_3,p_3)) \right)
	\mapsto
	\left(
	(q_3,p_3), 
	(q_1,p_1),
	(q_2,p_2)) \right).
	\eqnd
(This map is part of the symmetric monoidal structure in $\lioustr$, hence in $\lioustr[\eqs^{-1}]$.)  Then in $\lioustr[\eqs^{-1}]$, $T^*\tau$ is homotopic to the identity.
\end{lemma}

\begin{remark}
\label{remark. the least formal part of monoidal structure}
 
Lemma~\ref{lemma. main monoidal lemma} is the only non-formal ingredient of exhibiting a symmetric monoidal structure on $\lioustr[\eqs^{-1}]$. The main issue is that there is no straightforward geometric description of the morphism spaces in  $\lioustr[\eqs^{-1}]$ (in contrast to  $ \lioustrstab[(\eqs^\dd)^{-1}]$). All we have is the definition as a localization. Thus, we must somehow argue that the categorically formal process of localization recovers the geometric information of $(123)$ being isotopic to the identity.
\end{remark}

\subsection{Proof of Theorem~\ref{theorem. lioustab is symmetric monoidal}}
Assume Lemma~\ref{lemma. main monoidal lemma} for the moment.

\begin{proof}[Proof of Theorem~\ref{theorem. lioustab is symmetric monoidal}]
Clearly, $\lioustr$ has a symmetric monoidal structure given by direct products---of objects, and of morphisms. (Here we are using that all morphisms in $\lioustr$ are {\em strict} sectorial embeddings.) 

By Proposition~\ref{prop. localization inherits symmetric monoidal structure}, we thus obtain a symmetric monoidal structure on the $\infty$-category
	\eqnn
	\lioustr[\eqs^{-1}]
	\eqnd
and the symmetric monoidal functor 
	\eqnn
	\lioustr \to \lioustr[\eqs^{-1}]
	\eqnd
has the universal property that any symmetric monoidal functor $\lioustr \to \cD$ sending Liouville equivalences to equivalences in $\cD$ will admit a symmetric monoidal, homotopy-coherent factorization through $\lioustr[\eqs^{-1}]$.

Assuming Lemma~\ref{lemma. main monoidal lemma}, we see that $T = T^*[0,1]$ satisfies the hypotheses of Proposition~\ref{prop. symmetric T}. 
Invoking Proposition~\ref{prop. symmetric T}, we conclude that 
	$
	\lioustr[\eqs^{-1}]^{\dd} 
	= 
	\stab_{T^*[0,1]}\lioustr[\eqs^{-1}]
	$ 
thus\footnote{Here, we are utilizing notation from Section~\ref{section. stabilizing and localization} and from the Proof of Proposition~\ref{prop. symmetric T}.} inherits a symmetric monoidal structure, and the symmetric monoidal functor
	\eqnn
	\lioustr[\eqs^{-1}] \to \lioustr[\eqs^{-1}]^{\dd}
	\eqnd
is universal for symmetric monoidal functors out of $\lioustr[\eqs^{-1}]$ that send $T^*[0,1]$ to an invertible object. 

Finally, we have an equivalence of $\infty$-categories as follows:
	\eqnn
	\xymatrix{
	\lioustr[\eqs^{-1}]^{\dd} \ar[rr]_-{\sim}^-{\text{Cor~\ref{cor. stab loc is loc stab}}}
	&& \lioustrstab[(\eqs^{\dd})^{-1}]
	} 
	\eqnd
The result follows by endowing $\lioulocal \simeq \lioustrstab[(\eqs^{\dd})^{-1}]$ with the induced symmetric monoidal structure.
\end{proof}

\subsection{Proof of Lemma~\ref{lemma. main monoidal lemma}}

\nc{\mfldcmpct}{\mfld^{\mathsf{cmpct}}}
In this proof, we will replace $T^*[0,1]$ with $T^*[-1,1]$ for convenience. (We would like the origin to be in the interior of the interval.)

Let $\mfld$ denote the category (in the classical sense) of smooth manifolds, possibly with corners, and whose morphisms are smooth codimension zero embeddings; it is symmetric monoidal under direct product. We let $\mfldcmpct \subset \mfld$ denote the full subcategory of {\em compact} smooth manifolds, possibly with corners. There is a symmetric monoidal functor
	\eqn\label{eqn. T* functor}
	\mfldcmpct \to \lioustr,
	\qquad
	Q \mapsto T^*Q.
	\eqnd
Because this functor is symmetric monoidal, it sends the permutation  map $\tau$ of $[-1,1] \times [-1,1] \times [-1,1]$ to $T^*\tau$. 

\begin{defn}\label{defn. isotopy equivalence}
Let $X$ and $Y$ be objects of $\mfldcmpct$. We say that a smooth, codimension zero embedding $\overline{f}: X \to Y$ is an {\em isotopy equivalence} if there exist a smooth, codimension zero embedding $\overline{g}: Y \to X$ and isotopies $\overline{f}\overline{g} \sim \id_Y, \overline{g}\overline{f} \sim \id_X$.
\end{defn}

\begin{example}\label{example. every oriented f is an equivalence}
If $X = Y = [-1,1]^n$ and if $\overline{f}$ is orientation-preserving, then $\overline{f}$ is an isotopy equivalence. This is a result of the following stronger fact:
\end{example}

\begin{prop}\label{prop. O_n}
Fix $n \geq 0$ and a real number $0<\epsilon \leq 1$. The following spaces are all homotopy equivalent:
\enum
	\item\label{item. all f} The space of smooth, codimension-zero embeddings $\overline{f}: [-1,1]^n \to [-1,1]^n$.
	\item\label{item. small f} The space of those $\overline{f}$ such that the image of $\overline{f}$ is contained in $[-\epsilon,\epsilon]^n$.
	\item\label{item. all j} The space of smooth, codimension-zero embeddings $\overline{j} : [-1,1]^n \to \RR^n$. 
	\item\label{item. j_0} The space of those $\overline{j}$ which send the origin to the origin.
	\item\label{item. linear j} The space of those $\overline{j}$ which are (restrictions of) linear maps.
	\item\label{item. f preserving 0} The space of those $\overline{f}: [-1,1]^n \to [-1,1]^n$ which preserve the origin.
\enumd
In particular, all these spaces are homotopy equivalent to the orthogonal group $O(n)$.

Now fix an orientation on $[-1,1]^n$ and $\RR^n$. If we consider the subspaces of the above consisting only of orientation-preserving maps, they are all homotopy equivalent to $SO(n)$. 
\end{prop}

\begin{proof}
	$\ref{item. all f} \sim \ref{item. small f}$: There is a smooth isotopy from $[-1,1]^n$ to $[-\epsilon,\epsilon]^n$, for example by scaling by $s$ as we take $s$ from $1$ to $\epsilon$.

	$\ref{item. small f}\sim \ref{item. all j} $: Let $F_\epsilon$ denote the space of all smooth, codimension-zero embeddings $[-1,1]^n \to \RR^n$ with image contained in $[-\epsilon,\epsilon]$. We have weak homotopy equivalences of simplicial sets as follows: 
		\eqnn
		\sing(F_\epsilon) \xrightarrow{\sim}
		\colim\left(\sing F_\epsilon \into \sing F_{2\epsilon}\into \ldots \right)
		\xrightarrow{\cong} \sing(\ref{item. all j}).
		\eqnd
		The first arrow is an equivalence because the inclusions $\sing(F_{n\epsilon}) \to \sing(F_{(n+1)\epsilon})$ are trivial cofibrations of simplicial sets. (The colimit is modeled as a honest ``union'' of simplicial sets; because all arrows in the colimit are cofibrations, this also models the homotopy colimit.) $\sing(\ref{item. all j})$ refers to the singular complex of the space \ref{item. all j}; so the final arrow is an isomorphism because  $[-1,1]^n$ is compact. 
	
	$\ref{item. all j} \sim \ref{item. j_0} $: Given $\overline{j}$, consider the $s$-dependent isotopy $(s,x) \mapsto \overline{j}(x) - s\overline{j}(0)$ for $x \in [-1,1]^n$.
	
	$\ref{item. j_0} \sim\ref{item. linear j} $: Consider the $s$-dependent isotopy $(s,x) \mapsto {\frac{1}{s}}\overline{j}(sx)$. (This is the usual trick of homotoping a map to its derivative at the origin.)
	
	$\ref{item. j_0}\sim\ref{item. f preserving 0}$: Let $F'_\epsilon$ denote the collection of those $j:[-1,1]^n \to \RR^n$ which preserve the origin, and which have image inside the cube $[-\epsilon,\epsilon]^n$. Because scaling respects the origin, we have a sequential diagram of trivial cofibrations
	\eqnn
	F'_\epsilon \into F'_{2\epsilon} \into F'_{4\epsilon} \into \ldots
	\eqnd
	whose (homotopy) colimit is equivalent to \ref{item. j_0}. When $\epsilon=1$, we see that the inclusion of $\ref{item. f preserving 0}$ as the first term of the sequential diagram exhibits a homotopy equivalence to \ref{item. j_0}.
	
	$\ref{item. linear j} \sim O(n)$: \ref{item. linear j} gives us the equivalence to the space $GL(n)$, which in turn is homotopy equivalent to $O(n)$ by Gram-Schmidt.
	
	The last claim of the proposition (regarding orientation-preserving versions) is proven by running through the exact same isotopies.
	\end{proof}

\begin{remark}\label{remark. 0 preserving is all f}
In fact, the obvious inclusion $\ref{item. f preserving 0} \to \ref{item. all f}$ factors the homotopy equivalence from \ref{item. f preserving 0} to \ref{item. j_0}, so we see that this obvious inclusion is a homotopy equivalence.
\end{remark}

\begin{notation}
We let $\eqs \subset \mfldcmpct$ denote the collection of isotopy equivalences. Given an object $X$ of $\mfldcmpct$, we let $\mfldcmpct_{/X}$ denote the slice category, and let
	\eqnn
	\eqs_{/X} \subset \mfldcmpct_{/X}
	\eqnd
be the collection of morphisms obtained by pulling back $\eqs$ along the forgetful functor $\mfldcmpct_{X} \to \mfldcmpct$.

Concretely, an object of $\mfldcmpct_{/X}$ is the data of a smooth, codimension zero embedding $j: A \to X$, and a morphism from $j$ to $j'$ is the data of a smooth, codimension zero embedding $f: A \to A'$ satisfying $j'f = j$. Such a morphism is in $\eqs_{/X}$ if and only if $f$ is an isotopy equivalence.
\end{notation}

\begin{remark}
The map $\mfldcmpct \to \mfldcmpct[\eqs^{-1}]$ induces a diagram
	\eqnn
	\xymatrix{
	\mfldcmpct_{/X} \ar[r] \ar[d] 
		&\mfldcmpct[\eqs^{-1}]_{/X} \ar[d] \\
	\mfldcmpct \ar[r] 
		& \mfldcmpct[\eqs^{-1}]
	}
	\eqnd
hence, by the universal property of localizations, we obtain a map
	\eqn\label{eqn. slice map localization}
	\mfldcmpct_{/X}[(\eqs_{/X})^{-1}] \to \mfldcmpct[\eqs^{-1}]_{/X}.
	\eqnd
\end{remark}

To prove Lemma~\ref{lemma. main monoidal lemma}, we will prove that $\tau$ is homotopic to $\id_{[-1,1]^3}$ in the localization $\mfldcmpct[\eqs^{-1}]$. The strategy to accomplish this follows techniques we learned from Lemma~5.4.5.10 of~\cite{higher-algebra}\footnote{Another approach would be to utilize techniques from Proposition~2.22 of~\cite{aft-2}.}, with slight modifications stemming from a small complication: $[-1,1]^3$ (because of its corners) does not admit a cover by open disks. The reason for considering ``$\mathrm{ex}$''-subscripted versions of smooth objects below is to overcome this issue.

Fix a smooth manifold $M$ with $\dim M = n$, and fix a smooth, codimension-zero embedding $i: M \into \overline{M}$ where $\overline{M}$ is a smooth {\em compact} manifold (possibly with corners), and for which $i(M) = \overline{M} \setminus \del \overline{M}$. (Here, $\del \overline{M}$ is the locus of all boundary and corner strata of $\overline{M}$.) The prototypical example for us is $i: (-1,1)^N \into [-1,1]^N$. (And in fact, all examples for us will be diffeomorphic to this example.) 

We also fix an orientation on $M$. 

Given this data, we define the following:
\nc{\diskex}{\mathrm{Disk}_{\mathrm{ex}}}
\nc{\embex}{\mathrm{Emb}_{\mathrm{ex}}}
\nc{\oriented}{\mathrm{or}}

\begin{notation}
 We let $B_M^\delta$ denote the category where an object is the data of a smooth, orientation-preserving,
 codimension zero embedding $j: (-1,1)^n \into M$ with the property that $j$ extends to a smooth embedding\footnote{That $\overline{j}$ is smooth, is an injection, and has derivative everywhere given by a linear isomorphism, is all important in what follows. For example, that $\overline{j}$ is an injection allows us to identify the closure of the image of $j$ with the image of $\overline{j}$.} $\overline{j}:[-1,1]^n \into \overline{M}$. A morphism from $j$ to $j'$ is the data of a smooth map $f: (-1,1)^n \to (-1,1)^n$ for which $j'f = j$. (Note it is then automatic that $f$ extends to a smooth embedding $\overline{f}: [-1,1]^n \to [-1,1]^n$.)
\end{notation}

\begin{remark}
There is a functor $B_M^\delta \to \mfldcmpct_{/\overline{M}}$ by sending $j: (-1,1)^n \to M$ to $\overline{j}$. We note that every morphism $f$ necessarily maps to an isotopy equivalence $\overline{f}$ (Example~\ref{example. every oriented f is an equivalence})
so we have an induced map of localizations 
	\eqn\label{remark. BM delta to slice}
	|B_M^\delta| \to \mfldcmpct_{/\overline{M}}[(\eqs_{/\overline{M}})^{-1}].
	\eqnd
Here, $|B_M^\delta|$ is the Kan complex obtained by localizing $B_M^\delta$ with respect to every morphism; in other words, it is the $\infty$-groupoid completion of the category $B_M^\delta$.
\end{remark}

\begin{notation}
We let $\diskex(\overline{M})$ denote the poset whose elements are open subsets $U \subset M$ that are equal to the image of some $j$ as above. (Equivalently, these are open subsets $U \subset M$ for which the closures $\overline{U}$ (taken in $\overline{M}$) are diffeomorphic to the cube $[-1,1]^n$.) The poset relation is the usual inclusion relation of subsets. 
\end{notation}

\begin{remark}
There is a functor 
		\eqn\label{eqn. p BM}
		p: B_M^\delta \to \diskex(\overline{M})
		\eqnd
---called take the image of $j$---and this functor is an equivalence of categories. In fact, $p$ is a coCartesian fibration classifying a functor from $\diskex(\overline{M})$ to the category of preorders. More precisely, thinking of every preorder as a category, and hence as a simplicial set (by taking the nerve), this coCartesian fibration classifies the map
		\eqn\label{eqn. functor PU}
		\diskex(\overline{M}) \to \sset,
		\qquad
		U \mapsto p^{-1}(U).
		\eqnd.
\end{remark}	

\begin{notation}[$\embex^{\oriented}$]
Given oriented manifolds $A$ and $B$, let $\emb^{\oriented}(A,B)$ be the space of smooth, codimension zero orientation-preserving embeddings, and---given further the data of smooth compactifications $A \subset \overline{A}, B \subset \overline{B}$ to manifolds with corners---let $\embex^{\oriented}(\overline{A},\overline{B}) \subset \emb^{\oriented}(A,B)$ denote the subspace consisting of those $j: A\to B$ that smoothly extend to an embedding $\overline{A} \to \overline{B}$. 
\end{notation}

\begin{lemma}\label{lemma. pullback at point of M}
Fix an embedding $j: (-1,1)^n \to M$ which extends to a smooth embedding $\overline{j}: [-1,1]^n \to \overline{M}$. Let $\embex^{\oriented,0}((-1,1)^n,(-1,1)^n) \subset \embex^{\oriented}((-1,1)^n,(-1,1)^n)$ denote the subspace consisting of those $f$ that respect the origin.

The diagram
	\eqn\label{eqn. lemma pullback at point of M}
	\xymatrix{
	\embex^{\oriented,0}((-1,1)^n,(-1,1)^n) \ar[r]^-{j \circ} \ar[d]
		&\embex^{\oriented}((-1,1)^n,M) \ar[d]^{\ev_0}\\
	\ast \ar[r]^-{j(0)}
		&M 
	}
	\eqnd
is a homotopy pullback square.
\end{lemma}

\begin{proof}
\nc{\germ}{\mathrm{Germ}}
This is a version of Remark~5.4.1.11~\cite{higher-algebra} for manifolds with corners. We provide the arguments for the reader's convenience.

Letting $C(t):= (-t,t)^n$ with choice of compatification $[-t,t]^n$, we have a filtered diagram of homotopy equivalences parametrized by $t \in (0,1]$
	\eqnn
	n \mapsto \sing\embex^{\oriented}(C(1/2^n),M).
	\eqnd
We let $\germ(M)$ denote the colimit simplicial set---informally, a vertex of $\germ(M)$ is the data of a point $x \in M$ and the germ of an embedding of a cube sending the origin of the cube to $x$. It is easy to see that $\germ(M)$ is a Kan complex. Restriction defines a homotopy equivalence 
	\eqn\label{eqn. emb disks is germ}
	\sing\embex^{\oriented}( (-1,1)^n ,M)
	\xrightarrow{\sim} \germ(M).
	\eqnd
Now choose an element $j \in \embex^{\oriented}((-1,1)^n,M)$. The commutative diagram	
	\eqn\label{eqn. germ pullback}
	\xymatrix{
	\germ( (-1,1)^n) \ar[r] \ar[d] & \germ(M) \ar[d] \\
	\sing((-1,1)^n) \ar[r]^{j} & \sing(M)
	}
	\eqnd
is obviously a pullback diagram, and the vertical arrows are Kan fibrations. (Here, we note the importance of having taken $\sing(M)$ rather than $\sing(\overline{M})$.) Thus~\eqref{eqn. germ pullback} is a homotopy pullback diagram.

By applying the equivalence~\eqref{eqn. emb disks is germ} we see that the diagram~\eqref{eqn. lemma pullback at point of M} is equivalent to the diagram~\eqref{eqn. germ pullback}, and we are finished.
\end{proof}

\begin{notation}[$\cC_M$ and $B_M$]
We define a topologically enriched category $\cC_M$ with two objects---$(-1,1)^n$ and $M$---and morphism spaces as follows:
		\eqnn
		\hom( (-1,1)^n,M) := \embex^{\oriented}((-1,1)^n,M),
		\qquad
		\hom( (-1,1)^n, (-1,1)^n) := \embex^{\oriented}( (-1,1)^n, (-1,1)^n),
		\eqnd
	and
		\eqnn
		\qquad
		\hom(M,M) = \ast,
		\qquad
		\hom( M, (-1,1)^n) = \emptyset.
		\eqnd
		We let $B\embex^{\oriented}$ denote the full subcategory spanned by $(-1,1)^n$. We then define $B_M$ to be the fiber product $\infty$-category
			\eqnn
			B_M:= N(B\embex^{\oriented}) \times_{N(\cC_M)} N(\cC_M)_{/M}.
			\eqnd
\end{notation}

\begin{remark}
Informally, this is the $\infty$-category whose objects are orientation-preserving, codimension-zero embeddings $j: (-1,1)^n \to M$ that smoothly extend to the chosen compactifications. A morphism from $j$ to $j'$ is the data of a map $f: (-1,1)^n \to (-1,1)^n$ in $\embex^{\oriented}((-1,1)^n, (-1,1)^n)$ together with an isotopy from $j'f$ to $j$---an isotopy through smooth embeddings that extend to the compactifications. 

$B_M$ is in fact a Kan complex.
\end{remark}

\begin{prop}\label{prop. BM contractible when M is cube}
$B_M$ is homotopy equivalent to $M$. In particular, if $M \cong (-1,1)^n$ with chosen compactification $\overline{M} \cong [-1,1]^n$, then $B_M$ is weakly contractible. 
\end{prop}

\begin{proof}
Our proof of this parallels Remark~5.4.5.2 of~\cite{higher-algebra}, with extra caution sprinkled in (because we are dealing with manifolds with corners). 

Let $\cC_M' \subset \cC_M$ denote the subcategory for which 
	\eqnn
	\hom_{\cC_{M'}}( (-1,1)^n,(-1,1)^n) \subset \hom_{\cC_M}( (-1,1)^n,(-1,1)^n)
	\eqnd
consists only of those embeddings that respect the origin (and with all other hom spaces unchanged). 

First, we see that inclusion $\cC_M' \into \cC_M$ is a (Dwyer-Kan) equivalence of Kan-complex-enriched categories. This follows from Remark~\ref{remark. 0 preserving is all f}.

Thus, defining $B_M' = N(\hom_{\cC_{M'}}( (-1,1)^n,(-1,1)^n)) \times_{N(\cC_{M'})} N(\cC_{M})_{/M}$, the induced map $B_M' \to B_M$ is an equivalence.

On the other hand, let $\cC_M''$ denote the simplicially enriched category with the same objects as $\cC_M'$, but where the only endomorphism of $(-1,1)^n$ is the identity, and where we declare $\hom_{\cC_M''}( (-1,1)^n, M) = M$ (with other hom spaces unchanged). Then ``evaluation at the origin'' defines a functor $\cC_M' \to \cC_M''$. Though this is certainly not an equivalence, we claim that the induced map
	\eqn\label{eqn. BM' to BM''}
	B_M' \to B_M''
	\eqnd 
is. Here, $B_M''$ is defined analogously to $B_M$ and $B_M'$:
	\eqnn
	B_M'' := \Delta^0 \times_{N(\cC_M'')} N(\cC_{M}'')_{/M}.
	\eqnd
To see \eqref{eqn. BM' to BM''} is an equivalence, choose an object $j: (-1,1)^n \to M$ of $B_M'$. This choice induces a diagram
	\eqnn
	\xymatrix{
	\hom_{\cC_{M}'}( (-1,1)^n,(-1,1)^n) \ar[r]^-{j \circ} \ar[d]
		&\hom_{\cC_{M'}}( (-1,1)^n,M) \ar[r] \ar[d]^{\ev_0}
		&B_{M}' \ar[d] \\
	\ast = \hom_{\cC_{M}''}( (-1,1)^n,(-1,1)^n) \ar[r]^-{j(0)}
		&M = \hom_{\cC_{M}''}( (-1,1)^n,M) \ar[r] 
		&B_{M}''
	}
	\eqnd
where each row is a homotopy fiber sequence.
The left square is a homotopy pullback square by Lemma~\ref{lemma. pullback at point of M}. 

We thus observe a diagram
	\eqnn
	\xymatrix{
		& \Omega_j B_M' \ar[rrr] \ar[dd] \ar[dl]_{\sim}
			&&& \ast \ar[dd]  \ar[dl] \\
		\Omega_{j(0)}M \ar[dd] \ar[rrr]
		&&& \ast \ar[dd] \\
		& \hom_{\cC_{M}'}( (-1,1)^n,(-1,1)^n) \ar[rrr]^{j \circ} \ar[dd] \ar[dl]
			&&&\hom_{\cC_{M'}}( (-1,1)^n,M) \ar[dd]  \ar[dl]_{\ev_0} \\
	\ast \ar[dd] \ar[rrr]
		&&& M \ar[dd] \\
		& \ast \ar[rrr] \ar[dl]
			&&& B_M' \ar[dl] \\
	\ast \ar[rrr] 
		&&&B_M''
	}
	\eqnd
where the topmost face is a homotopy pullback square; in particular, the map $\Omega_j B_M' \to \Omega_{j(0)} M$ is a homotopy equivalence (as indicated).

So it suffices to show that the map $B_M' \to M$ induces a bijection on connected components. This is obvious, and our proof is finished.
\end{proof}

\begin{notation}[$q$ and $\tilde B$]
Note that if $U$ is any object of $\diskex(\overline{M})$, we by definition have a smooth manifold with corners $\overline{U}$ defined by taking the closure of $U$. This results in a functor
		\eqn\label{eqn.B_U functor}
		\diskex(\overline{M}) \to \ssets,
		\qquad
		U \mapsto B_U.
		\eqnd
We let $q: \tilde B \to \diskex(\overline{M})$ denote the coCartesian fibration classifying this functor. 	(The map $q$ can be concretely modeled applying the relative nerve construction of~\cite[Section 3.2.5]{htt} to the functor~\eqref{eqn.B_U functor}.) So for example, a vertex of $B_U$ is a pair $(U, j: (-1,1)^n \to U)$ where $j$ extends to a smooth embedding $[-1,1]^n \to \overline{U}$.

Now let $p$ be the map from~\eqref{eqn. p BM}. The inclusion $p^{-1}(U) \to B_U$ (from the strictly commuting diagrams to the isotopy-commuting diagrams) is natural in $U$, so defines a map of coCartesian fibrations:
		\eqnn
		\xymatrix{
		B_M^\delta \ar[rr] \ar[dr]_p && \widetilde{B} \ar[dl]^q\\
		& \diskex(\overline{M})
		}
		\eqnd 
\end{notation}
	
\begin{lemma}\label{lemma. p q equivalences}
$p$ and $q$ are weak homotopy equivalences of simplicial sets. In particular, the map $B_M^\delta \to \tilde B$ is a weak homotopy equivalence of simplicial sets.
\end{lemma}

\begin{proof}
That $p$ is a weak equivalence is obvious because $p$ is an equivalence of categories. 
$q$ is a coCartesian fibration, so to show it is a weak equivalence, it suffices to show that $q$ has contractible fibers. This follows from Proposition~\ref{prop. BM contractible when M is cube}. 
\end{proof}
	
\begin{notation}[$\psi'$]
	There is also the natural functor 
		\eqn\label{eqn. psi'}
		\psi': \widetilde{B} \to B_M
		\eqnd
	obtained by post-composing any $j$ with the inclusion $U \subset M$. 
\end{notation}

\begin{lemma}\label{lemma. psi' equivalence}
$\psi'$ is a weak homotopy equivalence.
\end{lemma}

\begin{proof}
One can copy-paste the (second half of) the proof of Lemma~5.4.5.10 in~\cite{higher-algebra}. The only important note is that we have deliberately taken our $U$ to be open so that the collection of $U$ forms an open cover of $M$ (which is a smooth manifold with no boundary or corners).
\end{proof}

\begin{proof}[Proof of Lemma~\ref{lemma. main monoidal lemma}.]
By Lemmas~\ref{lemma. p q equivalences} and~\ref{lemma. psi' equivalence} we obtain a composition of weak homotopy equivalences $B_M^\delta \to \widetilde{B} \to B_M$, and in particular, a homotopy equivalence of Kan complexes 
	\eqnn
	|B_M^\delta| \to B_M.
	\eqnd
Choosing an inverse to this homotopy equivalence, consider the composition
	\eqnn
	B_M \xrightarrow{\sim} |B_M^{\delta}| \xrightarrow{\eqref{remark. BM delta to slice}} (\mfldcmpct_{/\overline{M}})[(\eqs_{/\overline{M}})^{-1}] \xrightarrow{\eqref{eqn. slice map localization}} \mfldcmpct[\eqs^{-1}]_{/\overline{M}}.
	\eqnd
When $M = [-1,1]^3$, in $B_M$ there is an edge given by an isotopy-commuting diagram 
	\eqnn
	\xymatrix{
	[-1,1]^3 \ar[dr]^{\tau} \ar[rr]^{\id} && [-1,1]^3\ar[dl]_{\id} \\
	& [-1,1]^3
	}
	\eqnd
because $\tau$ is orientation-preserving.
Its image in the slice $\mfldcmpct[\eqs^{-1}]_{/[-1,1]^3}$ exhibits a homotopy in $\mfldcmpct[\eqs^{-1}]$ from $\tau$ to $\id_{[-1,1]^3}$.

Thus, the natural functor of $\infty$-categories $\mfldcmpct[\eqs^{-1}] \to \lioustr[\eqs^{-1}]$---induced by \eqref{eqn. T* functor}---shows that $T^*\tau$ is homotopic to $\id_{T^*[-1,1]^3}$ in $\lioustr[\eqs^{-1}]$. 
\end{proof}

\clearpage

\section{Equivalent localizations}
\label{section. equivalent localizations}
There are four natural candidates for morphisms to invert (i.e., to declare to be equivalences) between Liouville sectors:
\enum
\item (Bordered) deformation equivalences (Definition~\ref{defn. deformation equivalence}),
\item Sectorial equivalences (Definition~\ref{defn. sectorial equivalence}), 
\item Trivial inclusions (Definition~\ref{defn. trivial inclusion} below),
\item Movie inclusions (Definition~\ref{defn. movie inclusion} below). 
\enumd

Here we prove Theorem~\ref{thm. three localizations of lioustr}, which shows that the localizations of $\lioustr$ with respect to the first two are equivalent, and that localization with respect to all four are equivalent upon stabilization.

\begin{defn}[Trivial inclusions, following Section~2.4 of~\cite{gps}]\label{defn. trivial inclusion}
A {\em trivial inclusion} is a sectorial embedding $f: M \to N$ for which there exists:
	\enum
	\item A smooth family of 1-forms $\{\lambda_t\}_{t\in[0,1]}$ on $M$ rendering each $(M,\lambda_t)$ a Liouville sector, and
	\item A smooth isotopy $\{f_t: M \to N\}_{t \in [0,1]}$
	\enumd
	such that
	\enum[(i)]
	\item For all $t$, $f_t: (M,\lambda_t) \to N$ is a sectorial embedding,
	\item $f_0 = f$ and $\lambda_0 = \lambda_M$, and
	\item $f_1(M) = N$. (In light of the other conditions, this condition is equivalent to demanding that $f_1$ is an isomorphism of Liouville sectors between $(M,\lambda_1)$ and $N$.)
	\enumd
	We emphasize here that the $\lambda_t$ for differing $t$ may not be related to each other by compactly supported deformations. In particular, the $f_t$ do not need to form an isotopy through sectorial embeddings for a fixed Liouville structure on $M$.
\end{defn}

\begin{defn}[Movie inclusions]\label{defn. movie inclusion}
Let $S = [0,1]$, fix some $1>\epsilon>0$ and declare $S_0 = [0,\epsilon]$ or $S_0 = [1-\epsilon, 1]$.
Finally, fix a collared, bordered, $S$-parametrized exact family of Liouville structures on $M$, and let $M \times T^*S$ be the associated movie (Construction~\ref{construction. movie}) and assume that the $S$-parametrized exact family is constant in $S_0 \subset S$. There is then a unique Liouville structure on $M$ for which the inclusion $M \tensor T^*S_0 \into M \times T^*S$ is a strict sectorial embedding.

For any choice of $\epsilon$, either choice of $S_0$, and any choice of exact family as above, we call the strict sectorial embedding $M \tensor T^*S_0 \into M \times T^*S$ a {\em movie inclusion of intervals}. 
\end{defn}

\begin{remark}
A prototypical trivial cofibration of topological spaces is an inclusion $X \times \{0\} \into X \times [0,1]$. We view movie inclusions as the sectorial analogue of such ``endpoint inclusions.''
\end{remark}

\begin{example}
Fix $0 < a_i < b_i < 1$ for $i = 1, \ldots, n$. 
By allowing ourselves to vary $M$, the codimension zero inclusion
	\eqnn
	T^*([a_1,b_1] \times \ldots \times [a_n,b_n])
	\into
	T^*[0,1]^n
	\eqnd
may be written as a composition of movie inclusions and Liouville isomorphisms. For example, when $n=1$, the inclusion $[a,b] \into [0,1]$ may be factored as such a composition by doing the following. Choose 
	\begin{itemize}
	\item Any diffeomorphism $\phi: \RR \to \RR$ taking $[a,b]$ to the interval $[a,1]$ (and fixing $a$). 
	\item Any diffeomorphism $\psi: [\phi^{-1}(0),b] \to [0,b]$ which restricts to the identity along $[a,b]$. 
	\end{itemize}
Consider the diagram
	\eqnn
	\xymatrix{
	\RR \ar[r]^\phi_{\cong}
		& \RR \ar@{=}[r]
		& \RR \ar[r]^{\phi^{-1}}_{\cong}
		& \RR \\
	[a,b] \ar[u]_{\subset}  \ar[r]_-{\cong}
		& [a, 1] \ar[r]^{\iota} \ar[u]^{\subset}
		& [0,1] \ar[r]_-{\cong} \ar[u]^{\subset}
		& [\phi^{-1}(0),b]  \ar[u]^{\subset} \ar[r]^-{\psi}_-{\cong}
		& [0,b] \ar[r]^{\iota}
		& [0,1]
	}
	\eqnd
where each $\iota$ is a standard inclusion.
Then the bottom horizontal arrows, upon passage to cotangent bundles, exhibit the standard inclusion $T^*[a,b] \into T^*[0,1]$ as a composition of movie inclusions and Liouvile isomorphisms.

For $n=2$, note that $T^*$ of the inclusion $[a_1,1] \times [a_2,1] \into [0,1]^2$ can be written as a composition of two movie inclusions, where we change which of the two components of the square are treated as the $M$ component (in the notation of Definition~\ref{defn. movie inclusion}). This is what we mean by ``varying $M$.''

Now suppose $j: [0,1]^n \into [0,1]^n$ is  any smooth, codimension zero embedding (that need not respect the stratification of the cubes). Then we may find a commutative diagram of embeddings as follows:
    \begin{equation}\nonumber
    \xymatrix{
    [0,1]^n \ar[r]^j \ar[d]_{h'}
    & [0,1]^n \ar[r]^{h}
    & [0,1]^n \\
    [0,1]^n \ar@{-->}[urr]^{j'}_{\cong}
    }
    \end{equation}
Here, the arrows labeled $h$ and $h'$ are compositions of the form $[0,1]^n \cong [\epsilon,1-\epsilon]^n \into [0,1]^n$ -- i.e., diffeomorphisms followed by a standard inclusion -- so that the images of $h$ and $h'$ are each bounded away from the boundary of $[0,1]^n$. Then there exists a {\em diffeomorphism} $j' : [0,1]^n \to [0,1]^n$ extending the top horizontal composition in the diagram. That $hj = j' h'$ shows that, upon inverting $T^*h$ and $T^*h'$ (in particular, upon inverting movie inclusions) the arbitrary embedding $j$ may be obtained up to homotopy as a composition of $T^*$ of diffeomorphisms and movie inclusions.

In particular, inverting movie inclusions renders $T^*$ of isotopy equivalences of cubes into homotopy-invertible morphisms. (This will also be a consequence of Theorem~\ref{thm. three localizations of lioustr} below.)
\end{example}

\begin{prop}\label{prop. equivalent oo categorical localizations}
Let $\cC$ be an $\infty$-category and let $S \subset T$ be two collections of morphisms in $\cC$. Assume every morphism of $T$ becomes invertible in the localization $\cC[S^{-1}]$ (more precisely, that the image of $T$ under the map $\cC \to \cC[S^{-1}]$ is contained in the collection of equivalences of $\cC[S^{-1}]$). Then the natural map
	\eqnn
	\cC[S^{-1}] \to
	\cC[T^{-1}]
	\eqnd
is an equivalence of $\infty$-categories.
\end{prop}

\begin{proof}
We first set some notation: Let $i_S$ and $i_T$ denote the localization maps from $\cC$ to $\cC[S^{-1}]$ and $\cC[T^{-1}]$, respectively. The natural map $h: \cC[S^{-1}] \to \cC[T^{-1}]$, by the universal property of localizations, induces a homotopy $i_T \sim h \circ i_S$.

By hypothesis, the universal property of localizations guarantees 
\enum
\item a map $j: \cC[T^{-1}] \to \cC[S^{-1}]$,
\item a homotopy $j \circ i_T \sim i_S$, and hence (by pulling back along $i_S$)
\item a homotopy $j \circ h \sim \id_{\cC[S^{-1}]}$.
\enumd
Similar reasoning shows that $h \circ j \sim \id_{\cC[T^{-1}]}$.
\end{proof}

\begin{prop}\label{prop. equivalence implications}
Let $f: X \to Y$ be a (not necessarily strict) sectorial embedding.
\enum[(a)]
\item\label{item. bordered eq is sectorial eq} $f$ is a sectorial equivalence if and only if $f$ lifts to (i.e., may be equipped with data realizing $f$ as being) a bordered deformation equivalence.
\item\label{item. trivial inclusion is bordered eq} If $f$ is a trivial inclusion, then $f$ is a bordered deformation equivalence (hence, by~\eqref{item. bordered eq is sectorial eq}, a sectorial equivalence).
\item\label{item. movie inclusion is bordered eq} If $f$ is a movie inclusion, then $f$ is a bordered deformation equivalence (hence, by~\eqref{item. bordered eq is sectorial eq}, a sectorial equivalence).
\enumd
\end{prop}

We note that the proof of this proposition will not rely on any of the results of Section~\ref{section. proof of localization}---this is logically necessary, as Section~\ref{section. proof of localization} relies on Remark~\ref{remark. max phi in Liou is equivalence} below, which in turn relies on the proof methods of the present proposition.

\begin{proof}
\eqref{item. bordered eq is sectorial eq}: It is obvious that any sectorial equivalence is a bordered deformation equivalence (by choosing, for example, compactly supported deformations). So now suppose $f:M \to N$ is a bordered deformation equivalence. This means the image of $f$ under the composition $\lioustr \to \lioudelta \to \lioudeltadef$ is an equivalence in $\lioudeltadef$. Moreover, the inclusion $\lioudelta \to \lioudeltadef$  is an equivalence of $\infty$-categories by Theorem~\ref{theorem. lioudelta is lioudeltadef}, so the existence of inverse data exhibiting $f$ as an equivalence in $\lioudeltadef$ implies the existence of inverse data exhibiting $f$ as an equivalence in $\lioudelta$; but by applying Proposition~\ref{prop. 2-simplex is an isotopy}, the 2-simplices exhibiting $f$ as an equivalence in $\lioudelta$ supply an inverse map $g: N \to M$ and isotopies of sectorial embeddings demonstrating that $f$ is a sectorial equivalence.

\eqref{item. trivial inclusion is bordered eq}: Given the data exhibiting $f: M \to N$ as a trivial inclusion (Definition~\ref{defn. trivial inclusion}), one may construct data exhibiting $f$ as a bordered deformation equivalence. In the notation of Definition~\ref{defn. deformation equivalence}, one may for example choose $g = f_1^{-1}$, and bordered deformations extending $(f_t)_*(\lambda_t^M)$.

\eqref{item. movie inclusion is bordered eq}: We use the notation from Definition~\ref{defn. movie inclusion}.  Let $\lambda_0$ denote the Liouville structure on $M$ associated to $S_0$, and choose an isotopy $\{j_t: S_0 \to [0,1]\}_{0 \leq t \leq 1}$ where $j_0$ is the inclusion and $j_1$ is a diffeomorphism. From hereon, by the sector $M$, we mean the sector $(M,\lambda_0)$. Then $j_t$ induces an isotopy, through exact sectorial embeddings, from $\id_M \times T^*j_0$ to $\id_M \times T^*j_1$. There is an obvious inverse (up to isotopy) from $M \tensor T^*[0,1]$ to $M \tensor T^*S_0$, so $\id_M \times T^*j_0$ is a bordered deformation equivalence from $M \tensor T^*S_0$ to $M \tensor T^*[0,1]$. (In fact, the deformations of Liouville structures may be chosen to be constant.)

Now note that, by definition of movie inclusion, there is a bordered deformation $\{\lambda^{M \times T^*[0,1]}_s\}_{s \in [0,1]}$ of Liouville forms from $\lambda_0 + {\bf p}d{\bf q}$ to the movie Liouville form~\eqref{eqn. movie form} on $M \times T^*[0,1]$. Thus, the ``identity map'' $M \tensor T^*[0,1] \to M \times T^*[0,1]$ (from the stabilization of $M$ to the movie of $M$), coupled with $\{\lambda^{M \times T^*[0,1]}_s\}_{s \in [0,1]}$, is a non-strict bordered deformation equivalence. The inverse is given by the ``identity map'' along with the reverse deformation $\{\lambda^{M \times T^*[0,1]}_{1-s}\}_{s \in [0,1]}$, where obviously the $s$-concatenation of the reverse and forward deformations may be deformed to the constant $s$-parametrized deformation.

Now we see that the strict map $\id_M \times T^*j_0$ from $M \tensor T^*S_0$ to the movie $M \times T^*[0,1]$ is a bordered deformation equivalence---one simply composes and concatenates the inverse data from the previous two paragraphs to exhibit the bordered deformation inverse to  $\id_M \times T^*j_0$.
\end{proof}

\begin{remark}
\label{remark. max phi in Liou is equivalence}
When $\max A' = \max A$, we claim that the map $\phi_{A' \subset A}$ is an equivalence of Liouville sectors. (This was used crucially in Construction~\ref{construction. alpha}.) To see why this claim is true, note we have a commutative diagram
	\eqnn
	\xymatrix{
	& M_{\max A'} \tensor T^*\Delta^{A'} \tensor T^*[-1,0]^{A \setminus A'} \ar[d]^{\sim}
	\\
	M_{{\max A'}} \tensor T^*[-1,0]^{A} \ar[d]^{\cong}
	\ar[r] \ar[ur]^{\sim}
		& M_{A'} \times T^*\Delta^{A'} \tensor T^*[-1,0]^{A \setminus A'}\ar[d]^{\phi_{A' \subset A}} 
	\\
	M_{\max A} \tensor T^*[-1,0]^A 
	\ar[r] \ar[dr]_{\sim}
	& M_A \times T^*\Delta^A \\
	& M_A \tensor T^*\Delta^A  \ar[u]^{\sim}.
	}
	\eqnd
The claim follows if we can prove that both (unlabeled) horizontal arrows are sectorial equivalences. This follows from more or less the same method as in Proposition~\ref{prop. equivalence implications}~\eqref{item. movie inclusion is bordered eq}. Namely, the inclusions of $[-1,0]^A$ into $\Delta^{A'} \times [-1,0]^{A \setminus A'}$ and $\Delta^A$ are both isotopy equivalences of smooth manifolds with corners. (Warning: this does not mean the inclusions are isotopic to diffeomorphisms; see Definition~\ref{defn. isotopy equivalence}.) Thus, the horizontal arrows are sectorial equivalences if the codomains are given the constant movie structure. One can then apply a bordered deformation from the constant movie Liouville structure to the movie structures of the original codomains of the horizontal arrows. As in the proof of Proposition~\ref{prop. equivalence implications}~\eqref{item. movie inclusion is bordered eq},   this exhibits the horizontal arrows as bordered deformation equivalences, and hence sectorial equivalences by Proposition~\ref{prop. equivalence implications}~\eqref{item. bordered eq is sectorial eq}.
\end{remark}

\begin{theorem}\label{thm. three localizations of lioustr}
In $\lioustr$, let $\eqs_{\triv}$ be the collection of strict trivial inclusions, $\eqs$ the collection of strict sectorial equivalences,  $\eqs_{\movie}$ the collection of (strict) movie inclusions, and $\eqs_{\defliou}$ the collection of strict sectorial embeddings that may be promoted to bordered deformation equivalences. Then the natural maps
	\enum[(a)]
	\item\label{item. def and eqs localizations are the same}
	$	\lioustr[\eqs^{-1}] \to
	\lioustr[(\eqs_{\defliou})^{-1}] $,
	\item\label{item. stabilized localizations are the same}
	$	\lioustr^{\dd}[(\eqs^{\dd})^{-1}] \to
	\lioustr^{\dd}[(\eqs_{\defliou}^{\dd})^{-1}] $,
	\item\label{item. triv and eqs localizations are the same stabilized}
	$	\lioustr^{\dd}[(\eqs_{\triv}^{\dd})^{-1}] \to \lioustr^{\dd}[(\eqs_{\defliou}^{\dd})^{-1}] $, and
	\item\label{item. mov and eqs localizations are the same stabilized}
	$	\lioustr^{\dd}[(\eqs_{\movie}^{\dd})^{-1}] \to \lioustr^{\dd}[(\eqs_{\defliou}^{\dd})^{-1}] $
	\enumd
are equivalences of $\infty$-categories. (Note that the final three equivalences are stabilized, while the first is not. Also note that the above natural maps are guaranteed by Proposition~\ref{prop. equivalence implications}.) 

In particular, whether one localizes $\lioustrstab$ by trivial inclusions, by sectorial equivalences, by bordered deformation equivalences, or by movie inclusions, one obtains an $\infty$-category equivalent to $\lioulocal$.
\end{theorem}

\begin{proof}

\eqref{item. def and eqs localizations are the same} This follows from Proposition~\ref{prop. equivalence implications}~\eqref{item. bordered eq is sectorial eq} and Proposition~\ref{prop. equivalent oo categorical localizations}.

\eqref{item. stabilized localizations are the same} This follows from the above by applying Lemma~\ref{lemma. increasing union of localizations is localization}. (See also Corollary~\ref{cor. stab loc is loc stab}.)

The rest of this proof is devoted to establishing \eqref{item. triv and eqs localizations are the same stabilized} and 
\eqref{item. mov and eqs localizations are the same stabilized}.

By Proposition~\ref{prop. equivalent oo categorical localizations} and the equivalence \eqref{item. stabilized localizations are the same},  we are reduced to showing that stabilized sectorial equivalences become homotopy invertible in $\cD$, where $\cD$ is either of $\lioustr^{\dd}[(\eqs_{\triv}^{\dd})^{-1}]$ or $\lioustr^{\dd}[(\eqs_{\movie}^{\dd})^{-1}]$.

So fix a strict sectorial equivalence $f: M \to N$. By definition, there is a (not necessarily strict) sectorial embedding $g: N \to M$ and isotopies $gf \sim \id_M, fg \sim \id_N$ through (not necessarily strict) sectorial embeddings. 
Endow $N$ with a form $\lambda^N_0$ for which $g$ is a strict map to $M$. Then, the isotopy $fg \sim \id_N$ and Proposition~\ref{prop. movies of maps} guarantee a strict map from $(N,\lambda^N_0) \tensor T^*[0,1]$ to a movie $N \times T^*[0,1]$ realizing this isotopy, fitting into the following commutative diagram in $\lioustr$:	
	\eqnn
	\xymatrix{
	(N,\lambda^N_0) \tensor T^*[0,\epsilon]
		\ar[d]^{\simeq}_{\id_N \times T^*\iota_0}
		\ar[drrrr]_{\simeq}^{\id_N \times T^*\iota_0} \\
	(N,\lambda^N_0) \tensor T^*[0,1] \ar[rrrr]_{\text{Prop~\ref{prop. movies of maps}}}
		&&&& N \times T^*[0,1] \\
	(N,\lambda^N_0) \tensor T^*[1-\epsilon,1] \ar[u]_{\simeq}^{\id_N \times T^*\iota_1}
		\ar[rr]^-{g \times \id}
		&& M \tensor T^*[1-\epsilon,1] \ar[rr]^-{f}
		&& N \tensor  T^*[1-\epsilon,1] \ar[u]^{\simeq}_{\id_N \times T^*\iota_1}
	}.
	\eqnd
	
Here, $\iota_0: [0,\epsilon] \to [0,1]$ and $\iota_1: [1-\epsilon,1] \to [0,1]$ are the standard inclusions. Note also we are identifying $f$ with its stabilization. The arrows labeled by $\simeq$ are simultaneously movie inclusions and trivial inclusions, so are homotopy invertible in $\cD$. 
It follows that in $\cD$, $f$ admits a right homotopy inverse.

Likewise, one has a commutative diagram as follows in $\lioustr$:
	\eqnn
	\xymatrix{
	M \tensor T^*[0,\epsilon]
		\ar[d]^{\simeq}_{\id_M \times T^*\iota_0}
		\ar[drrrr]_{\simeq}^{\id_M \times T^*\iota_0} \\
	M \tensor T^*[0,1] \ar[rrrr]_{\text{Prop~\ref{prop. movies of maps}}}
		&&&& M \times T^*[0,1] \\
	M \tensor T^*[1-\epsilon,1] \ar[u]_{\simeq}^{\id_M \times T^*\iota_1}
		\ar[rr]^-{ f}
		&& N \tensor T^*[1-\epsilon,1] \ar[rr]^-{g \times \id}
		&& (M,\lambda^M_1) \tensor  T^*[1-\epsilon,1] \ar[u]^{\simeq}_{\id_M \times T^*\iota_1}
	}.
	\eqnd
Here, we have chosen a form $\lambda^M_1$ to render $g$ strict, along with a movie structure on $M \times T^*[0,1]$ to render all maps strict and to invoke Proposition~\ref{prop. movies of maps}. The arrows labeled by $\simeq$ again become equivalences in $\cD$, and it follows that
$f$ admits a left homotopy inverse.

$f$ is thus an equivalence in  $\cD$.
\end{proof}

\subsection{The wrapped Fukaya category is unchanged by equivalences}

We have the following consequence of Theorem~\ref{thm. three localizations of lioustr}:

\begin{prop}\label{prop. wrapped category invariant under sectorial equivalences}
Fix Liouville sectors $M$ and $N$.
\enum[(a)]
\item\label{item. wrapped fukaya preserves sectorial equivalences} Let $f: M \to N$ be a sectorial equivalence. Then the induced map on the stabilized wrapped Fukaya categories $f_*: \cW^{\dd}(M) \to \cW^{\dd}(N)$ is an equivalence of $A_\infty$-categories.

\item\label{item. wrapped fukaya preserves bordered equivalences} Let $f: M \to N$ be a bordered deformation equivalence. Then the induced map on the stabilized wrapped Fukaya categories $f_*: \cW^{\dd}(M) \to \cW^{\dd}(N)$  is an equivalence of $A_\infty$-categories.

\item\label{item. wrapped fukaya preserves deformation equivalences of weinsteins} Now suppose $M$ and $N$ are Weinstein sectors, and that $f: M \to N$ is a bordered deformation equivalence. Then the induced map on wrapped Fukaya categories is an equivalence.
\enumd
\end{prop}

\begin{proof}
The work~\cite{gps} exhibited the wrapped Fukaya category functor
	\eqn\label{eqn. wrapped functor GPS}
	\lioustr \to \Ainftycatt
	\eqnd
and showed that trivial inclusions induce equivalences of $A_\infty$-categories. Thus, letting $\Ainftycat$ be the localization of $\Ainftycatt$ along quasi-equivalences, one has an induced functor $\cW$ as below:
	\eqnn
	\lioustr \to \lioustr[\eqs_{\triv}^{-1}] \xrightarrow{\cW} \Ainftycat.
	\eqnd 
There are natural stabilization transformations $\cW(M) \to \cW(M \times T^*[0,1])$, so we have an induced functor
	\eqnn
	\lioustr^{\dd} \to \Ainftycat,
	\qquad
	M \mapsto \cW^{\dd}(M) = \hocolim_k \cW(M \times T^* [0,1]^k)
	\eqnd
factoring through $\lioustr^{\dd}[(\eqs_{\triv}^{\dd})^{-1}] \simeq \lioustr[\eqs_{\triv}^{-1}]^{\dd}.$
Thus~\eqref{item. wrapped fukaya preserves sectorial equivalences} and~\eqref{item. wrapped fukaya preserves bordered equivalences} follow from Theorem~\ref{thm. three localizations of lioustr}~\eqref{item. stabilized localizations are the same} and \eqref{item. triv and eqs localizations are the same stabilized}.

\eqref{item. wrapped fukaya preserves deformation equivalences of weinsteins} This follows from (b) by noting that the natural maps $\cW(M) \to \cW(M \times T^*[0,1])$ are equivalences for Weinstein sectors, and hence so are the maps $\cW(M) \to \cW^{\dd}(M)$.
\end{proof}

\begin{remark}
Note that the proof of Proposition~\ref{prop. wrapped category invariant under sectorial equivalences} required no mention of holomorphic disks -- indeed, the only requisite arguments concerning holomorphic disks are fully contained in exhibiting the functor~\eqref{eqn. wrapped functor GPS}. The universal property of localizations, together with the geometry of {\em sectors} (without reference to branes or holomorphic disks) established in Theorem~\ref{thm. three localizations of lioustr}, allowed us to conclude the proposition.  
While it is not too onerous to prove Proposition~\ref{prop. wrapped category invariant under sectorial equivalences} by directly using Floer-theoretic methods, any such proof would first pass through the geometry used in proving Theorem~\ref{thm. three localizations of lioustr}, then require additional arguments about branes and holomorphic disks (by riffing on arguments already used in~\cite{gps}'s Lemma~3.33, for example).
\end{remark}

\clearpage
\section{Incorporating tangential structures}\label{section. tangential}
To define a $\ZZ$-graded (as opposed to two-periodic) wrapped Fukaya category, one must trivialize a natural map from a Liouville sector to $BU(1)$. One must also trivialize certain tangentially classified obstructions to render the coefficient ring $\ZZ$. Likewise, the construction of Nadler-Shende's microlocal category requires a trivialization of the stable normal Gauss map. So it is natural and necessary to consider not only Liouville sectors, but Liouville sectors equipped with particular tangential structures (such as trivializations).

Here we give a description of the $\infty$-categories of Liouville sectors equipped with  tangential structures, paying particular attention to the symmetric monoidal structures we can endow upon them. We omit some details that rely on classical differential-topological constructions (such as equipping a bundle with a locally finite collection of local trivializations, and observing that the space of such choices is contractible upon allowing for common refinements).

\begin{notation}
Throughout this section, $\ast$ refers to the one-point space.
\end{notation}

\subsection{Preliminaries: The \texorpdfstring{$\infty$}{infinity}-categories of (stable) vector bundles}
First, suppose that $\cC$ is a symmetric monoidal category, and fix an object $B \in \cC$. It is not the case that the slice category $\cC_{/B}$ (an object of which is the data of an object $X$ equipped with a map $X \to B$) has an induced symmetric monoidal structure. 

\begin{recollection}\label{recollection. slice categories are symmetric monoidal}
However, $\cC_{/B}$ can be rendered symmetric monoidal each time $B$ is given the structure of a commutative algebra object. Informally, the monoidal product $(X \to B) \tensor (Y \to B)$ is encoded by the composition $X \tensor Y \to B \tensor B \to B$. This holds true when $\cC$ is a symmetric monoidal $\infty$-category and $B$ is given an $E_\infty$-algebra structure. (See Notation~2.2.2.3 and Theorem~2.2.2.4 of~\cite{higher-algebra}; in the notation of loc. cit., we are concerned with the case $S = N(\mathcal{F}\!\operatorname{in}_\ast)$.)
\end{recollection}

\begin{example}
We let $\Top$ denote the $\infty$-category of topological spaces. We let $BO$ denote the usual classifying space of stable vector bundles, which we model concretely as the colimit of infinite Grassmannians:
	\eqnn
	BO := \colim \left( BO(1) \to BO(2) \to \ldots\right),
	\qquad
	BO(n) = \{ \text{$n$-planes in $\RR^\infty$}\}.
	\eqnd
We let $\Top_{/BO}$ denote the slice $\infty$-category. Concretely, an object of $\Top_{/BO}$ is the data of a topological space $X$ equipped with a map $f: X \to BO$. A 1-simplex in $\Top_{/BO}$ is the data of a homotopy-commuting diagram
	\eqnn
	\xymatrix{
	X\ar[dr]_f \ar[rr]^g && X' \ar[dl]^{f'} \\
	& BO
	}
	\eqnd
together with the data of the homotopy $f' \circ g \sim f$. Endowing $BO$ with the standard infinite-loop-space structure (which models the direct sum of stable vector bundles), one obtains a symmetric monoidal structure on $\Top_{/BO}$.
\end{example}

\begin{remark}
Suppose one has a model for the infinite-loop space structure on a space $B$ using the language of operads \`a la May and Boardman-Vogt. For example, one might use the linear isometries operad to model the infinite-loop space structure on $BO$ or $BU$. (See for example Chapter I of~\cite{may-Eoo-ring-spaces}.) Then one can endow the usual topologically enriched category $Top_{/B}$ with the structure of an $E_\infty$-algebra in topologically enriched categories. Taking the homotopy coherent nerve\footnote{By a version of Proposition~2.1.1.27 of~\cite{htt}, any symmetric monoidal, topologically enriched category yields a symmetric monoidal $\infty$-category by taking the homotopy coherent nerve.}, one obtains a symmetric monoidal $\infty$-category $N(Top_{/B})^\tensor$---concretely, the $E_\infty$-algebra structure, via the $\infty$-categorical Grothendieck construction (otherwise known as the unstraightening), models an $\infty$-category equipped with a coCartesian fibration over the (nerve of the multicategory of the particular model of the) $E_\infty$-operad, which in turn is fibered over $N(\mathcal{F}\!\operatorname{in}_*)$. One can then exhibit an equivalence of symmetric monoidal $\infty$-categories
	\eqnn
	N(Top_{/B})^\tensor \simeq \Top_{/B}^\tensor
	\eqnd
where the righthand side is the symmetric monoidal $\infty$-category from Recollection~\ref{recollection. slice categories are symmetric monoidal}. 

We only outline the proof: Both sides are most naturally modeled as having a coCartesian fibration not to $N(\mathcal{F}\!\operatorname{in}_*)$, but to the $\infty$-operad modeling (for example) the linear isometries operad. By definition of the nerve construction, one observes that every simplex of $N(Top_{/B})^\tensor$ defines a simplex of $\Top_{/B}^\tensor$ by the characterization of simplices of $\Top_{/B}^\tensor$ given in Notation~2.2.2.3 of~\cite{higher-algebra}. This map of simplicial sets is obviously essentially surjective and respects coCartesian edges; it remains to show that the map is an equivalence of $\infty$-categories, which follows from a straightforward computation of mapping spaces. (To compute mapping spaces on the lefthand side, one may invoke that $Top_{/B}$ arises as the nerve of a simplicial model category, then use the usual formula for mapping spaces in a multicategory associated to a symmetric monoidal category. The righthand side is computed using the usual fibration sequence for computing morphism spaces in slice categories\footnote{For example, Proposition~5.5.5.12 of~\cite{htt}.}, then the formula for computing mapping spaces in an $\infty$-operad\footnote{For example, (2) of Definition~2.1.1.10 of~\cite{higher-algebra}.}.)
\end{remark}

\nc{\vecbun}{\mathcal{V}\!\operatorname{ecBun}}
\nc{\vecbuntop}{\operatorname{VecBun}}

Let $\vecbuntop_\RR$ denote the topologically enriched category whose objects are vectors bundles $(E \to X)$, and whose morphisms are pairs $(f: X \to X', \tilde f: E \to E')$ where $f$ is continuous and $\tilde f$ is a bundle map (meaning it induces isomorphisms along fibers) over $f$. As an example, the object $(\RR^n \to \ast)$ given by the trivial rank $n$ bundle over a point has endomorphism space homeomorphic to the Lie group $GL(n)$. For simplicity, we will assume that every $X$ is (homotopy equivalent to) a CW complex---this will later allow us to use paracompactness arguments. 

We endow $\vecbuntop_\RR$ with the symmetric monoidal structure induced by Whitney sum and direct product of spaces; so the product of $(E \to X)$ and $(E' \to X')$ is given by the bundle $p_X^* E \oplus p_{X'}^* E'$ over $X \times X'$.  We let
	\eqnn
	(\vecbun_\RR)^\oplus
	\eqnd
denote the symmetric monoidal $\infty$-category obtained by taking the homotopy coherent nerve. (We note the change in font from $\vecbuntop$ to $\vecbun$.) There is a stabilizing endofunctor
	\eqnn
	\dd: \vecbun_\RR \to \vecbun_\RR,
	\qquad
	(E \to X) \mapsto ((E \oplus \underline{\RR}) \to X)
	\eqnd
which is modeled by taking a product with the object $(\RR \to \ast)$. (Here, $\underline{\RR}$ denotes the trivial vector bundle.) Obviously, the cyclic permutation action on $(\RR^{\oplus 3} \to \ast)$---which may be identified with an orientation-preserving element of $O(3)$---is homotopic to the identity. Thus the colimit of the diagram of $\infty$-categories
	\eqnn
	\vecbun_\RR \xrightarrow{\dd} \vecbun_\RR \xrightarrow{\dd} \ldots
	\eqnd
inherits a symmetric monoidal structure by Proposition~\ref{prop. symmetric T}. 

\begin{notation}
We let
	\eqnn
	(\vecbun_\RR^{\dd})^\oplus
	\eqnd
denote the resulting symmetric monoidal $\infty$-category. When considering complex vector bundles, we likewise let
	\eqnn
	(\vecbun_\CC^\dd)^\oplus
	\eqnd
denote the symmetric monoidal $\infty$-category of   CW complexes equipped with stable complex vector bundles. 
\end{notation}

\begin{remark}
The functor of taking a product with $(\RR \to \ast)$ is equivalent to taking a product with $(\underline{\RR} \to T^*[0,1])$ because $T^*[0,1] \to \ast$ is an equivalence in the $\infty$-category $\Top$. Thus, $\vecbun^{\dd}$ admits a model by stabilizing with respect to direct product with $(\underline{\RR} \to T^*[0,1])$. This model is more on-the-nose compatible with the stabilization of $\lioustr$. 
\end{remark}

\begin{remark}
Concretely, an object of 
$\vecbun_\RR^\dd$
is the data of a pair $\left( (E \to X), k \right)$---where $(E \to X)$ is a real vector bundle over a CW base, and $k\geq 0$ is some integer---up to an equivalence relation. The relation identifies $\left( (E \to X), k \right)$ with the object $\left( (E\oplus \underline{\RR} \to X), k+1 \right)$, where $\underline{\RR}$ is the trivial rank one real vector bundle over $X$.

In particular, there are objects $((\underline{\RR} \to \ast), k)$ which one may think of as a stable rank one vector bundle destabilized $k$ times. Each of these has a monoidal inverse given by $((\underline{\RR}^{\oplus k-1} \to \ast),0)$.
\end{remark}

\begin{lemma}\label{lemma. classical vecbun to BO symm mon}
There exists an equivalence of symmetric monoidal $\infty$-categories
	\eqnn
	(\vecbun_\RR^{\dd})^\oplus \simeq \Top_{/BO}^\tensor,
	\qquad
	(\vecbun_\CC^{\dd})^\oplus \simeq \Top_{/BU}^\tensor.
	\eqnd
\end{lemma}

\begin{proof}[Sketch of proof.]
There are several proof methods. Because this is a standard result, we sketch one proof method and omit the details. First, we recall that all $(E \to X) \in \vecbun_\RR$ are assumed to have a CW base, and that every CW complex is paracompact. 

$\vecbun_\RR^\dd$ can be written as the base of a right fibration
	\eqnn
	p: \widetilde{\vecbun_\RR^{\dd}} \to \vecbun_\RR^\dd.
	\eqnd
The vertices of $\widetilde{\vecbun_\RR^{\dd}}$ consist of vector bundles $(E \to X)$ equipped with data of local trivializations subordinate to a locally finite cover\footnote{Here we are using paracompactness.}, along with partitions of unity sufficient to define a map $j: E \to \RR^\infty$ for which $j|_{E_x}$ is a linear injection for all $x \in X$. Higher simplices are encoded by common refinements of such covering data, along with homotopies of trivializations and of the partitions of unity. (In particular, the ability to construct edges using refinements of covers is one utility of the model of $\infty$-categories; it would be more difficult to construct a topologically enriched category incorporating such refinements.) 

The fibers of $p$ are contractible, so $p$ is an equivalence. Modeling $BO$ as a colimit of Grassmannians of planes in $\RR^\infty$, we obtain a functor from $\widetilde{\vecbun_\RR^{\dd}}$ to $\Top_{/BO}^{\tensor}$; on an object, we assign the map taking $x \in X$ to $j(E_x)$. (The basic idea of this construction is age-old; see for example Lemma~5.3 of~\cite{milnor-stasheff}.) The details omitted are the constructions of $\widetilde{\vecbun_\RR^{\dd}}$, the symmetric monoidal $\infty$-category structure on $
	\widetilde{\vecbun_\RR^{\dd}}$ and the symmetric monoidal enrichment of the map to $\Top_{/BO}$. 
\end{proof}

\subsection{Tangential structures respect direct products}

\begin{lemma}\label{lemma. vecbun_CC from liou}
There exists a symmetric monoidal functor
	\eqnn
	(\lioustr)^{\tensor}  \to (\vecbun_\CC)^{\oplus}
	\eqnd
which sends a Liouville sector $X$ to $(TX \to X)$ with some choice of almost-complex structure compatible with $\omega_X$. 
\end{lemma}

\begin{proof}
There is a symmetric monoidal functor from $\lioustr$ to $\vecbun_\RR$ sending any Liouville sector to its tangent bundle. Now note the following two functors to the $\infty$-category of spaces:
	\eqnn
	(\lioustr)^{\op} \to \Top,
	\qquad
	X \mapsto \{\text{$\omega_X$-compatible almost-complex structures $J$}\}
	\eqnd
and
	\eqnn
	\vecbun_\RR^{\op} \to \Top,
	\qquad
	(E \to X) \mapsto \{\text{endomorphisms $J$ of $E$ that square to $-1$.}\}
	\eqnd
The first of these functors sends every $X$ to a contractible space, as is well-known; hence it admits a symmetric monoidal promotion. The second is a lax symmetric monoidal functor. The following diagram commutes up to natural transformation:
	\eqnn
	\xymatrix{
	(\lioustr)^{\op} \ar[dr]_{X \mapsto \{J\}} \ar[rr]^{X \mapsto TX} && \vecbun_\RR^{\op} \ar[dl]^{(E \to X) \mapsto \{J\}} \\
		& \Top
	}
	\eqnd
The natural transformation is the inclusion of $\omega$-compatible $J$ into the space of all $J$, and can obviously be upgraded to a lax monoidal natural transformation.   Because the above functors are (lax) symmetric monoidal, the  fibrations classifying these functors are $\infty$-operads, and the lax symmetric monoidal natural transformation results in a diagram of $\infty$-operads:
	\eqnn
	\xymatrix{
	\widetilde{(\lioustr)}^{\tensor} \ar[d]_{\sim} \ar[r] & \widetilde{\vecbun_\RR}^{\oplus}  \ar[d] \ar[r]^\sim & (\vecbun_\CC)^{\oplus}  \\
	(\lioustr)^{\tensor} \ar[r] & \vecbun_\RR^{\oplus} 
	}
	\eqnd
References for the $\infty$-operadic claims include Example 2.4.2.4,
Proposition 2.4.2.5, and
Remark 2.4.2.7 (to deal with the contravariance of the functors we classify) of~\cite{higher-algebra}.

We note that, as $\infty$-categories, the total space $\widetilde{\vecbun_\RR}$ of the classifying fibration is equivalent to $\vecbun_\CC$. The equivalence extends (as indicated by the rightmost horizontal arrow in the diagram) to an equivalence of symmetric monoidal $\infty$-categories. We also note that the left downward arrow is an equivalence of $\infty$-categories because the fibration has contractible fibers. It is then straightforward to see that all the unlabeled arrows in the diagram are maps of symmetric monoidal $\infty$-categories (not just maps of $\infty$-operads).

The desired symmetric monoidal functor is the composition of an inverse to the left vertical arrow with the top horizontal arrows.
\end{proof}

\begin{lemma}
There exists a symmetric monoidal functor
	\eqn\label{eqn. lioustr to BU spaces}
	(\lioustr)^{\tensor} \to \Top_{/BU}^{\tensor}
	\eqnd
which sends a Liouville sector $X$ to some map $(X \to BU)$ classifying the stable tangent bundle of $X$ (equipped with a compatible almost-complex structure).
\end{lemma}

\begin{proof}
Combine Lemma~\ref{lemma. classical vecbun to BO symm mon} with Lemma~\ref{lemma. vecbun_CC from liou}, composing with the natural functor $(\vecbun_\CC)^\oplus \to (\vecbun_\CC^\dd)^\oplus$ along the way.
\end{proof}

\subsection{Stabilized Liouville sectors with tangential structures}
The functor~\eqref{eqn. lioustr to BU spaces} sends sectorial equivalences to homotopy equivalences, so one obtains a symmetric monoidal map from the localization (Proposition~\ref{prop. localization inherits symmetric monoidal structure}):
	\eqnn
	(\lioustr[\eqs^{-1}])^{\tensor} \to (\Top_{/BU})^{\tensor}.
	\eqnd
Because the object $T^*[0,1]$ is sent to an invertible object in the target\footnote{In fact, the image of $T^*[0,1]$ under this functor is (equivalent to) the unit---it is a/the trivial stable vector bundle on $\ast \simeq T^*[0,1]$.}, we obtain a symmetric monoidal functor
	\eqnn
	(\lioulocal)^{\tensor} := (\lioustr[\eqs^{-1}]^{\dd})^{\tensor} \to \Top_{/BU}^{\tensor}. 
	\eqnd
(This follows from Theorem~\ref{theorem. lioustab is symmetric monoidal}, or from Proposition~\ref{prop. symmetric T}. See also Notation~\ref{notation. lioulocal}.)

At this moment, we may contemplate tangential structures as follows:

\begin{construction}[$\lioulocal_{/(B' \to B)}$]
\label{construction. Liou with structures}
Fix $E_\infty$-spaces $B$ and $B'$, along with $E_\infty$ maps $BU \to B \leftarrow B'$. 

We define the $\infty$-category of stable Liouville sectors equipped with a reduction of induced $\Omega B$-structures to $\Omega B'$-structures as the following pullback $\infty$-category:
	\eqnn
	\xymatrix{
	\lioulocal_{/(B' \to B)} \ar[rr] \ar[d] && \Top_{/(B' \to B)} \ar[d] \\
	\lioulocal \ar[r] & \Top_{/BU} \ar[r] & \Top_{/B}.
	}
	\eqnd
(The bottom horizontal arrow is depicted as a composition for explicitness.)  Here, $\Top_{/(B' \to B)}$ is the $\infty$-category whose objects are topological spaces $X$ equipped with (i) Maps to $B'$ and $B$, and (ii) a homotopy between $X \to B' \to B$ and $X \to B$.

By remembering the $E_\infty$-structure of the maps $BU \to B$ and $B' \to B$, one may promote this pullback diagram to be a pullback of symmetric monoidal $\infty$-categories, thereby endowing $\lioulocal_{/(B' \to B)}$ with a symmetric monoidal structure. (See Recollection~\ref{recollection. slice categories are symmetric monoidal}.)
\end{construction}

\begin{example}[Reductions to $\Omega B'$] 
Let $B = BU$.
Fix an $E_\infty$ map $B' \to BU$ from some $E_\infty$ space $B'$. Then one can contemplate stabilized Liouville sectors with a reduction of the stable complex structure of its tangent bundle to some $\Omega B'$-structure. The associated $\infty$-category of such sectors is the pullback of $\infty$-categories
	\eqnn
	\xymatrix{
	\lioulocal_{/(B' \to BU)} \ar[r] \ar[d] & \Top_{/(B' \to BU)} \ar[d] \\
	\lioulocal \ar[r] & \Top_{/BU}.
	}
	\eqnd
Informally, $(\lioulocal)_{/(B' \to BU)}$ is the $\infty$-category of stabilized Liouville sectors equipped with a reduction of the stable complex structure bundle's structure group to $\Omega B'$. Taking $B' = \ast$ to be the trivial/initial $E_\infty$-space, this is the data of a stable trivialization of the complex tangent bundle. One could also choose $B' = BG$ for some infinite-loop space $G$ (e.g., $GL(\infty)$) in which case objects of $\lioulocal_{/(B' \to BU)}$ are stabilized Liouville sectors with a reduction of the stable complex structure of its tangent bundle to a $G$-structure.
\end{example}

\begin{example}[Trivializations of induced structures] 
This time, fix an $E_\infty$ map $BU \to B$ and let $B' = \ast$. One may then contemplate the following pullback diagram, where the bottom horizontal arrow is depicted as a composition for explicitness:
	\eqnn
	\xymatrix{
	\lioulocal_{/(\ast \to B)} \ar[rr] \ar[d] && \Top_{/(\ast \to B)} \ar[d] \\
	\lioulocal \ar[r] & \Top_{/BU} \ar[r] & \Top_{/B}.
	}
	\eqnd
An object of $(\lioulocal)_{\ast \to B}$ is a stable Liouville sector $X$ equipped with a null-homotopy of the map $X \to B$. For example, when $BU \to BU/BO \simeq B(U/O) =: B$ is chosen to be the map classifying the stable Lagrangian Grassmanian bundle, such a null-homotopy exhibits a stable trivialization of this bundle. 
\end{example}
 
\begin{example}
There is another example relevant in the conjectural construction of (wrapped) Floer theory with various coefficient ring spectra. 

For example, choose a commutative coefficient ring spectrum $R$. Then $B\Pic(R)$---the classifying space of rank-one $R$-gerbes---inherits an action of $B\Pic(\SS)$ by virtue of $\SS$ being the initial ring. 
One can take 
\begin{itemize}
	\item $BU \to B = B^2 \Pic(\SS)$ to be a two-fold delooping of the complex $J$ homomorphism,  and 
	\item $B' = B\Pic(R)/B\Pic(\SS) \to B^2 \Pic(\SS)$ to be the map induced by the $B\Pic(\SS)$-action on $B\Pic(R)$.
\end{itemize}
See, for example, Section~1.3 of~\cite{lurie-rotation}.
\end{example}

\begin{remark}\label{remark. tilde liou}
For any choice of $E_\infty$ maps $BU \to B \leftarrow B'$, the map 
	$
	\lioulocal_{/(B' \to B)} \to \lioulocal
	$
is a right fibration. To put this remark in context, recall that a right fibration is a map that classifies contravariant functor from the base to the $\infty$-category of spaces---see~\cite[Section~2.1]{htt}. In our case, the contravariant functor may be informally describe as sending an object $X \in \lioulocal$ to the space of ways to reduce the stable $\Omega B$ bundle to have structure group $\Omega B'$.
\end{remark}

\subsection{Localization with tangential structures (Proof of Theorem~\ref{maintheorem. B structures})}
There is also the obvious strict analogue of $\lioulocal_{/(B' \to B)}$:

\begin{construction}[$(\lioustrstab)_{/(B' \to B)}$]
\label{construction. Lioustr with structures}
Fix $E_\infty$-spaces $B$ and $B'$, along with $E_\infty$ maps $BU \to B \leftarrow B'$. 
We define the following pullback $\infty$-category:
	\eqnn
	\xymatrix{
	(\lioustrstab)_{/(B' \to B)} \ar[rr] \ar[d] && \Top_{/(B' \to B)} \ar[d] \\
	\lioustrstab \ar[r] & \Top_{/BU} \ar[r] & \Top_{/B}.
	}
	\eqnd
(The bottom horizontal arrow is depicted as a composition for explicitness.)
\end{construction}

Note that $(\lioustrstab)_{/(B' \to B)}$ is typically not a category (in the usual sense), but is an $\infty$-category. Informally, this $\infty$-category only detects (stabilized) strict sectorial embeddings and no topology of the space of sectorial embeddings, but detects the homotopy type of lifts of $B$-structures to $B'$.

\begin{remark}[Symmetric monoidal structure]
 As in Construction~\ref{construction. Liou with structures}, one may promote this pullback diagram to be a pullback of symmetric monoidal $\infty$-categories, endowing $(\lioustrstab)_{/(B' \to B)}$ with a symmetric monoidal structure.
 \end{remark}
 
By the universal property of pullbacks, the functor $\lioustrstab \to \lioulocal$ (which factors the map to $\Top_{/BU}$) induces a functor
	\eqn\label{eqn. localization map with structures}
	(\lioustrstab)_{/(B' \to B)} \to \lioulocal_{/(B' \to B)}.
	\eqnd

\begin{theorem}[Theorem~\ref{maintheorem. B structures}]
\label{theorem. proof of B structure localization}
The map~\eqref{eqn. localization map with structures} is a localization. In fact, letting $\eqs_{/B' \to B} \subset (\lioustrstab)_{/(B' \to B)}$ be the pullback of $\eqs \subset \lioustrstab$, the induced map
	\eqnn
	(\lioustrstab)_{/(B' \to B)}[\eqs_{/B' \to B}^{-1}] \to \lioulocal_{/(B' \to B)}	
	\eqnd 
is an equivalence of $\infty$-categories.
\end{theorem}

\begin{proof}
This follows more generally from the following fact: If $\cC \to \cD$ is a localization, and $\cD' \to \cD$ is a right fibration, then the pullback induces a localization $\cC' \to \cD'$.

Here is a proof. The formation of colimits of $\infty$-categories commutes with pulling back along right fibrations.
So the question immediately reduces to the case where $\cC$ consists of a single morphism and $\cD$ is obtained by
inverting it; this can be modeled by the projection  $\Delta^1 \to \Delta^0$. In this case the right fibration $\cD' \to \cD$ is constant, so the result is trivial.
\end{proof}
\clearpage
\bibliographystyle{plain}
\bibliography{biblio}

\end{document}